\documentclass[12pt]{report}
\usepackage{graphicx}
\usepackage{amsmath,amssymb,xy,array}
\usepackage[T1]{fontenc}
\usepackage{amsfonts}
\usepackage{amsmath}
\usepackage{amssymb}
\usepackage{amsthm}
\usepackage[utf8]{inputenc}
\usepackage{hyperref}
\usepackage{fancyhdr}
\usepackage{tikz}
\usepackage{tikz-cd}
\usepackage{longtable}
\usepackage{geometry}
\usepackage{textcomp}
\usetikzlibrary{trees}
\usepackage{multirow}
\usepackage{youngtab}

\newcommand{\arXiv}[1]{\href{http://arxiv.org/abs/#1}{arXiv:#1}}
\def\bibaut#1{{\sc #1}}

\providecommand{\U}[1]{\protect\rule{.1in}{.1in}}
\providecommand{\U}[1]{\protect\rule{.1in}{.1in}}
\providecommand{\U}[1]{\protect\rule{.1in}{.1in}}
\providecommand{\U}[1]{\protect\rule{.1in}{.1in}}
\providecommand{\U}[1]{\protect\rule{.1in}{.1in}}
\input{xy}
\xyoption{all}
\usepackage{amsfonts}
\usepackage{amssymb,amsthm}
\usepackage{verbatim}
\usepackage{hyperref} 
\usepackage[T1]{fontenc}
\usepackage{amsmath}
\usepackage{enumerate}
\usepackage{tikz}
\usepackage{color}

\newcommand{\C}{\mathbb C}
\newcommand{\G}{\mathbb G}
\renewcommand{\P}{\mathbb P}
\newcommand{\R}{\mathbb R}
\newcommand{\Q}{\mathbb Q}
\newcommand{\Z}{\mathbb Z}
\DeclareMathOperator{\codim}{codim}
\newcommand{\rddots}{\reflectbox{$\ddots$}}
\newcommand{\Spec}{\operatorname{Spec}}
\newcommand{\Aut}{\operatorname{Aut}}
\DeclareMathOperator{\Bl}{Bl}

\DeclareMathOperator{\mult}{mult}

\DeclareMathOperator{\Exc}{Exc}

\DeclareMathOperator{\Sing}{Sing}

\DeclareMathOperator{\Supp}{Supp}

\DeclareMathOperator{\Sec}{Sec}
\DeclareMathOperator{\Stab}{Stab}

\DeclareMathOperator{\Eff}{Eff}
\DeclareMathOperator{\Nef}{Nef}
\DeclareMathOperator{\Mov}{Mov}
\DeclareMathOperator{\Pic}{Pic}

\renewcommand{\sec}{\mathbb{S}ec}
\DeclareMathOperator{\expdim}{expdim}
\DeclareMathOperator{\Sym}{Sym}
\DeclareMathOperator{\Cox}{Cox}
\DeclareMathOperator{\cone}{cone}
\DeclareMathOperator{\Amp}{Amp}
\DeclareMathOperator{\mov}{mov}
\newcommand{\binomm}[2]{\scriptstyle \binom{#1}{#2} \displaystyle}

\newcommand{\QED}{\ifhmode\unskip\nobreak\fi\quad {\rm Q.E.D.}} 

\newcommand\map{\dasharrow}

\newcommand{\f}{\varphi}

\newcommand{\N}{\mathbb{N}}
 
\renewcommand{\P}{\mathbb{P}}

\DeclareMathOperator{\sign}{sign}

\newtheorem{Theorem}{Theorem}[section]

\newtheorem{Conjecture}{Conjecture}

\newtheorem{Lemma}[Theorem]{Lemma}
\newtheorem{Proposition}[Theorem]{Proposition}
\newtheorem{Problem}{Problem}
\newtheorem{Corollary}[Theorem]{Corollary}

\theoremstyle{definition}

\newtheorem{Definition}[Theorem]{Definition}
\newtheorem{Remark}[Theorem]{Remark}

\newtheorem{Example}[Theorem]{Example}
\newtheorem{Notation}[Theorem]{Notation}

\hypersetup{pdfpagemode=UseNone}
\hypersetup{pdfstartview=FitH}

\begin{document}

\title{
Projective and birational geometry of Grassmannians and other special varieties
}
\author{Rick Antonio Rischter}
\date{\today}

%
%
%
%
%
%
%
%
%
%
%
%


    \begin{center}
    	\thispagestyle{empty}
    
        \Large
        Instituto Nacional de Matem\'atica Pura e Aplicada
        \vspace*{2.5cm}
        
        \Huge
        \textbf{Projective and birational geometry\\ of Grassmannians and\\ other special varieties}
        
        \Large
        
        \vspace{1.5cm}
        
        \textbf{Rick Rischter}\\

        \vspace{2.5cm}

        Advisor: Carolina Araujo\\
        Co-advisor: Alex Massarenti
        
        \vspace*{2.5cm}
        
        \large
Thesis presented to the Post-graduate Program in 
Mathematics at\\
Instituto Nacional de Matem\'atica Pura e Aplicada as
partial fulfillment \\of
the requirements for the degree of Doctor in Mathematics.\\

        \vfill
                
        \vspace{2cm}

        \large
        Rio de Janeiro\\
        \today
        
    \end{center}

\chapter*{Abstract}
    	\thispagestyle{empty}
The subject of this thesis is complex algebraic geometry.

In the first part of the thesis, we study a classical invariant of projective varieties, the ``secant defectivity''. Given a projective variety $X$, the secant variety of $X$ is defined as the union of all secant lines through two points of $X$. More generally, the $h-$secant variety of $X$ is the union of all linear spaces of dimension $h-1$ intersecting $X$ in at least $h$ points. The expected dimension of
the $h-$secant variety of $X$ depends only on the dimension of $X$, and the actual dimension
coincides with the expected one in most cases. Varieties whose $h-$secant variety has dimension
less than expected are special. They are called ``secant defective varieties''. Their classification is a classical and generally difficult problem in algebraic geoemetry. In this thesis we provide a new method to approach this problem using osculating spaces. We then apply this method to produce new results about secant defectivity of Grassmannians and Segre-Veronese varieties.

The second part is devoted to modern algebraic geometry. We study the birational
geometry of blow-ups of Grassmannians at points. To describe the birational geometry of a
variety, that is, to describe all the maps it admits, is a difficult problem in general. For a special class of varieties, called Mori dream spaces (MDS), the birational geometry is well behaved and can be codified in a finite combinatorial data. In this thesis we investigate when the blow-up of Grassmannians at general points is MDS, and describe their birational geometry in
some special cases.

\vspace{1cm}
\textbf{Keywords: } Grassmannians, Segre-Veronese varieties, osculating spaces, secant varieties,  secant defect, degenerations of rational maps, birational geometry, Mori dream spaces, Fano varieties, weak Fano varieties, spherical varieties, rational curves, blow-up, Mori chamber decomposition.

\chapter*{Resumo}
    	\thispagestyle{empty}
O tópico desta tese \'{e} geometria alg\'{e}brica complexa.

Na primeira parte desta tese, estudamos um invariante cl\'{a}ssico de variedades projetivas, a ``defeituosidade secante''. Dada uma variedade projetiva $X$, a variedade secante de $X$ \'{e} definida como a união de todas retas secantes por dois pontos de $X$. Mais geralmente, a variedade $h-$secante de $X$ \'{e} a união de todos os espaços lineares de dimens\~{a}o $h-1$ que intersectam $X$ em pelo menos $h$ pontos. A dimens\~{a}o esperada da variedade $h-$secante de $X$ depende apenas da dimens\~{a}o de $X$, e a sua dimens\~{a}o de fato coincide com a esperada na maioria dos casos. Variedades cuja variedade $h-$secante tem dimens\~{a}o menor que a esperada s\~{a}o especiais, elas s\~{a}o chamadas ``variedades secante defeituosas''. A sua classifica\c{c}\~{a}o \'{e} um problema cl\'{a}ssico e em geral dif\'{i}cil em geometria alg\'{e}brica. Nesta tese fornecemos um novo m\'{e}todo para abordar este problema usando espaços osculadores. Depois aplicamos este m\'{e}todo para produzir novos resultados a respeito de defeituosidade secante de Grassmannianas e variedades de Segre-Veronese.

A segunda parte \'{e} dedicada a geometria alg\'{e}brica moderna. Estudamos a geoemtria birracional de explosões de Grassmannianas em pontos. Descrever a geometria birracional de uma variedade, isto \'{e}, descrever todos os morfismos que ela admite, \'{e} um problema dif\'{i}cil em geral. Para uma classe especial de variedades, chamados Mori dream spaces (MDS), a geometria birracional \'{e} bem comportada e pode ser decodificada em uma informa\c{c}\~{a}o combinatória finita. Nesta tese investigamos quando explosões de Grassmannianas em pontos gerais s\~{a}o MDS, e descrevemos sua geometria birracional em casos especiais.

\vspace{1cm}
\textbf{Palavras-chave: } Grassmannianas,
variedades de Segre-Veronese, espaços osculadores, variedades secantes, defeitos secantes, degenerações de aplicações racionais,
geometria biracional, Mori dream spaces, variedades de Fano, variedades weak Fano, variedades esféricas, curvas racionais, explosões, decomposição em câmaras de Mori.    	
\newpage
\normalsize
\vspace*{5cm}
\begin{flushright}
    	\thispagestyle{empty}
Dedicated to my beloved wife Talita.
\end{flushright}
\newpage

\newgeometry{inner=3cm,outer=2.8cm,bottom=1cm,
top=1.7cm}
\chapter*{Agradecimentos}
\thispagestyle{empty}
\vspace{-0.6cm}
Antes de mais nada eu gostaria de agradecer muito à minha orientadora Carolina Araujo que me ajudou desde antes de eu começar oficialmente o doutorado. 
Agradeço também pelos cursos de geometria algébrica.
Sua orientação sempre me fez ver mais claramente os problemas mais difíceis, foi um privilégio trabalhar com ela. Além disso, agradeço a Carolina por ter me dado o interessante problema que gerou toda esta tese.

Agradeço também ao meu coorientador Alex Massarenti, especialmente por ter me ajudado muito no último ano de doutorado. Também sou agradecido por ter me mostrado como pensar na geometria algébrica da maneira italiana.

Sem a ajuda dos meus dois orientadores esta tese não seria possível.

Agradeço à minha amada esposa e companheira Talita, que me apoiou em toda essa longa estrada do doutorado, sem seu apoio incondicional eu não chegaria são ao fim do doutorado.

Agradeço à minha mãe Rosinete Rischter e ao meu irmão Fred Rischter.

Agradeço aos professores da UFES por terem alimentado meu interesse pela álgebra e matemática em geral, especialmente Alan, Florencio, Gilvan e Ricardo. Também agradeço aos meus colegas de graduação da UFES, especialmente ao Karlo. O convívio com eles me ajudou a ser a pessoa que sou hoje.

Agradeço aos professores do Instituto de Matemática e Computação da UNIFEI por apoiarem meu afastamento temporário das atividades didáticas da UNIFEI. Sem isto o doutorado seria muito mais difícil.

Agradeço aos meus colegas de mestrado pela excelente turma e grupo de estudos que tínhamos, especialmente Flaviano 
(in memoriam), Renan  e Sergio.

Agradeço aos amigos do doutorado do IMPA, especialmente ao Wállace que me ajudou mais de uma vez com a parte algébrica da geometria algébrica.

Agradeço muito ao Grupo do Café Especial pelo ambiente amigável e relaxante. Cada reunião do grupo gerava um assunto mais interessante e engraçado que a anterior.
Agradeço em especial a Caio Cesar por ser um ótimo companheiro de estudo e por ter me ajudado a estudar para o exame de qualificação, além de revisar versões preliminares desta tese.

Agradaço Enrique Arrondo, Cinzia Casagrande, Ana Maria Castravet,  Ciro Ciliberto,  Izzet Coskun, Olivier Debarre, Igor Dolgachev,  Letterio Gatto, Cesar Huerta e Damiano Testa  por produtivas conversas matemáticas.

Agradeço aos professores do IMPA pelas disciplinas e pelas ajudas extras, especialmente Eduardo Esteves, Elon Lages, Oliver Lorscheid, Reimundo Heluani e Karl St\"ohr.

Agradeço \`a banca de defesa, Alessio Corti, Letterio Gatto, Jorge Pereira e Karl St\"ohr, pela leitura cuidadosa da tese.

Agradeço aos funcionários do IMPA pelo excelente ambiente de trabalho proporcionado por eles, e ao CNPq pelo apoio financeiro.
\newpage
\restoregeometry

\vspace*{5cm}
    	\thispagestyle{empty}
As long as algebra and geometry have been separated, their progress have been slow and their uses limited, but when these two sciences have been united, they have lent each mutual forces, and have marched together towards perfection.
\begin{flushright}
\textit{Joseph-Louis Lagrange}
\end{flushright}
\newpage

\tableofcontents

\chapter*{Introduction}
\addcontentsline{toc}{chapter}{Introduction}

\hspace{0.5cm}

Algebraic geometry is the study of varieties given by zeros of homogeneous polynomials in projective spaces. These varieties appear naturally in mathematics and in other sciences. For this reason algebraic geometry plays a central role in mathematics.

During the nineteenth century, algebraic geometry experienced great progress with the Italian school. Among the most studied varieties were the Grassmannians and Segre-Veronese varieties.

Grassmannians parametrize linear spaces inside the projective space. They are among the most studied varieties in algebraic geometry, as they appear frequently not only in algebraic geometry but also in other areas of mathematics, in theoretical physics and in classical mechanics.

Veronese varieties are embeddings  of projective spaces of arbitrary degrees. A generalization of these are the Segre-Veronese varieties, which are embeddings of products of projective spaces of arbitrary multi-degrees. 

A problem that dates back to the Italian school is that of ``secant defectivity''. Given a projective variety $X$, the secant variety of $X$ is defined as the union of all secant lines through two points of $X$. More generally, the $h-$secant variety of $X$ is the union of all linear spaces of dimension $h-1$ intersecting $X$ in at least $h$ points. The expected dimension of
the $h-$secant variety of $X$ depends only on the dimension of $X$, and the actual dimension
coincides with the expected one in most cases. Varieties whose $h-$secant variety has dimension
less than expected are special. They are called ``secant defective varieties''. Their classification is a classical and generally difficult problem in algebraic geoemetry.
Defective surfaces and $3$-folds are already classified but for bigger dimension only special results are known.

Grassmannian and Segre-Veronese varieties can be used to parametrize decomposable tensors, and its secant varieties parametrize naturally sums of rank one tensors. Therefore, by studying secant varieties of Grassmannian and Segre-Veronese varieties one can understand better when a general tensor can be written as a sum of a fixed given number of decomposable tensors.

Defectivity of Veronese varieties was fully understood only in 1995 with the work of Alexander and Hirshowitz. Defectivity of Grassmannians and Segre-Veronese varieties is  understood only in some special cases so far.

The first part of this thesis is deveoted to provide a new method to approach the problem of secant defectivity using osculating spaces, see Theorem  \ref{lemmadefectsviaosculating}. For a smooth point $x\in X\subset\mathbb{P}^N$, the \textit{$k$-osculating space} $T_x^{k}X$ of $X$ at $x$ is  the smallest linear subspace 
where $X$ can be locally approximated up to order $k$ at $x$. We then apply this method to produce new results about secant defectivity of Grassmannians, see Chapter \ref{cap3}, and Segre-Veronese varieties, see Chapter \ref{cap4}.

Most of the content of this first part appears in the pre-prints:
\begin{itemize}
\item 
\uppercase{A. Massarenti, R. Rischter}, \textit{Non-secant defectivity via osculating projections}, 2016, \arXiv{1610.09332v1}.
\item
\uppercase{C. Araujo, A. Massarenti, R. Rischter}, \textit{On non-secant defectivity of Segre-Veronese varieties}, 2016, \arXiv{1611.01674}.
\end{itemize}

\vspace{1cm}
Modern algebraic geometry is more interested in intrinsic invariants of a variety, rather than  those depending on a fixed embbeding. For instance one usually wants to describe the birational geometry of a
variety, that is, to describe all the maps it admits. This is a difficult problem in general. For a special class of varieties, called Mori dream spaces (MDS), the birational geometry is well behaved and can be codified in a finite combinatorial data, namely its cone of effective divisors and a chamber decomposition on it, see Section \ref{sec51}.

Some examples of MDS are projective spaces, products of projective spaces, Grassmannians, and more generally spherical varieties.  A variety with an action of a reductive group is called spherical when it has a dense orbit by the action of a Borel subgroup. An important problem is to find new classes of MDS. For example, it is natural to ask when the blow-up of a MDS is a MDS. The blow-up is a fundamental operation in birational geometry.

It was completely determined in the work of Castravet and Tevelev and in the work of Mukai, when blow-ups $\Bl_{p_1,\dots,p_k}\P^n$ of the projective space $\P^n$ at $k$ general points $p_1,\dots,p_k$ are MDS. 
They showed that $\Bl_{p_1,\dots,p_k}\P^n, n\geq 5,$ is a MDS if and only if $n\leq k+3.$ They also determined completely the combinatorial data which describes the birational geometry of $\Bl_{p_1,\dots,p_k}\P^n$ when it is a MDS.

The second part of this thesis is devoted to investigate when the blow-up $\G(r,n)_k$ of a Grassmannian at $k$ general points is MDS, and describe its birational geometry in some special cases. 
We classify when $\G(r,n)_k$ is a spherical variety, see Theorem \ref{G(r,n)kspherical}.
We also classify when $\G(r,n)_k$ is a weak Fano variety, see Proposition \ref{G(r,n)_k weak Fano}. In this way we obtain new examples of MDS, see Theorem \ref{G(r,n)k MDS weak}.
In the special case $\G(1,n)_1,$ we are able to completely describe the birational geometry, see Theorem \ref{MCD G(1,n)}. 

\newpage

Throughout this thesis a variety is an integral, separated scheme of finite type over an algebraically closed field, and we work over the field of complex numbers. By a point we mean a closed point. We say that a closed point of a variety is a general point if it lies in the complement of a proper closed subscheme of it. The notion of general point depends upon the choice of
the proper closed subschemes to be avoided and should be clear from the context.

In Part \ref{part1} of this thesis the result that give us the main tool is Theorem \ref{lemmadefectsviaosculating}. In order to prove it we used Propositions \ref{cc} and \ref{p2}, the former need that the base field is the complex numbers and the latter need at least a infinite base field. Therefore we can not avoid use complex numbers.

In Part \ref{part2} we used Contraction Lemma,
\cite[Theorem 7.39]{De01}, which needs characteristic zero. Therefore we choose to fix throughout the thesis the complex numbers as base field.

\part{Secant defectivity via osculating projections}\label{part1}
\chapter{Overview}\label{intro1}
Secant varieties are classical objects in algebraic geometry.
The \textit{$h$-secant variety} $\sec_{h}(X)$ of a non-degenerate $n$-dimensional variety $X\subset\mathbb{P}^N$ is 
$$
\sec_h(X)=\overline{\bigcup_{p_1,\dots,p_h\in X}
\left\langle
p_1,\dots,p_h
\right\rangle,
}
$$
the Zariski closure 
of the union of all linear spaces spanned by collections of $h$ points of $X$. 
Secant varieties are central objects in both classical algebraic geometry \cite{CC01}, \cite{Za}, and applied mathematics \cite{La12}, \cite{LM}, \cite{LO}, \cite{MR}.
The \textit{expected dimension} of $\sec_{h}(X)$ is
$$
\expdim(\sec_{h}(X)):= \min\{nh+h-1,N\}.
$$
The actual dimension of $\sec_{h}(X)$ may be smaller than the expected one. This happens when there are infinitely many $(h-1)$-planes 
$h$-secant to $X$ passing trough a general point of  $\sec_{h}(X)$. 
Following \cite{Za}, we say that $X$ is \textit{$h$-defective} if 
$$
\dim(\sec_{h}(X)) < \expdim(\sec_{h}(X)).
$$

Most projective varieties are not defective. For instance, curves and hypersurfaces are not defective, on the other hand defective surfaces and threefolds have been classified \cite{Te21},\cite{CC03}.

Grassmannians together with Veronese and Segre varieties form the triad of varieties parametrizing rank one tensors. Hence, a general point of their $h$-secant variety corresponds to a tensor of a given rank depending on $h$. For this reason, secant varieties of Grassmannians, Veroneses and Segres are particularly interesting in problems of tensor decomposition \cite{CM}, \cite{CGLM}, \cite{La12}, \cite{Me06}, \cite{Me09}, \cite{GM16}.

Furthermore, secant varieties have been widely used to construct and study moduli spaces for all  possible additive decompositions of a general tensor into a given number of rank one tensors \cite{Do04}, \cite{DK93}, \cite{Ma16}, \cite{MM13}, \cite{RS00}, \cite{TZ11}, \cite{BGI11}.

The problem of determining the actual dimension of secant varieties, and its relation with the dimension of certain linear systems of hypersurfaces with double points, have a very long history in algebraic geometry, and can be traced back to the Italian school \cite{Ca37}, \cite{Sc08}, \cite{Se01}, \cite{Te11}.

Since then the geometry of secant varieties has been studied and used by many authors in various contexts \cite{CC01},\cite{CR06},\cite{IR08},  and the problem of secant defectivity has been widely studied for Veroneses, Segres and Grassmannians \cite{AH95}, \cite{AB13}, \cite{AOP09a}, \cite{AOP09b}, \cite{Bo13}, \cite{CGG03},\cite{CGG05}, \cite{CGG11}, \cite{LP13}, \cite{BBC12}, \cite{BCC11}.

In Chapter \ref{cap2} we present the following theorem, which is our main tool to study secant defectivity.

\begin{Theorem}
Let $X\subset \P^N$ be a projective variety 
having $m$-osculating regularity and strong $2$-osculating regularity.
Let $k_1,\dots,k_l\geq 1$ be integers such that the general osculating projection 
$\Pi_{T^{k_1,\dots,k_l}_{p_1,\dots,p_l}}$ is generically finite.
Then $X$ is not $h$-defective for $h\leq \displaystyle\sum_{j=1}^{l}h_m(k_j)+1,$ where
$h_m$ is as in Definiton \ref{defhowmanytangent}.
\end{Theorem}
 
Now we explain the strategy of its proof.

Given a non-degenerate $n$-dimensional variety $X\subset\mathbb{P}^N$, and general points  $x_1,\dots,$ $x_h\in X$, consider the 
linear projection with center $\left\langle T_{x_1}X,\dots,T_{x_h}X\right\rangle$,
$$
\tau_{X,h}:X\subseteq\mathbb{P}^N\dasharrow\mathbb{P}^{N_h}.
$$
By \cite[Proposition 3.5]{CC01}, if $\tau_{X,h}$ is generically finite then $X$ is not $(h+1)$-defective.
In general, however, it is hard to control the dimension of the fibers of the tangential projections $\tau_{X,h}$ as $h$ gets larger. 
We develop a new strategy, based on the more general  \emph{osculating projections} instead of just tangential projections.
For a smooth point $x\in X\subset\mathbb{P}^N$, the \textit{$k$-osculating space} $T_x^{k}X$ of $X$ at $x$ is roughly the smaller linear subspace 
where $X$ can be locally approximated up to order $k$ at $x$ (see Section \ref{osculating}).
Given $x_1,\dots, x_l\in X$ general points, we consider the linear projection with center $\left\langle T_{x_1}^{k_1}X,\dots, T_{x_l}^{k_l}X\right\rangle$,
$$
\Pi_{T^{k_1,\dots,k_l}_{x_1,\dots,x_l}}:X\subset\mathbb{P}^N\dasharrow\mathbb{P}^{N_{k_1,\dots,k_l}},
$$
and call it a \textit{$(k_1,\dots,k_l)$-osculating projection}.
Under suitable conditions, one can degenerate the linear span of several tangent spaces $T_{x_i}X$ into a subspace contained in a
single osculating space $T_x^{k}X$. 
We formalize this in Section \ref{degsection}.
So the tangential projections $\tau_{X,h}$ degenerates to a linear projection with center contained in the linear span of osculating spaces,
$\left\langle T_{p_1}^{k_1}X,\dots, T_{p_l}^{k_l}X\right\rangle$.
If $\Pi_{T^{k_1,\dots,k_l}_{p_1,\dots,p_l}}$ is generically finite, then $\tau_{X,h}$ is also generically finite, and one concludes that $X$ is not $(h+1)$-defective.
The advantage of this approach is that one has to consider osculating spaces at much less points than $h$, 
allowing to control the dimension of the fibers of the projection.

We would like to mention that, as remarked by  Ciliberto and Russo in \cite{CR06}, the idea that the behavior of osculating projections reflects the geometry of the variety itself was already present in the work of Castelnuovo \cite[Pages 186-188]{Ca37}.

Finally, we would like to stress that the machinery introduced in this thesis could be used to produce bounds, for the non secant defectivity of an arbitrary irreducible projective variety, once we know how its osculating spaces behave in families and when the projections from them are generically finite. 

In Chapter \ref{cap3} we investigate secant defectivity for Grassmannian varieties. 

It is well-known that the secant variety $\sec_h(\G(1,n))$, which is the locus of skew-symmetric matrices of rank at most $2h$, is almost always defective. Therefore, throughout the paper we assume $r\geq 2$. Only four defective cases are known
then, and we have the following conjecture proposed by Baur, Draisma, and de Graaf.

\begin{Conjecture}\cite[Conjecture 4.1]{BDdG07}
If $r\geq 2$ then $\G(r,n)$ is not $h$-defective with the following exceptions:
$$(r,n,h)\in\{(2,7,3),(3,8,3),(3,8,4),(2,9,4)\}.$$
\end{Conjecture}
In \cite{CGG05} Catalisano, Geramita, and  Gimigliano gave explicit bounds on $(r,n,h)$ for $\G(r,n)$ not to be $h$-defective. Later, in \cite{AOP09b} Abo, Ottaviani, and  Peterson, improved these bounds, and showed that the conjecture is true for $h\leq 6$. Finally, in \cite{Bo13} Boralevi further improved this result by proving the conjecture for $h\leq 12$.

To the best of our knowledge, the best asymptotic bound for $\sec_h(\G(r,n))$ to have expected dimension was obtained by Abo, Ottaviani, and  Peterson using monomial techniques.
\begin{Theorem}\cite[Theorem 3.3]{AOP09b}
If $r\geq 2$ and 
$$h\leq \frac{n-r}{3}+1$$
then $\sec_h(\G(r,n))$ has the expected dimension.
\end{Theorem}

As a direct consequence of our main results in Theorem \ref{maingrass} we get the following.

\begin{Theorem}
Assume that $r\geq 2$, set 
$$\alpha:=\left\lfloor \dfrac{n+1}{r+1} \right\rfloor$$
and write $r = 2^{\lambda_1}+\dots+2^{\lambda_s}+\varepsilon$, with $\lambda_1 > \lambda_2 >\dots >\lambda_s\geq 1$, $\varepsilon\in\{0,1\}$. If either
\begin{itemize}
\item[-] $h\leq (\alpha-1)(\alpha^{\lambda_1-1}+\dots+\alpha^{\lambda_s-1})+1$ or
\item[-] $n\geq r^2+3r+1$ and $h\leq \alpha^{\lambda_1}+\dots+\alpha^{\lambda_s}+1$
\end{itemize}
then $\G(r,n)$ is not $h$-defective.
\end{Theorem}
Note that the bounds in our main result gives that asymptotically the Grassmannian $\G(r,n)$ is not $(\frac{n+1}{r+1})^{\lfloor\log_2(r)\rfloor}$-defective, while \cite[Theorem 3.3]{AOP09b} yields that $\G(r,n)$ is not $\frac{n}{3}$-defective. In Subsection \ref{uglymath} we show that Theorem \ref{maingrass} improves \cite[Theorem 3.3]{AOP09b} for any $r\geq 4$. However, Abo,  Ottaviani, and Peterson in \cite{AOP09b} gave a much better bound, going with $n^2$, in the case $r=2$. 

In order to prove our results we describe osculating spaces of Grassmannians in Section \ref{osculating spaces grass}.
In Section \ref{projoscgrass} we give conditions ensuring that osculating projections of the Grassmannian are birational. In Section \ref{degtanosc} we show how osculating spaces of Grassmannians degenerate in families.

In Chapter \ref{cap4} we push forward our techniques to investigate secant defectivity of Segre-Veronese varieties. 

The problem of secant defectivity for  Veronese varieties was completely solved in \cite{AH95}. 
In that paper, Alexander and  Hirshowitz showed that, except for the degree $2$ Veronese embedding, which is almost always defective, 
the degree $d$ Veronese embedding of $\mathbb{P}^n$ is not $h$-defective except in the following cases: 
$$
(d,n,h)\in\{(4,2,5),(4,3,9),(3,4,7),(4,4,14)\}.
$$

For Segre varieties, very little is known. 
Segre products of two factors $\mathbb{P}^{n_1}\times\mathbb{P}^{n_2}\subset \mathbb{P}^{^{n_1n_2+n_1+n_2}}$ are almost always defective.
For Segre products  $\mathbb{P}^{1}\times\dots\times\mathbb{P}^{1}\subset\mathbb{P}^N$, the problem was completely 
settled in \cite{CGG11}.
In general, $h$-defectivity of Segre products $\mathbb{P}^{n_1}\times\dots\times\mathbb{P}^{n_r}\subset\mathbb{P}^N$ were classified only for $h\leq 6.$
(\cite{AOP09a}).

Now, let us consider Segre-Veronese varieties. These are products $\mathbb{P}^{n_1}\times\dots\times\mathbb{P}^{n_r}$ embedded by the complete 
linear system  $\big|\mathcal{O}_{\mathbb{P}^{n_1}\times\dots\times\mathbb{P}^{n_r}}(d_1,\dots, d_r)\big|$, $d_i>0$.
The problem  of secant defectivity  for  Segre-Veronese varieties has been solved in some very special cases, mostly for products of few factors 
\cite{CGG03}, \cite{AB09}, \cite{Ab10}, \cite{BCC11}, \cite{AB12}, \cite{BBC12}, \cite{AB13}.
Secant defective Segre-Veronese products  $\mathbb{P}^{1}\times\dots\times\mathbb{P}^{1}$, with arbitrary number of factor and degrees, 
were classified in \cite{LP13}. 
It is known that Segre-Veronese varieties are not h-defective for small values of $h$ (\cite[Proposition 3.2]{CGG03}):
except for the Segre product $\mathbb{P}^{1}\times\mathbb{P}^{1}\subset\mathbb{P}^3$, Segre-Veronese varieties 
$\mathbb{P}^{n_1}\times\dots\times\mathbb{P}^{n_r}$ are never $h$-defective for $h\leq \min\{n_i\}+1$.
In this thesis we improve this bound by taking into account degrees $d_1,\dots, d_r$ of the embedding.
We show that, asymptotically, Segre-Veronese varieties are not $h$-defective for $h\leq (\min\{n_i\})^{\lfloor \log_2(d-1)\rfloor}$,
where $d=d_1+\dots+ d_r$.
More precisely, our main result in Theorem~\ref{mainsv} can be rephrased as follows.

\begin{Theorem}\label{main_intro}
Let $\pmb{n}=(n_1,\dots,n_r)$ and $\pmb{d} = (d_1,\dots,d_r)$ be two $r$-tuples of positive integers, with $n_1\leq \dots \leq n_r$ and $d=d_1+\dots+d_r\geq 3$. 
Let $SV^{\pmb n}_{\pmb d}\subset\mathbb{P}^N$ be the product  $\mathbb{P}^{n_1}\times\dots\times\mathbb{P}^{n_r}$ embedded by the complete 
linear system  $\big|\mathcal{O}_{\mathbb{P}^{n_1}\times\dots\times\mathbb{P}^{n_r}}(d_1,\dots, d_r)\big|$. 
Write 
$$
d-1 = 2^{\lambda_1}+\dots+2^{\lambda_s}+\epsilon,
$$ 
with $\lambda_1 > \lambda_2 >\dots >\lambda_s\geq 1$, $\epsilon\in\{0,1\}$. 
Then $SV^{\pmb n}_{\pmb d}$ is not $(h+1)$-defective for 
$$
h\leq n_1((n_1+1)^{\lambda_1-1}+\dots + (n_1+1)^{\lambda_s-1})+1.
$$
\end{Theorem}

In order to prove our result we proceed as in the case of Grassmannians. In Section \ref{oscspaces} we describe osculating spaces of Segre-Veronese varieties. Then, in Section \ref{oscproj} we give conditions ensuring that osculating projections are birational. In Section \ref{degtanoscsv} we show how osculating spaces of Segre-Veronese varieties degenerate in families.

Most of the content of this first part appears in the pre-prints:
\begin{itemize}
\item 
\uppercase{Alex Massarenti, Rick Rischter}, \textit{Non-secant defectivity via osculating projections}, 2016, \arXiv{1610.09332v1}.
\item
\uppercase{Carolina Araujo, Alex Massarenti, Rick Rischter}, \textit{On non-secant defectivity of Segre-Veronese varieties}, 2016, \arXiv{1611.01674}.
\end{itemize}
\chapter{Secant defectivity via osculating projections}
\label{cap2}
In the present chapter we develop our method to study secant defectivity using osculating spaces.

In Section \ref{mainsection} we prove our main theorem concerning secant defects, Theorem \ref{lemmadefectsviaosculating}. In order to do this we make the proper definitions and preliminary results in the first three sections.
In Section \ref{secant} we recall the notions of secant varieties, secant defectivity and secant defect.
In Section \ref{osculating} we define osculating spaces and osculating projections.
In section \ref{degsection} we treat degenerations of osculating projections.

\section{Secant varieties}\label{secant}
We refer to \cite{Ru03} for a comprehensive survey on the subject of secant varieties.

Let $X\subseteq\P^N$ be a non-degenerate variety of dimension $n$ and let
$$\Gamma_h(X)\subseteq X\times \dots \times X\times\G(h-1,N)$$
be the closure of the graph of the rational map
$$\alpha: X\times  \dots \times X \dasharrow \G(h-1,N),$$
taking $h \leq N-n+1$ general points to their linear span $\langle x_1, \dots , x_{h}\rangle$.\\ Let $\pi_2:\Gamma_h(X)\to\G(h-1,N)$ be the natural projection. We denote
$$\mathcal{S}_h(X):=\pi_2(\Gamma_h(X))\subseteq\G(h-1,N).$$
Observe that $\Gamma_h(X)$ and $\mathcal{S}_h(X)$ are both irreducible of dimension $hn$. Finally, let
$$\mathcal{I}_h=\{(x,\Lambda) \: | \: x\in \Lambda\} \subseteq\P^N\times\G(h-1,N)$$
with natural projections $\pi_h$ and $\psi_h$ onto the factors. Furthermore, observe that $\psi_h:\mathcal{I}_h\to\G(h-1,N)$ is a $\P^{h-1}$-bundle on $\G(h-1,N)$.

\begin{Definition} Let $X\subseteq\P^N$ be a non-degenerate variety. The {\it abstract $h$-secant variety} is the irreducible variety
$$\Sec_{h}(X):=(\psi_h)^{-1}(\mathcal{S}_h(X))\subseteq \mathcal{I}_h.$$
The {\it $h$-secant variety} is
$$\sec_{h}(X):=\pi_h(\Sec_{h}(X))\subseteq\P^N.$$
It immediately follows that $\Sec_{h}(X)$ is a $(hn+h-1)$-dimensional variety with a $\P^{h-1}$-bundle structure over $\mathcal{S}_h(X)$. We say that $X$ is \textit{$h$-defective} if $$\dim\sec_{h}(X)<\min\{\dim\Sec_{h}(X),N\}.$$ 
The number 
$$\delta_h(X) = \min\{\dim\Sec_{h}(X),N\}-\dim\sec_{h}(X)$$ 
is called the \textit{$h$-defect} of $X$. When $X$ is $h-$defective and $(h+1)n+h<N$, then $X$ is $(h+1)-$defective. We say that  $X$ is defective if it is $h-$defective for some $h.$
\end{Definition}

\begin{Example}
If $X\subseteq \P^N$  is a non-degenerate hypersurface then $\sec_{2}(X)=\P^N.$
If $X\subseteq \P^N$  is a curve then $\dim \sec_{h}(X)=\min\{2h-1,N\}$ for any $h.$
In both cases $X$ is not defective.

If $X\subseteq \P^N, N\geq 3,$ is a non-degenerate surface, then $3\leq \dim \sec_{2}(X)\leq 5.$
It can be shown that $\dim \sec_{2}(X)=3$ if and only if $N=3,$ that is, $X$ is a hypersurface. Most of the surfaces have $\sec_{2}(X)$ with the expected dimension $5.$ The surfaces such that
$\dim \sec_{2}(X)=4$ have been classified a century ago by Terracini \cite{Te21}.
\end{Example}

\begin{Example}
Let $\nu_2^2:\mathbb{P}^2\rightarrow\mathbb{P}^{5}$ be the $2$-Veronese embedding of $\mathbb{P}^2$ and $X = V^2_2\subseteq\mathbb{P}^{5}$ the corresponding Veronese variety. Interpreting this as 
\begin{align*}
\nu_2^2:\mathbb{C}[x,y,z]_1\to&\mathbb{\C}[x,y,z]_2\\
[L]\mapsto &[L^2]
\end{align*}
we can see elements of $X$ as symmetric $3\times 3$ matrices of rank one. Then $\sec_{2}(X)$ can be seen as symmetric $3\times 3$ matrices of rank at most two, therefore it is a hypersurface of degree three. Thus $X$ is $2-$defective.
\end{Example}

To describe precisely $\sec_h X$ for a given projective variety $X$ is in general quite difficult, in order to study secant defectivity it is better to use Terracini's Lemma:

\begin{Theorem}\cite{Te11}\label{terracini}
Let $X\subseteq\mathbb{P}^N$ be a non-degenerate variety over a field of characteristic zero. Let $p\in\sec_h(X)$ be a general point lying in the linear span of $p_1,...,p_h\in X$. Then
$$T_p\sec_h(X) = \left\langle {T}_{p_1}X,...,{T}_{p_h}X\right\rangle.$$
\end{Theorem} 

\begin{Example}Consider again the Veronese surface $X\subseteq \P^5.$ We will show that given two general points $p_1,p_2\in X$ we have that $T_{p_1} X\cap T_{p_2} X$ is not empty, and therefore $X$ is $2-$defective by Terracini's Lemma \ref{terracini}.

Let $p_1,p_2\in X$ be general points. Consider the preimages 
$$q_1=(\nu_2^2)^{-1}(p_1),q_2=(\nu_2^2)^{-1}(p_2)\in \P^2$$ of $p_1,p_2,$ and 
the line $$l=\overline{q_1 q_2}\subseteq \P^2$$ connecting them. The image of this line by $\nu_2^2$ is a conic $\nu_2^2(l)=c\subseteq X.$ Call $v_1=T_{p_1}c,v_2=T_{p_2}c$ the tangent lines of $c$ at $p_1$ and $p_2.$
Since $c$ is a plane curve $v_1\cap v_2=\{p\}$ is not empty. But then
$$T_{p_1} X\cap T_{p_2} X\supset T_{p_1} c\cap T_{p_2} c=\{p\}\neq \varnothing.$$
\end{Example}

There is another method to approach this problem, developed by Chiantini and Ciliberto in \cite{CC01}.
Let $p_1,\dots,p_h\in X\subseteq\mathbb{P}^N$ be general points with tangent spaces $T_{p_i}X .$
We call the linear projection
$$\tau_{X,h}:X\subseteq\mathbb{P}^N\dasharrow\mathbb{P}^{N_h}$$
with center $\left\langle T_{p_1}X,\dots,T_{p_h}X\right\rangle$ a \textit{general $h$-tangential projection} of $X$. Set $X_h \!=\! \tau_{X,h}(X)$. 

\begin{Proposition}\cite[Proposition 3.5]{CC01}\label{cc}
Let $X\subset\P^N$ be a non-degenerate projective variety of dimension $n$, and $x_1,\dots,x_h\in X$  general points. Assume that 
$$
N-\dim(\left\langle T_{x_1}X,\dots,T_{x_h}X\right\rangle)-1\geq n.
$$
Then the general $h$-tangential projection $\tau_{X,h}:X\dasharrow X_h$ is generically finite if and only if 
$X$ is not $(h+1)$-defective.
\end{Proposition}

\begin{Example}Let $\nu_2^n:\mathbb{P}^n\rightarrow\mathbb{P}^{N_n}$ be the $2$-Veronese embedding of $\mathbb{P}^n$, with $N_n = \frac{1}{2}(n+2)(n+1)-1$, $X = V^n_2\subseteq\mathbb{P}^{N_n}$ the corresponding Veronese variety, and $x_1,\dots, x_{h}\in V^n_2$ general points, with $h\leq n-1.$ The linear system of hyperplanes in $\mathbb{P}^{N_n}$ containing $\left\langle T_{x_1}V^n_2,\dots, T_{x_{h}}V^n_2\right\rangle$ corresponds to the linear system of quadrics in $\mathbb{P}^n$ whose vertex contains $\Lambda=\left\langle\nu_2^{-1}(x_1),\dots, \nu_2^{-1}(x_{h})\right\rangle$. Therefore, we have the following commutative diagram
\[
  \begin{tikzpicture}[xscale=2.7,yscale=-1.2]
    \node (A0_0) at (0, 0) {$\mathbb{P}^n$};
    \node (A0_1) at (1, 0) {$V_2^n\subseteq\mathbb{P}^{N_n}$};
    \node (A1_0) at (0, 1) {$\mathbb{P}^{n-h}$};
    \node (A1_1) at (1, 1) {$V_2^{n-h}\subseteq\mathbb{P}^{N_{n-h}}$};
    \path (A0_0) edge [->]node [auto] {$\scriptstyle{\nu_2^n}$} (A0_1);
    \path (A0_0) edge [->,swap,dashed]node [auto] {$\scriptstyle{\pi_\Lambda}$} (A1_0);
    \path (A0_1) edge [->,dashed]node [auto] {$\scriptstyle{\tau_{X,h}}$} (A1_1);
    \path (A1_0) edge [->]node [auto] {$\scriptstyle{\nu_2^{n-h}}$} (A1_1);
  \end{tikzpicture}
  \]
where $\pi_{\Lambda}:\mathbb{P}^n\dasharrow\mathbb{P}^{n-h}$ is the projection from $\Lambda$. Hence $\tau_{X,h}$ has positive relative dimension, and Proposition \ref{cc} yields, as it is well-known, that $V_2^n$ is $h$-defective for any $h\leq n$.
\end{Example}

Using Proposition \ref{cc} Chiantini and Ciliberto were able to classify defective threefolds in \cite{CC03}.

Terracini's Lemma \ref{terracini} and Proposition \ref{cc} became more difficult to use in practice as soon as $h$ increases, that is, when we are dealing with more points. 
The new strategy developed in this thesis is to build upon Proposition \ref{cc} and give a new technique in which one could guarantee that 
$\sec_h X$ has the expected dimension but working with much less than $h$ points.

The idea is to consider an osculating space of $X$ at $p$ or order $s,$ see Section \ref{osculating}, and degenerate the span of several tangent spaces, say $h$ of them, inside this single one osculating space.
Then we consider the linear projection from this osculating space, and if it is generically finite then the projection from the span of the tangent spaces is generically finite as well, as we will see in Section \ref{mainsection}. And by Proposition \ref{cc} we get that $X$ is not $h+1$-defective.

The advantage here is that we only have to consider the linear projection from a single natural space instead of consider the projection from the span of several tangent spaces. The former in the practice may be much simpler.
In Chapters \ref{cap3} and \ref{cap4} we apply this new technique to two well studied varieties, the Grassmannian and the Segre-Veronese, and are able to find new bounds for theirs non-defectivity.

\section{Osculating spaces and osculating projections}\label{osculating}

Let $X\subset \P^N$ be a projective variety of dimension $n$, and $p\in X$ a smooth point.
Choose a local parametrization of $X$ at $p$:
$$
\begin{array}{cccc}
\phi: &\mathcal{U}\subseteq\mathbb{C}^n& \longrightarrow & \mathbb{C}^{N}\\
      & (t_1,\dots,t_n) & \longmapsto & \phi(t_1,\dots,t_n) \\
      & 0 & \longmapsto & p 
\end{array}
$$
For a multi-index $I = (i_1,\dots,i_n)$, set 
\begin{equation}\label{osceq}
\phi_I = \frac{\partial^{|I|}\phi}{\partial t_1^{i_1}\dots\partial t_n^{i_n}}.
\end{equation}
For any $m\geq 0$, let $O^m_pX$ be the affine subspace of $\mathbb{C}^{N}$ centered at $p$ and spanned by the vectors $\phi_I(0)$ with  $|I|\leq m$.

The $m$-\textit{osculating space} $T_p^m X$ of $X$ at $p$ is the projective closure  of  $O^m_pX$ in $\mathbb{P}^N$.
Note that $T_p^0 X=\{p\}$, and $T_p^1 X$ is the usual tangent space of $X$ at $p$.
When no confusion arises we will write $T_p^m$ instead of $T_p^mX$.

\begin{Example}\label{ratcurve}
Let $C_n\subseteq \P^n$ be a rational normal curve of degree $n$ and $p\in C_n.$ We may assume that 
$p=(1:0:\dots:0)\in (x_0=1)=\C^n\subseteq \P^n.$
Then 
$$
\begin{array}{cccc}
\phi: &\mathbb{C}& \longrightarrow & \mathbb{C}^{n}\\
      & t & \longmapsto & (t,t^2,\dots,t^n)
\end{array}
$$
is a local parametrization of $C_n$ in a neighorhood of $\phi(0)=p.$
Thus
$$
\begin{array}{cccc}
\dfrac{\partial^j \phi}{\partial t^j}(0)=e_j, \ j=1,\dots,n
\end{array}
$$
and
$$O^m_p X=\left\langle e_1,\dots, e_m \right\rangle\subseteq \C^n, m=1,\dots,n.$$
Therefore the osculating spaces of $C_n$ are
$$T^m_p X=\left\langle e_0,e_1,\dots, e_m \right\rangle\subseteq \P^n, m=1,\dots,n.$$
\end{Example}

Osculating spaces can be defined intrinsically. Let $\mathcal{L}$ be an invertible sheaf on $X$, $V = H^0(X,\mathcal{L})$, and $\Delta\subseteq X\times X$ the diagonal. The rank $\binom{n+m}{m}$ locally free sheaf
$$J_m(\mathcal{L}) = \pi_{1*}(\pi_{2}^{*}(\mathcal{L})\otimes \mathcal{O}_{X\times X}/\mathcal{I}_{\Delta}^{m+1})$$
is called the \textit{$m$-jet bundle} of $\mathcal{L}$. Note that the fiber of $J_m(\mathcal{L})$ at $p\in X$ is 
$$J_m(\mathcal{L})_p\cong H^0(X,\mathcal{L}\otimes\mathcal{O}_X/\mathfrak{m}_p^{m+1})$$
and the quotient map 
$$j_{m,p}:V\rightarrow H^0(X,\mathcal{L}\otimes\mathcal{O}_X/\mathfrak{m}_p^{m+1})$$
is nothing but the evaluation of the global sections and their derivatives of order at most $m$ at the point $p\in X$. Let 
$$j_m:V\otimes\mathcal{O}_X\rightarrow J_m(\mathcal{L})$$
be the corresponding vector bundle map. Then, there exists an open subset $U_m\subseteq X$ where $j_m$ is of maximal rank $r_m\leq \binom{n+m}{m}$.

The linear space $\mathbb{P}(j_{m,p}(V)) = T_p^m X\subseteq\mathbb{P}(V)$ is the $m$-\textit{osculating space} of $X$ at $p\in X$. The integer $r_m$ is called the \textit{general $m$-osculating dimension} of $\mathcal{L}$ on $X$.

Note that while the dimension of the tangent space at a smooth point is always equal to the dimension of the variety, higher order osculating spaces can be strictly smaller than expected even at a general point. In general, we have that the
\textit{$m$-osculating dimension} of $X$ at $p$ is
\begin{equation}\label{dimosc}
\dim(T_p^m X) = \min\left\{\binom{n+m}{n}-1-\delta_{m,p},N\right\}
\end{equation}
where $\delta_{m,p},$ is the number of independent differential equations of order $\leq m$ satisfied by $X$ at $p.$

Projective varieties having general $m$-osculating dimension smaller than expected were introduced and studied in \cite{Seg07}, \cite{Te12}, \cite{Bom19}, \cite{To29}, \cite{To46}, and more recently in \cite{PT90}, \cite{BPT92}, \cite{BF04}, \cite{MMRO13}, \cite{DiRJL15}.

In particular, these works highlight how algebraic surfaces with defective higher order osculating spaces contain many lines, such as rational normal scrolls, and developable surfaces, that is cones or tangent developables of curves. As an example, which will be useful later on this chapter, we consider tangent developables of rational normal curves.

\begin{Proposition}\label{tdrnc}
Let $C_n\subseteq\mathbb{P}^n$ be a rational normal curve of degree $n$ in $\mathbb{P}^n$, and let $Y_n\subseteq\mathbb{P}^n$ be its tangent developable. Then
$$\dim(T^m_pY_n) = \min\{m+1,n\}$$
for $p\in Y_n$ general, and $m\geq 1$.
\end{Proposition}
\begin{proof}
We may work on an affine chart. Then $Y_n$ is the surface parametrized by
$$
\begin{array}{cccc}
\phi: & \mathbb{A}^2 & \longrightarrow & \mathbb{A}^n\\ 
 & (t,u) & \mapsto & (t+u,t^2+2tu,\dots ,t^n+nt^{n-1}u)
\end{array} 
$$
Note that 
$$\frac{\partial^m \phi}{\partial t^{m-k} \partial u^k}=0$$
for any $k\geq 2$. Furthermore, we have
$$\frac{\partial^m\phi}{\partial t^m}-\frac{\partial^m\phi}{\partial t^{m-1}\partial u} = u\frac{\partial^{m+1}\phi}{\partial t^m\partial u}$$
for any $n\geq m\geq 1$.

Therefore, for any $n \geq m\geq 1$ we get just two non-zero partial derivatives of order $m$, and one partial derivative is given in terms of smaller order partial derivatives. Furthermore, in the notation of (\ref{dimosc}) we have $\delta_{m,p} = \frac{m(m+1)}{2}-1$ for any $1\leq m\leq n-1$, where $p\in Y_n$ is a general point.
\end{proof}

Let $p_1,\dots,p_l\in X\subseteq\mathbb{P}^N$ be general points and $k_1,\dots,k_l\geq 0$ integers.
We will call the linear projection
\begin{equation}
\Pi_{T^{k_1,\dots,k_l}_{p_1,\dots,p_l}}:X\subseteq\mathbb{P}^N\dasharrow\mathbb{P}^{N_{k_1,\dots,k_l}}
\end{equation}
with center 
$$\left\langle T^{k_1}_{p_1}X,\dots,T^{k_l}_{p_l}X\right\rangle$$ a \textit{general $(k_1,\dots,k_l)$-osculating projection} of $X$.

\begin{Example}
Recall the Example \ref{ratcurve} and let $p,q\in C_n$ be general points. We may suppose that $p=e_0,q=e_n,$ and then 
$$T^m_p X=\left\langle e_0,e_1,\dots, e_m \right\rangle,
T^m_q X=\left\langle e_n,e_{n-1},\dots, e_{n-m} \right\rangle\subseteq \P^n, m=1,\dots,n.$$
Therefore, given non-negative integers $a,b$ such that 
$a+b\leq n-2$ we can write a general $(a,b)-$osculating projection from $C_n$
$$
\begin{array}{rl}
\Pi_{T^{a,b}_{p,q}}:C_n\subseteq\mathbb{P}^n&\dasharrow\mathbb{P}^{n-a-b-2}\\
(x_0:\dots:x_n)&\mapsto
(x_{a+1}:\dots:x_{n-b-1})\end{array}.
$$
Composing with 
$$
\begin{array}{rl}
\gamma:\P^1&\to C_n\subseteq \P^n\\
(t:s)&\mapsto
(s^n:s^{n-1}t:\dots :t^n)\end{array}
$$
we get
$$
\begin{array}{rl}
\Pi_{T^{a,b}_{p,q}}\circ\gamma:\P^1&\dasharrow\mathbb{P}^{n-a-b-2}\\
(t:s)&\mapsto
(s^{n-a-1}t^{a+1}:\dots:s^{b+1}t^{n-b-1})=
(s^{n-a-b-2}:s^{n-a-b-3}t:\dots:t^{n-a-b-2})
\end{array}
$$
We conclude that $\Pi_{T^{a,b}_{p,q}}$ is the constant map when $a+b=n-2$ and is birational otherwise.
\end{Example}

For our strategy to work we have to be able to control the dimension of the fibers of these general osculating projections. 
In Chapters \ref{cap3} and \ref{cap4} we describe explicitly the osculating spaces and give sufficient conditions for the osculation projections to be generically finite for the Grassmannian and the Segre-Veronese varieties. 

\section{Degenerating projections and osculating spaces}
\label{degsection}

In order to study the fibers of general tangential projections via osculating projections we need to understand how the fibers of rational maps behave under specialization. We refer to 
\cite[Section 20]{GD64} for the general theory of rational maps relative to a base scheme.

\begin{Proposition}\label{p1}
Let $C$ be a smooth and irreducible curve, $X\rightarrow C$ an integral scheme flat over $C$, and $\phi:X\dasharrow \mathbb{P}^n_{C}$ be a rational map of schemes over $C$. Let $d_0 = \dim(\overline{\phi_{|X_{t_0}}(X_{t_0})})$ with $t_0\in C$. Then for $t\in C$ general we have $\dim(\overline{\phi_{|X_{t}}(X_{t})})\geq d_0$.

In particular, if there exists $t_0\in C$ such that $\phi_{|X_{t_0}}:X_{t_0}\dasharrow\mathbb{P}^n$ is generically finite, then for a general $t\in C$ the rational map $\phi_{|X_{t}}:X_{t}\dasharrow\mathbb{P}^n$ is generically finite as well.
\end{Proposition}
\begin{proof}
Let us consider the closure $Y = \overline{\phi(X)}\subseteq \mathbb{P}^n_C$ of the image of $X$ through $\phi$. By taking the restriction $\pi_{|Y}:Y\rightarrow C$ of the projection $\pi:\mathbb{P}^n_C\rightarrow C$ we see that $Y$ is a scheme over $C$.

Note that since $Y$ is an irreducible and reduced scheme over the curve $C$ we have that $Y$ is flat over $C$. In particular, the dimension of the fibers $\pi_{|Y}^{-1}(t)=Y_t$ is a constant $d = \dim(Y_t)$ for any $t\in C$.
  
For $t\in C$ general the fiber $\pi_{|Y}^{-1}(t)=Y_t$ contains $\phi_{|X_t}(X_t)$ as a dense subset. Therefore, we have $d = \dim(\phi_{|X_t}(X_t))\leq \dim(X_t)$ for $t\in C$ general.

Then, since $\phi_{|X_{t_0}}(X_{t_0})\subseteq Y_{t_0}$ we have $\dim(\overline{\phi_{|X_{t_0}}(X_{t_0})})\leq d = \dim(\overline{\phi_{|X_t}(X_t)})$ for $t\in C$ general.
Now, assume that $\dim(X_{t_0}) = \dim(\phi_{|X_{t_0}}(X_{t_0}))\leq d$. Therefore, we get 
$$\dim(X_{t_0})\leq d\leq \dim(X_t) = \dim(X_{t_0})$$ 
that yields $d = \dim(X_{t_0}) = \dim(X_t)$ for any $t\in C$. Hence, for a general $t\in C$ we have 
$$\dim(X_t) = \dim(\overline{\phi_{|X_t}(X_t)})$$
that is $\phi_{|X_{t}}:X_{t}\dasharrow \overline{\phi_{|X_t}(X_t)}\subseteq\mathbb{P}^n$ is generically finite.
\end{proof}

Now, let $C$ be a smooth and irreducible curve, $X\subset\mathbb{P}^N$ an irreducible and reduced projective variety, and $f:\Lambda\rightarrow C$ a family of $k$-dimensional linear subspaces of $\mathbb{P}^n$ parametrized by $C$. 

Let us consider the invertible sheaf $\mathcal{O}_{\mathbb{P}^n\times C}(1),$ and the sublinear system $|\mathcal{H}_{\Lambda}|\subseteq |\mathcal{O}_{\mathbb{P}^n\times C}(1)|$ given by the sections of $\mathcal{O}_{\mathbb{P}^n\times C}(1)$ vanishing on $\Lambda\subset\mathbb{P}^n\times C$. We denote by $\pi_{\Lambda|X\times C}$
the restriction of the rational map $\pi_{\Lambda}:\mathbb{P}^n\times C\dasharrow \mathbb{P}^{n-k-1}\times C$ of schemes over $C$ induced by $|\mathcal{H}_{\Lambda}|$.

Furthermore, for any $t\in C$ we denote by $\Lambda_t\cong\mathbb{P}^k$ the fiber $f^{-1}(t)$, and by $\pi_{\Lambda_t|X}$ the restriction to $X$ of the linear projection $\pi_{\Lambda_t}:\mathbb{P}^n\dasharrow\mathbb{P}^{n-k-1}$ with center $\Lambda_t$. 

\begin{Proposition}\label{p2}
Let $d_0 = \dim(\overline{\pi_{\Lambda_{t_0}|X}(X)})$ for $t_0\in C$. Then 
$$\dim(\overline{\pi_{\Lambda_{t}|X}(X)})\geq d_0$$
for $t\in C$ general.

Furthermore, if there exists $t_0\in C$ such that $\pi_{\Lambda_{t_0}|X}:X\dasharrow\mathbb{P}^{n-k-1}$ is generically finite then $\pi_{\Lambda_{t}|X}:X\dasharrow\mathbb{P}^{n-k-1}$ is generically finite for $t\in C$ general.
\end{Proposition}
\begin{proof}
The rational map $\pi_{\Lambda|X\times C}:X\times C\dasharrow \mathbb{P}^{n-k-1}\times C$ of schemes over $C$ is just the restriction of the relative linear projection $\pi_{\Lambda}:\mathbb{P}^n\times C\dasharrow \mathbb{P}^{n-k-1}\times C$ with center $\Lambda$. 

Therefore, the restriction of $\pi_{\Lambda|X\times C}$ to the fiber $X_t\cong X$ of $X\times C$ over $t\in C$ induces the linear projection from the linear subspace $\Lambda_t$, that is
$\pi_{\Lambda|X_t} = \pi_{\Lambda_t|X}$
for any $t\in C$.  Now, to conclude it is enough to apply Proposition \ref{p1} with $\phi = \pi_{\Lambda|X\times C}$.  
\end{proof}

Essentially, Propositions \ref{p1} and \ref{p2} say that the dimension of the general fiber of the special map is greater or equal than the dimension of the general fiber of the general map. Therefore, when the special map is generically finite the general one is generically finite as well. We would like to stress that in this case, under suitable assumptions, \cite[Lemma 5.4]{AGMO16} says that the degree of the map can only decrease under specialization.

Next we formalize the idea of degenerating osculating spaces. We start with a simple example of this phenomenon.

\begin{Example}\label{ratcurveII}
Consider again $p,q\in C_n\subset \P^n$ of Example \ref{ratcurve}, and $a,b\geq 0$ integers such that $a+b\leq n-2.$ We will consider the open set 
$$(x_n=1)=\C^n\subset \P^n,$$
and the map
$$
\begin{array}{rl}
\gamma:\C&\to \C^n\\
t&\mapsto (t,t^2,\dots,t^n)\end{array}.
$$
Now we consider the family of linear spaces
$$T_t=\left\langle T^a_p, T^b_{\gamma(t)}
\right\rangle, t\in \C\backslash 0,$$
parametrized by $\C\backslash 0.$
Such family has a flat limit $T_0\in \G(\dim T_t,n).$
Note that
$$\begin{array}{ll}
T^a_p=\left\langle e_0,\dots,e_a
\right\rangle;\\
T^b_{\gamma(t)}=\left\langle e_n(t),e_{n-1}(t),\dots,e_{n-b}(t)
\right\rangle, t\in \C\backslash 0,
\end{array}$$
where
$$\begin{array}{ll}
e_n(t)=(1:t:t^2:\dots:t^n)\\
e_{n-1}(t)=(1:1:2t:\dots:nt^{n-1})\\
\vdots\\
e_{n-b}(t)=(1:0:\dots:0:b!:(b+1)!t:\dots:
\frac{(n-1)!}{(n-b-1)!}t^{n-b-1}:
\frac{n!}{(n-b)!}t^{n-b}).
\end{array}$$
To avoid knotty computations we do only the case $b=1.$ We have then
\begin{align*}
T_t&=\left\langle e_0,\dots,e_a,
(1:t:t^2:\dots:t^n),(1:1:2t:\dots:nt^{n-1})
\right\rangle\\
&=(F_n=F_{n-1}=\dots=F_{n-a-2}=0)\subset \P^n
\end{align*}
where $F_j=x_j-2x_{j-1}t+x_{j-2}t^2,j=n-a-2,\dots,n.$
Therefore the flat limit $T_0$ is 
$$T_0=(x_{n-a-2}=\dots=x_{n-1}=x_{n}=0)=
\left\langle e_0,\dots,e_{n-a-3}
\right\rangle=T^{a+2}_p=T^{a+b+1}_p.$$
\end{Example}

In Chapters \ref{cap3} and \ref{cap4} we will make computations similar to Example \ref{ratcurveII} for Grassmannian and Segre-Veronese varieties. Note that in Example \ref{ratcurveII} we can degenerate $\left\langle T^a_p, T^b_q \right\rangle$ inside $T^{a+b+1}_p$ but not inside $T^{a+b}_p.$
In general we may be able to degenerate $\left\langle T^a_p, T^b_q \right\rangle$
inside $T^{a+b+1}_p,$ but can not be expected that the flat limit be equal to a given osculating space.
The definitions ahead formalize what we need.

\begin{Definition}\label{moscularity}
Let $X\subset\mathbb{P}^N$ be a projective variety.
 
We say that $X$ has \textit{$m$-osculating regularity} if the following property holds. 
Given general points $p_1,\dots,p_{m}\in X$  and  integer $k\geq 0$, 
there exists a smooth curve $C$ and morphisms $\gamma_j:C\to X$, $j=2,\dots,m$, 
such that  $\gamma_j(t_0)=p_1$, $\gamma_j(t_\infty)=p_j$, and the flat limit $T_0$ in $\G(dim(T_t),N)$ of the family of linear spaces 
$$
T_t=\left\langle T^{k}_{p_1},T^{k}_{\gamma_2(t)},\dots,T^{k}_{\gamma_{m}(t)}\right\rangle,\: t\in C\backslash \{t_0\}
$$
is contained in $T^{2k+1}_{p_1}$.

We say that $X$ has \textit{strong $2$-osculating regularity} if the following property holds. 
Given general points $p,q\in X$ and  integers $k_1,k_2\geq 0$,
there exists a smooth curve $\gamma:C\to X$ such that  $\gamma(t_0)=p$, $\gamma(t_\infty)=q$ 
and the flat limit $T_0$ in $\G(dim(T_t),N)$ of the family of linear spaces 
$$
T_t=\left\langle T^{k_1}_p,T^{k_2}_{\gamma(t)}\right\rangle,\: t\in C\backslash \{t_0\}
$$
is contained in $T^{k_1+k_2+1}_p$.
\end{Definition}

Example \ref{ratcurveII} shows that the rational normal curve $C_n$ has $2$-strong osculating regularity and we will see in the next chapters that   Grassmannians and Segre-Veronese have as well.
Next, we give a example of a variety that does not have $2$-osculating regularity.

\begin{Example}\label{developableoscdef}
Let us consider the tangent developable $Y_n\subseteq\mathbb{P}^n$ of a degree $n$ rational normal curve $C_n\subseteq\mathbb{P}^n$ as in Proposition \ref{tdrnc}.

Note that two general points $p = \phi(t_1,u_1)$, $q = \phi(t_2,u_2)$ in $Y_n$ can be joined by a smooth rational curve. Indeed, we may consider the curve 
$$\xi(t) = (t_1+t(t_2-t_1)+u_1+t(u_2-u_1),\dots, (t_1+t(t_2-t_1))^n+n(t_1+t(t_2-t_1))^{n-1}(u_1+t(u_2-u_1)))$$ 
Now, let $\gamma:C\rightarrow Y_n$ be a smooth curve with $\gamma(t_0)=p$ and $\gamma(t_\infty)=q$, and let $T_{t_0}$ be the flat limit of the family of linear spaces 
$$T_t=\left\langle T_{p},T_{\gamma(t)}\right\rangle,\: t\in C\backslash \{t_0\}$$
Now, one can prove that if $n\geq 5$ then $T_{p}Y_n\cap T_{q}Y_n = \varnothing$ by a straightforward computation, or alternatively by noticing that by \cite{Ba05} $Y_n$ is not $2$-secant defective, and then by Terracini's lemma \cite[Theorem 1.3.1]{Ru03} $T_{p}Y_n\cap T_{q}Y_n = \varnothing$. Now, $T_{p}Y_n\cap T_{q}Y_n = \varnothing$ implies that $\dim(T_t) = 5$ for any $t\in C$. On the other hand, by Proposition \ref{tdrnc} we have $\dim(T^3_pY_n) = 4$. Hence, $T_{t_0}\nsubseteq T^3_pY_n$ as soon as $n\geq 5$.
\end{Example}

Given a projective variety having $m$-osculating regularity and strong $2$-osculating regularity we introduce a function $h_m:\N_{\geq 0}\to \N_{\geq 0}$ counting how many tangent spaces we can degenerate to a higher order osculating space. 

\begin{Definition}\label{defhowmanytangent}
Given an integer $m\geq 2$ we define a function
\begin{align*}
h_m:\N_{\geq 0}\to \N_{\geq 0}
\end{align*}
as follows: $h_m(0)=0$. For any $k\geq 1$ write
$$k+1=2^{\lambda_1}+2^{\lambda_2}+\dots+2^{\lambda_l}+\varepsilon$$
where $\lambda_1>\lambda_2>\dots>\lambda_l \geq 1$, $\varepsilon\in \{0,1\}$, and define
$$h_m(k):=m^{\lambda_1-1}+m^{\lambda_2-1}+\dots+m^{\lambda_l-1}.$$
In particular $h_m(2k)=h_m(2k-1)$ and $h_2(k)=\left\lfloor \dfrac{k+1}{2}\right\rfloor$.
\end{Definition}

\begin{Example}
For instance 
$$h_m(1)=h_m(2)=1,h_m(3)=m,h_m(5)=m+1,h_m(7)=m^2,h(9)=m^2+1$$
and since $23=16+4+2+1=2^4+2^2+2^1+1$ we have $h_m(22)=m^3+m+1$. In particular, if $X$ has $m$-osculating regularity and $2$-strong osculating regularity, then we can degenerate 
$\left\langle T^{1}_{p_1}X,\dots,T^{1}_{p_{m^2+1}}X\right\rangle$
inside $T^9_p X$ because $h_m(9)=m^2+1.$
\end{Example}

Let $X\subset\mathbb{P}^N$ be a rational variety of dimension $n$, $p_1,\dots,p_m\in X$ general points. We reinterpret the notion of $m$-osculating regularity in Definition \ref{moscularity} in terms of limit linear systems and collisions of fat points. 

Let $\mathcal{H}\subseteq |\mathcal{O}_{\mathbb{P}^n}(d)|$ be the sublinear system of $|\mathcal{O}_{\mathbb{P}^n}(d)|$ inducing the birational map $i_{\mathcal{H}}:\mathbb{P}^n\dasharrow X\subset\mathbb{P}^N$, and $q_i = i_{\mathcal{H}}^{-1}(p_i)$. 

Then $X$ has $m$-osculating regularity if and only if there exists smooth curves $\gamma_i:C\rightarrow \mathbb{P}^n$, $i = 2,\dots,m$, with $\gamma_i(t_0) = q_1$, $\gamma_i(t_{\infty}) = q_i$ for $i = 1,\dots,m$, such that the limit linear system $\mathcal{H}_{t_0}$ of the family of linear systems $\mathcal{H}_t$ given by the hypersurfaces in $\mathcal{H}$ having at least multiplicity $s+1$ at $q_1,\gamma_2(t),\dots,\gamma_m(t)$ contains the linear system $\mathcal{H}^{2s+2}_{q_1}$ of degree $d$ hypersurfaces with multiplicity at least $2s+2$ at $q_1$.

Indeed, if $p_i = i_{\mathcal{H}}(q_i)$ for $i = 1,\dots,m$ then the linear system of hyperplanes in $\mathbb{P}^N$ containing 
$$T_t = \left\langle T_{p_1}^s,T_{i_{\mathcal{H}}(\gamma_2(t))}^s,\dots, T_{i_{\mathcal{H}}(\gamma_m(t))}^s\right\rangle$$
corresponds to the linear system $\mathcal{H}_t$. Similarly, the linear system of hyperplanes in $\mathbb{P}^N$ containing $T_{p_1}^{2s+1}$ corresponds to the linear system $\mathcal{H}^{2s+2}_{q_1}$.

Therefore, the problem of computing the $m$-osculating regularity of a rational variety can be translated in terms of limit linear systems in $\mathbb{P}^n$ given by colliding a number of fat points. This is a very hard and widely studied subject \cite{CM98}, \cite{CM00}, \cite{CM05}, \cite{Ne09}.

\section{Secant defectivity via osculating projections}
\label{mainsection}

In this section we use the notions developed in previous sections to study the dimension of secant varieties. We do this reinterpreting Proposition \ref{cc} in terms of osculating projections.

The method goes as follows. 
If $X\subset\mathbb{P}^N$ has $m$-osculating regularity, one degenerates a general $m$-tangential projection 
into a linear projection with center contained in $T^{3}_p$X. 
Then one further degenerates a general osculating projection $T^{(3,\dots,3)}_{p_1, \dots, p_m}$
into a linear projection with center contained in $T^{7}_qX$.
By proceeding recursively, one degenerates a general $h$-tangential projection 
into a linear projection with center contained in a suitable 
linear span of osculating spaces, and then check whether this projection is generically finite.

\begin{Theorem}\label{lemmadefectsviaosculating}
Let $X\subset \P^N$ be a projective variety 
having $m$-osculating regularity and strong $2$-osculating regularity.
Let $k_1,\dots,k_l\geq 1$ be integers such that the general osculating projection 
$\Pi_{T^{k_1,\dots,k_l}_{p_1,\dots,p_l}}$ is generically finite.
Then $X$ is not $h$-defective for $h\leq \displaystyle\sum_{j=1}^{l}h_m(k_j)+1,$ where
$h_m$ is as in Definiton \ref{defhowmanytangent}.
\end{Theorem}

\begin{proof}
Let us consider a general tangential projection $\Pi_{T}$ where 
$$T = \left\langle T^1_{p_1^1},\dots, T^1_{p_1^{h_m(k_1)}},\dots, T^1_{p_l^1},\dots,T^1_{p_l^{h_m(k_l)}}\right\rangle$$
and $p_1^1 = p_1,\dots, p_l^1 = p_l$. Our argument consists in specializing the projection $\Pi_{T}$ several times in order to reach a generically finite projection. For seek of notational simplicity along the proof we will assume $l = 1$. For the general case it is enough to apply the same argument $l$ times.

Let us begin with the case $k_1+1 = 2^{\lambda}$. Then $h_{m}(k_1) = m^{\lambda-1}$. Since $X$ has $m$-osculating regularity we can degenerate $\Pi_{T}$, in a family parametrized by a smooth curve, to a projection $\Pi_{U_1}$ whose center $U_1$ is contained in 
$$V_1 = \left\langle T^{3}_{p_1^1}, T^3_{p_1^{m+1}},\dots, T^3_{p_1^{m^{\lambda-1}-m+1}}\right\rangle.$$ 
Again, since $X$ has $m$-osculating regularity we may specialize, in a family parametrized by a smooth curve, the projection $\Pi_{V_1}$ to a projection $\Pi_{U_2}$ whose center $U_2$ is contained in
$$V_2 = \left\langle T^{7}_{p_1^1}, T^7_{p_1^{m^2+1}},\dots, T^7_{p_1^{m^{\lambda-1}-m^2+1}}\right\rangle.$$
Proceeding recursively in this way in last step we get a projection $\Pi_{U_{\lambda-1}}$ whose center $U_{\lambda -1}$ is contained in 
$$V_{\lambda-1} = T^{2^{\lambda}-1}_{p_1^1}.$$
When $k_1+1 = 2^{\lambda}$ our hypothesis means that $\Pi_{T^{k_1}_{p_1^1}}$ is generically finite. Therefore, $\Pi_{U_{\lambda-1}}$ is generically finite, and applying Proposition \ref{p2} recursively to the specializations in between $\Pi_{T}$ and $\Pi_{U_{\lambda-1}}$ we conclude that $\Pi_{T}$ is generically finite as well.

Now, more generally, let us assume that 
$$k_1+1 = 2^{\lambda_1}+\dots + 2^{\lambda_s}+\varepsilon$$
with $\varepsilon\in\{0,1\}$, and $\lambda_1 > \lambda_2 > \dots > \lambda_s\geq 1$. Then
$$h_m(k_1) = m^{\lambda_1-1}+\dots + m^{\lambda_s-1}.$$
By applying $s$ times the argument for $k_1+1 = 2^{\lambda}$ in the first part of the proof we may specialize $\Pi_{T}$ to a projection $\Pi_{U}$ whose center $U$ is contained in 
$$V = \left\langle T^{2^{\lambda_1}-1}_{p_1^1}, T^{2^{\lambda_2}-1}_{p_1^{m^{\lambda_1-1}+1}},\dots, T^{2^{\lambda_s}-1}_{p_1^{m^{\lambda_1-1}+\dots+m^{\lambda_{s-1}-1}+1}}\right\rangle.$$ 
Finally, we use the strong $2$-osculating regularity $s-1$ times to specialize $\Pi_V$ to a projection $\Pi_{U^{'}}$ whose center $U^{'}$ is contained in 
$$V^{'} = T_{p_1^1}^{2^{\lambda_1}+\dots +2^{\lambda_s}-1}.$$
Note that $T_{p_1^1}^{2^{\lambda_1}+\dots +2^{\lambda_s}-1} = T^{k_1}_{p_1^1}$ if $\varepsilon = 0$, and $T_{p_1^1}^{2^{\lambda_1}+\dots +2^{\lambda_s}-1} = T^{k_1-1}_{p_1^1}\subset T^{k_1}_{p_1^1}$ if $\varepsilon = 1$. In any case, since by hypothesis $\Pi_{T^{k_1}_{p_1^1}}$ is generically finite, again by applying Proposition \ref{p2} recursively to the specializations in between $\Pi_{T}$ and $\Pi_{U^{'}}$ we conclude that $\Pi_{T}$ is generically finite. Therefore, by Proposition \ref{cc} we get that $X$ is not $(\sum_{j=1}^{l}h_m(k_j)+1)$-defective.
\end{proof}

As a corollary of the proof we get the following result if the variety has only $m$-osculating regularity but not necessarily strong $2$-osculating regularity.

\begin{Theorem}
Let $X\subset \P^N$ be a projective variety having $m$-osculating regularity. Let $k_1,\dots,k_l\geq 1$ be integers such that the general osculating projection 
$\Pi_{T^{k_1,\dots,k_l}_{p_1,\dots,p_l}}$ is generically finite.
Then $X$ is not $h$-defective for $h\leq \displaystyle \left(\sum_{j=1}^{l}m^{\lfloor \log_2(k_j+1)\rfloor -1}\right)+1$.
\end{Theorem}
 
In Chapters \ref{cap3} and \ref{cap4} we will apply Theorem \ref{lemmadefectsviaosculating} for Grassmannian and Segre-Veronese varieties and obtain bounds on non-defectivity of them.

\chapter{Secant defectivity of Grassmannians}\label{cap3}
In this chapter we apply the technique developed in the previous chapter to study defectivity of Grassmannians, obtaining the following theorem:

\begin{Theorem}
Assume that $r\geq 2$, set 
$$\alpha:=\left\lfloor \dfrac{n+1}{r+1} \right\rfloor$$
and let $h_\alpha$ be as in Definition \ref{defhowmanytangent}. If either
\begin{itemize}
	\item[-] $n\geq r^2+3r+1$ and $h\leq\alpha h_{\alpha}(r-1)$ or
	\item[-] $n< r^2+3r+1$, $r$ is even, and 
	$h\leq (\alpha-1) h_{\alpha}(r-1)+
	h_\alpha(n-2-\alpha r)$ 	or
	\item[-] $n< r^2+3r+1$, $r$ is odd, and 
	$h\leq (\alpha-1) h_{\alpha}(r-2)+h_\alpha(\min\{n-3-\alpha(r-1),r-2\})$
\end{itemize}
then $\G(r,n)$ is not $(h+1)$-defective.
\end{Theorem}

In the first section of the present chapter we fix notation for the Grassmannian to be used throughout this thesis, and recall some results regarding Schubert varieties to be used in the next section and in some places in Part \ref{part2}.
In Section \ref{osculating spaces grass} we describe osculating spaces to the Grassmannian and give some interesting characterizations of some special Schubert varieties. In Section \ref{projoscgrass} we give sufficient conditions to osculating projections from Grassmannians to be birational.
In Section \ref{degtanosc} we show that $\G(r,n)$ has strong $2$-osculating regularity and $\lfloor \frac{n+1}{r+1}\rfloor$-osculating regularity.
Finally, in Section \ref{grassnodef} we use the results in the previous sections together with Theorem \ref{lemmadefectsviaosculating} to prove our main result concerning Grassmannian secant defects.

\section{ Grassmannaians and Schubert varieties}
\label{Schubert varieties}
Throughout this thesis we always view the Grassmannian $\G(r,n)$ of $r$-spaces inside $\P^n$ as a projective variety in its Pl\"ucker embedding, that is the morphism induced by the determinant of the universal quotient bundle $\mathcal{Q}_{\G(r,n)}$ on $\G(r,n)$:
$$
\begin{array}{cccc}
\f_{r,n}: &\G(r,n)& \longrightarrow & \P^N:=\P(\bigwedge^{r+1}\C^{n+1})\\
      & \left\langle v_0,\dots,v_r\right\rangle & \longmapsto & [v_0\wedge \dots\wedge v_r]
\end{array}
$$
where $N = \binom{n+1}{r+1}-1$.

When $n<2r+1$ there is a natural isomorphism 
$\G(r,n)\cong \G(n-r-1,n)$ and $n> 2(n-r-1)+1,$
therefore from now on we assume that $n \geq 2r+1.$ We will denote by $e_0,\dots,e_n\in \C^{n+1}$ both the vectors of the canonical basis of $\mathbb{C}^{n+1}$ and the corresponding points in $\P^n=\P(\C^{n+1})$.
We will denote by $p_I$ the Pl\"ucker coordinates on $\mathbb{P}^N$.

Now we give some simple properties of Schubert varieties and its singularities. We follow Billey and Lakshmibai's book \cite{Lak00}.

Fix a complete flag 
$$\mathcal{F}_{\bullet}:\quad \{0\}=V_0\subset V_1\subset \cdots \subset V_{n+1}=\C^{n+1},$$
and let $\lambda=(\lambda_1,\dots,\lambda_{r+1})$ be 
a partition of $|\lambda|=\sum\lambda_j$ such that 
$$n-r\geq \lambda_1\geq \dots \geq 
\lambda_{r+1}\geq 0.$$
The \textit{Schubert variety} associated to 
$\mathcal{F}_{\bullet}$ and $\lambda$ is defined by
$$\Sigma_{\lambda}(\mathcal{F}_{\bullet})=\{ [U]\in \G(r,n);\ dim (U \cap V_{n-r+i-\lambda_i})\geq i \mbox{ for each }  1\leq i\leq r+1\}.$$
We will omit the flag from the notation when no confusion can arise. We also omit the zeros and use powers to denote repeated indexes in the partition $\lambda.$ For instance $(3,3,2,2,2,2,1,0)=(3^2,2^4,1).$
It is easy to see that
$$\begin{cases}
\Sigma_{(0,0,\dots,0)}=\Sigma_0=\G(r,n);\\
\Sigma_{(n-r,n-r,\dots,n-r)}=\Sigma_{(n-r)^{r+1}}=
[V_{r+1}] \mbox{ is a point};\\
\Sigma_{((n-r-i)^{r+1-i},0^{i})}=\{[U]\in \G(r,n);
\dim (U \cap V_{r+1})\geq r+1-i\}=:R_i', i=1,\dots,r;\\
\Sigma_{(1,0,\dots,0)}=\Sigma_1=
\{[U]\in \G(r,n);\dim (U \cap V_{n-r})\geq 1\}=D
\mbox{ is a divisor},
\end{cases}$$

It is well known that $\codim(\Sigma_{\lambda})=|\lambda|$
and that 
$$\Sigma_{\lambda}\subset \Sigma_{\mu}
\Leftrightarrow \lambda \geq \mu$$
where 
$$\lambda=(\lambda_1,\dots,\lambda_{r+1})\geq 
(\mu_1,\dots,\mu_{r+1}) =\mu \iff \lambda_i\geq \mu_i,i=1,\dots,r+1.$$
 In order to describe the singular locus of 
$\Sigma_\lambda,$ for each partition we introduce an associated complementary partition 
$\widetilde \lambda=(\widetilde\lambda_1,\dots
\widetilde\lambda_{r+1})$ defined by $\widetilde{\lambda}_j=n-r-\lambda_j.$

Then $\dim(\Sigma_\lambda)=|\widetilde\lambda|$
and $\Sigma_{\lambda}\subset \Sigma_{\mu}
\Leftrightarrow
\widetilde\lambda \leq \widetilde\mu.$
Each partition $\lambda$ can be written in a unique way as 
$$\lambda=(p_1^{q_1},\dots,p_d^{q_d})=
(\underbrace{p_1,\dots,p_1}_{\mbox{$q_1$ times}},\dots,
\underbrace{p_d,\dots,p_d}_{\mbox{$q_d$ times}}),$$
where $p_d\neq 0,$ and we say that $\lambda$ has $d$ rectangles. This language is justified by the Ferrers diagram.
We say that $\Sigma_\lambda$ has $d$ rectangles when $\widetilde\lambda$ has $d$ rectangles.

\begin{Example}
The non-trivial Schubert varieties of $\G(1,4)$ are of eight types:
$$ \Sigma_1,\Sigma_2=R_1',\Sigma_3, \Sigma_{(1,1)},
\Sigma_{(2,1)},\Sigma_{(3,1)},\Sigma_{(2,2)},\Sigma_{(3,2)}.$$
The Ferrers diagrams, corresponding to the complementary partitions, are
$$\yng(2,3),\quad\yng(1,3),\quad\yng(3),
\quad\yng(2,2),\quad\yng(1,2),\quad\yng(2),
\quad\yng(1,1),\quad\yng(1).$$
From the Ferres diagrams we can see whether $\Sigma_\lambda \subset \Sigma_\mu,$
for instance 
$\Sigma_1\supset\Sigma_2\supset\Sigma_3\not\supset \Sigma_{(1,1)},$ but 
$\Sigma_1\supset \Sigma_{(1,1)}.$
Moreover, $\Sigma_3,\Sigma_{(1,1)},
\Sigma_{(3,1)},\Sigma_{(2,2)},$ and
$\Sigma_{(3,2)}$ have one rectangle, while the other three have two rectangles.
\end{Example}

Now, we illustrate the concept of hook of a diagram.

\begin{Example}\label{ExampleSchubert}
Consider $X=\Sigma_{(2^3,1)}\subset \G(4,9).$
Then the complementary partition is 
$\widetilde\lambda=(3^3,4,5),$ $X$ has three rectangles and its Ferrers diagram is
$$\yng(3,3,3,4,5)$$
There are two hooks to remove from this Ferrers diagram:
$$\young(\ \ *,\ \ *,\ \ *,\ \ **,\ \ \ \ \ )
\quad
\young(\ \ \ ,\ \ \ ,\ \ \ ,\ \ \ *,\ \ \ **)$$
and removing them we get two new Ferrers diagrams
$$\yng(2,2,2,2,5)\quad\quad\yng(3,3,3,3,3)$$
corresponding to $\widetilde\lambda_1=(2^4,5)$ and
$\widetilde\lambda_2=(3^5).$
The diagram on the right has no hook to remove, and the diagram on the left has exactly one hook to remove:
$$\young(\ *,\ *,\ *,\ *,\ ****)
\quad\to\quad
\yng(1,1,1,1,1)$$
and the resulting Ferrers diagram corresponds to $\widetilde{\lambda_3}=(1^5).$
\end{Example}

We have then the following characterization of singularities of Schubert varieties.

\begin{Theorem}[Theorem 9.3.1 \cite{Lak00}]\label{singSchubert}
The singular locus of the Schubert variety $\Sigma_\lambda$ is the union of the Schubert varieties 
$\Sigma_\mu$ indexed by the set of all partitions $\mu$ where $\widetilde\mu$ is obtained by removing a hook from $\widetilde\lambda.$
\end{Theorem}

\begin{Example}\label{ExampleSchubert2}
Considering Example \ref{ExampleSchubert},
Theorem \ref{singSchubert} implies that
$$\Sing(\Sigma_{(2^3,1)})=\Sigma_{(3^4,2)}\bigcup 
\Sigma_{2^5}.$$
Moreover, from Theorem \ref{singSchubert} follows that $\Sigma_{2^5}$ is smooth and $\Sing(\Sigma_{(3^4,2)})=\Sigma_{4^5},$ which is smooth as well.
\end{Example}

We can also determine the multiplicity of a Schubert variety along another Schubert variety.

\begin{Theorem}[Theorem 9.4.49 \cite{Lak00}]\label{schubertmult}
Set $\lambda=(\lambda_1,\dots,\lambda_{r+1})\leq
\mu=(\mu_1,\dots,\mu_{r+1}),$
$t=(t_1,\dots,t_{r+1}),s=(s_1,\dots,s_{r+1}),$
where $s_i=\# \{j;\mu_j-j<\lambda_i-i\},$ and $t_i=n-r+i-\lambda_i, i=1,\dots,r+1.$
Then $mult_{\Sigma_{\mu}} \Sigma_{\lambda}$ is the absolute value of
$$\det \begin{pmatrix}
\binom{t_1}{-s_1}&\dots &\binom{t_{r+1}}{-s_{r+1}}\\
\binom{t_1}{1-s_1}&\dots 
&\binom{t_{r+1}}{1-s_{r+1}}\\
\vdots&&\vdots\\
\binom{t_1}{r-s_1}&\dots 
&\binom{t_{r+1}}{r-s_{r+1}}
\end{pmatrix}.$$
\end{Theorem}

\begin{Example}\label{exampleSchubertmult2}
Lets compute $m:=\mult_{\Sigma_{(3^4,2)}} \Sigma_{(2^3,1)},$ see Example \ref{ExampleSchubert2}.
We have $r=4,n=9,$ and
$$\begin{cases}
\lambda&=(2,2,2,1,0)\\
\mu&=(3,3,3,3,2)\\
(\lambda_i-i)_i&=(1,0,-1,-3,-5)\\
(\mu_j-j)_j&=(2,1,0,-1,-3)\\
t&=(4,5,6,8,10)\\
s&=(3,2,1,0,0)
\end{cases}$$
Then, by Theorem \ref{schubertmult}
$$m=\left| \det \begin{pmatrix}
\binom{4}{-3}&\binom{5}{-2}&\binom{6}{-1}
&\binom{8}{0}&\binom{10}{0}\\
\binom{4}{-2}&\binom{5}{-1}&\binom{6}{0}
&\binom{8}{1}&\binom{10}{1}\\
\binom{4}{-1}&\binom{5}{0}&\binom{6}{1}
&\binom{8}{2}&\binom{10}{2}\\
\binom{4}{0}&\binom{5}{1}&\binom{6}{2}
&\binom{8}{3}&\binom{10}{3}\\
\binom{4}{1}&\binom{5}{2}&\binom{6}{3}
&\binom{8}{4}&\binom{10}{4}
\end{pmatrix}
\right|
=\det \left|\begin{pmatrix}
0&0&0
&1&1\\
0&0&1
&8&10\\
0&1&6
&\binom{8}{2}&\binom{10}{2}\\
1&5&\binom{6}{2}
&\binom{8}{3}&\binom{10}{3}\\
4&\binom{5}{2}&\binom{6}{3}
&\binom{8}{4}&\binom{10}{4}
\end{pmatrix}\right|
=14.$$
\end{Example}

\begin{Example}\label{exampleSchubertmult}
Consider 
$$\Sigma_{(n-2)^3}=R_0'\subset
\Sigma_{((n-3)^2,0)}=R_1' \subset
\Sigma_{(n-4,0,0)}=R_2'\subset 
\Sigma_{(1,0,0)}=D\subset \G(2,n).$$
Here $\lambda=(1,0,0)\leq \mu^j, j=0,1,2$
where $R_j'=\Sigma_{\mu^j}.$ We have
$ t=(n-1,n+1,n+2),$ and $s^0=(0,0,0),s^1=(1,1,0),s^2=(2,1,0).$
Therefore, by Theorem \ref{schubertmult}
\begin{align*}
\mult_{R_0'}D&=\left| \det \begin{pmatrix}
\binom{n-1}{0}&\binom{n+1}{0}&\binom{n+2}{0}\\
\binom{n-1}{1}&\binom{n+1}{1}&\binom{n+2}{1}\\
\binom{n-1}{2}&\binom{n+1}{2}&\binom{n+2}{2}
\end{pmatrix}
\right|=3,\\
\mult_{R_1'}D&=\left| \det \begin{pmatrix}
\binom{n-1}{-1}&\binom{n+1}{-1}&\binom{n+2}{0}\\
\binom{n-1}{0}&\binom{n+1}{0}&\binom{n+2}{1}\\
\binom{n-1}{1}&\binom{n+1}{1}&\binom{n+2}{2}
\end{pmatrix}
\right|=\left| \det \begin{pmatrix}
0&0&1\\
\binom{n-1}{0}&\binom{n+1}{0}&\binom{n+2}{1}\\
\binom{n-1}{1}&\binom{n+1}{1}&\binom{n+2}{2}
\end{pmatrix}
\right|=2,\\
\mult_{R_2'}D&=\left| \det \begin{pmatrix}
\binom{n-1}{-2}&\binom{n+1}{-1}&\binom{n+2}{0}\\
\binom{n-1}{-1}&\binom{n+1}{0}&\binom{n+2}{1}\\
\binom{n-1}{0}&\binom{n+1}{1}&\binom{n+2}{2}
\end{pmatrix}
\right|=\left| \det \begin{pmatrix}
0&0&1\\
0&1&\binom{n+2}{1}\\
\binom{n-1}{0}&\binom{n+1}{1}&\binom{n+2}{2}
\end{pmatrix}
\right|=1.\\
\end{align*}

\end{Example}

In the next section we will generalize this last example.
\section{Osculating spaces}\label{osculating spaces grass}

Set
$$\Lambda:=\left\{ I\subset \{0,\dots,n\}, |I|=r+1 \right\}.$$

For each $I=\{i_0,\dots,i_r\}\in \Lambda$ let $e_I\in\G(r,n)$ be the point corresponding to $e_{i_0}\wedge\dots\wedge e_{i_r}\in\bigwedge^{r+1} \C^{n+1}$.
Furthermore, we define a distance on $\Lambda$ as 
$$d(I,J)=|I|-|I\cap J|=|J|-|I\cap J|$$ 
for each $I,J\in\Lambda$. Note that, with respect to this distance, the diameter of $\Lambda$ is $r+1$.

In the following we give an explicit description of osculating spaces of Grassmannians at fundamental points.

\begin{Proposition}\label{oscgrass}
For any $s\geq 0$ we have 
$$T^s_{e_I}(\G(r,n))= \left\langle e_J \: | \: d(I,J)\leq s\right\rangle = \{p_J=0 \: | \: d(I,J)>s\}\subseteq\P^N.$$
In particular, $T^s_{e_I}(\G(r,n))=\P^N$ for any $s\geq r+1$.
\end{Proposition}
\begin{proof}
We may assume that $I\!=\!\{0,\dots,r\}$ and consider the usual parametrization of $\G(r,n):$ 
$$\phi:\C^{(r+1)(n-r)}\rightarrow \G(r,n)$$
given by 
$$
A=(a_{ij})=\begin{pmatrix}
1& \dots & 0 &a_{0,r+1}& \dots & a_{0n}\\
\vdots & \ddots & \vdots & \vdots &   \ddots &\vdots\\
0 & \dots & 1 & a_{r,r+1} & \dots & a_{rn}\\
\end{pmatrix}
\mapsto (\det (A_J))_{J\in \Lambda}
$$
where $A_J$ is the $(r+1)\times(r+1)$ matrix obtained from $A$ considering just the columns indexed by $J$.

Note that each variable appears in degree at most one in the coordinates of $\phi$. Therefore, differentiating two times with respect to the same variable always yields zero.

Thus, in order to describe the osculating spaces we may take into account just partial derivatives
with respect to different variables. Moreover, since the degree of $\det (A_J)$ with respect to $a_{i,j}$ is at most $r+1$ all partial derivatives of order greater or equal than $r+2$ are zero. Hence, it is enough to prove the proposition for $s\leq r+1$.

Given $J=\{j_0,\dots,j_r\}\subset \{0,\dots,n\}$, $k\in \{0,\dots,r\}$, and $k'\in \{r+1,\dots,n\}$ we have
$$ \dfrac{\partial \det (A_J)}{\partial a_{k,k'}}=\begin{cases}
0 &\mbox{ if } k'\notin J\\
(-1)^{l+1+k'} \det (A_{J,k,k'}) &\mbox{ if } k'=j_l
\end{cases}$$
where $A_{J,k,k'}$ denotes the submatrix of $A_{J}$ obtained deleting the line indexed by $k$ and the column indexed by $k'$. More generally, for any $m\geq 1$ and for any
\begin{align*}
J&=\{j_0,\dots,j_r\}\subset \{0,\dots,n\},\\
K'&=\{k_1',\dots,k_m'\}\subset \{r+1,\dots,n\}, \\
K&=\{k_1,\dots,k_m\}\subset \{0,\dots,r\}
\end{align*}
we have
$$ \dfrac{\partial^m \det (A_J)}{\partial a_{k_1,k_1'}\dots\partial a_{k_m,k_m'}}=
\begin{cases}
(\pm 1) \det (A_{J,(k_1,k_1'),\dots,(k_m,k_m')} )
&\mbox{ if }K'\!\subset\! J \mbox{ and } |K|\!=\!|K'|\!=\!m\leq d\\
0 &\mbox{ otherwise }
\end{cases}$$
where $d=d(J,\{0,\dots,r\})=\deg(\det(A_J))$. Therefore
$$ \dfrac{\partial^m \det (A_J)}{\partial a_{k_1,k_1'}\dots\partial a_{k_m,k_m'}}(0)=\begin{cases}
\pm 1 &\mbox{ if } J=K'\bigcup \left(\{ 0,\dots, r\}\backslash K\right)\\
0 &\mbox{ otherwise }
\end{cases}$$
and
$$\dfrac{\partial^m \phi}{\partial a_{k_1,k_1'}\dots\partial a_{k_m,k_m'}}(0)=
\pm  e_{K'\cup \left(\{ 0,\dots, r\}\setminus K\right)}.$$
Note that $d\left(K'\cup \left(\{ 0,\dots, r\}\backslash K\right),\{ 0,\dots, r\}\right)=m$, and that any $J$ with $d(J,\{0,\dots,r\})=m$ may be written in the form $K'\cup \left(\{ 0,\dots, r\}\backslash K\right)$.

Finally, we get that
$$\left\langle \dfrac{\partial^{|I|}\phi}{\partial^I a_{i,j}}(0)\: \big| \: |I|= m\right\rangle
=\left\langle e_J\: | \: d(J,\{0,\dots r\})=m\right\rangle,$$
which proves the statement.
\end{proof}

Notice that using Proposition \ref{oscgrass} one can describe the osculating space of any given point of the Grassmannian.

Now, it is easy to compute the dimension of the osculating spaces of $\G(r,n)$.

\begin{Corollary}
For any point $p\in \G(r,n)$ we have
$$\dim T^s_p \G(r,n)=\sum_{l=1}^s \binom{r+1}{l}\binom{n-r}{l}$$
for any $0\leq s\leq r$, while $T^s_p \G(r,n)=\P^N$ for any $s\geq r+1$.
\end{Corollary}
\begin{proof}
Since $\G(r,n)\subset\P^{N}$ is homogeneous under the action the algebraic subgroup 
$$\Stab(\G(r,n))\subset PGL(N+1)$$
stabilizing it, there exists an automorphism $\alpha\in PGL(N+1)$ inducing an automorphism of $\G(r,n)$ such that $\alpha(p) = e_I$. Moreover, since $\alpha\in PGL(N+1)$ we have that it induces an isomorphism between $T^s_p \G(r,n)$ and $T^s_{e_I} \G(r,n)$. Now, the computation of $\dim \G(r,n)$ follows, by standard combinatorial arguments, from Proposition \ref{oscgrass}.
\end{proof}

We end this section by describing the loci
$\G(r,n)\cap T^i_p\G(r,n)$ in two different ways.
These characterizations will not be used in Part \ref{part1}
to prove non-secant defectivity results
but may be interesting on their own.
%
%
 
\begin{Notation}\label{notation1}
Let $p\in \G(r,n)$ be a point and for any non negative integer $i\leq r+1$ consider the $i$-th osculating space $T^i_p \G(r,n)\subset \P^N$ of $\G(r,n)$ at $p.$ We denote by $R_i=R_i(p)$ the subvariety of the Grassmannian defined by
$$R_i:=\G(r,n)\bigcap T^i_p\G(r,n).$$
\end{Notation}

In particular $R_0=\{p\}.$ The variety
$$R_1=R=\G(r,n)\cap T_p\G(r,n)$$
will be used in the construction of the Mori chamber decomposition of $\G(1,n)_1,$ the blow-up of $\G(1,n)$ at one point,  in Chapter \ref{cap5}. However, we will not use the results of this section in Chapter \ref{cap5}, we rather make Chapter \ref{cap5} self contained giving simplified proofs.

The $R_i$ can be characterized in two more ways.

\begin{Lemma}\label{Rslemma}
Choose a complete Flag $\{0\}=V_0\subset V_1\subset \cdots \subset V_{n+1}=\C^{n+1}$ of linear spaces in $\C^{n+1}$,
with $V_{r+1}$ corresponding to the point $p\in \G(r,n).$ Define the following irreducible subvarieties of $\G(r,n):$ 
$$R_i'(p)=R_i':=\{ [U]\in \G(r,n); dim (U \cap V_{r+1})\geq r+1-i \}$$
for $i=0,1,\dots, r+1.$ 
Moreover, for any $1\leq i\leq r+1$ define
$$R_i''(p)=R_i'':=\!\!\!\!\!\!
\bigcup_{\substack{c \mbox{ \tiny{rational curve} }\\ p\in c\subset \G(r,n) \\ deg(c)=i}}\!\!\! \!\!\! c$$
as the locus of the degree $i$ rational curves contained in $\G(r,n)$ and passing through $p.$
Then $R_i=R_i'=R_i''$ for any $1\leq i\leq r+1.$
\end{Lemma}

\begin{proof}
First we prove that $R_{r+1}''=\G(r,n).$

We have to prove that any two points of $\G(r,n)$ can be connected by a degree $r+1$ rational curve.
Let $p,q$ be points in $\G(r,n)$ corresponding to linear $r$-spaces $V_p$ and $V_q$ in $\P^n.$
Take lines $L_0,\dots,L_r$ on $\P^n$ that intersects $V_p$ and $V_q$ and such that they generate a linear $(2r+1)$-space. 
Call $a=L_0\cap V_p$ and $b=L_0\cap V_q.$
Choose isomorphisms $\phi_i:L_0\to L_i$ for $i=1,\dots,r$ such that
$\phi_i(a)=L_i \cap V_p$ and $\phi_i(b)=L_i \cap V_q.$
Define the degree $r$ rational normal scroll 
$$S=\bigcup_{x\in L_0}\left\langle x,
\phi_1(x),\dots,\phi_r(x)
\right\rangle.$$
Then $c=\{[V]\in \G(r,n);V\subset S\}$ is rational normal curve on $\G(r,n)$
of degree $r=deg(S)$ and passing through $p$ and $q$ by the way we chose the $\phi_i$.

Next we prove that $R_i'=R_i''$ for $1\leq i\leq r+1.$ 

Let $q\in R_i'=R_i'(p).$
Then by the definition of $R_i'$ the spaces $V_p$ and $V_q$ have intersection with dimension at least $r-i.$ Therefore, there is a space $L\cong \P^{r+i}$ containing both. Thus, there is some subGrassmannian $\G(r,r+i)\cong \G(i-1,r+i)$ of $\G(r,n)$ such that
$p,q\in \G(i-1,r+i).$
Then, by the first part there is a rational normal curve on $\G(i-1,r+i)\subset \G(r,n)$
of degree $i$ passing through $p$ and $q,$ that is, $q\in R_i''.$ This proves $R_i'\subset R_i''.$

Conversely, let $q\in R_i'',$ then there is a rational normal curve $c$ of degree $i$
passing through $p$ and $q.$ Set $X=\bigcup_{[V]\in c} V\subset \P^n$.
Then $X$ is a rational normal scroll of degree $i$ and dimension $r+1.$ Let $L=\P^s$ be the span of $X$ in $\P^n.$
Therefore $X$ is a non degenerate subvariety of $L$ of minimal degree, more precisely
$deg(X)=dim(L)-dim(X)+1,$ that is, $i=s-(r+1)+1,\ s=r+i.$ For details 
about rational normal scrolls and varieties of minimal degree see \cite[Chapter 19]{Harris92}.
This means $V_p$ and $V_q$ are contained in $L\cong \P^{r+i},$
that is, $dim(V_p\cap V_q)\geq r- i,$ that is, $q\in R_i'.$ This proves that $R_i''\subset R_i'.$

Now it is enough to show that $R_i=R_i'.$ This follows from the description of $T^i_p$ in Lemma \ref{oscgrass}.
\end{proof}

These subvarieties $R_i$ have good properties.

\begin{Lemma} 
With the above notations we have
\begin{enumerate}
	\item $\{p\}=R_0\subsetneq R_1\subsetneq \cdots \subsetneq R_r \subsetneq R_{r+1}=\G(r,n).$
	\item $dim(R_i)=i(n+1-i)$ for $i=0,\dots, r+1.$ In particular, $R_i$ is a divisor of $\G(r,n)$ if and
	only if $i=r$ and $n=2r+1.$
	\item $Sing(R_i)=R_{i-1}$ for $i=1,\dots, r.$
\end{enumerate}
\end{Lemma}
\begin{proof}
Item $(1)$ is clear, for the proof of items $(2)$ and $(3)$ see Lemma \ref{keylemma}.
\end{proof}

Lemma \ref{keylemma} below generalizes the computation of Example \ref{exampleSchubertmult}.

\begin{Lemma}\label{keylemma}
With the above notations we have
\begin{enumerate}
\item $\dim(R_i')=i(n+1-i),i=0,\dots,r+1.$\\
\item $\Sing(R_i')=R_{i-1}', i=1,\dots,r.$\\
\item $\mult_{R_i'} D=r+1-i,i=0,\dots,r.$\\
\item If $W\subset \C^{r+1}$ is a $(r+1)$-dimensional vector space, then 
$$\mult _{[W]}D=\dim(W\cap V_{r+1}).$$
\end{enumerate}
\end{Lemma}
\begin{proof}
Item $(1)$ follows from the formula $\dim(\Sigma_\lambda)=|\widetilde\lambda|$ of the dimension of a Schubert variety.
Item $(2)$ follows directly from Theorem \ref{singSchubert}.
Item $(4)$ follows from item $(3)$ and the fact that  $[W]$ is a general point of $R_{r+1-\dim(W\cap V_{r+1})}'.$ We only need to prove item $(3).$

Let $0\leq i\leq r.$ 
The integer vectors of Theorem \ref{singSchubert} in our case are
$$\begin{cases}
\lambda&=(1,0^r)\\
(\lambda_i-i)_i&=(0,-2,-3,\dots,-r-1)\\
\mu&=((n-r-i)^{r+1-i},0^i)\\
(\mu_j-j)_j&=(n-r-i-1,\dots,
\underbrace{n-2r-1}_{\mbox{entry $r+1-i$} },\underbrace{-r-2+i}_{\mbox{entry $r+2-i$} },
\dots,-r-1)\\
t&=(n-r,n-r+2,n-r+3,\dots,n+1)\\
s&=(i,\dots,i,
\underbrace{i-1,\dots,1,0}_
{\mbox{$i$ entries total}})
\end{cases}$$
Now we only have to compute the determinant of 
$$M=\binom{t_l}{k-1-s_l}_{1\leq k,l\leq r+1}.$$
In order to do this we write $M$ in blocks
$$M=\begin{pmatrix}
M_1&M_2\\ M_3&M_4
\end{pmatrix},$$
where $M_1$ is of size $i\times(r+1-i),$ and the other three have the adequate size.
Note that $M_1=0$ because 
$k-1-s_l=k-1-i<0$ for $k=1,\dots,i,l=1,\dots,r+1-i.$
Moreover, $M_2$ is zero above its secondary diagonal 
(entries $(k,l)$ such that $k+l=r+2$) and have ones on it because
$$k-1-s_l=k-1-(r+1-l)=k+l-r-2.$$
Therefore $|\det(M)|=|\det(M_3)|.$ Next, note that
$$M_3=\begin{pmatrix}
\binom{n-r}{0}&\binom{n-r+2}{0}&
\dots&\binom{n+1-i}{0}\\
\binom{n-r}{1}&\binom{n-r+2}{1}&
\dots&\binom{n+1-i}{1}\\
&&&\\
\binom{n-r}{r-i}&\binom{n-r+2}{r-i}&
\dots&\binom{n+1-i}{r-i}\\
\end{pmatrix}$$
is a $r+1-i$ size Vandermonde square matrix.
Thus,
\begin{align*}
&\det(M_3)=
\dfrac{\prod_{1\leq q<p\leq r+1-i}(t_p-t_q)}
{1!2!\dots(r-i)!}\\
&=\frac{\prod_{p>q>1} (n-r+p-(n-r+q))\prod_{p>1} (n-r+p-(n-r))}{1!2!\cdots (r-i)!}\\
&=\frac{\prod_{p>q>1} (p-q)\prod_{p>1} p}{1!2!\cdots (r-i)!}=\frac{(r-1-i)!(r-2-i)!\cdots 1! (r+1-i)!}{1!2!\cdots (r-i)!}\\
&=\frac{1!2!\cdots(r-1-i)!\widehat{(r-i)!}(r+1-i)!}{1!2!\cdots (r-1-i)! (r-i)!}=r+1-i.
\end{align*}
\end{proof}

\section{Osculating projections}\label{projoscgrass}
In this section we study linear projections of Grassmannians from their osculating spaces. In order to help the reader get acquainted with the ideas of the proofs, we start by studying in detail projections from a single osculating space.

Let $0\leq s\leq r$ be an integer, and $I\in \Lambda$. By Proposition \ref{oscgrass} the projection of $\G(r,n)$ from $T_{e_I}^s$ is given by
\begin{align*}
\Pi_{T_{e_I}^s}:\G(r,n)&\dasharrow \P^{N_{s}}\\
(p_I)_{I\in \Lambda}
&\mapsto (p_J)_{J\in \Lambda \: |\: d(I,J)>s}
\end{align*}
Moreover, given $I'\!=\!\{i_0',\dots,i_s'\}\subset I$ with $|I'|\!=\! s+1$
we can consider the linear projection
\begin{align*}
\pi_{I'}:\P^n&\dasharrow \P^{n-s-1}\\
(x_i)
&\mapsto (x_i)_{i\in \{0,\dots,n\}\setminus I'}
\end{align*}
which in turn induces the linear projection
\begin{align*}
\Pi_{I'}:\G(r,n)&\dasharrow \G(r,n-s-1)\\
[V]&\mapsto [\pi_{I'}(V)]\\
(p_I)_{I\in \Lambda}
&\mapsto(p_J)_{J\in \Lambda \: | \: J\cap I'=\emptyset}
\end{align*}

Note that the fibers of $\Pi_{I'}$ are isomorphic to $\G(r,r+s+1)$. More precisely, let 
$y\in \G(r,n-s-1)$ be a point, and consider a general point $x\in \overline{\Pi_{I'}^{-1}(y)}\subset \G(r,n)$ 
corresponding to an $r$-plane $V_x\subset \P^n$. Then we have  
$$\overline{\Pi_{I'}^{-1}(y)}=\G\left(r,\left\langle  V_x, e_{i_0'},\dots,e_{i_s'}\right\rangle\right)\subset \G(r,n).$$

On the other hand, a priori it is not at all clear what are the fibers of $\Pi_{T_{e_I}^s}$. In general the image of $\Pi_{T_{e_I}^s}$ is very singular, and its fibers may not be connected. In what follows we study the general fiber of the map $\Pi_{T_{e_I}^s}$ by factoring it through several projections of type $\Pi_{I'}$. 

\begin{Lemma}\label{oscfactors}
If $s=0,\dots, r$ and $I'\subset I$ with $|I'|=s+1$, then the rational map $\Pi_{I'}$ factors through $\Pi_{T_{e_I}^s}$. Moreover, $\Pi_{T_{e_I}^r}=\Pi_I$.
\end{Lemma}
\begin{proof}
Since $J\cap I'=\emptyset\Rightarrow d(I,J)>s$ the center of $\Pi_{T_{e_I}^s}$ in contained in the center of $\Pi_{I'}$. Furthermore, if $s=r$ then $J\cap I=\emptyset\Leftrightarrow d(I,J)>r.$
\end{proof}

Now, we are ready to describe the fibers of $\Pi_{T_{e_I}^s}$ for $0\leq s\leq r$.

\begin{Proposition}\label{oscprojbirational}
The rational map $\Pi_{T_{e_I}^s}$ is birational for every $0\leq s\leq r-1$, and 
$$\Pi_{T_{e_I}^r}:\G(r,n)\dasharrow \G(r,n-r-1)$$ 
is a fibration with fibers isomorphic to $\G(r,2r+1)$.
\end{Proposition}
\begin{proof}
For the second part of the statement it is enough to observe that $\Pi_{T_{e_I}^r}=\Pi_I$. Now, let us consider the first claim. Since $\Pi_{T_{e_I}^s}$ factors through $\Pi_{T_{e_I}^{s-1}}$ it is enough to prove that $\Pi_{T_{e_I}^{r-1}}$ is birational. By Lemma \ref{oscfactors} for any $I_j = I\setminus\{i_j\}$, there exists a rational map $\tau_j$ such that the following diagram is commutative
\[
  \begin{tikzpicture}[xscale=3.5,yscale=-1.5]
    \node (A0_0) at (0, 0) {$\G(r,n)$};
    \node (A1_0) at (1, 1) {$\G(r,n-r)$};
    \node (A1_1) at (1, 0) {$W\subseteq\mathbb{P}^{N_s}$};
    \path (A0_0) edge [->,swap, dashed] node [auto] {$\scriptstyle{\Pi_{I_j}}$} (A1_0);
    \path (A1_1) edge [->, dashed] node [auto] {$\scriptstyle{\tau_j}$} (A1_0);
    \path (A0_0) edge [->, dashed] node [auto] {$\scriptstyle{\Pi_{T_{e_I}^{r-1}}}$} (A1_1);
  \end{tikzpicture}
\]
where $W=\overline{\Pi_{T_{e_I}^{r-1}}(\G(r,n))}$. Now, let $x\in W$ be a general point, and $F\subset\G(r,n)$ be the fiber of $\Pi_{T_{e_I}^{r-1}}$ over $x$. Set $x_j= \tau_j(x) \in \G(r,n-r)$, and denote by $F_j \subset \G(r,n)$ the fiber of $\Pi_{I_j}$ over $x_j$. Therefore 
\begin{equation}\label{intfib1}
F\subseteq \bigcap_{j=0}^{r} F_j.
\end{equation}
Now, note that if $y\in F$ is a general point corresponding to an $r$-plane $V_y\subset \P^n$ we have
$$F_j=\G(r,\left\langle V_y,e_{i_0},\dots,\widehat{e_{i_j}},\dots,e_{i_r}\right\rangle)$$ 
and hence
$$\bigcap_{j=0}^{r} F_j=\bigcap_{j=0}^{r}\G(r,\left\langle V_y,e_{i_0},\dots,\widehat{e_{i_j}},\dots,e_{i_r}\right\rangle)=\G(r,V_y)=\{y\}.$$
The last equality and (\ref{intfib1}) force $F=\{y\}$, and since we are working in characteristic zero $\Pi_{T_{e_I}^{r-1}}$ is birational.
\end{proof}

\begin{Remark}
We remark that Proposition \ref{oscprojbirational} was inspired by a problem of birational geometry, see Remark \ref{connection}.
\end{Remark}

Our next aim is to study linear projections from the span of several osculating spaces. In particular, we want to understand when such a projection is birational as we did in Proposition \ref{oscprojbirational} for the projection from a single osculating space.

Clearly, there are some natural numerical constraints regarding how many coordinate points of $\G(r,n)$ we may take into account, and the order of the osculating spaces we want to project from.

First of all, by Proposition \ref{oscprojbirational} the order of the osculating spaces cannot exceed $r-1$. Furthermore, since in order to carry out the computations, we need to consider just coordinate points of $\G(r,n)$ corresponding to linearly independent linear subspaces of dimension $r+1$ in $\mathbb{C}^{n+1}$ we can use at most 
$$\alpha:=\left\lfloor \dfrac{n+1}{r+1} \right\rfloor$$ 
of them.

Now, let us consider the points $e_{I_1},\dots,e_{I_\alpha}\in\G(r,n)$ where 
\begin{equation}\label{I1Ialpha}
I_1=\{0,\dots,r\},\dots,I_\alpha=\{(r+1)(\alpha-1),\dots,(r+1)\alpha-1\}\in \Lambda.
\end{equation}

Again by Proposition \ref{oscgrass} the projection from the span of the osculating spaces of $\G(r,n)$ 
of orders $s_1,\dots,s_l$ at the points $e_{I_1},\dots,e_{I_l}$ is given by
\begin{align*}
\Pi_{T_{e_{I_1},\dots,e_{I_l}}^{s_1,\dots,s_l}}:\G(r,n)&\dasharrow \P^{N_{s_1,\dots, s_l}}\\
(p_I)_{I\in \Lambda}&\mapsto (p_J)_{J\in \Lambda\: | \: d(I_1,J)>s_1,\dots,d(I_l,J)>s_l}
\end{align*}
whenever $\{J\in \Lambda\: | \: d(I_1,J)\leq s_1 \mbox{ or } \dots \mbox{ or } d(I_l,J)\leq s_l\}\neq \Lambda$, and $l\leq \alpha$.

Furthermore, for any 
$I_1'=\{i_0^1,\dots,i_{s_1}^1\}\subset I_1,\dots,I_{l}'=\{i_0^l,\dots,i_{s_l}^l\}\subset I_l$ we consider the projection
\begin{align*}
\pi_{I_1',\dots,I_l'}:\P^n&\dasharrow \P^{n-l-\sum_1^l s_i}\\
(x_i)_{i=0,\dots,n}
&\mapsto (x_i)_{i\in\{0,\dots,n\}\setminus (I_1'\cup\dots\cup I_l')}
\end{align*}
where $l\leq \alpha$ and $n-l-\sum_1^l s_i\geq r+1$. The map $\pi_{I_1',\dots,I_l'}$ in turn induces the projection
\begin{align*}
\Pi_{I_1',\dots,I_l'}:\G(r,n)&\dasharrow \G\left(r,n-l-\textstyle\sum_1^l s_i\right)\\
[V]&\mapsto [\pi_{I_1',\dots,I_l'}(V)]\\
(p_I)_{I\in \Lambda}
&\mapsto(p_J)_{J\in \Lambda\: |\: J\cap \left(I_1'\cup\dots\cup I_l' \right)=\emptyset}
\end{align*}

\begin{Lemma}\label{oscfactorsII}
Let $I_1,\dots,I_\alpha$ be as in (\ref{I1Ialpha}), $l,s_1,\dots,s_l$ be integers such that $0\leq s_j\leq r-1$, and $0<l\leq \min\{\alpha,n-r-1-\sum_i s_i\}$. Then for any $I_1'\!=\!\{i_0^1,\dots,i_{s_1}^1\}\!\subset\! I_1,\dots, I_l'\!=\!\{i_0^l,\dots,i_{s_l}^l\}\!\subset\! I_l$ with $|I'_j|=s_j+1$ the rational maps $\Pi_{T_{e_{I_1},\dots,e_{I_l}}^{s_1,\dots,s_l}}$ and $\Pi_{I_1',\dots,I_l'}$ are well-defined and the latter factors through the former.
\end{Lemma}
\begin{proof}
Note that $J\cap \left(I_1'\cup\dots\cup I_l'\right)=\emptyset$ yields $d(I_1,J)>s_1,\dots,d(I_l,J)>s_l$.
Note also that the $I_j$'s are disjoint since the $I_j'$'s are.
Furthermore, since $\sum(s_i+1)=l+\sum s_i\leq n-r-1$ and $n\geq 2r+1,$ there are at least $r+2$ elements in $\{0,\dots,n\}\backslash \left(I_1'\cup\dots\cup I_l' \right)$. If $k_1,\dots,k_{r+2}$ are such elements, then 
$$K_j:=\{k_1,\dots,\widehat{k_j},\dots,k_{r+2}\}\in \{J\in \Lambda\: | \: d(I_j,J)> s_j,\: j=1,\dots,l\}$$
for any $j=1,\dots,r+2$ forces 
$\{J\in \Lambda \: | \: d(I_1,J)\leq s_1 \mbox{ or } \dots \mbox{ or } d(I_l,J)\leq s_l\}\neq \Lambda$.
\end{proof}

Now, we are ready to prove the main result of this section.

\begin{Proposition}\label{oscprojbir}
Let $I_1,\dots,I_\alpha$ be as in (\ref{I1Ialpha}), $l,s_1,\dots,s_l$ be integers such that $0\leq s_j\leq r-1$, and $0<l\leq \min\{\alpha,n-r-1-\sum_i s_i\}$. Then the projection $\Pi_{T_{e_{I_1},\dots,e_{I_l}}^{s_1,\dots,s_l}}$ is birational.
\end{Proposition}
\begin{proof}
For any collection of subsets $I_i'\subset I_i$ with $|I_i'|= s_i+1$ set $I'=\bigcup_i I_i'$. By Lemma \ref{oscfactorsII} there exists a rational map $\tau_{I_1',\dots,I_l'}$ fitting in the following commutative diagram 
\[
  \begin{tikzpicture}[xscale=3.5,yscale=-1.5]
    \node (A0_0) at (0, 0) {$\G(r,n)$};
    \node (A1_0) at (1, 1) {$\G(r,n-l-\sum_1^l s_i)$};
    \node (A1_1) at (1, 0) {$W\subseteq\mathbb{P}^{N_{s_1,\dots s_l}}$};
    \path (A0_0) edge [->,swap, dashed] node [auto] {$\scriptstyle{\Pi_{I_1',\dots,I_l'}}$} (A1_0);
    \path (A1_1) edge [->, dashed] node [auto] {$\scriptstyle{\tau_{I_1',\dots,I_l'}}$} (A1_0);
    \path (A0_0) edge [->, dashed] node [auto] {$\scriptstyle{\Pi_{T_{e_{I_1},\dots,e_{I_l}}^{s_1,\dots,s_l}}}$} (A1_1);
  \end{tikzpicture}
\]
where $W= \overline{\Pi_{T_{e_{I_1},\dots,e_{I_l}}^{s_1,\dots,s_l}}(\G(r,n))}$. Now, let $x\in W$ be a general point, and $F\subset \G(r,n)$ be the fiber of $\Pi_{T_{e_{I_1},\dots,e_{I_l}}^{s_1,\dots,s_l}}$ over $x$. Set $x'= \tau_{I_1',\dots,I_l'}(x) \in \G(r,n-l-\sum_1^l s_i)$, and denote by
$$F_{I_1',\dots,I_l'} \subset \G(r,n)$$
the fiber of $\Pi_{I_1',\dots,I_l'}$ over $x'$. Therefore
\begin{equation}\label{inc1}
F\subseteq \bigcap_{I_1',\dots,I_l'} F_{I_1',\dots,I_l'}
\end{equation}
where the intersection runs over all the collections of subsets $I_i'\subset I_i$ with $|I_i'|= s_i+1$. Now, if $y\in F$ is a general point corresponding to an $r$-plane $V_y\subset \P^n$ we have
$$F_{I_1',\dots,I_l'}=\G\left(r,\left\langle V_y,e_j \: | \: j\in I'\right\rangle\right)$$ 
and hence
\begin{equation}\label{inc2}
\bigcap_{I_1',\dots,I_l'} F_{I_1',\dots,I_l'}=
\bigcap_{I_1',\dots,I_l'}\G\left(r,\left\langle V_y,e_j \: | \: j\in I'\right\rangle\right)=\G(r,V_y)=\{y\}
\end{equation}
where again the first intersection is taken over all the subsets $I_i'\subset I_i$ with $|I_i'|= s_i+1$. 

Finally, to conclude it is enough to observe that (\ref{inc1}) and (\ref{inc2}) yield $F=\{y\}$, and since we are working in characteristic zero $\Pi_{T_{e_{I_1},\dots,e_{I_l}}^{s_1,\dots,s_l}}$ is birational.
\end{proof}

In what follows we just make Proposition \ref{oscprojbir} more explicit. 

\begin{Corollary}\label{oscprojbirationalII}
Set $\alpha:=\left\lfloor \dfrac{n+1}{r+1} \right\rfloor$
and let $I_1,\dots,I_\alpha$ be as in (\ref{I1Ialpha}).
Then $\Pi_{T_{e_{I_1},\dots,e_{I_{\alpha-1}}}^{r-1,\dots,r-1}}$ is birational. Furthermore, if $n\geq r^2+3r+1$ then $\Pi_{T_{e_{I_1},\dots,e_{I_\alpha}}^{r-1,\dots,r-1}}$ is birational.

Now, set $r':=n-2-\alpha r$ and $r'':=\min\{n-3-\alpha (r-1),r-2\}$.
If $2r+1<n<r^2+3r+1$ then 
\begin{itemize}
\item[-] $r-1\geq r'\geq 0$ and $\Pi_{T_{e_{I_1},\dots,e_{I_{\alpha-1}},e_{I_\alpha}}^{r-1,\dots,r-1,r'}}$ is birational;
\item[-] $r''\geq 0$ and $\Pi_{T_{e_{I_1},\dots,e_{I_{\alpha-1}},e_{I_\alpha}}^{r-2,\dots,r-2,r''}}$ is birational.
\end{itemize}
\end{Corollary}
\begin{proof}
First we apply Proposition \ref{oscprojbir} with $l=\alpha-1$ and $s_1=\dots=s_{\alpha-1}=r-1$. In this case the constraint is $\alpha-1\leq n-r-1-(\alpha-1)(r-1)$, that is $\alpha\leq \dfrac{n-r-1}{r}+1$. Note that this is always the case since
$$\alpha\leq\dfrac{n+1}{r+1}\leq \dfrac{n-1}{r}=\dfrac{n-r-1}{r}+1.$$

If $l=\alpha$ and $s_1=\dots=s_{\alpha}=r-1$
the constraint in Proposition \ref{oscprojbir} is $\alpha\leq n-r-1-\alpha(r-1)$, which is equivalent to $\alpha\leq \dfrac{n-r-1}{r}$. Now, it is enough to observe that 
$$\dfrac{n+1}{r+1}\leq\dfrac{n-r-1}{r} \Longleftrightarrow n\geq r^2+3r+1.$$

If $n\geq r^2+3r+1,$ then the claim follows from the inequalities $\alpha\leq \dfrac{n+1}{r+1}\leq \dfrac{n-r-1}{r}$.

Now assume that $n< r^2+3r+1$. First we check that $r'=n-2-\alpha r\leq r-1,$ that is 
$\alpha \geq \dfrac{n-1-r}{r}.$ That follows from
$$\alpha\geq \dfrac{n+1}{r+1}-1= \dfrac{n-r}{r+1}
\geq\dfrac{n-r-1}{r}$$
whenever $n\geq 2r+1.$
Next we check that $r',r''\geq 0.$
If $2r+1<n< 3r+2$ then $\alpha=2$, and $r'=n-2-2r\geq 0$.
If $n\geq 3r+2$ we have
$$\alpha=\left\lfloor \dfrac{n+1}{r+1} \right\rfloor\leq \dfrac{n+1}{r+1}\leq \dfrac{n-2}{r}$$
and then $r'=n-2-\alpha r\geq 0$. Furthermore, note that 
$$\alpha=\left\lfloor \dfrac{n+1}{r+1} \right\rfloor\leq \dfrac{n+1}{r+1}\leq \dfrac{n-3}{r-1}$$
and then $r''=n-3-\alpha (r-1)\geq 0$.

Now, we apply Proposition \ref{oscprojbir} with $l=\alpha, s_1=\dots = s_{\alpha-1}=r-1$ and $s_\alpha=r'$.
In this case the constraint in Proposition \ref{oscprojbir} is 
$\alpha\leq  n-r-1-(\alpha-1)(r-1)-r'$
that is $r'\leq n-2-\alpha r$.

Finally, if $l=\alpha, s_1=\dots = s_{\alpha-1}=r-2$ and $s_\alpha=r''$, then the constraint in Proposition \ref{oscprojbir} is
$\alpha\leq  n-r-1-(\alpha-1)(r-2)-r''$, that is $r''\leq  n-3-\alpha(r-1)$.
\end{proof}

\section{Degenerating osculating spaces}\label{degtanosc}

We begin by studying how the span of two osculating spaces degenerates in a flat family of linear spaces parametrized by $\mathbb{P}^1$.

We recall that the Grassmannian $\G(r,n)$ is rationally connected by rational normal curves of degree $r+1$. Indeed, if $p,q\in\G(r,n)$ are general points, corresponding to the $r$-planes $V_p,V_q\subseteq\mathbb{P}^n$, we may consider a rational normal scroll $X\subseteq\mathbb{P}^n$ of dimension $r+1$ containing $V_p$ and $V_q$. Then the $r$-planes of $X$ correspond to the points of a degree $r+1$ rational normal curve in $\G(r,n)$ joining $p$ and $q$.

The first step consists in studying how the span of two osculating spaces at two general points $p,q\in\G(r,n)$ behaves when $q$ approaches $p$ along a degree $r+1$ rational normal curve connecting $p$ and $q$.

\begin{Proposition}\label{limitosculatingspacesgrass}
The Grassmannian $\G(r,n)$ has strong $2$-osculating regularity.
\end{Proposition}
\begin{proof}
Let $p,q\in\G(r,n)\subseteq\mathbb{P}^N$ be general points, and $k_1,k_2\geq 0$ integers.
We may assume that $k_1+k_2\leq r-1,$
otherwise $T^{k_1+k_2+1}_p=\P^N$ by Proposition \ref{oscgrass}.

Consider $\gamma,$ the degree $r+1$ rational normal curve connecting $p$ and $q$ as above.
We may assume that $p=e_{I_1},q=e_{I_2}$, 
see  (\ref{I1Ialpha}),
and that $\gamma:\P^1\to\G(r,n)$ is the rational normal curve given by 
$$\gamma([t:s])= (se_0+te_{r+1})\wedge\dots\wedge (se_r+te_{2r+1}).$$
We can work on the affine chart $s=1$ and set $t=(t:1)$. Consider the points $$e_0,\dots,e_n,
e_0^{t}=e_0+te_{r+1},\dots,e_r^{t}=e_r+te_{2r+1},e_{r+1}^{t}=e_{r+1},\dots,e_n^{t}=e_n\in\P^n$$
and the corresponding points of $\P^N$ 
$$e_I=e_{i_0}\wedge\dots\wedge e_{i_r},e_I^{t}=e_{i_0}^{t}\wedge\dots\wedge e_{i_r}^{t},\: I\in \Lambda.$$
By Proposition \ref{oscgrass} we have 
$$T_t:=\left\langle
T_{p}^{k_1},T_{\gamma(t)}^{k_2}
\right\rangle
=\left\langle e_I\: | \: d(I,I_1)\leq k_1; \: e_I^{t}\: |\: d(I,I_1)\leq k_2\right\rangle
,\: t\neq 0,$$
and
$$T^{k_1+k_2+1}_p=\left\langle e_I\: | \: d(I,I_1)\leq k_1+k_2+1\right\rangle
=\{p_I=0\: | \: d(I,I_1)> k_1+k_2+1\}.$$
It is enough to prove that 
$T_0\subset T^{k_1+k_2+1}_p.$
In order to prove this it is enough to exhibit, for any index $I\in \Lambda$ with $d(I,I_1)> k_1+k_2+1$, a hyperplane $H_I\subset\mathbb{P}^N$ of type
$$p_I+t\left( \sum_{J\in \Lambda, \ J\neq I}f(t)_{I,J} p_J\right)=0$$
such that $T_t\subset H_I$ for $t\neq 0$, where $f(t)_{I,J}\in \C[t]$ are polynomials. Clearly, taking the limit for $t\mapsto 0$, this will imply that $T_0\subseteq \{p_I = 0\}$.

In order to construct such a hyperplane we need to introduce some other definitions. We define
$$\Delta(I,l):= \left\{(I\setminus J)\cup (J+r+1)|\: J\subset I\cap I_1,\: |J|=l\right\}
\subset \Lambda$$
for any $I\in \Lambda,\: l\geq 0$, where $L+\lambda:=\{i+\lambda;\: i\in L\}$ is the translation of the set $L$ by the integer $\lambda$. Note that $\Delta(I,0)=\{I\}$ and $\Delta(I,l)=\emptyset$ for $l$ big enough. For any $l>0$ set
$$\Delta(I,-l):=\left\{ J|\: I\in\Delta(J,l) \right\} \subset \Lambda;$$
$$s_I^+:=\max_{l\geq 0}\{\Delta(I,l)\neq \emptyset\}\in\{0,\dots,r+1\};$$
$$s_I^-:=\max_{l\geq 0}\{\Delta(I,-l)\neq \emptyset\}\in\{0,\dots,r+1\};$$
$$\Delta(I)^+:=\bigcup_{0\leq l} \Delta(I,l)=
\!\!\!\!\bigcup_{0\leq l \leq s_I^+} \!\!\!\! \Delta(I,l);$$
$$\Delta(I)^-:=\bigcup_{0\leq l} \Delta(I,-l)=
\!\!\!\!\bigcup_{0\leq l \leq s_I^-} \!\!\!\!\Delta(I,-l).$$
Note that $0\leq s_I^-\leq d(I,I_1),0\leq s_I^+\leq r+1-d(I,I_1)$, and for any $l$ we have
$$J\in\Delta(I,l)\Rightarrow d(J,I)=|l|, d(J,I_1)=d(I,I_1)+l,d(J,I_2)=d(I,I_2)-l.$$
In order to get acquainted with the rest of the proof the reader may keep reading the proof taking a look to Example \ref{exampledeg} where we follow the same lines of the proof in the case $(r,n) = (2,5)$.

Now, we write the $e_I^t$'s with $d(I,I_1)< k_2,$ in the basis $e_J,J\in\Lambda$. For any $I\in\Lambda$ we have
\begin{align*}
e_{I}^{t}
&=e_I+t\!\!\!\!\sum_{J\in \Delta(I,1)}\!\!\!\!\left(\sign(J){e_J}\right)+\dots
+t^{l}\!\!\!\!\!\sum_{J\in \Delta(I,l)}\!\!\!\!\left(\sign(J){e_J}\right)+\dots
+t^{s_I^+}\!\!\!\!\!\!\!\!\sum_{J\in \Delta(I,s_I^+)}\!\!\!\!\!\left(\sign(J){e_J}\right)\\
&=\sum_{l=0}^{s_I^+}\left(t^l\!\!\!\!\sum_{J\in \Delta(I,l)}\sign(J) e_J\right)
=\!\!\!\!\sum_{J\in \Delta(I)^+}\!\!\!\!\left( t^{d(I,J)}\sign(J)e_J\right)
\end{align*}
where $\sign(J)=\pm 1$. Note that $\sign(J)$ depends on $J$ but not on $I$, hence we may replace $e_J$ by $\sign(J)e_J$, and write 
\begin{align*}
e_{I}^{t}=\sum_{J\in \Delta(I)^+}\!\!\!\!t^{d(I,J)}e_J.
\end{align*}
Therefore, we have
\begin{align*}
T_t=&\Big< e_I \: | \: d(I,I_1)\leq k_1;\:
\sum_{J\in \Delta(I)^+}\!\!\!\!\left( t^{d(I,J)}e_J\right) \: | \: d(I,I_1)\leq k_2
\Big >.
\end{align*}
Next, we define $$ \Delta:=\left\{ I \: | \:  d(I,I_1)\leq k_1\right\}\bigcup
\left(\bigcup_{d(I,I_1)\leq k_2} \!\!\!\!\Delta(I)^+\right)\subset \Lambda.$$
Let $I\in \Lambda$ be an index with $d(I,I_1)> k_1+k_2+1$. If $I\notin \Delta$ then $T_t\subset \{p_I=0\}$ for any $t\neq 0$ and we are done.

Now, assume that $I\in \Delta$. For any $e_K^t$ with non-zero Pl\"ucker coordinate $p_I$ we have $I\in \Delta(K)^+$, that is $K\in \Delta(I)^-$. Now, we want to find a hyperplane $H_I$ of type
\begin{align}\label{hyperplane}
F_I=\sum_{J\in \Delta(I)^-}t^{d(I,J)}c_J p_J =0
\end{align}
where $c_J\in\C$ with $c_I\neq 0$, and such that $T_t\subset H_I$ for $t\neq 0$. Note that then we can divide the equation by $c_I$, and get a hyperplane $H_I$ of the required type:
$$p_I+\frac{t}{c_I}\left(\sum_{J\in \Delta(I)^-, \ J\neq I}t^{d(J,I)-1} c_J p_J\right)=0$$
In the following we will write $s_I^-=s$ for short. Since 
$$\left|\Delta(I)^-\right|
=\sum_{l=0}^{s}\left|\Delta(I,-l)\right|
=1+s+\binom{s}{2}+\dots+\binom{s}{s -1}+1=2^{s}$$
in equation (\ref{hyperplane}) there are $2^s$ variables $c_J$. Now, we want to understand what conditions we get by requiring $T_t\subseteq \{F_I=0\}$ for $t\neq 0$.

Given $K\in \Delta(I)^-$ we have $s_K^+\geq d(I,K)$ and
\begin{align*}
F_I(e_K^t)&=
F\left(\sum_{L\in \Delta(K)^+}\!\!\!\!\left( t^{d(K,L)}e_L\right)\right)
=F\left(\sum_{l=0}^{s_K^+}\left(t^l\!\!\!\!\sum_{L\in \Delta(K,l)} e_L\right)\right)
=F\left(\sum_{l=0}^{d(I,K)}\left(t^l\!\!\!\!\sum_{L\in \Delta(K,l)} e_L\right)\right)\\
&\stackrel{(\ref{hyperplane})}{=}\sum_{J\in \Delta(I)^-\cap \Delta(K)^+}t^{d(I,K)-d(J,K)}c_J 
\left(  t^{d(J,K)}\right)
=t^{d(I,K)}\left[\sum_{J\in \Delta(I)^-\cap \Delta(K)^+}c_J \right]
\end{align*}
that is 
\begin{align*}
F_I(e_K^t)=0\ \forall t\neq 0 \Leftrightarrow \sum_{J\in \Delta(I)^-\cap \Delta(K)^+}\!\!\!\!\!\!c_J =0.
\end{align*}
Note that this is a linear condition on the coefficients $c_J$, with $J\in \Delta(I)^-$. Therefore,
\begin{align}\label{eqns3}
T_t\subset \{F_I=0\} \mbox{ for } t\neq 0 &\Leftrightarrow 
\begin{cases}
F_I(e_L)=0\   &\forall L\in \Delta(I)^-\cap B[I_1,k_1]  \\ 
F_I(e_K^t)=0 \ \forall t\neq 0 \  &\forall K\in \Delta(I)^-\cap B[I_1,k_2] 
\end{cases}\\&\nonumber
\Leftrightarrow 
\begin{cases}
c_L=0  &\forall L\in \Delta(I)^-\cap B[I_1,k_1]\\
\displaystyle\sum_{J\in \Delta(I)^-\cap \Delta(K)^+}\!\!\!\!\!\!\!\!\!\!\!\!c_J \quad
=0 \ &\forall K\in \Delta(I)^-\cap B[I_1,k_2]
\end{cases}
\end{align}
where $B[J,u]:=\{K\in \Lambda |\: d(J,K)\leq u\}$. The number of conditions on the $c_J$'s, $J\in \Delta(I)^-$ is then
$$c:=\left|\Delta(I)^-\cap B[I_1,k_1]\right|+\left|\Delta(I)^-\cap B[I_1,k_2]\right|.$$

The problem is now reduced to find a solution of the linear system given by the $c$ equations (\ref{eqns3}) in the $2^s$ variables $c_J$'s, $J\in \Delta(I)^-$ such that $c_I\neq 0$. Therefore, it is enough to find $s+1$ complex numbers $c_I=c_0\neq 0,c_1,\dots, c_{s}$ satisfying the following conditions 
\begin{align}\label{eqns4}
\begin{cases}
c_{j}=0  &\forall j=s,\dots,d-k_1\\
\displaystyle\sum_{l=0}^{d(I,K)}\left|\Delta(I)^-\cap \Delta(K,l)\right|c_{d(I,K)-l} =0 
\ &\forall K\in \Delta(I)^-\cap B[I_1,k_2]
\end{cases}
\end{align}
where $d=d(I,I_1)>k_1+k_2+1$. Note that (\ref{eqns4}) can be written as
\begin{align*}
\begin{cases}
c_{j}=0  &\forall j=s,\dots,d-k_1\\
\displaystyle\sum_{k=0}^{j}\binom{j}{j-k}c_{k} =0 
\ &\forall j=s,\dots,d-k_2
\end{cases}
\end{align*}
that is
\begin{align}\label{eqns5}
\begin{cases}
c_{s}=0\\
\vdots\\
c_{d-k_1}=0\\
\end{cases}
\begin{cases}
\binom{s}{0}c_{s}+\binom{s}{1}c_{s-1}+\cdots+\binom{s}{s-1}c_{1}+\binom{s}{s}c_{0}=0\\
\vdots\\
\binom{d-k_2}{0}c_{d-k_2}+\binom{d-k_2}{1}c_{d-k_2-1}+\cdots+\binom{d-k_2}{d-k_2-1}c_{1}+\binom{d-k_2}{d-k_2}c_{0}=0
\end{cases}
\end{align}
Now, it is enough to show that the linear system (\ref{eqns5}) admits a solution with $c_0\neq 0$. If $s<d-k_2,$ the system (\ref{eqns5}) reduces to $c_s=\dots=c_{d-k_1}=0$. In this case we may take $c_0=1, c_1=\dots,c_s=0$. Note that $d-k_1>k_2+1\geq 1$ and we can use the hyperplane $p_I=0$.

From now on assume that $s\geq d-k_2$. Since $c_s=\dots=c_{d-k_1}=0$ we may consider the second set of conditions in (\ref{eqns5}) and translate the problem into checking that the system (\ref{eqns6}) admits a solution involving the variables $c_0,c_1,\dots,c_{d-k_1+1}$ with $c_0\neq 0$. Note that (\ref{eqns5}) takes the following form:
\begin{align}\label{eqns6}
\begin{cases}
\binom{s}{s-(d-k_1+1)}c_{d-k_1+1}+\binom{s}{s-(d-k_1)}c_{d-k_1}+
\cdots+\binom{s}{s-1}c_{1}+\binom{s}{s}c_{0}=0\\
\vdots\\
\binom{d-k_2}{k_1-1-k_2}c_{d-k_1+1}+\binom{d-k_2}{k_1-k_2}c_{d-k_1}+
\cdots+\binom{d-k_2}{d-k_2-1}c_{1}+\binom{d-k_2}{d-k_2}c_{0}=0
\end{cases}
\end{align}
Therefore, it is enough to check that the $(s-d+k_2+1)\times (d-k_1+1)$ matrix
\begin{align}\label{eqns7}
M=
\begin{pmatrix}
\binom{s}{s-(d-k_1+1)}&\binom{s}{s-(d-k_1)}&
\cdots&\binom{s}{s-1}\\
\vdots&\vdots& &\vdots\\
\binom{d-k_2}{k_1-1-k_2}&\binom{d-k_2}{k_1-k_2}&
\cdots&\binom{d-k_2}{d-k_2-1}
\end{pmatrix}
\end{align}
has maximal rank. Note that $s\leq d$ and $d>k_1+k_2+1$ yield $s-d+k_2+1< d-k_1+1$. Then it is enough to show that the $(s-d+k_2+1)\times (s-d+k_2+1)$ submatrix
\begin{align*}
M'=&
\begin{pmatrix}
\binom{s}{s-(s-d+k_2+1)}&\binom{s}{s-(s-d+k_2)}&
\cdots&\binom{s}{s-1}\\
\vdots&\vdots& &\vdots\\
\binom{d-k_2}{d-k_2-(s-d+k_2+1)}&\binom{d-k_2}{d-k_2-(s-d+k_2)}&
\cdots&\binom{d-k_2}{d-k_2-1}
\end{pmatrix}\\
=&
\begin{pmatrix}
\binom{s}{d-k_2-1}&\binom{s}{d-k_2}&
\cdots&\binom{s}{s-1}\\
\vdots&\vdots& &\vdots\\
\binom{d-k_2}{2d-2k_2-s-1}&\binom{d-k_2}{2d-2k_2-s}&
\cdots&\binom{d-k_2}{d-k_2-1}
\end{pmatrix}=
\begin{pmatrix}
\binom{s}{s+1-d+k_2}&\binom{s}{s-d+k_2}&
\cdots&\binom{s}{1}\\
\vdots&\vdots& &\vdots\\
\binom{d-k_2}{s+1-d+k_2}&\binom{d-k_2}{s-d+k_2}&
\cdots&\binom{d-k_2}{1}
\end{pmatrix}
\end{align*}
has non-zero determinant. Since the determinant of $M'$ is equal to the determinant of the matrix of binomial coefficients 
$$M'':=\left(\binom{i}{j}\right)_{\substack{ \hspace{-0.7cm}d-k_2\leq i\leq s\\ 1\leq j\leq s+1-d+k_2}}.$$
it is enough to observe that since $d-k_2>k_1+1\geq 1$ by \cite[Corollary 2]{GV85} we have $\det(M')=\det(M'')\neq 0$.
\end{proof}
In the following example we work out explicitly the proof of Proposition \ref{limitosculatingspacesgrass}.

\begin{Example}\label{exampledeg}
Consider the case $(r,n)=(2,5).$ Then $I_1=\{0,1,2\}, I_2=\{3,4,5\}$. Let us take 
$$I^1=\{0,1,2\},I^2=\{0,1,3\},I^3=\{0,4,5\}.$$
Then we have 
$$
\begin{array}{ll}
\Delta(I^1,1)=\{\{1,2,3\},\{0,2,4\},\{0,1,5\}\} & \Delta(I^3,1)=\{\{3,4,5\}\}\\
\Delta(I^1,2)=\{\{0,4,5\},\{1,3,5\},\{2,3,4\}\} & \Delta(I^3,-1)=\{\{0,1,5\},\{0,2,4\}\}\\
\Delta(I^1,3)=\{\{3,4,5\}\} &\Delta(I^3,-2)=\{\{0,1,2\}\} \\
\Delta(I^2,1)=\{\{0,3,4\}\} &
\end{array}
$$
and $\Delta(I^j,l)=\emptyset$ for any other pair $(j,l)$ with $1\leq j\leq 3$ and $l\neq 0$. Therefore 
$$s^+_{I^1}=3,s^-_{I^1}=0,s^+_{I^2}=1,s^-_{I^2}=0,s^+_{I^3}=1,s^-_{I^3}=2$$
while
$$d(I^1,I_1)=0,d(I^2,I_1)=1,d(I^3,I_1)=2.$$
Let us work out the case $k_1=0,k_2=1$. Here $T_p^{k_1}$ is just the point $e_{012}$ and the generators of $T^{k_2}_{\gamma(t)}$ are
$$e_{012}^t,e_{123}^t,e_{024}^t,e_{015}^t,
e_{124}^t,e_{125}^t,
e_{023}^t,e_{025}^t,
e_{013}^t,e_{014}^t
$$
We can write them on the basis $(e_I)_{I\in \Lambda}$ as
\begin{align}\label{etonthebasisexampleI}
\begin{cases}
e_{012}^t&=e_{012}+t(e_{123}+e_{024}+e_{015})+t^2(e_{045}+e_{135}+e_{234})+t^3e_{345}\\
e_{123}^t&=e_{123}+t(e_{135}+e_{234})+t^2e_{345}\\
e_{024}^t&=e_{024}+t(e_{045}+e_{234})+t^2e_{345}\\
e_{015}^t&=e_{015}+t(e_{045}+e_{135})+t^2e_{345}\\
\end{cases}
\end{align}
and
\begin{align}\label{etonthebasisexampleII}
\begin{cases}
e_{124}^t&=e_{124}+te_{145}\\
e_{125}^t&=e_{125}+te_{245}\\
e_{023}^t&=e_{023}+te_{035}\\ 
e_{025}^t&=e_{025}+te_{235}\\
e_{013}^t&=e_{013}+te_{034}\\ 
e_{014}^t&=e_{014}+te_{134}
\end{cases}
\end{align}
Now, given $I\in \Lambda$ with $d(I,I_1)>2=k_1+k_2+1$ we have to find a hyperplane $H_I$ of type
$$c_Ip_I+t \!\!\!\!\sum_{J\in \Delta(I,-1)}\!\!\!\! c_J p_J+
t^2 \!\!\!\!\sum_{J\in \Delta(I,-2)}\!\!\!\! c_J p_J+
t^3 \!\!\!\!\sum_{J\in \Delta(I,-3)}\!\!\!\! c_J p_J=0$$
such that $c_I\neq 0$, and $T_t\subseteq H_I$ for every $t\neq 0.$

In this case it is enough to consider $I=\{3,4,5\}$. Note that $e_{345}$ appears in (\ref{etonthebasisexampleI}) but does not in (\ref{etonthebasisexampleII}). In the notation of the proof of Proposition \ref{limitosculatingspacesgrass} we have $d=s=3, d-k_1=3,d-k_2=2$, and we are looking for
$$c_0=c_{345}\neq 0, c_1=c_{045}=c_{135}=c_{234}, c_2=c_{123}=c_{024}=c_{015},c_3=c_{012}$$
satisfying the following system:
\begin{align}\label{eqnsexample}
\begin{cases}
c_{3}=0\\
\binom{3}{0}c_{3}+\binom{3}{1}c_{2}+\binom{3}{2}c_{1}+\binom{3}{3}c_{0}=0\\
\binom{2}{0}c_{2}+\binom{2}{1}c_{1}+\binom{2}{2}c_{0}=0
\end{cases}
\end{align}
Note that the matrix
\begin{equation*}
M=
\begin{pmatrix}
\binom{3}{1}&\binom{3}{2}\\
\binom{2}{0}&\binom{2}{1}
\end{pmatrix}=
\begin{pmatrix}
3&3\\
1&2
\end{pmatrix}
\end{equation*}
has maximal rank. Therefore, there exist complex numbers $c_I=c_0\neq 0,c_1,c_2,c_3$ satisfying system (\ref{eqnsexample}). For instance, we may take $c_0=3,c_1=-2,c_2=1,c_3=0$ corresponding to the hyperplane  
$$3p_{345}-2t(p_{045}+p_{135}+p_{234})+t^2(p_{123}+p_{024}+p_{015})=0$$
and taking the limit for $t\mapsto 0$ we get the equation $p_{345}=0$.
\end{Example}
Next, we see how the span of several osculating spaces on Grassmannian degenerate.

\begin{Proposition}\label{limitosculatingspacesgrassII}
The Grassmannian $\G(r,n)$ has $\lfloor \frac{n+1}{r+1}\rfloor$-osculating regularity.
\end{Proposition}
\begin{proof}
Let $p_1,\dots,p_{\alpha}\in \G(r,n)\subseteq\mathbb{P}^{N}$ be general points with $\alpha = \lfloor\frac{n+1}{r+1}\rfloor$,
$k\leq (r-1)/2$ a non-negative integer, and $\gamma_j:\P^1\to \G(r,n)$
a degree $r+1$ rational normal curve with $\gamma_j(0)=p_1$ and $\gamma_j(\infty)=p_j,$ for every $j=2,\dots,\alpha.$
Let us consider the family of linear spaces 
$$T_t=\left\langle T^{k}_{p_1},T^{k}_{\gamma_2(t)},\dots,T^{k}_{\gamma_{\alpha}(t)}
\right\rangle,\: t\in \P^1\backslash \{0\}$$
parametrized by $\P^1\backslash\{0\}$, and let $T_0$ be the flat limit of $\{T_t\}_{t\in \P^1\backslash \{0\}}$ in
$\G(\dim(T_t),N)$.
We have to show that $T_0\subset T^{2k+1}_p.$

If $\alpha=2$ it follows from the Proposition \ref{limitosculatingspacesgrass}. Therefore, we may assume that $\alpha\geq 3$, $p_j=e_{I_j}$ (\ref{I1Ialpha}) and that $\gamma_j:\P^1\to\mathbb{P}^{N}$ is the rational curve given by 
$$\gamma_j([t:s])=\left(se_0+te_{(r+1)(j-1)}\right)\wedge\dots \wedge \left(se_{r}+te_{(r+1)j-1}\right).$$
We can work on the affine chart $s=1$ and set $t=(t:1)$. Consider the points 
$$e_0,\dots,e_{n},
e_0^{j,t}=e_0+te_{(r+1)(j-1)},
\dots,e_{r}^{j,t}=e_{r}+te_{(r+1)j-1},e_{r+1}^{j,t}=e_{r+1},\dots,e_{n}^{j,t}=e_{n}\in \P^{n}$$
and the corresponding points in $\mathbb{P}^{N}$ 
$$e_I=e_{i_0}\wedge\dots\wedge e_{i_{r}},e_I^{j,t}=e_{i_0}^{j,t}\wedge\dots\wedge e_{i_{r}}^{j,t},\:
I=\{i_0,\dots,i_r\}\in \Lambda,$$
for $j=2,\dots,\alpha.$
By Proposition \ref{oscgrass} we have 
$$T_t=\left\langle e_I\: | \: d(I,I_1)\leq k; \: e_I^{j,t}\: |\: d(I,I_1)\leq k, j=2,\dots,\alpha\right\rangle
,\: t\neq 0$$
and
$$T^{2k+1}_{p_0}=\left\langle e_I\: | \: d(I,I_1)\leq 2k+1\right\rangle
=\{p_I=0\: | \: d(I,I_1)> 2k+1 \}.$$
Therefore, as in Proposition \ref{limitosculatingspacesgrass}, in order to prove that $T_0\subset T^{2k+1}_p$ it is enough to exhibit, for any index $I\in \Lambda$ with $d(I,I_1)> 2k+1$, a hyperplane $H_I\subset\mathbb{P}^{N}$ of type
$$p_I+t\left( \sum_{J\in \Lambda, \ J\neq I}f(t)_{I,J} p_J\right)=0$$
such that $T_t\subset H_I$ for $t\neq 0$, where $f(t)_{I,J}\in \C[t]$ are polynomials. The first part of the proof goes as in the proof of Proposition \ref{limitosculatingspacesgrass}. Given $I\in \Lambda$ we define 
$$\Delta(I,l)_j:=\left\{(I\setminus J)\cup(J+(j-1)(r+1))| J\subset I\cap I_1, |J|=l\right\}\subset \Lambda$$
for any $I\in \Lambda,l\geq 0, j=2,\dots, \alpha,$ where $L+\lambda:=\{i+\lambda; i\in L\}$ is the 
translation of the set $L$ by the integer $\lambda.$ Note that $\Delta(I,0)_j=\{I\}$ and $\Delta(I,l)_j=\emptyset$ for $l$ big enough. For any $l>0$ set
$$\Delta(I,-l)_j:=\left\{ J|\: I\in\Delta(J,l)_j \right\} \subset \Lambda;$$
$$s(I)^+_j:=\max_{l\geq 0}\{\Delta(I,l)_j\neq \emptyset\}\in\{0,\dots,r+1\};$$
$$s(I)^-_j:=\max_{l\geq 0}\{\Delta(I,-l)_j\neq \emptyset\}\in\{0,\dots,r+1\};$$
$$\Delta(I)^+_j:=\bigcup_{0\leq l} \Delta(I,l)_j=
\!\!\!\!\bigcup_{0\leq l \leq s(I)^+_j} \!\!\!\! \Delta(I,l)_j;$$
$$\Delta(I)^-_j:=\bigcup_{0\leq l} \Delta(I,-l)_j=
\!\!\!\!\bigcup_{0\leq l \leq s(I)^-_j} \!\!\!\!\Delta(I,-l)_j.$$
Note that $0\leq s(I)^-_j\leq d(I,I_1),0\leq s(I)^+_j\leq r+1-d(I,I_1)$, and for any $l$ we have
$$J\in\Delta(I,l)_j\Rightarrow d(J,I)=|l|, d(J,I_1)=d(I,I_1)+l,d(J,I_j)=d(I,I_j)-l.$$
 Now, we write $e_I^{j,t}$, $d(I,I_1)< k$, in the basis $e_J,J\in\Lambda$. For any $I\in\Lambda$ we have
\begin{align*}
e_{I}^{j,t}=\!\!\!\!\sum_{J\in \Delta(I)^+_j}\!\!\!\!\left( t^{d(I,J)}\sign(J)e_J\right)
\end{align*}
where $\sign(J)=\pm 1$. Since $\sign(J)$ does depend on $J$ but not on $I$ we can replace $e_J$ by $\sign(J)e_J$. Then, we may write 
\begin{align*}
e_{I}^{t}=\sum_{J\in \Delta(I)^+_j}\!\!\!\!\left( t^{d(I,J)}e_J\right).
\end{align*}
and
\begin{align*}
T_t=&\left\langle e_I \: | \: d(I,I_1)\leq k;\:
\sum_{J\in \Delta(I)^+_j}\!\!\!\!\left( t^{d(I,J)}e_J\right) \: \big| \: d(I,I_1)\leq k,\: 2\leq j\leq \alpha\right \rangle.
\end{align*}
Next, we define 
$$ \Delta:=\left\{ I \: | \:  d(I,I_1)\leq k\right\}\bigcup
\left(\bigcup_{2\leq j\leq \alpha}\bigcup_{d(I,I_1)\leq k} \!\!\!\!\Delta(I)^+_j\right)\subset \Lambda.$$
Let $I\in \Lambda$ be an index with $d(I,I_1)=:D> 2k+1$. If $I\notin \Delta$ then $T_t\subset \{p_I=0\}$ for any $t\neq 0$ and we are done. Now, assume that $I\in \Delta$, and $I\in \Delta(K_1,l_1)_2 \bigcap \Delta(K_2,l_2)_3$ with 
$$d(K_1,I_1),d(K_2,I_1)\leq k.$$
Consider the following sets
\begin{align*}
I^0:&=I\cap I_1\\
I^1:&=I\cap(K_1+(r+1))\subset I_2\\
I^2:&=I\cap(K_2+2(r+1))\subset I_3\\
I^3:&=I\setminus (I^0\cup I^1 \cup I^2)
\end{align*}
Then $|I^1|=l_1,|I^2|=l_2$. Set $u:=|I^3|$, then
$$d(I,I_1)=l_1+l_2+u\leq l_1+l_2+2u= d(K_1,I_1)+d(K_2,I_1)\leq 2k$$
contradicting $d(I,I_1)>2k+1$. Therefore, there is a unique $j$ such that 
$$I\in \bigcup_{d(J,I_1)\leq k} \!\!\!\!\Delta(J)^+_j.$$
Note that $\Delta(I,-s(I)^-_j)$ has only one element, say $I'$. Then 
$$k+1-D+s(I)_j^-=k+1-d(I,I_1)+d(I,I')=k+1-d(I_1,I')>0.$$
Now, consider the set of indexes
$$\Gamma:=\left\{ I \right\} \cup \Delta(I,-1)_j \cup \dots \cup \Delta(I,-(k+1-D+s(I)^-_j))_j
=\!\!\!\!\!\!\!\!\bigcup_{0\leq l \leq k+1-D+s(I)^-_j} \!\!\!\!\!\!\!\!\Delta(I,-l)_j \subset \Lambda$$
Our aim now is to find a hyperplane of the form
\begin{equation}\label{eq1}
H_I= \left \{\sum_{J\in \Gamma } t^{d(I,J)}c_J p_J=0 \right\}
\end{equation}
such that $T_t\subset H_I$ and $c_I\neq 0$.

First, we claim that 
\begin{equation}\label{eq2}
J\in \Gamma\Rightarrow J \notin  \bigcup_{\substack{2\leq i\leq \alpha \\ i\neq j} }
\bigcup_{d(I,I_1)\leq k} \!\!\!\!\Delta(I)^+_i.
\end{equation}
Indeed, assume that $J\in \Delta(I,-l)_j\cap\Delta(K,m)_i$ for some $K\in \Lambda$ with 
$$d(K,I_1)\leq k \mbox{ and } i\neq j, 0\leq l \leq k+1-D+s(I)^-_j, m\geq 0.$$
Since $J\in \Delta(I,-l)_j$ then
\begin{equation*}
|J\cap I_j|=|I\cap I_j|-l\geq s(I)_j^- -l\geq D-(k+1)>k
\end{equation*}
On other hand, since $J\in \Delta(K,m)_i$ with $j\neq i$ we have
\begin{equation*}
|J\cap I_j|=|K\cap I_j|\leq d(K,I_1)\leq k.
\end{equation*}
A contradiction. Now, (\ref{eq2}) yields that the hyperplane $H_I$ given by (\ref{eq1}) is such that 
$$\left\langle e_I\: | \: d(I,I_1)\leq k; \: \sum_{J\in \Delta(I)^+_i}\!\!\!\!
t^{d(I,J)} c_{(I,J)}e_J\: | \: d(I,I_1)\leq k,\ i=2,\dots,\alpha, i\neq j \right\rangle \subset H_I , t\neq 0.$$
Therefore
$$
T_t\subset H_I, t \neq 0 \Longleftrightarrow \left\langle \sum_{J\in \Delta(I)^+_j}\!\!\!\!
t^{d(I,J)} e_J\: | \: d(I,I_1)\leq k \right\rangle \subset H_I , t\neq 0
.$$
Now, arguing as in the proof of Proposition \ref{limitosculatingspacesgrass} we obtain
\begin{equation}\label{eq4II}
T_t\subset H_I, t \neq 0 \Longleftrightarrow \sum_{J\in \Delta(K)^+_j \cap \Gamma}\!\!\!\! c_{J}=0
\ \ \forall  K\in \Delta(I)^-_j\cap B[I_1,k]
\end{equation}
and the problem is now reduced to find a solution of the linear system given by the 
$|\Delta(I)^-_j\cap B[I_1,k]|$ equations (\ref{eq4II}) in the $|\Delta(K)^+_j \cap \Gamma|$ variables $c_J$, $J\in\Delta(K)^+_j \cap \Gamma$, such that $c_I\neq 0$. We set $c_J=c_{d(I,J)}$ and, as in the proof of Proposition \ref{limitosculatingspacesgrass}, we consider the linear system 
\begin{equation}\label{linsys}
\displaystyle\sum_{l=0}^{k+1-D+s(I)^-_j}\binom{D-i}{D-l-i}c_{l} =0 
\quad \forall i=D-s(I)^-_j,\dots,k
\end{equation}
with $k+2-D+s(I)^-_j$ variables $c_0,\dots,c_{k+1-D+s(I)^-_j}$ and $k+1-D+s(I)^-_j$ equations, where $D=d(I,I_1)$. Finally, arguing exactly as in the last part of the proof of Proposition \ref{limitosculatingspacesgrass} we have that (\ref{linsys}) admits a solution with $c_0\neq 0$.
\end{proof}

\section{Non-secant defectivity of Grassmannian varieties}\label{grassnodef}
In this section we put together the results of sections \ref{mainsection}, \ref{projoscgrass} and \ref{degtanosc} to study the dimension of secant varieties of Grassmannians.
First we state our main result on non-defectivity of Grassmannians, next we discuss the asymptotic behavior of the bounds, and finally we give some examples and a corollary. We then conclude showing that our bounds improve those in \cite{AOP09b} for $r\geq 2$, although we note that Abo, Ottaviani and Peterson have given in \cite{AOP09b} much better bounds for $r=2.$

\begin{Theorem}\label{maingrass}
Assume that $r\geq 2$, set 
$$\alpha:=\left\lfloor \dfrac{n+1}{r+1} \right\rfloor$$
and let $h_\alpha$ be as in Definition \ref{defhowmanytangent}. If either
\begin{itemize}
	\item[-] $n\geq r^2+3r+1$ and $h\leq\alpha h_{\alpha}(r-1)$ or
	\item[-] $n< r^2+3r+1$, $r$ is even, and 
	$h\leq (\alpha-1) h_{\alpha}(r-1)+
	h_\alpha(n-2-\alpha r)$ 	or
	\item[-] $n< r^2+3r+1$, $r$ is odd, and 
	$h\leq (\alpha-1) h_{\alpha}(r-2)+h_\alpha(\min\{n-3-\alpha(r-1),r-2\})$
\end{itemize}
then $\G(r,n)$ is not $(h+1)$-defective.
\end{Theorem}
\begin{proof}
Since by Propositions \ref{limitosculatingspacesgrass} and \ref{limitosculatingspacesgrassII}
the Grassmannian $\G(r,n)$ has strong $2$-osculating regularity and $\alpha-$osculating regularity,
it is enough to apply Corollary \ref{oscprojbirationalII} together with Theorem \ref{lemmadefectsviaosculating}.
\end{proof}
Note that if we write 
\begin{equation}\label{222}
r = 2^{\lambda_1}+2^{\lambda_2}+\dots + 2^{\lambda_s}+\varepsilon
\end{equation}
with $\lambda_1 >\lambda_2>\dots >\lambda_s\geq 1$, $\varepsilon\in\{0,1\}$, then 
$$h_{\alpha}(r-1) = \alpha^{\lambda_1-1}+\dots + \alpha^{\lambda_s-1}.$$ 
Therefore, the first bound in Theorem \ref{maingrass} gives 
$$h\leq \alpha^{\lambda_1}+\dots + \alpha^{\lambda_s}.$$
Furthermore, just considering the first summand in the second and third bound in Theorem \ref{maingrass} we get that $\G(r,n)$ is not $(h+1)$-defective for
$$h\leq (\alpha-1)(\alpha^{\lambda_1-1}+\dots + \alpha^{\lambda_s-1}).$$ 
Finally, note that (\ref{222}) yields $\lambda_1 = \lfloor\log_2(r)\rfloor$. Hence, asymptotically we have $h_{\alpha}(r-1)\sim \alpha^{\lfloor\log_2(r)\rfloor-1}$, and by Theorem \ref{maingrass} $\G(r,n)$ is not $(h+1)$-defective for 
$$h\leq \alpha^{\lfloor\log_2(r)\rfloor} = \left(\frac{n+1}{r+1}\right)^{\lfloor\log_2(r)\rfloor}.$$

\begin{Example}


In order to help the reader in getting a concrete idea of the order of growth of the bound in Theorem \ref{maingrass} for $n\geq r^2+3r+1$ we work out some cases in the following table:
\begin{center}
\begin{tabular}{|c|c|l|}
\hline 
$r$ & $r^2+3r+1$ & $h$\\ 
\hline 
$4$ & $29$ & $\left(\frac{n+1}{5}\right)^2+1$\\ 
\hline 
$6$ & $55$ & $\left(\frac{n+1}{7}\right)^2+\left(\frac{n+1}{7}\right)+1$\\ 
\hline 
$8$ & $89$ & $\left(\frac{n+1}{9}\right)^3+1$\\ 
\hline 
$10$ & $131$ & $\left(\frac{n+1}{11}\right)^3+\left(\frac{n+1}{11}\right)+1$\\ 
\hline 
$12$ & $181$ & $\left(\frac{n+1}{13}\right)^3+\left(\frac{n+1}{13}\right)^2+1$\\ 
\hline 
$14$ & $239$ & $\left(\frac{n+1}{15}\right)^3+\left(\frac{n+1}{15}\right)^2+\left(\frac{n+1}{15}\right)+1$\\ 
\hline
$16$ & $305$ & $\left(\frac{n+1}{17}\right)^4+1$\\
\hline 
\end{tabular} 
\end{center}
\end{Example}

Thanks to Theorem \ref{maingrass} it is straightforward to get a linear bound going with $\frac{n}{2}$.

\begin{Corollary}\label{maincor}
Assume that $r\geq 2$, and set 
$$\alpha:=\left\lfloor \dfrac{n+1}{r+1} \right\rfloor$$
If either
\begin{itemize}
	\item[-] $n\geq r^2+3r+1$ and $h\leq\left\lfloor \dfrac{r}{2} \right\rfloor\alpha+1$ or
	\item[-] $n< r^2+3r+1$, $r$ is even, and 
	$h\leq	\left\lfloor \dfrac{n+1}{2} \right\rfloor-\dfrac{r}{2}$ or
	\item[-] $n< r^2+3r+1$, $r$ is odd, and $h\leq \min
	\left\{ \dfrac{r-1}{2} \alpha+1,\left\lfloor \dfrac{n}{2} \right\rfloor-\dfrac{r-1}{2}\right\}$
\end{itemize}
then $\G(r,n)$ is not $h$-defective.
\end{Corollary}
\begin{proof}
Since $\alpha \geq 2$ we have $h_\alpha(k)\geq h_2(k)=\left\lfloor \dfrac{k+1}{2} \right\rfloor$. In particular, $h_\alpha(r-1)\geq \left\lfloor \dfrac{r}{2} \right\rfloor$ and $h_\alpha(r-2)\geq \left\lfloor \dfrac{r-1}{2} \right\rfloor$.

Now, it is enough to observe that
\begin{align*}
\dfrac{r}{2} (\alpha-1)+\left\lfloor \dfrac{n-2-\alpha r+1}{2} \right\rfloor+1=
\left\lfloor \dfrac{n+1}{2} \right\rfloor-\dfrac{r}{2}
\end{align*}
for $r$ even, and 
\begin{align*}
\dfrac{r-1}{2} (\alpha-1)+\left\lfloor \dfrac{n-3-\alpha (r-1)+1}{2} \right\rfloor+1=
\left\lfloor \dfrac{n}{2} \right\rfloor-\dfrac{r-1}{2}
\end{align*}
for $r$ odd, and to apply Theorem \ref{maingrass}.
\end{proof}

\subsection{Comparison with Abo-Ottaviani-Peterson bound}\label{uglymath}
Finally, we show that Corollary \ref{maincor} strictly improves \cite[Theorem 3.3]{AOP09b} for $r\geq 4$, whenever $(r,n)\notin\left\{(4,10),(5,11)\right\}$.

For $r\geq 4,n\geq 2r+1$ we define the following functions of $r$ and $n$:
\begin{align*}
a:=\left\lfloor \dfrac{r}{2} \right\rfloor\left\lfloor \dfrac{n+1}{r+1} \right\rfloor,\ 
a':=\left\lfloor \dfrac{n-1}{2} \right\rfloor-\dfrac{r}{2},\ 
a'':=\left\lfloor \dfrac{n}{2} \right\rfloor-\dfrac{r+1}{2},\ 
b:=\left\lfloor \dfrac{n-r}{3} \right\rfloor
\end{align*}
First we show that $a>b$. Indeed, if $r>2$ is even then
$$a=\dfrac{r}{2}\left\lfloor \dfrac{n+1}{r+1} \right\rfloor>\dfrac{r}{2}\cdot\dfrac{n-r}{r+1}>
\dfrac{n-r}{3}\geq\left\lfloor \dfrac{n-r}{3} \right\rfloor=b$$
and if $r>5$ is odd then
$$a=\dfrac{r-1}{2}\left\lfloor \dfrac{n+1}{r+1} \right\rfloor>\dfrac{r-1}{2}\cdot\dfrac{n-r}{r+1}>
\dfrac{n-r}{3}\geq\left\lfloor \dfrac{n-r}{3} \right\rfloor=b.$$
Furthermore, if $r=5$ we write $n=6\lambda+\varepsilon$ with $\varepsilon\in\{-1,0,1,2,3,4\}$. Then we have
$$a=2\left\lfloor \dfrac{6\lambda+\varepsilon+1}{6} \right\rfloor=2\lambda>
2\lambda+\left\lfloor \dfrac{\varepsilon-5}{3} \right\rfloor=\left\lfloor \dfrac{6\lambda+\varepsilon-5}{3} \right\rfloor=b.$$

Now, we assume that $n<r^2+3r+1$ and we show that $a'>b$ if $r$ is even and $(r,n)\neq (4,10)$, and that $a''>b$ if $r$ is odd and $(n,r)\neq (5,11)$. Note that $a'(4,10)=a''(5,11)=b(4,10)=b(5,11)=2$. If $r$ is even
$$a'=\left\lfloor \dfrac{n-1}{2} \right\rfloor-\dfrac{r}{2}>
\dfrac{n-1}{2}-1-\dfrac{r}{2}=\dfrac{n-r-3}{2}>\dfrac{n-r}{3}=b$$
whenever $n>r+9.$ Similarly, if $r$ is odd and $n>r+9$ we have 
$$a''=\left\lfloor \dfrac{n}{2} \right\rfloor-\dfrac{r+1}{2}>
\dfrac{n}{2}-1-\dfrac{r+1}{2}=\dfrac{n-r-3}{2}>\dfrac{n-r}{3}=b.$$
Now, if $r>8$ then $n\geq 2r+1\Rightarrow n>r+9$. A finite number of cases are left, namely
$$(r,n)\in\left\{(r,n);\: r=4,5,6,7,8 \mbox{ and } r^2+3r+1>r+9\geq n \geq 2r+1 \right\}$$
These cases can be easily checked one by one.

\chapter{Secant defectivity of Segre-Veronese varieties}\label{cap4}
In this chapter we apply the technique developed in Chapter \ref{cap2} to study defectivity of Segre-Veronese varieties, obtaining the following theorem:

\begin{Theorem}
The Segre-Veronese variety $SV^{\pmb n}_{\pmb d}$ is not $h$-defective where
$$h\leq n_1h_{n_1+1}(d-2)+1$$
and $h_{n_1+1}(\cdot)$ is as in Definition \ref{defhowmanytangent}.
\end{Theorem}

In the first section of this chapter we recall the definition of Segre-Veronese varieties.
In Section \ref{oscspaces} we describe its osculating spaces. In Section \ref{oscproj} we give sufficient conditions to osculating projections from Segre-Veronese varieties to be birational.
In Section \ref{degtanoscsv} we show that the
Segre-Veronese variety $SV_{\pmb{d}}^{\pmb{n}}\subseteq\mathbb{P}^{N(\pmb{n},\pmb{d})}$ 
has strong $2$-osculating regularity and
$(n_1+1)$-osculating regularity.
In Section \ref{mainsec} we use the results in the previous sections together with Theorem \ref{lemmadefectsviaosculating} to prove our main result concerning Segre-Veronese secant defects.

\section{Segre-Veronese varieties}\label{Segre-Veronese varieties}
In this section we recall the definition of Segre-Veronese varieties and fix some notation to be used throughout this chapter. 
And then we provide some examples of defective Segre-Veronese varieties.

Let $\pmb{n}=(n_1,\dots,n_r)$ and $\pmb{d} = (d_1,\dots,d_r)$ be two $r$-uples of positive integers, with $n_1\leq \dots \leq n_r$.
Set $d=d_1+\dots+d_r$,  $n=n_1+\dots+n_r$, and $N(\pmb{n},\pmb{d})=\prod_{i=1}^r\binom{n_i+d_i}{n_i}-1$. 

Let $V_1,\dots, V_r$ be vector spaces of dimensions $n_1+1\leq n_2+1\leq \dots \leq n_r+1$, and consider the product
$$
\mathbb{P}^{\pmb{n}} = \mathbb{P}(V_1^{*})\times \dots \times \mathbb{P}(V_r^{*}).
$$
The line bundle 
$$
\mathcal{O}_{\mathbb{P}^{\pmb{n}} }(d_1,\dots, d_r)=\mathcal{O}_{\mathbb{P}(V_1^{*})}(d_1)\boxtimes\dots\boxtimes \mathcal{O}_{\mathbb{P}(V_1^{*})}(d_r)
$$
induces an embedding
$$
\begin{array}{cccc}
\sigma\nu_{\pmb{d}}^{\pmb{n}}:
&\mathbb{P}(V_1^{*})\times \dots \times \mathbb{P}(V_r^{*})& \longrightarrow &
\mathbb{P}(\Sym^{d_1}V_1^{*}\otimes\dots\otimes \Sym^{d_r}V_r^{*})
=\mathbb{P}^{N(\pmb{n},\pmb{d})},\\
      & (\left[v_1\right],\dots,\left[v_r\right]) & \longmapsto & [v_1^{d_1}\otimes\dots\otimes v_r^{d_r}]
\end{array}
$$ 
where $v_i\in V_i$.
We call the image 
$$
SV_{\pmb{d}}^{\pmb{n}}= \sigma\nu_{\pmb{d}}^{\pmb{n}}(\mathbb{P}^{\pmb{n}} ) \subset \mathbb{P}^{N(\pmb{n},\pmb{d})}
$$ 
a \textit{Segre-Veronese variety}. It is a smooth  variety of dimension $n$ and degree 
$$\frac{(n_1+\dots +n_r)!}{n_1!\dots n_r!}d_1^{n_1}\dots d_r^{n_r}$$
in $\mathbb{P}^{N(\pmb{n},\pmb{d})}.$ 

When $r = 1$, $SV_{d}^{n}$ is a Veronese variety. In this case we write $V_d^n$ for $SV_{d}^{n}$, and $v_{d}^{n}$ for the Veronese embedding.
When $d_1 = \dots = d_r = 1$, $SV_{1,\dots,1}^{\pmb{n}}$ is a Segre variety. 
In this case we write $S^{\pmb{n}}$ for $SV_{1,\dots,1}^{\pmb{n}}$, and  $\sigma^{\pmb{n}}$ for the Segre embedding.
Note that 
$$
\sigma\nu_{\pmb{d}}^{\pmb{n}}=\sigma^{\pmb{n}'}\circ \left( \nu_{d_1}^{n_1}\times \dots \times \nu_{d_r}^{n_r}\right),
$$
where $\pmb{n}'=(N(n_1,d_1),\dots,N(n_r,d_r))$.

\subsection{Some examples of secant defective Segre-Veronese varieties}
In the notation of Section \ref{secant} for  any $h$ consider the $h$-secant map $\pi_h:\Sec_{h}(X)\to\P^N$ and define
$$VSP_G^X(p,h):=\pi_h^{-1}(p)$$
where $p\in \mathbb{P}^N$ is a general point. Note that when $X = V_{d}^{n}$, and $F\in\mathbb{P}^{N(n,d)}$ is a general polynomial we recover the classical variety of sums of powers 
$$VSP_G^{V_d^{n}}(F,h) = VSP(F,h)$$
parametrizing additive decompositions of $F\in k[x_0,...,x_n]_{d}$ as sum of $d$-powers of linear forms \cite{Do04}. 

For instance, for homogeneous polynomials in two variables by \cite[Theorem 3.1]{MM13} we have that if $h > 1$ is a fixed integer, and $d$ in an integes such that $h\leq d\leq 2h-1$ then $VPS(F,h)\cong\mathbb{P}^{2h-d-1}$ where $F\in k[x_0,x_1]_d$ is general. Furthermore, a very simple description of $VSP(F,4)$, where $F$ is a general cubic polynomial in three variables, is at hand. 

\begin{Proposition}\label{vspF4}
Let $F\in k[x_0,x_1,x_2]_3$ be a general homogeneous polynomial. Then the variety of sums of powers $VSP(F,4)$ is isomorphic to the projective plane $\mathbb{P}^{2}$.
\end{Proposition}
\begin{proof}
Let $F\in k[x_0,x_1,x_2]$ be a general cubic polynomial. The partial derivatives of $F$ are three quadric polynomials $\frac{\partial F}{\partial x_0}$, $\frac{\partial F}{\partial x_1}$, $\frac{\partial F}{\partial x_2}$ that generate a projective plane $H_{\partial}$ in the $\mathbb{P}^{5}$ parametrizing plane conics. Since $F\in k[x_0,x_1,x_2]_3$ is general we have $H_{\partial}\cap V_2^{2}= \varnothing$.

Since any additive decomposition $\{l_1,l_2,l_3,l_4\}$ of $F$ induces an additive decomposition of $\frac{\partial F}{\partial x_0}$, $\frac{\partial F}{\partial x_1}$, $\frac{\partial F}{\partial x_2}$ the linear space $\left\langle l_1^2,l_2^2,l_3^2,l_4^2\right\rangle$ contains $H_{\partial}$. 

Note that if $\dim(\left\langle l_1^2,l_2^2,l_3^2,l_4^2\right\rangle) = 2$ then $\left\langle l_1^2,l_2^2,l_3^2,l_4^2\right\rangle = H_{\partial}$, and $H_{\partial}\cap V_2^{2}\neq \varnothing$. A contradiction. 
Therefore, since the $3$-planes of $\mathbb{P}^5  = \mathbb{P}(k[x_0,x_1,x_2]_2)$ containing $H_{\partial}\cong\mathbb{P}^2$ are parametrized by $\mathbb{P}^2$, we get a morphism 
$$
\begin{array}{cccc}
\phi: & VSP(F,4) & \longrightarrow & \mathbb{P}^2\\
 & \{l_{1},\dots,l_4\} & \longmapsto & \left\langle l_1^2,\dots,l_4^2\right\rangle
\end{array}
$$
Now, since $\deg(V_2^2) = 4$ any $3$-plane containing $H_{\partial}$ intersects $V_2^2$ in a $0$-dimensional subscheme of length four. Therefore, $\phi$ is injective. To conclude that $\phi$ is in isomorphism it is enough to observe that by \cite[Proposition 3.2]{Do04} $VSP(F,4)$ is a smooth, irreducible surface.  
\end{proof}

Finally, we give some examples of secant defective Segre-Veronese varieties.

\begin{Proposition}\label{def}
The Segre-Veronese varieties $SV^{(1,1)}_{(2,2)}$, $SV^{(1,1,1)}_{(1,1,2)}$ and $SV^{(1,1,1,1)}_{(1,1,1,1)}$ are $3$-defective. Furthermore, $SV^{(2,2,2)}_{(1,1,1)}$ is $4$-defective.
\end{Proposition}
\begin{proof}
First, let us consider products of $r$ copies of $\mathbb{P}^1$ with $r\in\{2,3,4\}$ and multi-degree $(d_1,\dots,d_r)$ such that $d_1+\dots+d_r = 4$. 

Let $p_1,p_2,p_3\in SV_{\pmb{d}}^{\pmb{n}}$ be three general points, and let $p\in \left\langle p_1,p_2,p_3\right\rangle$ be a general point. Therefore, $p\in \mathbb{P}^{N(\pmb{n},\pmb{d})}$ is a general points of $\sec_{h}(SV_{\pmb{d}}^{\pmb{n}})$.

Up to an automorphism of $SV_{\pmb{d}}^{\pmb{n}}$ we may assume that $p_1,p_2,p_3$ lie on the image of the small diagonal that is a degree four rational nomal curve $C$.

Let $H$ be the $4$-plane spanned by $C$. Then we may interpret $p\in H$ as a general degree four polynomial $F_p\in k[x_0,x_1]_4$. 

The variety parametrizing, $3$-secant planes to $C$ passing through $F_p$ is nothing but the variety of sums of powers $VSP(F_p,3)$. Now, \cite[Theorem 3.1]{MM13} yields that 
$$\dim(VSP(F_p,3)) = 1$$
Therefore, the general fiber of the map 
$$\pi_3:Sec_3(SV_{\pmb{d}}^{\pmb{n}})\rightarrow\sec_3(SV_{\pmb{d}}^{\pmb{n}})$$
is at least $1$-dimensional, and for the first three Segre-Veronese varieties in the statement this is in enough to conclude that they are $3$-defective.

A similar argument applies to $SV^{(2,2,2)}_{(1,1,1)}\subset\mathbb{P}^{26}$. In this case we may move four general points $p_1,\dots,p_4\in SV^{(2,2,2)}_{(1,1,1)}$ on the image of the diagonal, that is a Veronese surface $V_3^2$. Now, we may interpret $p\in \left\langle p_1,\dots,p_4\right\rangle$ as a general polynomial $F_p\in k[x_0,x_1,x_2]_3$. By Proposition \ref{vspF4} we have
$$\dim(VSP(F_p,4)) = 2$$
Hence, the general fiber of the map 
$$\pi_4:Sec_4(SV^{(2,2,2)}_{(1,1,1)})\rightarrow\sec_4(SV^{(2,2,2)}_{(1,1,1)})$$
is at least $2$-dimensional, and therefore $\sec_4(SV^{(2,2,2)}_{(1,1,1)})\subset\mathbb{P}^{26}$ is at most $25$-dimensional.
\end{proof}
 
\section{Osculating spaces}\label{oscspaces}

Given $v_1, \dots, v_{d_j}\in V_j$, we denote by $v_1\cdot \dots \cdot v_{d_j} \in \Sym^{d_j}V_j$  the symmetrization of
$v_1\otimes \dots \otimes v_{d_j}$.

Hoping that no confusion arise, we write $(e_0,\dots,e_{n_j})$ for a fixed a basis of each $V_j$. 
Given a $d_j$-uple $I=(i_1,\dots,i_{d_j})$, with $0\leq i_1 \leq \dots \leq i_{d_j} \leq n_j$, we denote by 
$e_{I}\in \Sym^{d_j}V_j$  the symmetric product $e_{i_{1}}\cdot\dots\cdot e_{i_{d_j}}$.

For each $j\in\{1, \dots, r\}$, consider a $d_j$-uple $I^j=(i^j_1,\dots,i^j_{d_j})$, with $0\leq i^j_1 \leq \dots \leq i^j_{d_j} \leq n_j$, and
set 
\begin{equation*}\label{vector}
I = (I^1,\dots,I^r)=((i_{1}^{1},\dots,i_{d_1}^{1}),
(i_{1}^{2},\dots,i_{d_2}^{2}),\dots,(i_{1}^{r},\dots,i_{d_r}^{r})).
\end{equation*}
We denote by $e_I$ the vector
$$
e_I = e_{I^1}\otimes e_{I^2} \otimes \cdots \otimes e_{I^r}\in \Sym^{d_1}V_1\otimes\dots\otimes \Sym^{d_r}V_r,
$$
as well as the corresponding point in 
$\mathbb{P}(\Sym^{d_1}V_1^{*}\otimes\dots\otimes \Sym^{d_r}V_r^{*})=\mathbb{P}^{N(\pmb{n},\pmb{d})}$.
When $I^j=(i_j,\dots,i_j)$ for every $j\in\{1, \dots, r\}$, for some $0\leq i_j \leq n_j$, we have
$$
e_I=\sigma\nu_{\pmb{d}}^{\pmb{n}}(\left[e_{i_1}\right],\dots,\left[e_{i_r}\right])
\in SV_{\pmb{d}}^{\pmb{n}}\subset \mathbb{P}^{N(\pmb{n},\pmb{d})}.
$$
In this case we say that $e_I$ is a \emph{coordinate point} of $SV_{\pmb{d}}^{\pmb{n}}$.

\begin{Definition}\label{distance}
Let $n$ and $d$ be positive integers, and set 
$$
\Lambda_{n,d}=\{I=(i_1,\dots,i_{d}),0\leq i_1 \leq \dots \leq i_{d} \leq n\}.
$$
For $I,J\in \Lambda_{n,d}$, we define their distance $d(I,J)$ as the number of different coordinates.
More precisely, write $I=(i_1,\dots,i_{d})$ and $J=(j_1,\dots,j_{d})$. 
There are $r\geq 0$, distinct indices $\lambda_1, \dots, \lambda_r\subset \{1, \dots, d\}$, and 
distinct indices $\tau_1, \dots, \tau_r\subset \{1, \dots, d\}$ such that $i_{\lambda_k}=j_{\tau_k}$ for every $1\leq k\leq r$, and 
$\{i_\lambda | \lambda\neq \lambda_1, \dots, \lambda_r\}\cap \{j_\tau | \tau\neq \tau_1, \dots, \tau_r\}=\emptyset$.
Then $d(I,J)=d-r$.
Note that $\Lambda_{n,d}$ has diameter $d$ and size $\binom{n+d}{n}=N(n,d)+1$.

Let $\pmb{n}=(n_1,\dots,n_r)$ and $\pmb{d} = (d_1,\dots,d_r)$ be two $r$-uples of positive integers, and set 
$$
\Lambda=\Lambda_{\pmb{n},\pmb{d}}=\Lambda_{n_1,d_1}\times \dots \times \Lambda_{n_r,d_r}.
$$
For $I=(I^1,\dots,I^r),J=(J^1,\dots,J^r)\in \Lambda$, 
we define their distance as 
$$
d(I,J)=d(I^1,J^1)+\dots+d(I^r,J^r).
$$
Note that $\Lambda$ has diameter $d$ and size $\prod_{i=1}^r\binom{n_i+d_i}{n_i}=N(\pmb{n},\pmb{d})+1$.
\end{Definition}

\begin{Example}
Let us consider the case $r=2,n_1=1,n_2=3,d_1=2,d_2=3$. Then $n=4,d=5,\pmb{n}=(1,3),\pmb{d}=(2,3)$ and
$$\Lambda=\{I=\left((a,b),(c,d,e)\right); 0\leq a\leq b\leq 1, 0\leq c\leq d\leq e\leq 3\}.$$
We have, for instance, 
\begin{align*}
&d\left(((0,0),(0,1,2)),((0,0),(1,2,3))\right)=0+1=1,\\
&d\left(((1,1),(0,0,1)),((0,0),(2,3,3))\right)=2+3=5,\\
&d\left(((0,1),(1,1,1)),((1,1),(0,0,1))\right)=1+2=3.
\end{align*}
\end{Example}

We can now state the main result of this section.

\begin{Proposition}\label{oscsegver}
Let the notation and assumptions be as above.
Set $I^1=(i_1,\dots,i_1),$ $\dots,$ $I^r=(i_r,\dots,i_r)$, with $0\leq i_j \leq n_j$, and $I = (I^1,\dots,I^r)$. 
Consider the point 
$$
e_I=\sigma\nu_{\pmb{d}}^{\pmb{n}}(\left[e_{i_1}\right],\dots,\left[e_{i_r}\right]) \in SV_{\pmb{d}}^{\pmb{n}}.
$$
For any $s\geq 0$, we have 
$$
T^s_{e_I}(SV_{\pmb{d}}^{\pmb{n}})= \left\langle e_J \: | \: d(I,J)\leq s\right\rangle. 
$$
In particular, $T^s_{e_I}(SV_{\pmb{d}}^{\pmb{n}})=\P^{N(\pmb{n},\pmb{d})}$ for any $s\geq d.$
\end{Proposition}

\begin{proof}
We may assume that $I^1=(0,\dots,0),\dots,I^r=(0,\dots,0)$. 
Write $\big(z_K\big)_{K\in \Lambda}$, for coordinates in $\P^{N(\pmb{n},\pmb{d})}$, and 
consider the rational parametrization
$$
\phi:\mathbb{A}^{\prod n_i}\rightarrow SV_{\pmb{d}}^{\pmb{n}}\cap \big(z_I\neq 0\big)\subset \mathbb{A}^{N(\pmb{n},\pmb{d})}
$$
given by 
$$
A=(a_{j,i})_{j=1,\dots, r,i=1,\dots,n_j}
\mapsto \left( \prod_{j=1}^r \prod_{k=1}^{d_j} a_{j,i_{k}^{j}}\right)_{K=(K^1,\dots,K^r)\in \Lambda \backslash \{I\}},
$$
where $K^{j}=(i_{1}^j,\dots,i_{d_j}^j)\in \Lambda_{n_j,d_j}$ for each $j=1,...,r$.

For integers $l$ and $m$, we write $\deg_{l,m}K$ for the degree of the polynomial
$$
\phi(A)_{K}:=\prod_{j=1}^r \prod_{k=1}^{d_j} a_{j,i_{k}^{j}}
$$
with respect to $a_{l,m}$. Then $0\leq \deg_{l,m}K\leq d_l$, and the degree of $\phi(A)_{K}$
with respect to all the variables $a_{j,i}$ is at most $d$. 
One computes:
$$
\left(\dfrac{\partial^{\lambda_1+\dots+\lambda_t} \phi(A)}
{\partial^{\lambda_1}a_{l_1,m_1}\dots\partial^{\lambda_t}a_{l_t,m_t}}\right)_{K}\!\!\!=\!
\begin{cases}
0 &\mbox{if } \deg_{l_j,m_j} K<\lambda_j \mbox{ for some } j.\\
\dfrac{\prod_{j=1}^t (\deg_{l_j,m_j}K) !\phi(A)_{K}}
{\prod_{j=1}^t (\deg_{l_j,m_j}K-\lambda_j) ! a_{l_j,m_j}^{\lambda_j}} &\mbox{otherwise }
\end{cases}
$$
For $A = 0$ we get
$$
\left(\dfrac{\partial^{\lambda_1+\dots+\lambda_t} \phi(0)}
{\partial^{\lambda_1}a_{l_1,m_1}\dots\partial^{\lambda_t}a_{l_t,m_t}}\right)_{K}\!\!=
\begin{cases}
0 &\mbox{if } \deg_{l_j,m_j} K\neq\lambda_j \mbox{ for some } j.\\
\prod_{j=1}^t (\deg_{l_j,m_j}K) ! &\mbox{otherwise }
\end{cases}$$
Therefore 
$$
\dfrac{\partial^{\lambda_1+\dots+\lambda_t} \phi(0)}
{\partial^{\lambda_1}a_{l_1,m_1}\dots\partial^{\lambda_t}a_{l_t,m_t}}=
\left( \lambda_1!\right)\cdots\left( \lambda_t!\right)e_J,  
$$
where $J\in \Lambda$ is characterized by
$$
\deg_{l,m}J =
\begin{cases}
\lambda_j &\mbox{if } (l,m)=(l_j,m_j) \mbox{ for some } j.\\
0 &\mbox{otherwise }
\end{cases}
$$
Note that $d(J,I)\!=\!\lambda_1+\dots+\lambda_t$. Conversely every $J\in \Lambda$
with $d(J,I)\!=\!\lambda_1+\dots+\lambda_t$ can be obtained in this way. Therefore, for every $0\leq s \leq d$, we have
$$
\left\langle \dfrac{\partial^s \phi(0)}{\partial^{\lambda_1}a_{l_1,m_1}\dots\partial^{\lambda_t}a_{l_t,m_t}}
\: \big| \: 1\leq l_1,\dots,l_t\leq r, 1\leq m_j \leq n_j, j=1,\dots,t
\right\rangle=
\left\langle e_J | d(J,I)=s\right\rangle,
$$
and hence $T^s_{e_I}(SV_{\pmb{d}}^{\pmb{n}})= \left\langle e_J \: | \: d(I,J)\leq s\right\rangle$.
\end{proof}

Now, it is easy to compute the dimension of the osculating spaces of $SV_{\pmb{d}}^{\pmb{n}}$.

\begin{Corollary}
For any point $p\in SV^{\pmb n}_{\pmb d}$ we have
$$
\dim T^s_p SV^{\pmb n}_{\pmb d}=\sum_{l=1}^s 
\sum_{\substack{0\leq l_1\leq d_1,\dots, 0\leq l_r \leq d_r \\ l_1+\dots +l_r=l}}
\!\!\!\!\binom{n_1+l_1-1}{l_1}\cdots\binom{n_r+l_r-1}{l_r}
$$
for any $0\leq s\leq d$, while $T^s_{p}(SV_{\pmb{d}}^{\pmb{n}})=\P^{N(\pmb{n},\pmb{d})}$ for any $s\geq d$.\\
In particular, 
$$
\dim T^s_p V^{n}_{d}=n+\binom{n+1}{2}+\dots+\binom{n+s-1}{s}
$$
for any $0\leq s\leq d$.
\end{Corollary}

\begin{proof}
Since $SV^{\pmb n}_{\pmb d}\subset\P^{N(\pmb{n},\pmb{d})}$ is homogeneous under the action the algebraic subgroup 
$$\Stab(SV^{\pmb n}_{\pmb d})\subset PGL(N(\pmb{n},\pmb{d})+1)$$
stabilizing it, there exists an automorphism $\alpha\in PGL(N(\pmb{n},\pmb{d})+1)$ inducing an automorphism of $SV^{\pmb n}_{\pmb d}$ such that $\alpha(p) = e_I$. Furthermore, since $\alpha\in PGL(N(\pmb{n},\pmb{d})+1)$ we have that it induces an isomorphism between $T^s_p SV^{\pmb n}_{\pmb d}$ and $T^s_{e_I} SV^{\pmb n}_{\pmb d}$. Now, the computation of $\dim T^s_p SV^{\pmb n}_{\pmb d}$ follows, by standard combinatorial computations, from Proposition \ref{oscsegver}.
\end{proof}

\section{Osculating projections}\label{oscproj}
In this section we study linear projections of Segre-Veronese varieties from their osculating spaces.
We follow the notation introduced in the previous section.

We start by analyzing projections of Veronese varieties from osculating spaces at coordinate points.
We consider a Veronese variety $V_{d}^{n}\subset \mathbb{P}^{N(n,d)}$, $d\geq 2$, and a coordinate point $e_{\underline{i}}=e_{(i,\dots,i)}\in V_{d}^{n}$
for some $i\in\{0,1,\dots,n\}$.
We write $(z_I)_{I\in \Lambda_{n,d}}$ for the coordinates in $\mathbb{P}^{N(n,d)}$. 
The linear projection
\begin{align*}
\pi_{i}:\P^{n}&\dasharrow \P^{n-1}\\
(x_j)
&\mapsto (x_j)_{j\neq i}
\end{align*}
induces the linear projection
\begin{align}\label{projonepoint}
\Pi_{i}:V_{d}^{n}&\dasharrow V_{d}^{n-1}\\
(z_I)_{I\in \Lambda_{n,d}}
&\mapsto (z_I)_{I\in \Lambda_{n,d} \: | \: i\notin I}\nonumber
\end{align}
making the following diagram commute
\[
\begin{tikzpicture}[xscale=6.5,yscale=-1.7]
    \node (A0_0) at (0, 0) {$\mathbb{P}^{n}$};
    \node (A0_1) at (1, 0) {$V_{d}^{n}\subseteq \mathbb{P}^{N(n,d)}$};
    \node (A1_0) at (0, 1) {$\mathbb{P}^{n-1}$};
    \node (A1_1) at (1, 1) {$V_{d}^{n-1}\subseteq \mathbb{P}^{N(n-1,d)}$};
    \path (A0_0) edge [->]node [auto] {$\scriptstyle{\nu_{d}^{n}}$} (A0_1);
    \path (A1_0) edge [->]node [auto] {$\scriptstyle{\nu_{d}^{n-1}}$} (A1_1);
    \path (A0_1) edge [->,dashed]node [auto] {$\scriptstyle{\Pi_{i}}$} (A1_1);
    \path (A0_0) edge [->,dashed]node [auto,swap] {$\scriptstyle{\pi_{i}}$} (A1_0);
  \end{tikzpicture} 
\]

\begin{Lemma}\label{projoscveronese}
Consider the projection of the Veronese variety $V_{d}^{n}\subset \mathbb{P}^{N(n,d)}$, $d\geq 2$, from 
the osculating space ${T_{e_{\underline{i}}}^s}$ of order $s$ at the point $e_{\underline{i}}=e_{(i,\dots,i)}\in V_{d}^{n}$, $0\leq s\leq d-1$:
$$
\Gamma_i^s:V_{d}^{n}\dasharrow \P^{N(n,d,s)}.
$$
Then $\Gamma_i^s$ is birational for any $s\leq d-2$, while $\Gamma_i^{d-1}=\Pi_{i}$.
\end{Lemma}

\begin{proof}
The case $s=d-1$ follows from Proposition~\ref{oscsegver} and the expression in \eqref{projonepoint} above, observing that, for any $J\in\Lambda_{n,d}$, 
$$
d(J,(i,\dots,i))=d \Leftrightarrow i\notin J.
$$ 

Since $\Gamma_i^{d-2}$ factors through $\Gamma_i^{j}$ for every $0\leq j\leq d-3$, it is enough to prove  birationality of $\Gamma_i^{d-2}$.
We may assume that $i\neq 0$, and consider the collection of indices
$$
J_0=(0,\dots,0,0),J_1=(0,\dots,0,1),\dots,J_{n}=(0,\dots,0,n)\in \Lambda_{(n,d)}.
$$
Note that $d(J_j,(i,\dots,i))\geq d-1$ for any $j\in\{1, \dots, n\}$. 
So we can define the linear projection
\begin{align*}
\gamma:\P^{N(n,d,s)} & \dasharrow \P^{n} .\\
(z_J)_{J \: |\: d(I,J)>d-2} &\mapsto (z_{J_0},\dots,z_{J_{n}})
\end{align*}
The composition 
\begin{align*}
\gamma\circ \Gamma_i^{d-2}\circ \nu_{d}^{n}:\mathbb{P}^n&\dasharrow\mathbb{P}^n\\
(x_0:\cdots:x_{n})&\mapsto (x_0^{d-1}x_0:\cdots x_0^{d-1}x_{n})=(x_0:\cdots:x_{n}) 
\end{align*}
is the identity, and thus $\Gamma_i^{d-2}$ is  birational.
\end{proof}

Now we turn to Segre-Veronese varieties. 
Let $SV_{\pmb{d}}^{\pmb{n}}\subset \P^{N(\pmb{n},\pmb{d})}$ be a Segre-Veronese variety, and consider a coordinate point
$$
e_I = e_{i_1}^{d_1}\otimes e_{i_2}^{d_2}\otimes\dots\otimes e_{i_r}^{d_r}\in SV_{\pmb{d}}^{\pmb{n}},
$$
with $0\leq i_j\leq n_j$, $I=\big((i_1,\dots,i_1),\dots, (i_r,\dots,i_r)\big)$. 
We write $(z_I)_{I\in \Lambda}$ for the coordinates in $\mathbb{P}^{N(\pmb{n},\pmb{d})}$.
Recall from Proposition~\ref{oscsegver} that the linear projection of $SV_{\pmb{d}}^{\pmb{n}}$ 
from the osculating space $T_{e_I}^s$ of order $s$ at $e_I$ is given by 
\begin{align}\label{eq:osc_proj}
\Pi_{T_{e_I}^s}:SV_{\pmb{d}}^{\pmb{n}}&\dasharrow \P^{N(\pmb{n},\pmb{d},s)}\\
(z_J)
&\mapsto (z_J)_{J\in\Lambda \: | \: d(I,J)>s}\nonumber
\end{align}
for every $s\leq d -1$.


In order to study the fibers of $\Pi_{T_{e_I}^s}$, we define auxiliary rational maps 
$$
\Sigma_l:SV_{\pmb{d}}^{\pmb{n}}\dasharrow \P^{N_l}
$$
for each $l\in \{1,\dots, r\}$ as follows. The map $\Sigma_1$ is the composition of the product map 
$$
\Gamma_{i_1}^{d_1-2} \times \prod_{j=2}^r \Pi_{i_j}:
V_{d_1}^{n_1}\times \dots \times V_{d_r}^{n_r}
\dasharrow  \mathbb P^{N(n_1,d_1,d_1-2)} \times  \prod_{j=2}^r \P^{N(n_j-1,d_j)}
$$ 
with the Segre embedding 
$$
\mathbb P^{N(n_1,d_1,d_1-2)} \times  \prod_{j=2}^r \P^{N(n_j-1,d_j)}\hookrightarrow \P^{N_1}.
$$
The other maps $\Sigma_l$, $2\leq l\leq r$, are defined analogously. 
In coordinates we have:
\begin{align}\label{eq_sigmal}
\Sigma_l:SV_{\pmb{d}}^{\pmb{n}}&\dasharrow \P^{N_l} \ , \\
(z_J)
&\mapsto (z_J)_{J\in\Lambda_l}\nonumber
\end{align}
where $\Lambda_l=\big\{J=(J^1,\dots,J^r)\in \Lambda \ \big| \ d(J^l,(i_l,\dots,i_l))\geq d_l-1 \text{ and }
i_j\not\in J^j \text{ for } j\neq l\}.$


\begin{Proposition}\label{projosconept}
Consider the projection of the Segre-Veronese variety $SV_{\pmb{d}}^{\pmb{n}}\subset \P^{N(\pmb{n},\pmb{d})}$ from 
the osculating space $\Pi_{T_{e_I}^s}$ of order $s$ at the point 
$e_I = e_{i_1}^{d_1}\otimes e_{i_2}^{d_2}\otimes\dots\otimes e_{i_r}^{d_r}\in SV_{\pmb{d}}^{\pmb{n}}$, $0\leq s\leq d-1$:
$$
\Pi_{T_{e_I}^s}:SV_{\pmb{d}}^{\pmb{n}} \dasharrow \P^{N(\pmb{n},\pmb{d},s)}.
$$
Then $\Pi_{T_{e_I}^s}$ is birational for any $s\leq d-2$, while $\Pi_{T_{e_I}^{d-1}}$ fits in the following commutative diagram:
\[
\begin{tikzpicture}[xscale=6.5,yscale=-1.7]
    \node (A0_0) at (0, 0) {$\mathbb{P}^{n_1}\times\dots\times\mathbb{P}^{n_r}$};
    \node (A0_1) at (1, 0) {$SV_{\pmb{d}}^{\pmb{n}}\subseteq \mathbb{P}^{N(\pmb{d},\pmb{n})}$};
    \node (A1_0) at (0, 1) {$\mathbb{P}^{n_1-1}\times\dots\times\mathbb{P}^{n_r-1}$};
    \node (A1_1) at (1, 1) {$SV_{\pmb{d}}^{\pmb{n}-\pmb{1}}\subseteq \mathbb{P}^{N(\pmb{d},\pmb{n}-\pmb{1})}$};
    \path (A0_0) edge [->]node [auto] {$\scriptstyle{\sigma\nu_{\pmb{d}}^{\pmb{n}}}$} (A0_1);
    \path (A1_0) edge [->]node [auto] {$\scriptstyle{\sigma\nu_{\pmb{d}}^{\pmb{n}-\pmb{1}}}$} (A1_1);
    \path (A0_1) edge [->,dashed]node [auto] {$\scriptstyle{\Pi_{T_{e_I}^{d -1}}}$} (A1_1);
    \path (A0_0) edge [->,dashed]node [auto,swap] {$\scriptstyle{\pi_{i_1}\times\dots\times \pi_{i_r}}$} (A1_0);
  \end{tikzpicture} 
\]
where $\pmb{n}-\pmb{1}=(n_1-1,\dots,n_r-1).$
Furthermore, the closure of the fiber of $\Pi_{T_{e_I}^{d -1}}$ is the Segre-Veronese variety 
$SV_{\pmb{d}}^{1,\dots,1} $.
\end{Proposition}

\begin{proof}
The case $s=d-1$ follows from the expressions in \eqref{projonepoint} and \eqref{eq:osc_proj}, and Lemma~\ref{projoscveronese}.

Since $\Pi_{T_{e_I}^{d-2}}$ factors through $\Pi_{T_{e_I}^{j}}$ for every $0\leq j\leq d-3$, it is enough to prove  birationality of $\Pi_{T_{e_I}^{d-2}}$.

First note that $\Pi_{T_{e_I}^{d-2}}$ factors the map $\Sigma_l$ for any $l=1,\dots,r$. 
This follows from the expressions in \eqref{eq:osc_proj} and \eqref{eq_sigmal}, observing that 
$$
J=(J^1,\dots,J^r) \in \Lambda_l\Rightarrow 
d(J,I)\geq d_l-1+\sum_{j\neq l} d_j =d-1>d-2.
$$
We write $\tau_l:\mathbb P^{N(\pmb{n},\pmb{d},d-2)}\dasharrow \mathbb{P}^{N_l}$ for the projection making the 
following diagram  commute: 
\[
  \begin{tikzpicture}[xscale=3.5,yscale=-1.5]
    \node (A0_0) at (0, 0) {$SV_{\pmb{d}}^{\pmb{n}}$};
    \node (A1_0) at (1, 1) {$\mathbb{P}^{N_l}$};
    \node (A1_1) at (1, 0) {$\mathbb P^{N(\pmb{n},\pmb{d},d-2)}$};
    \path (A0_0) edge [->,swap, dashed] node [auto] {$\scriptstyle{\Sigma_l}$} (A1_0);
    \path (A1_1) edge [->, dashed] node [auto] {$\scriptstyle{\tau_l}$} (A1_0);
    \path (A0_0) edge [->, dashed] node [auto] {$\scriptstyle{\Pi_{T_{e_I}^{d-2}}}$} (A1_1);
  \end{tikzpicture}
\]

Take a general point
$$
x\in \overline{\Pi_{T_{e_I}^{d-2}}\left( SV_{\pmb{d}}^{\pmb{n}} \right)}\subseteq\mathbb P^{N(\pmb{n},\pmb{d},d-2)},
$$ 
and set $x_l=\tau_l(x)$, $l=1,\dots,r$. 
Denote by $F\subset \mathbb P^{\pmb{n}} $ the closure of the  fiber of $\Pi_{T_{e_I}^{d-2}}$
over $x$, and by $F_l$ the closure of the  fiber of $\Sigma_l$ over $x_l$. 
Let $y\in F\subset F_l$ be a general point, and write
$y=\sigma\nu^{\pmb{n}}_{\pmb{d}}(y_1,\dots,y_r)$, with $y_j\in \mathbb{P}^{n_j}, j=1,\dots,r$. 
By Lemma \ref{projoscveronese}, $F_l$ is the image under $\sigma\nu^{\pmb{n}}_{\pmb{d}}$ of 
$$
\langle y_1,e_{i_1}\rangle \times \dots \times \langle y_{l-1},e_{i_{l-1}}\rangle \times y_l \times
\langle y_{l+1},e_{i_{l+1}}\rangle \times \dots \times\langle y_r,e_{i_r}\rangle
\subset \mathbb{P}^{\pmb{n}},
$$
$l=1,\dots,r$.  It follows that $F=\{y\}$, and so
$\Pi_{T_{e_I}^{d-2}}$ is birational.
\end{proof}

Next study linear projections from the span of several osculating spaces at coordinate points, and investigate 
when they are birational.

\medskip

We start with the case of a Veronese variety $V_{d}^{n}\subset \mathbb{P}^{N(n,d)}$, with coordinate points 
$e_{\underline{i}}=e_{(i,\dots,i)}\in V_{d}^{n}$, $i\in\{0,1,\dots,n\}$.
For $m\in \{1, \dots, n\}$, let $\pmb{s}=(s_0,\dots,s_m)$ be an $(m+1)$-uple of positive integers, and set $s=s_0+\dots+s_m$.
Let $e_{\underline{i}_0}, \dots, e_{\underline{i}_{m}}\in V_{d}^{n}$ be distinct coordinate points, and denote by 
$T_{e_{\underline{i}_{0}}, \dots, e_{\underline{i}_{m}}}^{s_0,\dots, s_m}\subset \mathbb{P}^{N(n,d)}$ the linear span 
$\left\langle T_{e_{\underline{i}_{0}}}^{s_0},\dots,T_{e_{\underline{i}_{m}}}^{s_m}\right\rangle$.
By Proposition~\ref{oscsegver}, the projection of $V_{d}^{n}$ from 
$T_{e_{\underline{i}_{0}}, \dots, e_{\underline{i}_{m}}}^{s_0,\dots, s_m}\subset \mathbb{P}^{N(n,d)}$ is given by:
\begin{align}
\Gamma^{s_0,\dots, s_m}_{e_{\underline{i}_{0}}, \dots, e_{\underline{i}_{m}}}:V_{d}^{n}&\dasharrow \P^{N(n,d,\pmb{s})},\\
(z_I)_{I\in \Lambda_{n,d}}
&\mapsto (z_J)_{J\in \Lambda_{n,d}^{\pmb{s}}}\nonumber
\end{align}
whenever $\Lambda_{n,d}^{\pmb{s}}=\{J\in \Lambda_{n,d} \ |\  d\big(J,(j,\dots,j)\big)>s_j \text{ for } j=0,\dots,m\}$ is not empty.

\begin{Lemma}\label{projoscveroneseseveralpoints}
Let the notation be as above, and assume that  $d\geq 2$ and $0\leq s_j\leq d-2$ for $j=0,\dots,m$.
\begin{itemize}
\item[(\textit{a})] If $n\leq d$ and $s \leq n(d-1)-2$, then $\Gamma^{s_0,\dots, s_m}_{e_{\underline{i}_{0}}, \dots, e_{\underline{i}_{m}}}$ is birational onto its image.
\item[(\textit{b})] If $n\leq d$ and $s = n(d-1)-1$, then $\Gamma^{s_0,\dots, s_m}_{e_{\underline{i}_{0}}, \dots, e_{\underline{i}_{m}}}$ is a constant map.
\item[(\textit{c})] If $n> d$, then  $\Gamma^{d-2,\dots, d-2}_{e_{\underline{i}_{0}}, \dots, e_{\underline{i}_{n}}}$ is birational onto its image.
\end{itemize}
\end{Lemma}

\begin{proof}
Assume that $n\leq d$ and $s \leq n(d-1)-2$.
In order to prove that $\Gamma^{s_0,\dots, s_m}_{e_{\underline{i}_{0}}, \dots, e_{\underline{i}_{m}}}$ is birational, 
we will exhibit $J_0,\dots,J_{n}\in \Lambda_{n,d}^{\pmb{s}}$ , and linear projection 
\begin{align}\label{proj_cremona}
\gamma: \mathbb{P}^{N(n,d)}&\dasharrow \P^{n}\\
(z_I)_{I\in \Lambda_{n,d}^{\pmb{s}}}
&\mapsto (z_{J_j})_{j=0,\dots,n}\nonumber
\end{align}
such that the composition $\gamma\circ\Gamma^{s_0,\dots, s_m}_{e_{\underline{i}_{0}}, \dots, e_{\underline{i}_{m}}}\circ \nu_{d}^{n}:\P^{n}\dasharrow \P^{n}$
is the standard Cremona transformation of $\mathbb{P}^n$.
The $d$-uples $J_j\in \Lambda_{n,d}^{\pmb{s}}$ are constructed as follows. 
Since $n\leq d$ we can take $n$ of the coordinates of $J_j$ to be $0,1,\dots, \widehat{j},\dots, n$. 
The condition  $s \leq n(d-1)-2$ assures that we can complete the $J_j$'s by choosing $d-n$ common coordinates 
in such a way that, for every $i,j\in \{0,\dots, n\}$, we have $d\big(J_j,(i,\dots,i)\big)>s_i$
(i.e., $J_j$ has at most $(d-s_i-1)$ coordinates equal to to $i$). 
This gives $J_j\in \Lambda_{n,d}^{\pmb{s}}$ for every $j\in \{0,\dots, n\}$.
For the linear projection \eqref{proj_cremona} given by these $J_j$'s, we have that 
$\gamma\circ\Gamma^{s_0,\dots, s_m}_{e_{\underline{i}_{0}}, \dots, e_{\underline{i}_{m}}}\circ \nu_{d}^{n}:\P^{n}\dasharrow \P^{n}$
is the standard Cremona transformation of $\mathbb{P}^n$.

Now assume that $n\leq d$ and $s = n(d-1)-1$. 
If $J\in \Lambda_{n,d}^{\pmb{s}}$,
then $J$ has at most $d-s_i-1$ coordinates equal to $i$ for any $i\in \{0,\dots,n\}$. Since 
$$
\sum_{j=0}^{n} \left( d-s_j-1\right)=(n+1)(d-1)-s=d,
$$
there is only one possibility for $J$, i.e., $\Lambda_{n,d}^{\pmb{s}}$ has only one element, and so
$\Gamma^{s_0,\dots, s_m}_{e_{\underline{i}_{0}}, \dots, e_{\underline{i}_{m}}}$ is a constant map.

Finally, assume that $n>d.$ 
Set $K_0=\{0,\dots,{n-d}\}$. For any $j\in K_0$, set
$$
(J_{K_0})_j:=(j, {n-d+1},\dots, {n}),
$$
and note that $d\big((J_{K_0})_j,(i,\dots,i)\big)>d-2$ for every $i\in \{0,\dots, n\}$.
Thus $(J_{K_0})_j\in \Lambda_{n,d}^{\pmb{d-2}}$ for every $j\in {K_0}$. 
So we can define the linear projection
\begin{align*}
\gamma_{K_0}: \P^{N(n,d,\pmb{d-2})}&\dasharrow \P^{n-d}.\\
(z_I)_{I\in \Lambda_{n,d}^{\pmb{d-2}}}
&\mapsto (z_{(J_{K_0})_j})_{j\in {K_0}}
\end{align*}
The composition $\gamma_{K_0}\circ\Gamma^{d-2,\dots, d-2}_{e_{\underline{i}_{0}}, \dots, e_{\underline{i}_{n}}}\circ \nu_{d}^{n}:\P^{n}\dasharrow \P^{n-d}$
is the linear projection given by 
\begin{align*}
\gamma_{K_0}\circ\Gamma^{d-2,\dots, d-2}_{e_{\underline{i}_{0}}, \dots, e_{\underline{i}_{n}}}\circ \nu_{d}^{n}:\P^{n}&\dasharrow \P^{n-d}.\\
(x_i)_{i\in\{0,\dots,n\}}&\mapsto (x_i)_{i\in {K_0}}
\end{align*}
Analogously, for each subset $K\subset \{0,\dots,n\}$ with $n-d+1$ distinct elements,
we define a linear projection $\gamma_{K}: \P^{N(n,d,\pmb{d-2})}\dasharrow \P^{n-d}$ such that the composition 
$\gamma_{K}\circ\Gamma^{d-2,\dots, d-2}_{e_{\underline{i}_{0}}, \dots, e_{\underline{i}_{n}}}\circ \nu_{d}^{n}:\P^{n}\dasharrow \P^{n-d}$
is the linear projection given by
\begin{align*}
\gamma_{K}\circ\Gamma^{d-2,\dots, d-2}_{e_{\underline{i}_{0}}, \dots, e_{\underline{i}_{n}}}\circ \nu_{d}^{n}:\P^{n}&\dasharrow \P^{n-d}.\\
(x_i)_{i\in\{0,\dots,n\}}&\mapsto (x_i)_{i\in {K}}
\end{align*}
This shows that $\Gamma^{d-2,\dots, d-2}_{e_{\underline{i}_{0}}, \dots, e_{\underline{i}_{n}}}$ is birational. 
\end{proof}

The following is an immediate consequence of Lemma \ref{projoscveroneseseveralpoints}.

\begin{Corollary}\label{projoscveroneseseveralpointscor}
Let the notation be as above, and assume that  $d\geq 2$. Then
\begin{itemize}
\item[(\textit{a})] $\Gamma^{d-2,\dots,d-2}_{e_{\underline{i}_{0}}, \dots, e_{\underline{i}_{n-1}}}$ is birational.
\item[(\textit{b})] If $n\geq 2$ then
$\Gamma^{d-2,\dots,d-2,\min\{n,d\}-2}_{e_{\underline{i}_{0}}, \dots, e_{\underline{i}_{n}}}$ is birational, while
$\Gamma^{d-2,\dots,d-2,\min\{n,d\}-1}_{e_{\underline{i}_{0}}, \dots, e_{\underline{i}_{n}}}$ is not.
\item[(\textit{c})] If $d\geq 3$ then
$\Gamma^{d-3,\dots,d-3,\min\{2n,d\}-2}_{e_{\underline{i}_{0}}, \dots, e_{\underline{i}_{n}}}$ is birational, while
$\Gamma^{d-3,\dots,d-3,\min\{2n,d\}-1}_{e_{\underline{i}_{0}}, \dots, e_{\underline{i}_{n}}}$ is not.
\end{itemize}
\end{Corollary}

Now we turn to Segre-Veronese varieties. 
Let $SV_{\pmb{d}}^{\pmb{n}}\subset \P^{N(\pmb{n},\pmb{d})}$ be a Segre-Veronese variety, 
and write $(z_I)_{I\in \Lambda_{n,d}}$ for coordinates in $\mathbb{P}^{N(n,d)}$. 
Consider the coordinate points $e_{I_0},e_{I_1},\dots,e_{I_{n_1}}\in SV_{\pmb{d}}^{\pmb{n}}$, where
$$
I_j=((j,\dots,j),\dots,(j,\dots,j))\in \Lambda.
$$
(Recall that $n_1\leq\dots \leq n_r$.)
Let $\pmb{s}=(s_0,\dots,s_m)$ be an $(m+1)$-uple of positive integers, and set $s=s_0+\dots+s_m$.
Denote by $T_{e_{I_0},\dots,e_{I_m}}^{s_0,\dots,s_m}\subset \P^{N(\pmb{n},\pmb{d})}$ the linear span of the osculating spaces 
$T_{e_{I_0}}^{s_0},\dots,T_{e_{I_m}}^{s_m}$.
By Proposition~\ref{oscsegver}, the projection of $SV_{\pmb{d}}^{\pmb{n}}$ from 
$T_{e_{I_0},\dots,e_{I_m}}^{s_0,\dots,s_m}$ is given by:
\begin{align}\label{eq:multi_osc_proj}
\Pi_{T_{e_{I_0},\dots,e_{I_m}}^{s_0,\dots,s_m}}:
SV_{\pmb{d}}^{\pmb{n}}&\dasharrow \P^{N(\pmb{d},\pmb{n},\pmb{s})}\\
(z_J)_{J \in \Lambda}
&\mapsto (z_J)_{J \in \Lambda^{\pmb{s}}} \nonumber
\end{align}
whenever $\Lambda^{\pmb{s}}\!=\!\{J \in \Lambda\: |\: d(I_j,J)>s_j \ \forall j\}$ is not empty. 

\begin{Proposition}\label{projoscseveralpoints} 
Let the notation be as above, and assume that $r,d\geq 2$. 
Then the projection  
$\Pi_{T_{e_{I_0},\dots,e_{I_{n_1-1}}}^{d-2,\dots,d-2}}:SV_{\pmb{d}}^{\pmb{n}} \dasharrow \P^{N(\pmb{d},\pmb{n},{\pmb{d-2}})}$ 
is birational.
\end{Proposition}

\begin{proof}
For each $l\in\{1,\dots, r\}$, set 
$$
\Lambda_l=\left\{J\!=\!(J^1,\dots,J^r)\in \Lambda \: \big| \:
\begin{cases} 
0,\dots, n_1-1\not \in J^j \text{ if } j\neq l \\
d\big(J^l,(i,\dots, i)\big)\geq d_l-1 \ \forall i\in \{0,\dots, n_1-1\}
\end{cases}
\right\},
$$
and consider the linear projection 
\begin{align}\label{Sigma_l}
\Sigma_l:SV_{\pmb{d}}^{\pmb{n}}
&\dasharrow \mathbb P^{N_l}.\\
{\left(z_J\right)}_{J\in\Lambda}&\mapsto
{\left( z_J\right)}_{J\in\Lambda_l}\nonumber
\end{align}
Note that $\Lambda_l\subset \Lambda^{\pmb{d-2}}$, and so there is a linear projection
$\tau_l:\P^{N(\pmb{d},\pmb{n},\pmb{d-2})}\dasharrow \mathbb P^{N_l}$
such that $\Sigma_l=\tau_l\circ \Pi_{T^{d-2,\dots,d-2}_{p_0,\dots,p_{n_1-1}}}$.

The restriction of $\Sigma_l\circ \sigma\nu^{\pmb{n}}_{\pmb{d}}$ to 
$$
\{pt\} \times \dots \{pt\} \times \P^{n_l} \times \{pt\} \times \dots \{pt\}
$$
is isomorphic to the osculating projection 
$$
\Gamma^{d_l-2,\dots,d_l-2}_{e_{\underline{i}_{0}}, \dots, e_{\underline{i}_{n_1-1}}}:V_{d_l}^{n_l}\dasharrow \P^{N(n_l,d_l,\pmb{d_l-2})}.
$$
This is birational by Corollary~\ref{projoscveroneseseveralpointscor}. 
For $j\neq l$, the restriction of $\Sigma_l\circ \sigma\nu^{\pmb{n}}_{\pmb{d}}$ to 
$$
\{pt\} \times \dots \{pt\} \times \P^{n_j} \times \{pt\} \times \dots \{pt\}
$$
is isomorphic to the projection with center $\left\langle e_0,\dots,e_{n_1-1} \right\rangle$.
Arguing as in the last part of the proof of Proposition \ref{projosconept}, we conclude that 
$\Pi_{T_{e_{I_0},\dots,e_{I_{n_1-1}}}^{d-2,\dots,d-2}}$ is birational.
\end{proof}

\section{Degenerating osculating spaces}\label{degtanoscsv}

In this section we show that the Segre-Veronese variety $SV_{\pmb d}^{\pmb n}\subseteq\mathbb{P}^{N(\pmb{n},\pmb{d})}$ has 
strong $2$-osculating regularity, and $(n_1+1)$-osculating regularity.
We follow the notation introduced in the previous sections.

\begin{Proposition}\label{limitosculatingspacessegver}
The Segre-Veronese variety $SV_{\pmb{d}}^{\pmb{n}}\subseteq\mathbb{P}^{N(\pmb{n},\pmb{d})}$ has 
strong $2$-osculating regularity.
\end{Proposition}

\begin{proof}
Let $p,q\in SV_{\pmb{d}}^{\pmb{n}}\subseteq\mathbb{P}^{N(\pmb{n},\pmb{d})}$ be general points.
There is a projective automorphism of $SV_{\pmb{d}}^{\pmb{n}}\subseteq\mathbb{P}^{N(\pmb{n},\pmb{d})}$ mapping 
$p$ and $q$ to the coordinate points $e_{I_0}$ and $e_{I_1}$. 
These points are connected by the degree $d$ rational normal curve defined by 
$$
\gamma([t:s])=(se_{I_0}+te_{I_1})^{d_1}\otimes\dots\otimes (se_{I_0}+te_{I_1})^{d_r}.
$$
We work in the affine chart $(s=1)$, and set $t=(t:1)$. 
Given integers $k_1,k_2\geq 0$, consider the family of linear spaces 
$$
T_t=\left\langle T^{k_1}_{e_{I_0}},T^{k_2}_{\gamma(t)}\right\rangle,\: t\in \C\backslash \{0\}.
$$
We will show that the flat limit $T_0$  of $\{T_t\}_{t\in \C\backslash \{0\}}$ in
$\G(\dim(T_t),N(\pmb{n},\pmb{d}))$ is contained in $T^{k_1+k_2+1}_{e_{I_0}}$.

We start by writing the linear spaces $T_t$ explicitly. 
For $j=1,\dots, r$, we define the vectors 
$$
e_0^{t}=e_0+te_1,e_1^t=e_1,e_2^t=e_2,\dots,e_{n_j}^{t}=e_{n_j}\in V_{j}.
$$
Given $I^j=(i_1,\dots,i_{d_j})\in \Lambda_{n_j,d_j}$, we denote by 
$e_{I^j}^{t}\in \Sym^{d_j}V_j$  the symmetric product $e_{i_1}^{t}\cdot\ldots\cdot e_{i_{d_j}}^{t}$.
Given $I=(I^1,\dots,I^r)\in \Lambda=\Lambda_{\pmb{n},\pmb{d}}$, we denote by 
$e_I^{t}\in \mathbb{P}^{N(\pmb{n},\pmb{d})}$
the point corresponding to 
$$
e_{I^1}^{t}\otimes\dots\otimes e_{I^r}^{t}\in \Sym^{d_1}V_1\otimes\dots\otimes \Sym^{d_r}V_r.
$$
By Proposition \ref{oscsegver} we have 
$$
T_t=\left\langle e_I\: | \: d(I,I_0)\leq k_1; \: e_I^{t}\: |\: d(I,I_0)\leq k_2\right\rangle,\: t\neq 0.
$$
We shall write $T_t$ in terms of the basis $\{e_J|J\in\Lambda\}$. 
Before we do so, it is convenient to introduce some additional notation. 

\begin{Notation}\label{defdelta}
Let $I\in \Lambda_{n,d}$, and write:
\begin{equation}\label{defI}
I=(\underbrace{0,\dots,0}_{a\text{ times}},\underbrace{1,\dots,1}_{b \text{ times}},i_{a+b+1},\dots,i_d),
\end{equation}
with $a,b\geq 0$ and $1<i_{a+b+1}\leq\dots\leq i_d$.
Given $l\in \Z$, define $\delta^l(I)\in \Lambda_{n,d}$ as 
$$
\delta^l(I)=
(\underbrace{0,\dots,0}_{a-l\text{ times}},\underbrace{1,\dots,1}_{b+l \text{ times}},i_{a+b+1},\dots,i_d),
$$
provided that $-b\leq l\leq a$. 

Given $I=(I^1,\dots,I^r)\in \Lambda$ and $\pmb{l}=(l_1,\dots,l_r)\in \Z^r$, define
$$
\delta^{\pmb{l}}(I)=(\delta^{l_1}(I^1),\dots,\delta^{l_r}(I^r))\in \Lambda,
$$
provided that each $\delta^{l_j}(I_j)$ is defined. 
Let $l\in \Z$. If $l\geq 0$, set 
$$
\Delta(I,l)= \left\{ \delta^{\pmb{l}}(I) | \pmb{l}=(l_1,\dots,l_r), l_1,\dots,l_r\geq 0, \: l_1+\dots +l_r=l\:\right\}\subset \Lambda.
$$
If $l<0$, set 
$$
\Delta(I,l)=\left\{ J|\: I\in\Delta(J,-l) \right\} \subset \Lambda.
$$
Define also:
$$
s_I^+=\max_{l\geq 0}\{\Delta(I,l)\neq \emptyset\}\in\{0,\dots,d\}= d-d(I,I_0) ,
$$
$$
s_I^-=\max_{l\geq 0}\{\Delta(I,-l)\neq \emptyset\}\in\{0,\dots,d\}= d-d(I,I_1),
$$
$$
\Delta(I)^+=\bigcup_{0\leq l} \Delta(I,l)=\!\!\!\!\bigcup_{0\leq l \leq s_I^+} \!\!\!\! \Delta(I,l), \text{ and }
$$
$$
\Delta(I)^-=\bigcup_{0\leq l} \Delta(I,-l)=\!\!\!\!\bigcup_{0\leq l \leq s_I^-} \!\!\!\!\Delta(I,-l).
$$
Note that if $J\in\Delta(I,l)$, then $d(J,I)=|l|$, $d(J,I_0)=d(I,I_0)+l$, and $d(J,I_1)=d(I,I_1)-l$.
Note also that, if $J\in \Delta(I)^-\cap \Delta(K)^+$, then $d(I,K)=d(I,J)+d(J,K)$.
\end{Notation}

Now we write each vector $e_I^t$ with $d(I,I_0)< k_2$ in terms of the basis $\{e_J|J\in\Lambda\}$. 

First, we consider the Veronese case. Let $I=(i_1,\dots,i_d)\in \Lambda_{n,d}$ be as in \eqref{defI}, so that $s_I^+=a$.
We have:
\begin{align*}
e_I^t&=(e_0^t)^a (e_1^t)^b e_{i_{a+b+1}}^t\cdots e_{i_d}^t
=(e_0+te_1)^a e_1^b e_{i_{a+b+1}}\cdots e_{i_d}=\\
&=e_0^a e_1^b e_{i_{a+b+1}}\cdots e_{i_d}
+t\binom{a}{1}e_0^{a-1}e_1^{b+1} e_{i_{a+b+1}}\cdots e_{i_d}+\dots
+t^a  e_1^{b+a} e_{i_{a+b+1}}\cdots e_{i_d}=\\
&=\sum_{l=0}^a t^l\binom{a}{l}e_0^{a-l}e_1^{b+l} e_{i_{a+b+1}}\cdots e_{i_d}
=\sum_{l=0}^a t^l\binom{a}{l}e_{\delta^l (I)}.
\end{align*}

In the Segre-Veronese case, for any $I=(I^1,\dots,I^r)\in\Lambda$, we have
\begin{align}\label{etonthebasis}
e_{I}^{t}
&=\!\!\!\!\sum_{J=(J^1,\dots,J^r)\in \Delta(I)^+}\!\!\!\!
t^{d(I,J)} c_{(I,J)}e_J,
\end{align}
where $ c_{(I,J)}=\binom{s_{I^1}^+}{d(I^1,J^1)}\cdots \binom{s_{I^r}^+}{d(I^r,J^r)}$. 
So we can rewrite the linear subspace $T_t$ as 
\begin{align}\label{Ttonthebasis}
T_t=&\Big< e_I \: | \: d(I,I_0)\leq k_1;\:
\!\!\!\!\sum_{J\in \Delta(I)^+}\!\!\!\!
t^{d(I,J)} c_{(I,J)}e_J\: | \: d(I,I_0)\leq k_2
\Big >.
\end{align}
For future use, we define the set indexing coordinates $z_I$ that do not vanish on some generator of $T_t$:
$$ 
\Delta=\left\{ I \: | \:  d(I,I_0)\leq k_1\right\}\bigcup \left(\bigcup_{d(I,I_0)\leq k_2} \!\!\!\!\Delta(I)^+\right)\subset \Lambda.
$$

On the other hand, by Proposition \ref{oscsegver}, we have 
$$
T^{k_1+k_2+1}_{e_{I_0}}=\left\langle e_I\: | \: d(I,I_0)\leq k_1+k_2+1\right\rangle
=\{z_I=0\: | \: d(I,I_0)> k_1+k_2+1 \}.
$$
In order to prove that $T_0\subset T^{k_1+k_2+1}_{e_{I_0}}$, we will define a family of linear subspaces $L_t$ 
whose flat limit at $t=0$ is $T^{k_1+k_2+1}_{e_{I_0}}$, and such that $T_t\subset L_t$ for every $t\neq 0$. 
(Note that we may assume that  $k_1+k_2\leq d-2$, for otherwise $T^{k_1+k_2+1}_{e_{I_0}}=\mathbb{P}^{N(\pmb{n},\pmb{d})}$.) 
For that, it is enough to 
exhibit, for each pair $(I,J)\in \Lambda^2$ with $d(I,I_0)> k_1+k_2+1$, 
a polynomial $f(t)_{(I,J)}\in \C[t]$ so that the hyperplane $(H_I)_t\subset\mathbb{P}^{N(\pmb{n},\pmb{d})}$ defined by
$$
z_I+t\left( \sum_{J\in \Lambda, \ J\neq I}f(t)_{(I,J)} z_J\right)=0
$$
satisfies $T_t\subset (H_I)_t$ for every $t\neq 0$.
If $I\notin \Delta$, then we can take $f(t)_{(I,J)}\equiv 0$ $\forall J\in \Lambda$. 
So from now on we assume that  $I\in \Delta$. 
We claim that it is enough to  find a hyperplane of type
\begin{align}\label{hyperplanesv}
F_I=\sum_{J\in \Delta(I)^-}t^{d(I,J)}c_J z_J =0,
\end{align}
with $c_J\in\C$ for $ J\in \Delta(I)^-$, $c_I\neq 0$, and such that $T_t\subset (F_I=0)$ for  $t\neq 0$. 
Indeed, once we find such $F_I$'s, we can take $(H_I)_t$ to be 
$$
z_I+\frac{t}{c_I}\left(\sum_{J\in \Delta(I)^-, \ J\neq I}t^{d(J,I)-1} c_J z_J\right)=0.
$$

In  \eqref{hyperplanesv}, there are $|\Delta(I)^-|$ indeterminates $c_J$. 
Let us analyze what conditions we get by requiring that $T_t\subseteq (F_I=0)$ for  $t\neq 0$. 
For any $e_K^t$ with non-zero coordinate $z_I$, we have $I\in \Delta(K)^+$, and so $K\in \Delta(I)^-$.
Given $K\in \Delta(I)^-$ we have
\begin{align*}
F_I(e_K^t)&\stackrel{(\ref{etonthebasis})}{=}\
F_I\left(\sum_{J\in \Delta(K)^+}\!\!\!\!
t^{d(K,J)} c_{(K,J)}e_J\right) = \\
&\stackrel{(\ref{hyperplanesv})}{=}\!\!\!\!\!\!\!\!\!\!\!\!\sum_{J\in \Delta(I)^-\cap \Delta(K)^+}
\!\!\!\!\!\!\!\!\!\!\!\! t^{d(I,K)-d(K,J)}c_J \left(  t^{d(K,J)} c_{(K,J)}\right)
=t^{d(I,K)}\left[\sum_{J\in \Delta(I)^-\cap \Delta(K)^+}\!\!\!\!\!\!\!\! \!\!\!\! c_{(K,J)}c_J \right].
\end{align*}
Thus:
$$
F_I(e_K^t)=0\: \forall\: t\neq 0 \Leftrightarrow\: \sum_{J\in \Delta(I)^-\cap \Delta(K)^+}\!\!\!\!\!\!c_{(K,J)}c_J =0.
$$
This is a linear condition on the coefficients $c_J$, with $J\in \Delta(I)^-$. Therefore
\begin{align}\label{eqns3sv}
T_t\subset (F_I=0) \mbox{ for } t\neq 0 &\Leftrightarrow 
\begin{cases}
F_I(e_L)=0\   &\forall L\in \Delta(I)^-\cap B[I_0,k_1]  \\ 
F_I(e_K^t)=0 \ \forall t\neq 0 \  &\forall K\in \Delta(I)^-\cap B[I_0,k_2] 
\end{cases}\\&\nonumber
\Leftrightarrow 
\begin{cases}
c_L=0  &\forall L\in \Delta(I)^-\cap B[I_0,k_1]\\
\displaystyle\sum_{J\in \Delta(I)^-\cap \Delta(K)^+}\!\!\!\!\!\!\!\!\!\!\!\!c_{(K,J)}c_J \quad
=0 \ &\forall K\in \Delta(I)^-\cap B[I_0,k_2],
\end{cases}
\end{align}
where $B[J,u]=\{K\in \Lambda |\: d(J,K)\leq u\}$. Set
$$
c=\left|\Delta(I)^-\cap B[I_0,k_1]\right|+\left|\Delta(I)^-\cap B[I_0,k_2]\right|.
$$
The problem is now reduced to finding a solution $(c_J)_{J\in \Delta(I)^-}$ of the linear system given by the $c$ equations (\ref{eqns3sv}) with $c_I\neq 0$.

In the following we write for short $s=s_I^-$, $\overline{s}=s_I^+$ and $D=d(I,I_0)>k_1+k_2+1$. 
We want to find $s+1$ complex numbers $c_I=c_0,c_1,\dots, c_{s}$ satisfying the following conditions 
\begin{align}\label{eqns4sv}
\begin{cases}
c_{j}=0  &\forall j=s,\dots,D-k_1\\
\displaystyle\sum_{l=0}^{d(I,K)}\left(c_{d(I,K)-l}\!\!\!\!\displaystyle\sum_{J\in \Delta(I)^-\cap\Delta(K,l)}
\!\!\!\!\!\!\!\!\!\!\!\!c_{(K,J)} \right)=0 
\ &\forall K\in \Delta(I)^-\cap B[I_0,k_2].
\end{cases}
\end{align}
For $0\leq l\leq d(I,K)$, we have
\begin{align*}
\displaystyle\sum_{J\in \Delta(I)^-\cap\Delta(K,l)}
\!\!\!\!\!\!\!\!\!\!\!\!c_{(K,J)}
&=\!\!\!\!\!\!\!\!\displaystyle\sum_{J\in \Delta(I)^-\cap\Delta(K,l)}
\!\!\!\!
\binom{s_{K^1}^+}{d(K^1,J^1)}\cdots \binom{s_{K^r}^+}{d(K^r,J^r)}
=\sum_{\substack{\pmb{l}=(l_1,\dots,l_r) \\ 0\leq l_1,\dots,l_r \\ l_1+\dots+l_r=l}}
\!\!\!\!\binom{s_{K^1}^+}{l_1}\cdots \binom{s_{K^r}^+}{l_r} = \\
&=\binom{s_{K^1}^+ +\dots+s_{K^r}^+}{l}=\binom{s_{K}^+}{l}=\binom{s_{I}^+ + d(I,K)}{l}.
\end{align*}
Thus the system (\ref{eqns4sv}) can be written as
\begin{align*}
\begin{cases}
c_{j}=0  &\forall j=s,\dots,D-k_1 \\
\displaystyle\sum_{k=0}^{j}\binom{\overline{s} +j}{j-k}c_{k} =0 
\ &\forall j=s,\dots,D-k_2,
\end{cases}
\end{align*}
that is
\begin{align}\label{eqns5sv}
\begin{cases}
c_{s}=0\\
\vdots\\
c_{D-k_1}=0\\
\end{cases}
\begin{cases}
\binom{\overline{s} +s}{0}c_{s}+\binom{\overline{s} +s}{1}c_{s-1}+\cdots+\binom{\overline{s} +s}{s}c_{0}=0\\
\vdots\\
\binom{\overline{s} +D-k_2}{0}c_{D-k_2}+\binom{\overline{s} +D-k_2}{1}c_{D-k_2-1}+\cdots
+\binom{\overline{s} +D-k_2}{D-k_2}c_{0}=0.
\end{cases}
\end{align}

We will show that the linear system (\ref{eqns5sv}) admits a solution with $c_0\neq 0$. 
If $s<D-k_2$, then the system (\ref{eqns5sv}) reduces to $c_s=\dots=c_{D-k_1}=0$. In this case we can take $c_0=1, c_1=\dots,c_s=0$. 
From now on assume that $s\geq D-k_2$. 
Since $c_s=\dots=c_{D-k_1}=0$  in (\ref{eqns5sv}), we are reduced to checking that the following system admits a solution 
$(c_i)_{0\leq i\leq D-k_1+1}$ with $c_0\neq 0$:
\begin{align*}
\begin{cases}
\binom{\overline{s} +s}{s-(D-k_1+1)}c_{D-k_1+1}+\binom{\overline{s} +s}{s-(D-k_1)}c_{D-k_1}+
\cdots+\binom{\overline{s} +s}{s}c_{0}=0\\
\vdots\\
\binom{\overline{s} +D-k_2}{k_1-1-k_2}c_{D-k_1+1}+\binom{\overline{s} +D-k_2}{k_1-k_2}c_{D-k_1}+
\cdots+\binom{\overline{s} +D-k_2}{D-k_2}c_{0}=0.
\end{cases}
\end{align*}
Therefore, it is enough to check that the $(s-D+k_2+1)\times (D-k_1+1)$ matrix
\begin{align*}
M=
\begin{pmatrix}
\binom{\overline{s} +s}{s-(D-k_1+1)}&\binom{\overline{s} +s}{s-(D-k_1)}&
\cdots&\binom{\overline{s}+s}{s-1}\\
\vdots&\vdots&\ddots &\vdots\\
\binom{\overline{s} +D-k_2}{k_1-1-k_2}&\binom{\overline{s} +D-k_2}{k_1-k_2}&
\cdots&\binom{\overline{s} +D-k_2}{D-k_2-1}
\end{pmatrix}
\end{align*}
has maximal rank. Since $s\leq D$ and $D>k_1+k_2+1$, we have $s-D+k_2+1< D-k_1+1$.
So it is enough to show that the $(s-D+k_2+1)\times (s-D+k_2+1)$ submatrix of $M$
\begin{align*}
M'=&
\begin{pmatrix}
\binom{\overline{s} +s}{s-(s-D+k_2+1)}&\binom{\overline{s} +s}{s-(s-D+k_2)}&
\cdots&\binom{\overline{s} +s}{s-1}\\
\vdots&\vdots& \ddots &\vdots\\
\binom{\overline{s} +D-k_2}{D-k_2-(s-D+k_2+1)}&\binom{\overline{s} +D-k_2}{D-k_2-(s-D+k_2)}&
\cdots&\binom{\overline{s} +D-k_2}{D-k_2-1}
\end{pmatrix} =\\
=&
\begin{pmatrix}
\binom{\overline{s} +s}{\overline{s} +s+1-D+k_2}&\binom{\overline{s} +s}{\overline{s} +s-D+k_2}&
\cdots&\binom{\overline{s} +s}{\overline{s} +1}\\
\vdots&\vdots& \ddots &\vdots\\
\binom{\overline{s} +D-k_2}{\overline{s} +s+1-D+k_2}&\binom{\overline{s} +D-k_2}{\overline{s} +s-D+k_2}&
\cdots&\binom{\overline{s} +D-k_2}{\overline{s}+1}
\end{pmatrix}
\end{align*}
has non-zero determinant. To conclude,  observe that the determinant of $M'$ is equal to the determinant of the matrix of binomial coefficients 
$$
M'':=\left(\binom{i}{j}\right)_{\substack{ \hspace{-0.7cm}
\overline{s} +D-k_2\leq i\leq \overline{s} + s\\ \overline{s} +1\leq j\leq \overline{s} +s+1-D+k_2}}.
$$
Since $D-k_2>k_1+1\geq 1$,  $\det(M')=\det(M'')\neq 0$ by \cite[Corollary 2]{GV85}.
\end{proof}

In the following example we work out explicitly the proof of Proposition \ref{limitosculatingspacessegver}.

\begin{Example}\label{exampledegsv}
Consider the case $SV_{(3,2)}^{(1,2)}\subset \P^{23}.$
Then $I_0=(000,00), I_1=(111,11)$ and we have 
$$d=d_1+d_2=3+2=5,s^+_{(000,00)}=5,s^+_{(000)}=3,s^+_{(00)}=2,$$
and
$$
\begin{array}{rl}
\Delta(I_0,1)&=\{\delta^{(1,0)}(I_0)=(001,00),\delta^{(0,1)}(I_0)=(000,01)\}\\
\Delta(I_0,2)&=\{\delta^{(2,0)}(I_0)=(011,00),\delta^{(1,1)}(I_0)=(001,01),\delta^{(0,2)}(I_0)=(000,11)\}\\
\Delta(I_0,3)&=\{\delta^{(3,0)}(I_0)=(111,00),\delta^{(2,1)}(I_0)=(011,01),\delta^{(1,2)}(I_0)=(001,11)\}\\
\Delta(I_0,4)&=\{\delta^{(3,1)}(I_0)=(111,01),\delta^{(2,2)}(I_0)=(011,11)\}\\
\Delta(I_0,5)&=\{\delta^{(3,2)}(I_0)=(111,11)\}\\
B[I_0,1]&=\{(000,00),(001,00),(000,01),(000,02)\}\\
B[I_0,2]&=B[I_0,1]\bigcup \{(011,00),(001,01),(001,02),(000,11),(000,12),(000,22)\}.
\end{array}
$$
Let us work out the case $k_1=2,k_2=1$. By Proposition \ref{oscsegver} we have 
$$T_t=\left\langle e_I\: | \: I\in B[I_0,2]; \: e_I^{t}\: |\:  I\in B[I_0,1]\right\rangle
,\: t\neq 0$$
and
$$T^{k_1+k_2+1}_p=\left\langle e_I\: | \:  I\in B[I_0,4]\right\rangle
=\{p_I=0\: | \:  d(I,I_0)=5 \}.$$
Now we have to write the generators of $T^{k_2}_{\gamma(t)}$ on the basis $(e_I)_{I\in \Lambda}:$
\begin{align}\label{etonthebasisexampleIsv}
\begin{cases}
e_{(000,00)}^t=&\!\!\!\! e_{(000,00)}+t\left(\binomm{3}{1}e_{(001,00)}+\binomm{2}{1}e_{(000,01)}\right)\\
&+t^2\left(\binomm{3}{2}e_{(011,00)}+\binomm{3}{1}\binomm{2}{1}e_{(001,01)}+\binomm{2}{2}e_{(000,11)}\right)\\
&+t^3\left(\binomm{3}{3}e_{(111,00)}+\binomm{3}{2}\binomm{2}{1}e_{(011,01)}+\binomm{3}{1}\binomm{2}{2}e_{(001,11)}\right)\\
&+t^4\left(\binomm{3}{3}\binomm{2}{1}e_{(111,01)}+\binomm{3}{2}\binomm{2}{2}e_{(011,11)}\right)
+t^5 e_{(111,11)}\\
e_{(001,00)}^t=&\!\!\!\! e_{(001,00)}+t\left(\binomm{2}{1}e_{(011,00)}+\binomm{2}{1}e_{(001,01)}\right)\\
&+t^2\left(\binomm{2}{2}e_{(111,00)}+\binomm{2}{1}\binomm{2}{1}e_{(011,01)}+\binomm{2}{2}e_{(001,11)}\right)\\
&+t^3\left(\binomm{2}{2}\binomm{2}{1}e_{(111,01)}+\binomm{2}{1}\binomm{2}{2}e_{(011,11)}\right)+t^4 e_{(111,11)}\\
e_{(000,01)}^t=&\!\!\!\! e_{(000,01)}+t\left(\binomm{3}{1}e_{(001,01)}+\binomm{1}{1}e_{(000,11)}\right)
+t^2\left(\binomm{3}{2}e_{(011,01)}+\binomm{3}{1}\binomm{1}{1}e_{(001,11)}\right)\\
&+t^3\left(\binomm{3}{3}e_{(111,01)}+\binomm{3}{2}\binomm{1}{1}e_{(011,11)}\right)+t^4 e_{(111,11)}\\
e_{(000,02)}^t=&\!\!\!\! e_{(000,02)}+t\left(\binomm{3}{1}e_{(001,02)}+\binomm{1}{1}e_{(000,12)}\right)
+t^2\left(\binomm{3}{2}e_{(011,02)}+\binomm{3}{1}\binomm{1}{1}e_{(001,12)}\right)\\
&+t^3\left(\binomm{3}{3}e_{(111,02)}+\binomm{3}{2}\binomm{1}{1}e_{(011,12)}\right)+t^4 e_{(111,12)}
\end{cases}
\end{align}
Now, given $I\in \Lambda$ with $d(I,I_0)>4=k_1+k_2+1$ we have to find a hyperplane $H_I$ of type
$$c_Ip_I+t \!\!\!\!\sum_{J\in \Delta(I,-1)}\!\!\!\! c_J p_J+
t^2 \!\!\!\!\sum_{J\in \Delta(I,-2)}\!\!\!\! c_J p_J+
t^3 \!\!\!\!\sum_{J\in \Delta(I,-3)}\!\!\!\! c_J p_J=0$$
such that $c_I\neq 0$, and $T_t\subseteq H_I$ for every $t\neq 0.$

In our case there is three such $I'$s, namely: 
$$A_1=I_1=(111,11),A_2=(111,12),A_3=I_1(111,22).$$

Set
$$\Delta:=B[I_0,2]\bigcup \left( \cup_{J\in B[I_0,1]} \Delta(J)^+\right).$$
Since $A_3\notin \Delta$ then $p_{A_3}=0$ works for it.
We have $A_1,A_2\in \Delta.$ Note that $A_2$ appears on (\ref{etonthebasisexampleIsv})
only on the the writing of $e_{(000,02)}^t.$ Also note that 
$$e_{(000,02)}^t\in \left\{ p_{(111,12)}-tp_{(111,02)}=0 \right\}, t\neq 0$$
and since $d((111,02),I_0)=4$ this hyperplane works for $A_2.$

It remains find the hyperplane $H_{A_1},$ we will illustrate the proof of Proposition \ref{limitosculatingspacessegver}
on this case. Observe that $s^+_{A_1}=\overline{s}=0,\ s^-_{A_1}=s=5$ and $D=d(A_1,I_0)=5.$

We are looking for numbers
\begin{align*}
c_0&=c_{(111,11)}\neq 0,c_1=c_{(111,01)}=c_{(011,11)},
c_2=c_{(111,00)}=c_{(011,01)}=c_{(001,11)},\\
c_3&=c_{(011,00)}=c_{(001,01)}=c_{(000,11)},
c_4=c_{(001,00)}=c_{(000,01)},
c_5=c_{(000,00)}
\end{align*}
satisfying the following system:
\begin{align}\label{eqnssvexample}
\begin{cases}
c_{5}=0\\
c_{4}=0\\
c_{3}=0\\
c_{5}+\binom{5}{1}c_{4}+\binom{5}{2}c_{3}+\binom{5}{3}c_{2}+\binom{5}{4}c_{1}+c_{0}=0\\
c_{4}+\binom{4}{1}c_{3}+\binom{4}{2}c_{2}+\binom{4}{3}c_{1}+c_{0}=0
\end{cases}
\end{align}
where the first three conditions came from the condition $T^2_p\subset H_{A_1}$ and the last two came from 
the condition $T_t\subset H_{A_1}, t\neq 0.$ Note that the matrix
\begin{equation*}
M=
\begin{pmatrix}
\binom{5}{3}&\binom{5}{4}\\
\binom{4}{2}&\binom{4}{3}
\end{pmatrix}=
\begin{pmatrix}
10&5\\
6&4
\end{pmatrix}
\end{equation*}
has maximal rank. Therefore, there exist numbers $c_0\neq 0,c_1,c_2,c_3,c_4,c_5$ satisfying the system (\ref{eqnssvexample}).
For instance, we may take $c_0=10,c_1=-4,c_2=1,c_3=c_4=c_5=0$ corresponding to the hyperplane  
$$10p_{(111,11)}-4t(p_{(111,01)}+p_{(011,11)})+t^2(p_{(111,00)}+p_{(011,01)}+p_{(001,11)})=0.$$
\end{Example}

In Proposition\ref{limitosculatingspacessegverII} below we prove that $SV_{\pmb d}^{\pmb n}$ has $(n_1+1)$-osculating regularity. But before we give a simple example illustrating the strategy of its proof.

\begin{Example}\label{exstrongdeg}
Set $d=d_1=4,n=n_1=2$ and $k=1.$ We want that 
$$T_t:=\left\langle T^1_{e_{I_0}},T^1_{\gamma_1(t)},T^1_{\gamma_2(t)}\right\rangle$$
be such that $T_0\subset T^3_{e_{I_0}}.$

Let $I=(1112)\in \Lambda$ and we want to find a family of hyperplanes 
$$H_I=\left\{p_I+t\sum_{I\neq J\in A} c_Jp_J=0\right\}$$
such that $T_t \subset H_I\ \forall t\neq 0$ for some set $A\subset \Lambda$ well chosen.

Note that if we choose 
$$A=\Delta(I)_1^-=\{(1112),(0112),(0012),(0002)\}$$
then we will have to impose
$$e_{(0000)}^{2,t},e_{(0001)}^{2,t},e_{(0002)}^{1,t} \in H_I.$$
This is bad because we have to deal with points on different rational normal curves $\gamma_1$ and $\gamma_2.$
We rather choose
$$A=\Gamma=\{(1112),(0112)\}$$
and only have to impose
$$e_{(0002)}^{1,t} \in H_I.$$

Doing this we use much less variables and conditions, and much better, we are in the same situation
of the proof of Proposition \ref{limitosculatingspacessegver}.
\end{Example}

\begin{Proposition}\label{limitosculatingspacessegverII}
The Segre-Veronese variety $SV_{\pmb{d}}^{\pmb{n}}\subseteq\mathbb{P}^{N(\pmb{n},\pmb{d})}$ has 
$(n_1+1)$-osculating regularity.
\end{Proposition}

\begin{proof}
We follow the same argument and computations as in the proof of Proposition~\ref{limitosculatingspacessegver}.

Given general points $p_0,\dots,p_{n_1}\in SV_{\pmb d}^{\pmb n}\subseteq\mathbb{P}^{N(\pmb n,\pmb d)}$, we may 
apply a projective automorphism of $SV_{\pmb{d}}^{\pmb{n}}\subseteq\mathbb{P}^{N(\pmb{n},\pmb{d})}$ and assume that 
$p_j=e_{I_j}$ for every $j$.
Each $p_j$, $j\geq 1$, is connected to $p_0$ by the degree $d$ rational normal curve defined by 
$$
\gamma_j([t:s])=(se_0+te_j)^{d_1}\otimes\dots\otimes (se_0+te_j)^{d_r}.
$$
We work in the affine chart $(s=1)$, and set $t=(t:1)$. 
Given $k\geq 0$, consider the family of linear spaces 
$$
T_t=\left\langle T^{k}_{p_0},T^{k}_{\gamma_1(t)},\dots,T^{k}_{\gamma_{n_1}(t)}\right\rangle,\: t\in \C\backslash \{0\}.
$$
We will show that the flat limit $T_0$  of $\{T_t\}_{t\in \C\backslash \{0\}}$ in
$\G(\dim(T_t),N(\pmb{n},\pmb{d}))$ is contained in $T^{2k+1}_{p_0}$.

We start by writing the linear spaces $T_t$ explicitly  in terms of the basis $\{e_J|J\in\Lambda\}$. 
As in the proof of Proposition~\ref{limitosculatingspacessegver}, it is convenient to introduce some additional notation. 

Given $I\in \Lambda_{n,d}$, we define $\delta^l_j(I),l\geq 0,$ as in Notation~\ref{defdelta}, 
with the only difference that this time we substitute $0$'s with $j$'s instead of $1$'s.
Similarly, for $I=(I^1,\dots,I^r)\in \Lambda$, $\pmb{l}=(l_1,\dots,l_r)\in \Z^r$, and $l\in \Z$, we define the sets 
$\Delta(I,l)_j, \Delta(I)^+_j, \Delta(I)^-_j \subset \Lambda$, and the integers
$s(I)^+_j,s(I)^-_j\in\{0,\dots,d\}$.

For $j=1,\dots, r$, we define the vectors 
$$
e_0^{j,t}=e_0+te_j,e_1^{j,t}=e_1,e_2^{j,t}=e_2,\dots,e_{n_j}^{j,t}=e_{n_j}\in V_j.
$$
Given $I^l=(i_1,\dots,i_{d_j})\in \Lambda_{n_j,d_j}$, we denote by 
$e_{I^j}^{j,t}\in \Sym^{d_j}V_j$  the symmetric product $e_{i_1}^{j,t}\cdot\ldots\cdot e_{i_{d_j}}^{j,t}$.
Given $I=(I^1,\dots,I^r)\in \Lambda=\Lambda_{\pmb{n},\pmb{d}}$, we denote by 
$e_I^{j,t}\in \mathbb{P}^{N(\pmb{n},\pmb{d})}$
the point corresponding to 
$$
e_{I^1}^{j,t}\otimes\dots\otimes e_{I^r}^{j,t}\in \Sym^{d_1}V_1\otimes\dots\otimes \Sym^{d_r}V_r.
$$
By Proposition \ref{oscsegver} we have 
$$
T_t=\left\langle e_I\: | \: d(I,I_0)\leq k; \: e_I^{j,t}\: |\: d(I,I_0)\leq k, j=1,\dots,n_1\right\rangle,\: t\neq 0.
$$
Now we write each vector $e_I^{j,t}$, with $I=(I^1,\dots,I^r)\in\Lambda$ such that $d(I,I_0)\leq k$, in terms of the basis $\{e_J|J\in\Lambda\}$:
\begin{align*}
e_{I}^{j,t}
&=\!\!\!\!\sum_{J=(J^1,\dots,J^r)\in \Delta(I)^+_j}\!\!\!\!
t^{d(I,J)} c_{(I,J)}e_J
\end{align*}
where $ c_{(I,J)}=\binom{s(I^1)^+_j}{d(I^1,J^1)}\cdots \binom{s(I^r)^+_j}{d(I^r,J^r)}$. 
So we can rewrite the linear subspace $T_t$ as 
\begin{align*}
T_t=&\left\langle e_I \: | \: d(I,I_0)\leq k;\:
\!\!\!\!\sum_{J\in \Delta(I)^+_j}\!\!\!\!
t^{d(I,J)} c_{(I,J)}e_J\: | \: d(I,I_0)\leq k,\ j=1,\dots,n_1
\right \rangle,
\end{align*}
and define the set
$$ 
\Delta=
\bigcup_{1\leq j\leq n_1} 
\bigcup_{d(J,I_0)\leq k} \!\!\!\!\Delta(J)^+_j\subset \Lambda.
$$

On the other hand, by Proposition \ref{oscsegver}, we have 
$$
T^{2k+1}_{p_0}=\left\langle e_I\: | \: d(I,I_0)\leq 2k+1\right\rangle=\{z_I=0\: | \: d(I,I_0)> 2k+1 \}.
$$

As in the proof of Proposition \ref{limitosculatingspacessegver},
in order to prove that $T_0\subset T^{2k+1}_{p_0}$,  it is enough to 
exhibit, for each  $I\in \Delta$ with $d(I,I_0)> 2k+1$, a family of hyperplanes of the form
\begin{equation}\label{eq1sv}
\left(F_I= \sum_{J\in \Gamma(I)} t^{d(I,J)}c_J z_J=0 \right)
\end{equation}
such that $T_t\subset (F_I=0)$ for $t\neq 0$, and $c_I\neq 0$. Here $\Gamma(I)\subset \Lambda$ is a suitable subset to be defined later. Let $I\in \Delta$ be such that $d(I,I_0)> 2k+1$.
We claim that there is a unique $j$ such that 
\begin{equation}\label{I->j}
I\in \bigcup_{d(J,I_0)\leq k} \!\!\!\!\Delta(J)^+_j.
\end{equation}
Indeed, assume that $I\in \Delta(J,l)_i$ and $I\in\Delta(K,m)_j$, with $d(J,I_0),d(K,I_0)\leq k$.
If $i\neq j$, then we must have 
$$
 d(J,I_0) \geq m \ \ \text{ and } \ \  d(K,I_0) \geq l.
$$
But then $d(I,I_0) = d(J,I_0) + l \leq d(J,I_0) + d(K,I_0)  \leq 2k$, contradicting the assumption that $d(I,I_0)> 2k+1$.
Let $J$ and $j$ be such that $d(J,I_0)\leq k$ and $I\in \Delta(J)^+_j$.
Note that $d(I,I_0)-s(I)^-_j=d(J,I_0) - s(J)^-_j\leq k$, and hence $k+1-d(I,I_0)+s(I)^-_j>0$.
We set $D=d(I,I_0)$ and define 
\begin{equation}\label{Gamma}
\Gamma(I)=\!\!\!\!\!\!\!\!\bigcup_{0\leq l \leq k+1-D+s(I)^-_j} \!\!\!\!\!\!\!\!\Delta(I,-l)_j \subset \Lambda.
\end{equation}
This is the set to be used in \eqref{eq1sv}. First we claim that 
\begin{equation}\label{eq2sv}
J\in \Gamma(I)\Rightarrow J \notin  \bigcup_{\substack{1\leq i\leq n_1 \\ i\neq j} }
\bigcup_{d(I,I_0)\leq k} \!\!\!\!\Delta(I)^+_i.
\end{equation}
Indeed, assume that $J\in \Delta(I,-l)_j$ with $0\leq l \leq k+1-D+s(I)^-_j$, and
$J\in \Delta(K)^+_i$  for some $K$ with $d(K,I_0)\leq k$.
If $i\neq j$, then 
$$
s(K)^-_j=s(J)^-_j=s(I)_j^- - l\geq D-(k+1)>k,
$$
contradicting the assumption that $d(K,I_0)\leq k$.
Therefore, if $F_I$ is as in \eqref{eq1sv} with $\Gamma(I)$ as in \eqref{Gamma}, then we have 
$$
\left\langle e_I\: | \: d(I,I_0)\leq k; \: \sum_{J\in \Delta(I)^+_i}\!\!\!\!
t^{d(I,J)} c_{(I,J)}e_J\: | \: d(I,I_0)\leq k,\ i=1,\dots,n_1, i\neq j \right\rangle \subset (F_I=0) , t\neq 0,$$
and thus
$$
T_t\subset (F_I=0), t \neq 0 \Longleftrightarrow \left\langle \sum_{J\in \Delta(I)^+_j}\!\!\!\!
t^{d(I,J)} c_{(I,J)}e_J\: | \: d(I,I_0)\leq k \right\rangle \subset (F_I=0) , t\neq 0.
$$

The same computations as in the  proof of Proposition \ref{limitosculatingspacessegver} yield
\begin{equation}
\label{eq3}T_t\subset (F_I=0), t \neq 0 \Longleftrightarrow \sum_{J\in \Delta(K)^+_j \cap \Gamma(I)}\!\!\!\! c_{J}c_{(K,J)}=0
\ \ \forall  K\in \Delta(I)^-_j\cap B[I_0,k].
\end{equation}
So the problem is reduced to finding a solution $(c_J)_{J\in\Gamma(I)}$ for the linear system \eqref{eq3} such that $c_I\neq 0$. 
We set $c_J=c_{d(I,J)}$ and reduce, as in the proof of Proposition \ref{limitosculatingspacessegver}, to the linear system
\begin{equation}\label{linsyssv}
\displaystyle\sum_{l=0}^{k+1-D+s(I)^-_j}\binom{d-i}{D-l-i}c_{l} =0 ,
\ \ \ D-s(I)^-_j\leq i \leq k
\end{equation}
in the variables $c_0,\dots,c_{k+1-D+s(I)^-_j}$. 
The argument used in the end of Proposition \ref{limitosculatingspacessegver} shows that the linear system (\ref{linsyssv}) admits a solution with $c_0\neq 0$.
\end{proof}

Proposition \ref{limitosculatingspacessegverII} says that $SV_{\pmb d}^{\pmb n}$ has $(n_1+1)$-osculating regularity. In particular, the Veronese variety $V_d^n$ has $(n+1)$-osculating regularity. Note that in principle $SV_{\pmb d}^{\pmb n}$ may have bigger osculating regularity. On other hand, there are surfaces not having $3$-osculating regularity as the example below shows.

\begin{Example}
Let us consider the rational normal scroll $X_{(1,7)}\subset\mathbb{P}^9$. A parametrization of $X_{(1,7)}$ is given by 
$$
\begin{array}{cccc}
\phi: & \mathbb{A}^2 & \longrightarrow & \mathbb{A}^9\\ 
 & (u,\alpha) & \mapsto & (\alpha u^7,\alpha u^6,\dots, \alpha u, \alpha, u)
\end{array} 
$$ 
Note that $\frac{\partial^2\phi}{\partial\alpha^2} = \frac{\partial^3\phi}{\partial\alpha^3} = \frac{\partial^3\phi}{\partial\alpha^2u} = 0$, while there are not other relations between the partial derivatives, up to order three, of $\phi$ at the general point of $X_{(1,7)}$. Therefore
$$\dim (T^3X_{(1,7)}) = 10-3-1 = 6.$$
On the other hand by \cite[Lemma 4.10]{DP96} we have that $\dim(\sec_3(X_{(1,7)})) = 7$. Hence by Terracini's lemma \cite[Theorem 1.3.1]{Ru03} the span of three general tangent spaces of $X_{(1,7)}$ has dimension seven. Therefore $X_{(1,7)}$ has not $3$-osculating regularity.
\end{Example}

\section{Non-secant defectivity of Segre-Veronese varieties}\label{mainsec}

In this section we study the dimension of secant varieties of Segre-Veronese varieties. 
First we state our main result, Theorem \ref{mainsv},
then we give some examples.

\begin{Theorem}\label{mainsv}
The Segre-Veronese variety $SV^{\pmb n}_{\pmb d}$ is not $h$-defective where
$$h\leq n_1h_{n_1+1}(d-2)+1$$
and $h_{n_1+1}(\cdot)$ is as in Definition \ref{defhowmanytangent}.
\end{Theorem}
\begin{proof}
We have shown in Propositions~\ref{limitosculatingspacessegver} and \ref{limitosculatingspacessegverII} that 
the Segre-Veronese variety $SV_{\pmb d}^{\pmb n}$ has 
strong $2$-osculating regularity, and $(n_1+1)$-osculating regularity.
The result then follows immediately from Proposition \ref{projoscseveralpoints} and Theorem~\ref{lemmadefectsviaosculating}.
\end{proof}

\begin{Remark}
Write 
$$
d-1 = 2^{\lambda_1}+2^{\lambda_2}+\dots + 2^{\lambda_s}+\epsilon
$$
with $\lambda_1 >\lambda_2>\dots >\lambda_s\geq 1$, $\epsilon\in\{0,1\}$, so that $\lambda_1 = \lfloor \log_2(d-1)\rfloor$.
By Theorem \ref{mainsv} $SV^{\pmb n}_{\pmb d}$ is not $h$-defective for
$$
h\leq n_1((n_1+1)^{\lambda_1-1}+\dots + (n_1+1)^{\lambda_s-1})+1.
$$
So we have that asymptotically  $SV^{\pmb n}_{\pmb d}$ is not $h$-defective for
$$
h\leq n_1^{\lfloor \log_2(d-1)\rfloor}.
$$
\end{Remark}


Recall  \cite[Proposition 3.2]{CGG03}:
except for the Segre product $\mathbb{P}^{1}\times\mathbb{P}^{1}\subset\mathbb{P}^3$, 
the Segre-Veronese variety $SV^{\pmb n}_{\pmb d}$ is not $h$-defective for $h\leq \min\{n_i\}+1$,
independently of $\pmb d$.
In the following table, for a few values of $d$, we compute the highest value of $h$ for which Theorem~\ref{mainsv} 
gives non $h$-defectivity of $SV^{\pmb n}_{\pmb d}$.

\medskip

\begin{center}
\begin{tabular}{|c|l|}
\hline 
$d = d_1+\dots +d_r$ & $h$\\ 
\hline 
$3$ & $n_1+1$\\ 
\hline 
$5$ & $n_1(n_1+1)+1$\\ 
\hline 
$7$ & $n_1((n_1+1)+1)+1$\\ 
\hline 
$9$ & $n_1(n_1+1)^2+1$\\ 
\hline 
$11$ & $n_1((n_1+1)^2+1)+1$\\ 
\hline 
$13$ & $n_1((n_1+1)^2+n_1+1)+1$\\ 
\hline
$15$ & $n_1((n_1+1)^2+(n_1+1)+1)+1$\\
\hline 
$17$ & $n_1(n_1+1)^3+1$\\
\hline
\end{tabular} 
\end{center}

\begin{Remark}
Note that the bound of Theorem \ref{mainsv} is sharp in some cases. For instance, by Proposition \ref{def} we know that $SV^{(1,1)}_{(2,2)}$, $SV^{(1,1,1)}_{(1,1,2)}$, $SV^{(1,1,1,1)}_{(1,1,1,1)}$ are $3$-defective, and $SV^{(2,2,2)}_{(1,1,1)}$ is $4$-defective. On the other hand $SV^{(1,1)}_{(2,2)}$, $SV^{(1,1,1)}_{(1,1,2)}$, $SV^{(1,1,1,1)}_{(1,1,1,1)}$ are not $2$-defective, and $SV^{(2,2,2)}_{(1,1,1)}$ is not $3$-defective
by Theorem \ref{mainsv}.
\end{Remark}

Now we show that in the Theorem \ref{mainsv} one can not in general change $n_1$ for $n_2.$

\begin{Example}
Consider $X=SV^{1,n}_{4,2}$ with $n\geq 2.$
Then $X$ is not $3$-defective.
If Theorem \ref{mainsv} could be improved from $n_1h_{n_1+1}(d-2)+1$
to $n_2h_{n_2+1}(d-2)+1$ then it would imply that $X$ is not $(n(n+1)+1)$-defective,
but $X$ is $(2n+3)$-defective by \cite[Theorem 3.4]{Ab08}.
\end{Example}

\part{Grassmannians' blow-ups and Mori dream spaces}\label{part2}
\chapter{Overview}\label{intro2}
Mori dream spaces play an important role in birational algebraic geometry. They were first introduced by Hu and Keel in \cite{HK}, and have been studied since then, see for instance
\cite{CT06,BCHM,Oka,CT,AM16}. 

\begin{Definition}
A normal projective variety $X$ is a \textit{Mori dream space}, MDS for short, if
\begin{itemize}
\item[(a)] $X$ is $\mathbb{Q}$-factorial and $\Pic(X)$ is finitely generated;
\item[(b)] $\Nef(X)$ is generated by finitely many semiample divisors;
\item[(c)] there exist finitely many small $\mathbb{Q}$-factorial modifications $f_i:X\dasharrow X_i$ such that each $X_i$ satisfies $\rm(b)$, and 
$$\Mov(X)=\bigcup_i \Nef(X_i).$$
\end{itemize}
\end{Definition}

The birational geometry of a Mori dream space $X$ can be encoded in some finite data, namely its cones of effective and movable divisors $\Eff(X)$ and $\Mov(X)$ together with a chamber decomposition on them, called the \textit{Mori chamber decomposition} of $X$.
Some classes of varieties are known to be Mori dream spaces:

\begin{itemize}
\item[-] A normal $\mathbb{Q}$-factorial projective variety of Picard number one is a Mori dream space if and only if $\Pic(X)$ is finitely generated.
\item[-] The image of a MDS is a MDS. More precisely, let $f:X\to Y$ be a contraction, see definition in Section \ref{sec51}, and $X,Y$ normal $\Q$-factorial projective varieties. If $X$ is a Mori dream space, then $Y$ is as well \cite{Oka}.
\item[-] Any projective $\mathbb{Q}$-factorial toric variety is a Mori dream space \cite{Re83,Cox}.
\item[-] Any projective $\mathbb{Q}$-factorial spherical  variety is a Mori dream space \cite{Brion93}, see Chapter \ref{cap6} for the definition and a discussion about some spherical varieties.
\item[-] Any smooth Fano or weak Fano (see Chapter \ref{cap7}), or more generally, any smooth log Fano variety is a Mori dream space \cite{BCHM}.
\end{itemize}

Now we recall a different characterization of Mori dream spaces.

\begin{Definition}
Let $X$ be a normal $\mathbb{Q}$-factorial projective variety with finitely generated and free Picard group and Picard number $\rho_X$. Let $D_1,...,D_{\rho_X}$ be a basis of Cartier divisors of $\Pic(X)$. We define the Cox ring of $X$ as follows 
$$\Cox(X) =\!\!\!\! \bigoplus_{m_1,...,m_{\rho_{X}}\in\mathbb{Z}}H^0\left(X,\sum_{i=1}^{\rho_{X}}m_{i}D_i\right).$$
Different choices of divisors $D_1,...,D_{\rho_X}$ yield isomorphic algebras.
\end{Definition}

For a comprehensive survey on Cox rings we refer to  \cite{coxrings}, and  for the proof of the following theorem we refer to \cite[Proposition 2.9]{HK}. 

\begin{Theorem}
A $\mathbb{Q}$-factorial projective variety $X$ with $\Pic(X)_{\mathbb{R}}\cong N^1(X)$ is a Mori dream space if and only if $\Cox(X)$ is finitely generated. In this case $X$ is a $GIT$ quotient of the affine variety $Y = \Spec(\Cox(X))$ by the action of a torus of dimension $\rho_X$. 
\end{Theorem}

We also observe that Cox introduced the Cox ring of a toric variety in \cite{Cox} refering to it as  the \textit{homogeneous total ring}, and proving that it is a polynomial ring. Now it is known that the Cox ring of a Mori dream space $X$ is a polynomial ring if and only if $X$ is a toric variety, \cite[Proposition 2.10]{HK}. 
In \cite{HKW} Hausen, Keicher and Wolf use $\Cox(X)$ to study the group of automorphisms of a Mori dream space $X.$
These results show how the geometry of $X$ is reflected in algebraic properties of $\Cox(X).$

A general problem is to determine whether a given variety $X$ is a MDS. Once this is the case, one may want to describe the ring $\Cox(X)$ and the Mori chamber decomposition of $\Eff(X).$ 

In practice, to determine exactly the Cox ring of a Mori dream space may be quite complicated but some work has been done, mainly in the surface case \cite{AHL,AL11,DLHHK15,AGL16}.
To determine the Mori chamber decomposition of a Mori dream space is also hard in general, although some successful attempts have been made  \cite{Hu15,BKR16}.

Next, we discuss two special classes of Mori dream spaces.

In \cite[Question 3.2]{HK} Hu and Keel asked if $\overline{M}_{0,n}$ is a Mori dream space. If $n = 4,5$ this is well known because $\overline{M}_{0,4}\cong\mathbb{P}^1$ and $\overline{M}_{0,5}$ is a del Pezzo surface of degree five. By \cite{HK} $\overline{M}_{0,n}$ is log Fano if and only if $n\leq 6$. In particular $\overline{M}_{0,6}$ is a Mori dream space.
In addition, Castravet in \cite{Ca09} gave an explicit proof of the finite generation of $\Cox(\overline{M}_{0,6}).$

Later, Castravet and Tevelev in \cite{CT} proved that $\overline{M}_{0,n}$ is not a Mori dream space for $n > 133.$ This result has been improved using the same techniques of Castravet and Tevelev.
Gonz\`alez and Karu in \cite{GK} showed that $\overline{M}_{0,n}$ is not a Mori dream space for $n > 12,$ and recently  Hausen,  Keicher and Laface observed in \cite{HKL} that the same is true for $n>10$ as well. Therefore, it remains open
whether $\overline{M}_{0,n}$ is a Mori dream space for $n = 7, 8, 9.$

The following result was proved in one direction by Mukai \cite{Mu04} and on the other by Castravet and Tevelev \cite{CT06}:
\begin{Theorem}
Let $X^n_k$ be the blow-up of $\P^n$ at $k$ points in general position, with $n\geq 2$ and $k\geq 0$.
Then  $X^n_k$ is a Mori dream space if and only if one of the following holds:
\begin{itemize}
	\item[-] $n=2$ and $k\leq 8$,
	\item[-] $n=3$ and $k\leq 7$,
	\item[-] $n=4$ and $k\leq 8$,
	\item[-] $n>4$ and $k\leq n+3$.
\end{itemize}
\end{Theorem}

In \cite{Mu05} Mukai gives explicitly the Mori chamber decomposition of  $X^n_k$ on the conditions above. In \cite{AM16} Araujo and Massarenti give a explicit log Fano structure of $X^n_k$ in the cases when it is a MDS.

The results in this dissertation arise from the interest in describing under which conditions the blow-up $\G(r,n)_k$ at $k$ general points of the Grassmannian $\G(r,n)$ is a Mori dream space. This was inspired by the analogous problem for $X^n_k.$
More precisely, one can consider two problems.

\begin{Problem}\label{q1}
For which triples $(r,n,k)$ is $\G(r,n)_k,$ the blow up of the Grassmannian $\G(r,n)$ at $k$ general points, a Mori dream space?
\end{Problem}

\begin{Problem}\label{q2}
If $\G(r,n)_k$ is a Mori dream space, describe the Mori chamber decomposition of its effective cone.
\end{Problem}

These two are our guideline problems for Part \ref{part2}. At the best of our knowledge these two problems are open in general.

In Chapter \ref{cap5} our aim is to the study Problem \ref{q2} for $r=k=1,$ and we prove the following.

\begin{Theorem}\label{teointro}
Let $n\geq 4$ and consider a basis of $\Pic(\G(1,n)_1)$ given by the pullback $H$ of an ample divisor on $\G(1,n),$
and by the class $E$ of the exceptional divisor.
Then $\G(1,n)_1$ is a Mori dream space  with 
effective cone \begin{center}
$\Eff(\G(1,n)_1)=\cone(E,H-2E).$
\end{center} 

Its Mori chamber decomposition is given by the walls $E,H,H-E,$ and $H-2E.$

There exists an isomorphism in codimension two $\eta:\G(1,n)_1\dasharrow \G(1,n)_1^+$ to another Mori dream space $\G(1,n)_1^+$ with 
$\Nef(\G(1,n)_1^+)=\cone(H-E,H-2E).$

The cone of movable divisors is 
$\Mov(\G(1,n)_1)\!=\!\cone(H,H-2E).$

The variety $\G(1,n)_1^+$ is a fibration over $\G(1,n-2)$ with fibers isomorphic to $\P^4,$ and if $n\geq 5$, then $\G(1,n)_1^+$ is a Fano variety.
\end{Theorem}

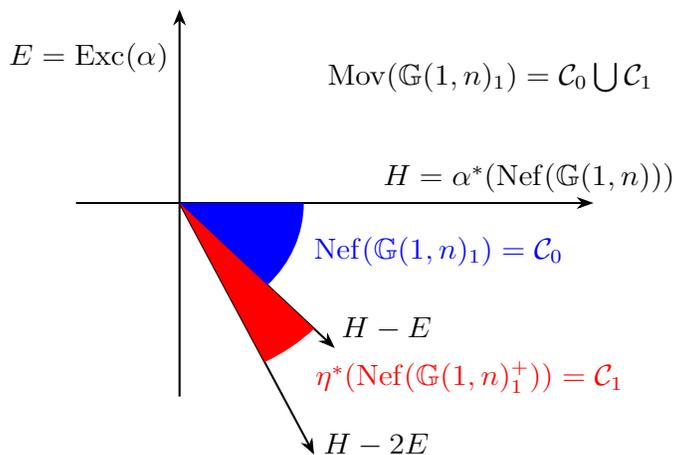
\begin{figure}[htb!]
\centering			
\resizebox{0.6\textwidth}{0.4\textwidth}{%
\begin{tikzpicture}[>=Stealth,scale=1.4]
\draw[->][thick] (0,-2) -- (0,2)   node[left,very near end]{$E=\Exc(\alpha)$};
\draw[->][thick] (-1,0) -- (4,0)   node[above,very near end]{$H=\alpha^*(\Nef (\G(1,n)))$};
\draw[->][thick] (0,0) -- (1.5,-1.5);  
\draw[->][thick] (0,0) -- (1.3,-2.6);

\draw (2.0,-1.3) node[very thick]{$H-E$};
\draw (1.9,-2.5) node[very thick]{$H-2E$};

\draw (3.0,1.3) node[very thick]{$\Mov(\G(1,n)_1)=\mathcal C_0 \bigcup \mathcal C_1$};

\fill[blue] (0,0) -- (1.2,0)
arc [start angle=0, end angle=-45, radius=1.2];
\fill[red] (0,0) -- (1.3,-1.3)
arc [start angle=-45, end angle=-63.5, radius=1.83];

\draw (2.5,-0.5) node[very thick,blue]{$\Nef(\G(1,n)_1)=\mathcal C_0$};
\draw (2.8,-1.8) node[very thick,red]{$\eta^*(\Nef(\G(1,n)_1^+))=\mathcal C_1$};
\end{tikzpicture}}
\end{figure}

Moreover, we show that the Mori chamber decomposition of $\G(1,n)_1$ is completely determined by the projections of $\G(1,n)$ from $p$, and from the tangent space $T_p\G(1,n)$.
We also explicitly construct the flip $\eta$ described in Theorem \ref{teointro}.

In Chapter \ref{cap6}  and Chapter \ref{cap7} we find instances where $\G(r,n)_k$ is a Mori dream space, giving in this way a partial answer to Problem \ref{q1}. The two chapters together yield the following.

\begin{Theorem}
$\G(r,n)_k$ is a Mori dream space if either
$$
\begin{cases}
k\!\!&\!=1;\ \mbox{ or }\\
k\!\!&\!=2 \mbox{ and } r=1 \mbox{ or } n=2r+1,n=2r+2; \mbox{ or }\\
k\!\!&\!=3 \mbox{ and } (r,n)\in \{(1,4),(1,5)\};\mbox{ or }\\
k\!\!&\!=4 \mbox{ and } (r,n)=(1,4)
\end{cases}
$$
\end{Theorem}

In Chapter \ref{cap6} We give a complete classification of the spherical varieties that can be obtain blowing-up general points in $\G(r,n)$.

\begin{Theorem}
$\G(r,n)_k$ is spherical if and only if one of the following holds
$$
\begin{cases}
r\!\!&\!=0 \mbox{ and } k\leq n+1;\ \mbox{ or }\\
k\!\!&\!=1;\ \mbox{ or }\\
k\!\!&\!=2 \mbox{ and } r=1 \mbox{ or } n=2r+1,n=2r+2; \mbox{ or }\\
k\!\!&\!=3 \mbox{ and } (r,n)=(1,5)
\end{cases}
$$
\end{Theorem}

In Chapter \ref{cap7} we give a complete classification of the   weak Fano varieties that can be obtain blowing-up general points in $\G(r,n)$. 

\begin{Proposition}
Let $\G(r,n)_k$ be the blow-up of the Grassmannian $\G(r,n)$ at $k\geq 1$ general points. Then
\begin{enumerate}
\item[(a)] $\G(r,n)_k$ is Fano if and only if $(r,n)=(1,3)$ and $k\leq 2.$
\item[(b)] $\G(r,n)_k$ is weak Fano if and only if one of the following holds.
\itemize
\item $(r,n)=(1,3)$ and $k\leq 2;$ or
\item $(r,n)=(1,4)$ and $k\leq 4.$
\end{enumerate}
\end{Proposition}

In the same chapter we also describe which blow ups of smooth quadrics in general points are weak Fano varieties.

\begin{Proposition}
Let $Q^n_k$ be the blow-up of a smooth quadric $Q^n\subset \P^{n+1}$ at $k\geq 1$ general points. Then
\begin{enumerate}
 \item[(a)]$Q_k^n$ is Fano if and only if either $k\leq 2$ or $n=2$ and $k\leq 7.$
 \item[(b)]$Q_k^n$ is weak Fano if and only if one of the following holds.
\itemize
\item $n=2$ and $k\leq 7;$ or
\item $n=3$ and $k\leq 6;$ or
\item $n\geq 4$ and $k\leq 2$.
\end{enumerate}
\end{Proposition}

In Chapter \ref{cap8} we address Problem \ref{q1} and Problem \ref{q2}. First we give a conjectural description of the Mori chamber decomposition of $\Eff(\G(r,n)_1).$

\begin{Conjecture}
Let $\G(r,n)_1$ be the blow-up of the Grassmannian at one point. Then 
$$\Mov(\G(r,n))=\begin{cases}
\cone(H,H-rE) &\mbox{ if } n=2r+1\\
\cone(H,H-(r+1)E) &\mbox{ if } n>2r+1
\end{cases}$$
Moreover, $E,H,H-E,\dots,H-(r+1)E$ are the walls of the Mori chamber decomposition of $\Eff(\G(r,n)_1).$
\end{Conjecture}

And we also describe a possible strategy for the proof of such result based on an generalization of the proof of Theorem \ref{teointro}. 

Second, we discuss possible approaches to Problem \ref{q2}. We discuss some difficulties that could appear as well.

\chapter{Blow-up of Grassmannians of lines at one point}\label{cap5}
In this chapter we describe the birational geometry of $\G(1,n)_1:$

\begin{Theorem}\label{MCD G(1,n)}
Let $n\geq 4$ and consider a basis of $\Pic(\G(1,n)_1)$ given by the pullback $H$ of an ample divisor on $\G(1,n),$
and by the class $E$ of the exceptional divisor.
Then $\G(1,n)_1$ is a Mori dream space  with 
effective cone \begin{center}
$\Eff(\G(1,n)_1)=\cone(E,H-2E).$
\end{center} 

Its Mori chamber decomposition is given by the walls $E,H,H-E,$ and $H-2E.$

There exists an isomorphism in codimension two $\eta:\G(1,n)_1\dasharrow \G(1,n)_1^+$ to another Mori dream space $\G(1,n)_1^+$ with 
$\Nef(\G(1,n)_1^+)=\cone(H-E,H-2E).$

The cone of movable divisors is 
$\Mov(\G(1,n)_1)\!=\!\cone(H,H-2E).$

The variety $\G(1,n)_1^+$ is a fibration over $\G(1,n-2)$ with fibers isomorphic to $\P^4,$ and if $n\geq 5$, then $\G(1,n)_1^+$ is a Fano variety.
\end{Theorem}
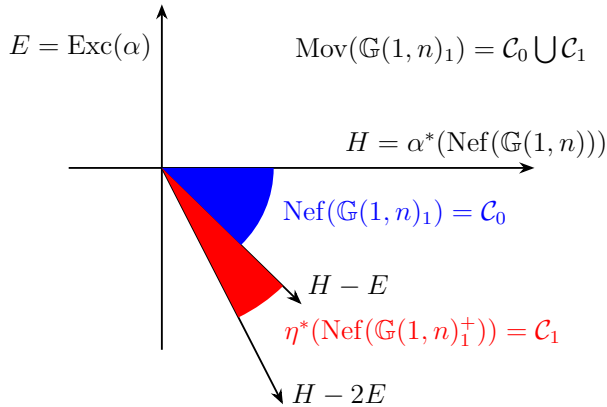
\vspace{-0.5cm}
\begin{figure}[htb!]
\centering			
\resizebox{0.54\textwidth}{0.36\textwidth}{%
\begin{tikzpicture}[>=Stealth,scale=1.4]
\draw[->][thick] (0,-2) -- (0,1.8)   node[left,very near end]{$E=\Exc(\alpha)$};
\draw[->][thick] (-1,0) -- (4,0)   node[above,very near end]{$H=\alpha^*(\Nef (\G(1,n)))$};
\draw[->][thick] (0,0) -- (1.5,-1.5);  
\draw[->][thick] (0,0) -- (1.3,-2.6);

\draw (2.0,-1.3) node[very thick]{$H-E$};
\draw (1.9,-2.5) node[very thick]{$H-2E$};

\draw (3.0,1.3) node[very thick]{$\Mov(\G(1,n)_1)=\mathcal C_0 \bigcup \mathcal C_1$};

\fill[blue] (0,0) -- (1.2,0)
arc [start angle=0, end angle=-45, radius=1.2];
\fill[red] (0,0) -- (1.3,-1.3)
arc [start angle=-45, end angle=-63.5, radius=1.83];

\draw (2.5,-0.5) node[very thick,blue]{$\Nef(\G(1,n)_1)=\mathcal C_0$};
\draw (2.8,-1.8) node[very thick,red]{$\eta^*(\Nef(\G(1,n)_1^+))=\mathcal C_1$};
\end{tikzpicture}}
\end{figure}

In the first section we recall basic definitions and known results, and give some examples. In Section \ref{secvarietiespho=2} we explain the particularities of Mori dream spaces whose Picard number is two. In Section \ref{Flipconstruction} we construct the flip and prove Theorem \ref{MCD G(1,n)}.

\section{Preliminary definitions}\label{sec51}

For a more complete survey on the basic definitions we refer to Lazarsfeld's book  \cite{Lazarsfeldvol1}. For a line bundle $L$ on a scheme $X$ the section ring of $L$ is the graded ring
$$R(X, L) :=\bigoplus_{n\in \N} H^0(X,L^{\otimes n}).$$

When there is no danger of confusion we will mix the notation of divisors and line bundles, e.g. writing $H^0(X,D)$
for $H^0(X,\mathcal O(D))$ for a divisor $D$. 
If $R(X,D)$ is finitely generated, and $D$ is effective, then there is an induced rational
map $$\varphi_D : X \dasharrow Proj(R(X,D))$$
which is regular outside the stable base locus 
$\bigcap_{n\in \N} Bs(|nD|)$ of $D.$

Let $X$ be a normal projective variety. We denote by $N^1(X)$ the real vector space of Cartier divisors modulo numerical equivalence and by $\rho_X = \dim(N^1(X))$ the Picard number of $X$. 
\begin{itemize}
\item[-] The \textit{effective cone} $\Eff(X)$ is the convex cone in $N^1(X)$ generated by classes of effective divisors. In general it is not a closed cone.
\item[-] The \textit{nef cone} $\Nef(X)$ is the convex cone in $N^1(X)$ generated by classes of divisors $D$ such that $D\cdot C\geq 0$ for any curve $C\subset X$. It is closed, but in general it is neither polyhedral nor rational. 
\item[-] A divisor $D\subset X$ is called \textit{movable} if its stable base locus has codimension at least two. The \textit{movable cone} $\Mov(X)$ is the convex cone in $N^1(X)$ generated by classes of movable divisors. In general, it is not closed.
\item[-] A line bundle $L$ is \textit{semiample} if   $L^{\otimes m}$ is globally generated for some $m>0.$ A divisor $D$ is semiample if the corresponding line bundle is so.
\end{itemize}
Remember also that the cone $\Amp(X)$ of ample divisors on $X$ is the interior of the cone of nef divisors. We have then
$$\overline{\Amp(X)}=\Nef(X)\subset \Mov(X)\subset\Eff(X)\subset N^1(X).$$

From now on we assume that $X$ is $\Q$-factorial.
A \textit{small $\mathbb{Q}$-factorial modification} of $X$ is a birational map $f:X\dasharrow Y$ to another normal $\mathbb{Q}$-factorial projective variety $Y$, such that $f$ is an isomorphism in codimension one.
A morphism $f:X\to Y$ between normal projective varieties is called a \textit{contraction} when it is surjective and has connected fibers. If $f:X\to Y$ is a contraction we define  the exceptional locus $\Exc(f)$ of $f$ as $X$ if $\dim(X)>\dim(Y)$ and as the smallest closed subset of $X$ such that 
$$f:X\setminus \Exc(f)\to Y\setminus f(\Exc(f))$$
is an isomorphism otherwise.

A contraction $f$ is called a \textit{small contraction} if $\Exc(f)$ has codimension at least two. Therefore, if $f$ is not small then $\Exc(f)$ is either a divisor (and $f$ is called a divisorial contraction) or $X$ itself (and $f$ is a fibration). A contraction $f:X\to Y$ is called a \textit{elementary contraction} when $\rho_X-\rho_Y=1,$ this is equivalent to saying that the contracted curves, 
$\{c: f(c)=\mbox{pt}\},$ generate an extremal ray of the Mori cone.

Assume there is a commutative diagram
\begin{center}
\begin{tikzcd}
X\arrow[rr,dashed,"\eta"]\arrow[rd,"f"']&&X^+\arrow[dl,"f^+"]\\
&Y&
\end{tikzcd}
\end{center}
such that $f,f^+$ are small elementary contractions, $\eta$ is a small $\mathbb{Q}$-factorial modification and $X,X^+$ are not isomorphic. If $D$ is an effective divisor in $X$ such that $D\cdot c<0$ for each curve $c$ such that $f(c)=\mbox{pt},$ then we call $\eta$ a $D$-negative \textit{flip}, or just a flip.

\begin{Definition} Let $D_1$ and $D_2$ be two movable $\Q$-Cartier divisors on $X$ with finitely generated section rings. Then we say that $D_1$ and $D_2$ are \textit{Mori equivalent} if the rational maps $\varphi_{D_i}$ have the same Stein factorization, i.e., there is an
isomorphism between their images which makes the obvious triangular diagram commutative.

Suppose now that $X$ is a projective variety such that $R(X, L)$ is finitely generated for all line bundles $L.$ By a \textit{Mori chamber} of $\Mov(X)$ we mean the closure of the cone spanned by a Mori equivalence class of divisors whose interior is open in $\Mov(X)$.
\end{Definition}

A Mori dream space is essentially a variety for which the movable cone is the union of finitely many Mori chambers. In the definition below we use the abuse of notation $f_i^*(\Nef(X_i))=\Nef(X_i)$
and we will keep using this from now on.

\begin{Definition}\label{Mori dream space}
A normal projective variety $X$ is a \textit{Mori dream space} if
\begin{itemize}
\item[(a)] $X$ is $\mathbb{Q}$-factorial and $\Pic(X)$ is finitely generated;
\item[(b)] $\Nef(X)$ is generated by finitely many semiample divisors;
\item[(c)] there exists finitely many small $\mathbb{Q}$-factorial modifications $f_i:X\dasharrow X_i$ such that each $X_i$ satisfies $\rm(a),(b)$, and 
$$\Mov(X)=\bigcup_i \Nef(X_i).$$
\end{itemize}
\end{Definition}

Observe that if $X$ is a Mori dream space then the $X_i$ are also Mori dream spaces, and $\Mov(X)$ has a finite number of Mori chambers.

It is shown in \cite[Proposition 1.11]{HK} that when $X$ is a Mori dream space we can in fact decompose $\Eff(X)$ in a finite number of Mori chambers as well. Moreover, we can describe a fan structure on the effective cone $\Eff(X)$, called the \emph{Mori chamber decomposition}. 
We refer to \cite[Proposition 1.11(2)]{HK} and \cite[Section 2.2]{Oka} for details.
There are finitely many birational contractions from $X$ to Mori dream spaces, denoted by $g_i:X\map Y_i$.
The set $\Exc^1(g_i)$ of exceptional prime divisors of $g_i$ has cardinality $\rho(X/Y_i)=\rho(X)-\rho(Y_i)$. If $g_i$ is a divisorial contraction then $\Exc(g_i)=\Exc^1(g_i).$
The maximal cones $\mathcal C_i$ of the Mori chamber decomposition of $\Eff(X)$ are of the form:
$$
\mathcal C_i \ = \cone \ \Big( \ g_i^*\big(\Nef(Y_i)\big)\ , \  \Exc^1(g_i) \ \Big).
$$
We call $\mathcal C_i$ or its interior $\mathcal C_i^{^\circ}$ a \emph{chamber} of $\Eff(X)$ or
a \emph{Mori chamber} of $\Eff(X)$.

Let $X$ be a Mori dream space and $\mathcal{C},\mathcal{C}'$ be two adjacent chambers, that is $\dim(\overline{\mathcal{C}}\bigcap\overline{\mathcal{C}'})=\rho_X -1,$ then we call $W=\overline{\mathcal{C}}\bigcap\overline{\mathcal{C}'}$ a wall of the Mori chamber decomposition of $\Eff(X).$ If $\rho_X=2$ each wall is determined by a single divisor $D,$ and in this case we sometimes say that $D$ is the wall. When $\rho_X=2$ and $\Eff(X)=\cone(D_1,D_2)$ we will also say that $D_1$ and $D_2$ are walls.

Let us conclude this introductory section with a concrete example of a Mori chamber decomposition. We denote by $\cone(v_1,\dots,v_n)$ the (closed) cone generated by the vectors $v_1,\dots,v_n$ in some vector space.

\begin{Example}\label{ExP^n_2}
See Example 3.7 of the Notes \cite{Ca12} for details.
Let $X$ be the blow-up of $\P^n, n\geq 2,$ at two points $p_1,p_2.$
Then $X$ is a Mori dream space because it is a toric variety.
Furthermore, $\Pic(X)$ is freely generated by the pullback  $H$ of a hyperplane in $\P^n$ and the exceptional divisors $E_1,E_2.$

The nef cone of $X$  is generated by $H,H-E_1$ and $H-E_2.$
There is a small $\mathbb{Q}$-factorial modification $f:X\dasharrow X'$ with $\Nef(X')$ being the cone generated by $H-E_1,H-E_2$ and $H-E_1-E_2.$ The movable cone of $X$ is $\Mov(X)=\Nef(X)\cup\Nef(X'),$
see Figure \ref{P^n_2}, and $\Nef(X),\Nef(X')$ are the two Mori   chambers of it.
The effective cone of $X$ is
$$Eff(X)=\cone(H-E_1-E_2,E_1,E_2)$$
and it has five Mori chambers, two of them are in the movable cone and were already explained. We now describe the other three.

If $\varphi:X\to Bl_{p_2}\P^n$ is the blow-up at $p_1$ then by
\cite[Proposition 1.11(2)]{HK} 
$$ \mathcal C_1=
\cone(\varphi^*(\Nef(Bl_{p_1}\P^n)),\Exc \varphi)=
\cone(H-E_2,H,E_1)$$
is a Mori chamber of $\Eff(X).$
Analogously 
$$ \mathcal C_2=
\cone(H-E_1,H,E_2)$$
is another Mori chamber.
The last Mori chamber is obtained considering the composition of the two blow-ups
$\psi:X\to \P^n:$
$$ \mathcal C_0=
\cone(\psi^*(\Nef(\P^n)),\Exc \psi)=
\cone(H,E_1,E_2).$$
The anticanonical divisor of $X$ is 
$-K_X=(n+1)H-(n-1)(E_1+E_2),$
and it is inside the interior of $\Nef(X)$ if $n=2,$
on the wall between the Mori chambers $\Nef(X)$ and $\Nef(X')$ if $n=3,$ and inside $\Nef(X')$ otherwise.
\end{Example}

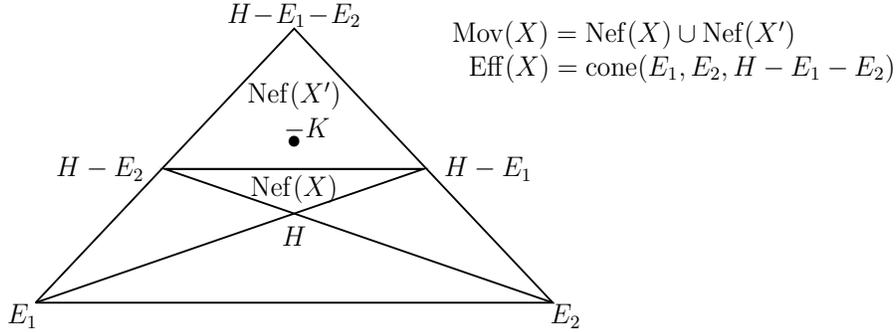
\begin{figure}[htb!]
\centering
\resizebox{0.8\textwidth}{0.3\textwidth}{%
\begin{tikzpicture}

\draw[thick] (0,0) -- (8,0) --(4,4.1)--(0,0) ;
\draw[thick] (2,2) -- (6,2) --(0,0) ;
\draw[thick] (6,2) -- (2,2) --(8,0) ;

\draw (-0.2,-0.2) 	 node[thick]{$E_1$};
\draw (8.2,-0.2) 	 node[thick]{$E_2$};
\draw (4,4.3) 	 node[thick]{$H\!-\!E_1\!-\!E_2$};
\draw (1,2) 	 node[thick]{$H-E_2$};
\draw (7,2) 	 node[thick]{$H-E_1$};
\draw (4,1) 	 node[thick]{$H$};

\draw (4,1.7) 	 node[thick]{$\Nef(X)$};
\draw (4,3.1) 	 node[thick]{$\Nef(X')$};
\draw (4,2.4) 	 node[thick]{$\bullet$};
\draw (4.2,2.6) 	 node[thick]{$-K$};

\draw (9.1,4) 	 node[thick]{$\Mov(X)=\Nef(X)\cup\Nef(X')$};
\draw (10,3.5) node[thick]{$\Eff(X)=\cone(E_1,E_2,H-E_1-E_2)$};
\end{tikzpicture}
}
\caption{Mori chamber decomposition of $Bl_{p_1,p_2}\P^n, n\geq 4$} 
\label{P^n_2}
\end{figure}

\section{Mori dream spaces with Picard Number Two}\label{secvarietiespho=2}

When $X$ is a variety of Picard number one
there is nothing interesting to say about the cones of divisors.
In the next case, when the Picard number is two,
the Mori chamber decomposition of the effective cone is simple because all cones are simplicial.
We will focus on this case.
When the Picard number is greater than two, the cones
may not be simplicial and many difficulties appear.

Now we state some simple lemmas about varieties whose Picard number is two. 

\begin{Lemma}\label{nef cone border}
Let $X,Y$ be normal projective varieties such that $X$ is $\Q$-factorial,
and $\varphi:X\to Y$ be a contraction.
Set $D:=\varphi^*(H)\in Pic(X),$ where $H\in Pic(Y)$ is an ample divisor.
Assume that the exceptional locus $E=\Exc(\varphi)$ of $\varphi$ is not empty. Then
\begin{enumerate}[(1)]
	\item $D \in\partial \Nef(X).$
	\item If $E=X$, then $D \in \partial \Mov(X)$ and $D\in\partial \Eff(X)$ as well.
 	\item If $E$ is a divisor, then $D\in\partial \Mov(X)$ and  $E\in \partial \Eff(X).$ Moreover, if $\rho(X)=2$ and $X$ is a Mori dream space, then $\cone(E,D)$ is a Mori chamber of $\Eff(X).$
\end{enumerate}
\end{Lemma}
\begin{proof}
For any irreducible curve $c$ in $X$ we have that $D\cdot c=\varphi^*(H)\cdot c=H\cdot  {\varphi}_* (c)$
is zero if $\varphi(c)$ is a point, and is positive otherwise.
Therefore $D$ is nef.
Since $E$ is not empty, there is at least one irreducible curve $c$ contracted by $\varphi$,
and then $D\cdot c=0.$ This shows that $D$ is $\Nef$ but not ample.
Now, the $\Nef$ cone is the closure of the ample cone, see \cite[Theorem 1.4.23]{Lazarsfeldvol1}, therefore $D\in \partial \Nef(X).$

If $E=X,$ then $\varphi=\varphi_{D}$ is not birational, this means that $D$ is not big.
Since the cone of big divisors is the interior of the effective cone,
see \cite[Theorem 2.2.26]{Lazarsfeldvol1}, then $D\in \partial \Eff(X).$ Now, since 
$\Nef(X)\!\subset\!\Mov(X)\!\subset\! \Eff(X)$ then
$D\in \partial \Mov(X).$

For the remaining of the proof assume that $E$ is a divisor. 

Let $c$ be a curve contracted by $\varphi.$
Therefore $\Supp(c)\subset E$ and thus
$c\cdot E<0.$
Indeed, if $c\cdot E\geq 0$ then there is a curve $c'$ numerically equivalent to $c$ such that $\Supp(c')\not\subset E,$ but this contradicts $E=\Exc(f).$

Next, $(\alpha D+\beta E)\cdot c=\beta E\cdot c<0$
for each $\beta>0.$ In particular, $|n(D+\beta E)|$ has $E$ in its base locus for every $\beta>0$ and $n\in \Z_{\geq 1},$ and  therefore   $D+\beta E$ is not movable. Since 
$\varphi=\varphi_{D}$ is birational, we have that $D$ is movable.
Noting that $D$ is movable, but 
$D+\beta E,\beta>0$ is not, we conclude that $D\in \partial \Mov(X).$

Pick $c$ as the class of a curve that covers a open dense of $X,$ that is, a movable curve.
Then $E\cdot c=0,$ and since the cone of movable curves is the dual of the effective cone of divisors, see \cite[Theorem 2.2]{BDPP}, then 
$E\in \partial\Eff(X).$

Finally, if $\rho(X)=2$ and $X$ is a Mori dream space then $\rho(Y)=1$ and 
$$\mathcal C=\cone(\varphi^*(\Nef(Y)),\Exc(\varphi))=\cone(D,E)$$
is a Mori chamber of $\Eff(X).$
\end{proof}

The following is a direct consequence of Lemma \ref{nef cone border}.
 
\begin{Corollary}\label{nef cone border2}
Let $\varphi_1:X\to Y_1,\varphi_2:X\to Y_2$ be two contractions as in the first part of Lemma \ref{nef cone border},
and set $E_1,E_2,D_1,D_2$ accordingly.
Assume that $\rho(X)=2, \ \Pic(X)\cong \Z^2,$ and that $D_1\cdot \R_{\geq 0}\neq D_2\cdot  \R_{\geq 0}.$ Then $\Nef(X)\!=\!\cone(D_1,D_2).$

Assume, in addition, that $\varphi_1,\varphi_2$ are not small contractions. Then $X$ is a Mori dream space and
\begin{enumerate}[(1)]
	\item $\Nef(X)=\Mov(X)=\cone(D_1,D_2)$ is one Mori chamber of $\Eff(X).$ 
	\item If $E_1\!=\!E_2\!=\!X,$ then $\Eff(X)\!=\!\Nef(X).$
	\item If $E_1$ is a divisor and $E_2=X,$ then $\Eff(X)$ have exactly two Mori chambers:
		$$\mathcal C_0=\cone(E_1,D_1),\ \mathcal C_1=\cone(D_1,D_2).$$		
	\item If $E_1,E_2$ are divisors, then $\Eff(X)$ has exactly three Mori chambers:
		$$\mathcal C_0=\cone(E_1,D_1),\ \mathcal C_1=\cone(D_1,D_2),\ \mathcal C_2=\cone(D_2,E_2).$$
\end{enumerate}
\end{Corollary}

The three cases appearing in Corollary \ref{nef cone border2} are, in a certain sense,
the simplest examples of Mori dream space with Picard number two.
We give below explicit examples of these three cases.

\begin{figure}[htb!]
\centering
\begin{tikzpicture}[>=Stealth,scale=1.4]
\fill[red!60!white] (0,0) -- (0.5,0)
arc [start angle=0, end angle=0-45, radius=0.5];
\draw[red!60!white] (1.2,-0.3) node[very thick]{$\Eff(X)$};
\draw[->][thick] (0,0) -- (1,0)
node[above,very near end]{$D_1$};
\draw[->][thick] (0,0) -- (0.8,-0.8)
node[left,very near end]{$D_2$};
\draw (0,1.5) node[very thick]{item $(2)$};

\fill[blue!60!white] (3,0) -- (3.7,0)
arc [start angle=0, end angle=-45, radius=0.7];
\draw[blue!60!white] (4.2,-0.3) node[very thick]{$\Nef(X)$};
\fill[red!60!white] (3,0) -- (3,0.5)
arc [start angle=90, end angle=-45, radius=0.5];
\draw[red!60!white] (3.6,0.7) node[very thick]{$\Eff(X)$};
\draw[->][thick] (3,0) -- (4,0)
node[above,very near end]{$D_1$};
\draw[->][thick] (3,0) -- (3,1)
node[left,very near end]{$E_1$};
\draw[->][thick] (3,0) -- (3.8,-0.8)
node[right,very near end]{$D_2$};
\draw (3.3,1.5) node[very thick]{item $(3)$};

\fill[blue!60!white] (6,0) -- (6.7,0)
arc [start angle=0, end angle=-45, radius=0.7];
\draw[blue!60!white] (7.2,-0.3) node[very thick]{$\Nef(X)$};
\fill[red!60!white] (6,0) -- (6,0.5)
arc [start angle=90, end angle=-57, radius=0.5];
\draw[red!60!white] (6.6,0.7) node[very thick]{$\Eff(X)$};
\draw[->][thick] (6,0) -- (7,0)
node[above,very near end]{$D_1$};
\draw[->][thick] (6,0) -- (6,1)
node[left,very near end]{$E_1$};
\draw[->][thick] (6,0) -- (6.8,-0.8)
node[right,very near end]{$D_2$};
\draw[->][thick] (6,0) -- (6.8,-1.2)
node[right,very near end]{$E_2$};
\draw (6.3,1.5) node[very thick]{item $(4)$};
\end{tikzpicture}
\end{figure}

\begin{Example}
Let $Y_1$ be a Mori dream space with Picard number one, $Y_2=\P^n$ and $X=Y_1\times Y_2,$
and consider the projections $\pi_i:X\to Y_i.$ Then $\rho(X)=2$ by \cite[Exercise II.6.1]{Har}.
Therefore by Corollary \ref{nef cone border2} we know that
$X$ is a Mori dream space with $$\Eff(X)=\Mov(X)=\Nef(X)=\cone(D_1,D_2),$$ where 
$H_i\in Pic(Y_i)$ are ample divisors and $D_i=\pi_i^*(H_i)$.
\end{Example}

\begin{Example}
Let $\varphi:X\to Y_1=\P^n$ be the blow-up of the projective space in one point $p$.
Consider also the linear projection from this point $\pi:Y_1\dasharrow Y_2=\P^{n-1}.$
The blow-up map resolves the projection, that is, there is a morphism $\psi:X\to Y_2$ making the diagram below commutative.
\begin{center}
\begin{tikzcd}
&X\arrow[dl,"\varphi"']\arrow[rd,"\psi"]&\\
\P^n\arrow[rr,dashed,"\pi"]&&\P^{n-1}
\end{tikzcd}
\end{center}
Denote by $H\in \Pic(X)$ the pullback of a general hyperplane section on $\P^n$
and by $E\in \Pic(X)$ the exceptional divisor of $\varphi.$ Therefore, by Corollary \ref{nef cone border2} we have 
that $X$ is a Mori dream space and $\Eff(X)$ has exactly two Mori chambers:
		$$\mathcal C_0=\cone(E,H),\ \mathcal C_1=\cone(H,H-E)=\Nef(X)=\Mov(X).$$	
\end{Example}

Before the next example we introduce some notation. We denote the real vector space of $1$-cycles modulo numerical equivalence of a given projective variety $X$ by $N_1(X).$
We denote the \textit{cone of curves} of $X,$ also known as the \textit{Mori cone} of $X,$ as
$$NE(X)=\left\{ \sum_{i=1}^k 
\alpha_i [C_i]\in N_1(X); 0\leq \alpha_i\in \R
, k\in \Z_{\geq 0} \right\}$$
where $C_i\subset X, i=1,\dots, k$ are irreducible curves. Given a contraction $f:X\to Y$ the classes of contracted curves constitute a subcone $\sigma$
of $NE(X),$ we use the notation $f=cont_{\sigma}.$
If $\sigma=R\cdot \R_{\geq 0}$ we write 
$f=cont_R$ instead.

\begin{Example}\label{exampleC}
This example applies in particular to $\G(1,3)_1,$ the blow-up of the Grassmannian $\G(1,3)\cong Q^4$
at one point.
Let $\varphi\!:\!X\to Y_1\!=\!Q^n\!\subset\! \P^{n+1}$ be the blow-up of the 
$n$-dimensional smooth quadric hypersurface in one point $p$,
and set $\pi:Y_1\dasharrow  Y_2=\P^n,\psi:X\to Y_2$ as in the previous example. Also, let
$E_0$  be the exceptional divisor of $\varphi,$ and $H_0\in Pic(X)$ the pullback of a general hyperplane section of $Q^n$.
\begin{center}
\begin{tikzcd}
&X\arrow[dl,"\varphi"']\arrow[rd,"\psi"]&\\
Q^n\arrow[rr,dashed,"\pi"]&&\P^{n}
\end{tikzcd}
\end{center}
Then $\psi$ is a divisorial contraction, and the exceptional divisor is $R_0,$ the strict transform
of $Q^n\cap T_p(Q^n)=\cone_p(Q^{n-2})$ by $\varphi.$ The class of $R_0$ in $\Pic(X)$ is $H_0-2E_0.$
The pullback of an ample divisor in $\P^n$ by $\psi$ is the pullback by $\varphi$
of a general hyperplane section in $Q^n$ passing through $p,$  and therefore has class $H_0-E_0.$ 

By Corollary \ref{nef cone border2}, $X$ is a Mori dream space and
$\Eff(X)$ has exactly three Mori chambers:
$$\mathcal C_0\!=\!\cone(E_0,H_0),
\mathcal C_1\!=\!\cone(H_0,H_0\!-\!E_0)
\!=\!\Nef(X)\!=\!\Mov(X),$$
$$\mbox{ and }
\mathcal C_2\!=\!\cone(H_0\!-\!E_0,H_0\!-\!2E_0).$$
We can also give additional information: $S_0:=\psi(R_0)$ is an $(n-2)$-dimensional smooth 
quadric hypersurface of $\psi(E_0)=\P^{n-1},$ which paremetrizes the lines in $Q^n$ through $p.$
By duality, the cone of curves of $X$ is given by
$$NE(X)=\Nef(X)^\vee=\cone(e_0,h_0-e_0),$$
where $h_0$ is the class of the strict transform by $\varphi$ of a general line in $Q^n$
and $e_0$ is the class of a general line in $E_0.$
We have $h_0\cdot H_0=-e_0\cdot E_0=1$ and $h_0\cdot E_0=e_0\cdot H_0=0.$
Moreover, $\varphi=cont_{e_0}$ and $\psi=cont_{h_0-e_0}.$
\end{Example}

Next we give another corollary of Lemma \ref{nef cone border} which allows us to prove that a variety $X$ of Picard number two is a Mori dream space if we have enough morphisms and birational maps from it.

\begin{Corollary}\label{nef cone border3}
Let $X$ be a normal $\Q$-factorial projective variety with Picard number two and $\Pic(X)\cong \Z^2$. Suppose there is a commutative diagram
\begin{center}
\begin{tikzcd}
X\arrow[d,"f_0"]
\arrow[dr,"g_0"]\arrow[r,dashed,"\eta_1"]&
X_1\arrow[d,"f_1"]
\arrow[dr,"g_1"]\arrow[r,dashed,"\eta_2"]&
\cdots\arrow[r,dashed,"\eta_r"]
\arrow[rd,"g_{r-1}"]&
X_r\arrow[d,"f_r"]
\arrow[dr,"g_r"]&\\
X_0&X_{-1}&\cdots &X_{-r}&X_{-r-1}
\end{tikzcd}
\end{center}
such that $X_{-r-1},\dots,X_r$ are normal projective varieties, $X_1,\dots,X_{r}$ are $\Q$-factorial, $f_0,g_0,\dots, f_r,g_r$ are elementary contractions,  $\eta_0,\dots,\eta_r$ are flips, and $f_0,g_r$ are not small contractions. Moreover, set
$$D_0=f_0^*(H_0),
D_1=\eta_1^*(f_1^*(H_1))
\dots,D_r=\eta_1^*\dots \eta_r^*f_r^*(H_r),D_{r+1}=g_r^*(H_{r+1}),$$
where $H_i\in \Pic(X_{-i})$ are ample,
$E=\Exc(f_0)$ and $E'=\Exc(g_r)$,
and assume that $D_0,\dots,D_{r+1}$ generate different rays of $\Pic(X).$

Then $X$ is a Mori dream space,
$\Nef(X)=\cone(D_0,D_1),\Mov(X)=cone(D_0,D_r),$ and $D_0,\dots,D_r$ are walls of the Mori chamber decomposition of $\Mov(X).$ Moreover, 
\begin{enumerate}[(1)]
\item If $E=E'=X$ then 
$\Eff(X)=cone(D_0,D_r).$
\item If $E=X$ and $E'$ is a divisor, then 
$\Eff(X)=cone(D_0,E')$ and $\cone(D_r,E')$ is a chamber of $\Eff(X).$
\item If both $E$ and $E$ are divisors, then 
$\Eff(X)=cone(E,E')$ and both $\cone(E,D_0)$ and
$\cone(D_r,E')$ are chambers of $\Eff(X).$
\end{enumerate}
\end{Corollary}

\section{Flip's construction and main result}\label{Flipconstruction}

In Example \ref{exampleC} we showed that the Mori chamber decomposition of $\Eff(\G(1,3)_1)$ has exactly three chambers and that $\Mov(\G(1,3)_1)=\Nef(\G(1,3)_1)$ has only one Mori chamber.

In this section we assume that $n\geq 4.$ We will see in the proof of Theorem \ref{MCD G(1,n)} that the Mori chamber decomposition of $\G(1,n)_1$ has also three chambers, one is the nef cone, another corresponds to the blow-up map itself
and a third one corresponds to a flip 
$\G(1,n)_1\dasharrow \G(1,n)_1^+$. We will also see that $\G(1,n)_1^+$ admits a structure of a $\P^4$-bundle over $\G(1,n-2).$
The purpose of the present section is to explain how to concretely construct this flip.

When we blow-up a point we resolve the linear projection from this point. 
The construction of the flip is based on a similar idea, we project linearly from the tangent space
and resolve this rational map using two successive blow-ups.

Consider a point $p\in\G(1,n)\subset \P^N,$ corresponding to a line $l_p\subset \P^n,$
and its embedded tangent space $T=T_p\G(1,n)\subset \P^N.$ Let $R$ be intersection $R:=T\cap\G(1,n).$
The following characterizations of $R$ are well known, see \cite[Example 3]{Hw06} 
and \cite[Exercise 6.9]{Harris92}.

\begin{Lemma}\label{Rcharacterizations}
$$R:=T\cap\G(1,n)=\!\!\!\!\!\!
\bigcup_{\substack{L \mbox{ \tiny{line} }\\ p\in L\subset \G(1,n) }}\!\!\!\!\!\!\!\!\!L \ =
\{[l]; l\subset \P^n \mbox{ line meeting } l_p\}  \cong \cone_p \left(\P^1\times \P^{n-2}\right).$$
\end{Lemma}
Next we consider some rational maps.
\begin{itemize}
	\item $\pi_p:\P^N\dasharrow \P^{N-1}$ the projection with center $p;$
	\item $\pi_T:\P^N\dasharrow \P^{N'},$ the projection with center $T,$ where $N'\!=\!N\!-\!2(n-1)\!-\!1;$ 
	\item $\pi_{\pi_p(T)}:\P^{N-1}\dasharrow \P^{N'}$ the projection with center ${\pi_p(T)};$
	\item The restrictions of these three projections:\\
	$\pi_0=\pi_{p|\G(1,n)}:\G(1,n)\dasharrow W:=\pi_p(\G(1,n)),$\\
	$\pi_R=\pi_{T|\G(1,n)}:\G(1,n)\dasharrow W_1:=\pi_T(\G(1,n)),$\\
	$\pi_1=\pi_{\pi_p(T)|W}:W\dasharrow W_1;$
	\item $\pi_{l_p}:\P^n\dasharrow \P^{n-2},$  the projection with center $l_p;$
	\item $\widetilde{\pi_{l_p}}:\G(1,n)\dasharrow \G(1,n-2),$  the rational map induced by $\pi_{l_p}.$
\end{itemize}

Consider also the blow-up $\alpha:\G(1,n)_1\to \G(1,n)$ at $p$ and the strict transform 
$\widetilde R\subset \G(1,n)_1$ of $R.$ By Lemma \ref{Rcharacterizations} we know that $R$ is a cone with vertex $p,$
thus $\widetilde R$ is smooth and the blow-up $\alpha_1:\G(1,n)_R\to\G(1,n)_1$
with center $\widetilde R$ is a smooth variety.

\begin{Lemma}\label{Projectionslemma} In the above notation
\begin{enumerate}
	\item $\pi_0$ is birational;
	\item $W_1\cong \G(1,n-2)$ and $\widetilde{\pi_{l_p}}=\pi_R;$ 
	\item $\alpha$ resolves $\pi_0,$ this is,  there exists a morphism $\beta:\G(1,n)_1\to W$ such that $\beta=\pi_0\circ \alpha;$
	\item $\alpha\circ \alpha_1$ resolves $\pi_R,$ this is,  there exists a morphism $\xi:\G(1,n)_R\to \G(1,n-2)$ such that $\xi=\pi_R\circ \alpha\circ \alpha_1;$
	\item $\xi$ is a fibration with fibers isomorphic to $\G(1,3)_1.$
\end{enumerate}
\end{Lemma}
\begin{figure}[htb!]
\begin{center}
\begin{tikzcd}
\G(1,n)_R \arrow[ddrr,"\xi"]\arrow[d,"\alpha_1"']&&\\
\G(1,n)_1 \arrow[d,"\alpha"']
\arrow[dr,"\beta"]&&\\
\G(1,n) \arrow[rr,dashed,"\pi_R",bend right=20]\arrow[r,dashed,"\pi_0"]&
W \arrow[r,dashed,"\pi_1"]& \G(1,n-2)\\
\end{tikzcd}.
\end{center}
\label{ProjectionsII}
\end{figure}
\begin{proof}
Items $(1)$ and $(3)$ are clear. To show $(2)$ we use Pl\"ucker coordinates $p_{ij},0\leq i<j\leq n$.
Since $\G(1,n)$ is homogeneous we may suppose that $p=(1:0:\dots:0)$ corresponds to 
$l_p=\begin{pmatrix}
1 & 0 & 0 & \dots &0 \\
0 & 1 & 0 & \dots &0 
\end{pmatrix}\subset \P^n,$
the line passing through $$x=(1:0\dots:0) 
\mbox{ and } y=(0:1:0:\dots:0)\in \P^n.$$
Then, by Lemma \ref{Rcharacterizations} we have
\begin{align*}
R=&\left\{\![l]; l\subset \P^n \mbox{ line meeting } l_p\right\}\!\!=\!
\left\{\!\left[
\begin{pmatrix}
a_0 & a_1 & 0   & \dots &0   \\
b_0 & b_1 & b_2 & \dots &b_n 
\end{pmatrix}\right]
\right\}\\
=&\left\{\!(a_0b_1-a_1b_0:a_0b_2:\dots:a_0b_n:a_1b_2:\dots:a_1b_n:0:\dots:0)\right\}\\
=&\G(1,n)\bigcap \left(p_{ij}\!=\!0;\forall i\neq 0,1\right)\subset \P^N.
\end{align*}
Therefore $\pi_R$ is given by
\begin{align*}
\pi_R: \ &\G(1,n)\dasharrow \P^{N'}\\
&(p_{ij})\mapsto (p_{ij}; i\neq 0,1)
\end{align*}
and $\pi_{l_p},\widetilde{\pi_{l_p}}$ are given by

\begin{minipage}[b]{0.45\linewidth}
\begin{align*}
\pi_{l_p}:  \P^n &\dasharrow \P^{n-2}\\
(x_0:\dots :x_n)&\mapsto (x_2:\dots :x_n)\\
\phantom{some text}
\end{align*}
\end{minipage}%
\begin{minipage}[b]{0.5\linewidth}
\begin{align*}
\widetilde{\pi_{l_p}}:  \G(1,n) &  \dasharrow \G(1,n-2) \\
\left[
\begin{pmatrix}
a_0 &  \dots &a_n   \\
b_0 &  \dots &b_n 
\end{pmatrix}\right]
&\mapsto 
\left[
\begin{pmatrix}
a_2 & \dots &a_n   \\
b_2 & \dots &b_n 
\end{pmatrix}\right].
\end{align*}
\end{minipage}\\
This shows $(2).$
Note that by Lemma \ref{Rcharacterizations} 
$\pi_R=\widetilde{\pi_{l_p}}$ has $R$ as indeterminacy loci. 

The linear projection $\pi_T$ has $T$ as base locus.
Consider the blow-up $f:Bl_p\P^N\to\P^N$ at $p$ and the blow-up 
$g:Bl_{\widetilde T}(Bl_p\P^N)\to Bl_p\P^N$ with center $\widetilde T,$ the strict transform of $T.$
By Lemma \ref{blowuporder} $f\circ g$ resolves the map $\pi_T.$
Now since $\alpha$ and $\alpha_1$ are restrictions of $f$ and $g,$ item $(4)$ follows.

Now, we prove item $(5).$ The fibers of $\pi_R$ are Grassmannians $\G(1,3)\subset \G(1,n)$ which contain $p.$
Therefore the fibers of $\alpha\circ \pi_R$ are isomorphic to $\G(1,3)_1.$
Note also that the intersection of $R$ with a fiber $\G(1,3)$ of $\pi_R,$
$$R'=R\cap \G(1,3)=T_p(\G(1,n))\cap \G(1,n)\cap \G(1,3)=T_p(\G(1,3))\cap\G(1,3),$$
is isomorphic to the $\cone_p(\P^1\times\P^1),$ and thus a divisor in $\G(1,3).$
Therefore the strict transform of $R'$ is a smooth divisor on the fiber $\G(1,3)_1$ of $\alpha\circ \pi_R,$
and then $\alpha_1$ is an isomorphism on such a fiber. 
\end{proof}

\begin{Lemma}\label{blowuporder}
Let $X$ be a noetherian scheme, $\pi:X\dasharrow Y$ a rational map with base scheme $Z\subset X$ reduced and defined by a coherent sheaf of ideals $\mathcal{I}_Z$, and $W\subset Z$ a subscheme with ideal sheaf $\mathcal{I}_W$. Let $b_W:X_{W}\rightarrow X$ be the blow-up of $X$ with respect to $\mathcal{I}_W$, and $b_{\widetilde{Z}}:X_{W,Z}\rightarrow X_W$ be the blow-up of $X_{W}$ along the strict transform of $Z$. Then there exists a morphism $\widetilde{\pi}:X_{W,Z}\rightarrow Y$ yielding a resolution of $\pi$.
\end{Lemma}
\begin{proof}
Let $b_Z:X_Z\rightarrow X$ be the blow-up of $X$ with respect to $\mathcal{I}_Z$. Then there exists a morphism $\overline{\pi}:X_Z\rightarrow Y$ yielding a resolution of $\pi$.
Let $g = b_{W}\circ b_{\widetilde{Z}}$. Since $b_W$ and $b_{\widetilde{Z}}$ are both blow-ups of coherent sheaves of ideals we have that $g^{-1}\mathcal{I}_Z\cdot\mathcal{O}_{X_{W,Z}}$ is an invertible sheaf. Therefore, the universal property of the blow-up \cite[Proposition 7.14]{Har} yields a morphism $h:X_{W,Z}\rightarrow X_Z$ such that $b_Z\circ h = g$. To conclude it is enough to set $\widetilde{\pi}=\overline{\pi}\circ h$.
\end{proof}

Next we fix a basis of classes of divisors and curves on $\G(1,n)_1$ and $\G(1,n)_R.$
Denote by 
\begin{itemize}
\item
$H$ a general hyperplane section in $\G(1,n),$ and its pullbacks in $\G(1,n)_1$ and $\G(1,n)_R;$
\item
$E$ the exceptional divisor of $\alpha$ in $\G(1,n)_1$ and its pullback in $\G(1,n)_R;$
\item
$F$ the exceptional divisor of $\alpha_1$ in $\G(1,n)_R;$
\item
$h$ a general line in $\G(1,n)$ and its strict transforms in $\G(1,n)_1$ and $\G(1,n)_R;$
\item
$e$ a general line in $E\subset \G(1,n)_1,$
and its strict transform in $\G(1,n)_R;$
\item
$f$ a general line inside a fiber of
$\alpha_{1|F}:F=\P(N_{\widetilde R\backslash \G(1,n)_1})\to\widetilde R.$
\end{itemize}
Therefore $f\cdot H=f\cdot E=0,f\cdot F=-1,$ 
$$N_1(\G(1,n)_1)=h\R\oplus e\R,N_1(\G(1,n)_R)=h\R\oplus e\R\oplus f\R,$$
$$N^1(\G(1,n)_R)=H\R\oplus E\R,N^1(\G(1,n)_R)=H\R\oplus E\R\oplus F\R,$$
and $h=H^*,e=-E^*,f=-F^*,$ where $D^*$ is the dual with respect to the intersection pairing
$$N_1(\G(1,n)_R)\times N^1(\G(1,n)_R)\to\R.$$

The curves contracted by $\xi$ are exactly those on the fiber $\G(1,3)_1,$
then by Example \ref{exampleC} we have the relative cone of curves of $\xi:\ NE(\xi)=NE(\G(1,3)_1)=\cone(e_0,h_0-e_0)$
(see Example \ref{exampleC} for the notation).
The restrictions of $H,E,F$ to $\G(1,3)_1$ are given by:
\begin{align*}
H_{|\G(1,3)_1}&=H_0,\\
E_{|\G(1,3)_1}&=E_0,\\
F_{|\G(1,3)_1}&=R_0=H_0-2E_0.
\end{align*}
Now we describe $e_0$ and $h_0-e_0$ in terms of the basis $h,e,f:$
\begin{align*}
e_0\cdot H&=e_0\cdot H_0=0\quad &(h_0-e_0)\cdot H&=(h_0-e_0)\cdot H_0=1\\
e_0\cdot E&=e_0\cdot E_0=-1\quad &(h_0-e_0)\cdot E&=(h_0-e_0)\cdot E_0=1\\
e_0\cdot F&=e_0\cdot (H_0-2E_0)=2\quad &(h_0-e_0)\cdot F&=(h_0-e_0)\cdot (H_0-2E_0)=-1.
\end{align*}
Therefore $e_0=e-2f$ and $h_0-e_0=h-e+f.$
Since, by \cite[Proposition 1.14]{De01} $NE(\xi)$ is an extremal subcone of $NE(\G(1,n)_R),$ then $h-e+f$ is an extremal ray of $NE(\G(1,n)_R).$
But $c:=h-e+f$ is $K_{\G(1,n)_R}$-negative, because 
$$-K_{\G(1,n)_R}=(n+1)H-(2(n-1)-1)E-(2(n-1)-n-1)F$$
and thus 
$$c\cdot (-K_{\G(1,n)_R})=(n+1)-[2(n-1)-1]+[2(n-1)-n-1]=1>0.$$

Then by Contraction Lemma,
\cite[Theorem 7.39]{De01}, there is a normal projective variety $\G(1,n)_1^+$
and a morphism $\Psi=cont_c:\G(1,n)_R\to \G(1,n)_1^+.$
Note that $\Psi_*:N_1(\G(1,n)_R)\to N_1(\G(1,n)_1)$ is surjective with kernel generated by $c.$ 
Fix the notation $\Psi_*(h)=h',\Psi_*(e)=e',\Psi_*(f)=f',$ for curves on  $\G(1,n)_1^+.$
Since $h'-e'+f'=0, $ or $f'=e'-h',$ we have $N_1(\G(1,n)_1)=h'\R\oplus e'\R$,
with intersection pairing such that $h'=H^*,e'=-E^*.$

Using again \cite[Proposition 1.14]{De01}
we have that $\xi$ factors through $\Psi,$ this means that there is a morphism
$\delta=cont_{e'-2f'}=cont_{2h'-e'}:\G(1,n)_1^+\to \G(1,n-2).$
Given a general conic in $\G(1,n)$ passing through $p$, consider the strict transform by $\alpha\circ\alpha_1,$ and then its image by $\Psi,$
the resulting curve has class $2h'-e'.$

Note that $c$ covers $R_0$ in each fiber $\G(1,3)_1\subset \G(1,n)_R,$
and therefore $c$ covers $F,$ this means that the exceptional locus of $\Psi$ is $\Exc(\Psi)=F.$
Thus $\Psi$ is a divisorial contraction and has the same exceptional divisor of $\alpha_1,$
hence there is a rational map $\eta:\G(1,n)_1\dasharrow \G(1,n)_1^+$ that is an isomorphism in codimension $2.$
This rational map yields an isomorphism 
$$\eta:\G(1,n)_1\backslash \widetilde R\ \widetilde\to \ \G(1,n)_1^+\backslash S$$ 
where $\dim(R)=n$ and $S:=\Psi(F)$ has codimension $2$ because $S$ has codimension $2$
in each fiber of $\delta$:
$$\Psi(F\cap \G(1,3)_1)\!=\!\psi(R_0)\!=\!S_0\cong Q^2\cong\P^1\times\P^1\subset
\P^3\!=\!\psi(E_0)\subset\P^4\!=\!\psi(\G(1,3)_1),$$
where $\psi:\G(1,3)_1\to \P^4$ is as in Example \ref{exampleC}.
We also get an isomorphism at the level of divisors
$$\eta^*:N^1(\G(1,n)_1)\ \widetilde\to \ N^1(\G(1,n)_1^+).$$

Consider the morphism $\tau=\beta\circ \alpha_1:\G(1,n)_R\to W.$
Note that $\tau$ has connected fibers and if
$H_W$ is a general hyperplane section in $W$ then
$$NE(\Psi)\!=\!c\cdot\R\subset NE(G(1,n)_R)\cap (H-E)^\perp \!=\!NE(G(1,n)_R)\cap \tau^*(H_W)^\perp
\!=\!NE(\tau).$$

By \cite[Proposition 1.14]{De01} we conclude that $\tau$ also factors through $\Psi,$
this means that there is a morphism $\gamma=cont_{c_1}:\G(1,n)_1^+\to W$ such that $\tau=\gamma \circ \Psi,$
for some curve class $c_1.$ Note that 
$\Exc(\beta)=\widetilde R,$ therefore
$\Exc(\tau)=F$ and $\Exc(\gamma)=\Psi(F)=S.$ This class $c_1$ has to be ortogonal to $H-E$ and effective,
since $f'=e'-h'$ does this job, by uniqueness of (the ray generated by) $c_1$ we can take
$c_1=e'-h'.$

\begin{figure}[htb!]
\begin{center}
\begin{tikzcd}
\G(1,n)_R \arrow[ddrr,"\xi",bend left=30]\arrow[d,"\alpha_1"']\arrow[dr,"\Psi"]&&\\
\G(1,n)_1 \arrow[dr,"\beta"']\arrow[d,"\alpha"']\arrow[r,dashed,"\eta"]
&\G(1,n)_1^+ \arrow[dr,"\delta"']\arrow[d,"\gamma"']&\\
\G(1,n) \arrow[rr,dashed,"\pi_R",bend right=20]\arrow[r,dashed,"\pi_0"']&
W \arrow[r,dashed,"\pi_1"']& \G(1,n-2)\\
\end{tikzcd}
\end{center}
\label{Projections}
\end{figure}
 
Lemma \ref{morphismsprops}
summarizes the properties of these morphisms.

\begin{Lemma}\label{morphismsprops} 
In the above notations,
\begin{itemize}
 \item $\alpha=\varphi_{|H|}=cont_e,$ $\Exc(\alpha)=E$ is a divisor;
 \item $\beta=\varphi_{|H-E|}=cont_{h-e},$ $\Exc(\beta)=R$ has codimension $n-2;$
 \item $\gamma=\varphi_{|H-E|}=cont_{e'-h'},$ $\Exc(\gamma)=S$ has codimension $2;$
 \item $\delta=\varphi_{|H-2E|}=cont_{2h'-e'},$ $\delta$ is a fibration with fibers isomorphic to $\P^4.$
\end{itemize}
\end{Lemma}

Now we are ready to prove 
Theorem \ref{MCD G(1,n)}.

\begin{proof} (of Theorem \ref{MCD G(1,n)})
We use the notation of this section.
Note that $\eta^*\delta^*(H_{\G(1,n-2)})=H-2E$ is  effective, where $H_{\G(1,n-2)}$ denotes a general hyperplane section on $\G(1,n-2).$ Note also that $\eta$ is a $(H-2E)$-flip since $(H-2E)\cdot (h-e)=-1<0.$ Now, using Corollary \ref{nef cone border3} and  Lemma \ref{morphismsprops} we get all the claims except $\G(1,n)_1^+$ being Fano. In order to prove this last fact it is enough to notice that
$$-K_{\G(1,n)_1^+}\!=\!(n+1)H-(2n-3)E\in \Amp(\G(1,n)_1^+)\!=\! int (\cone(H-E,H-2E)).$$
for any $n\geq 5$.
\end{proof}

\chapter{Spherical varieties}
\label{cap6}
In this chapter we classify which blow ups of Grassmannians at points in general position are spherical varieties.
\begin{Theorem}\label{G(r,n)kspherical}
$\G(r,n)_k$ is spherical if and only if one of the following holds
$$
\begin{cases}
r\!\!&\!=0 \mbox{ and } k\leq n+1;\ \mbox{ or }\\
k\!\!&\!=1;\ \mbox{ or }\\
k\!\!&\!=2 \mbox{ and } r=1 \mbox{ or } n=2r+1,n=2r+2; \mbox{ or }\\
k\!\!&\!=3 \mbox{ and } (r,n)=(1,5)
\end{cases}
$$
\end{Theorem}

Since spherical varieties are MDS (Theorem \ref{sphericalisMDS}), we obtain in this way new examples of MDS.

\begin{Theorem}
Let $\G(r,n)_k$ be the blow-up of the Grassmannian $\G(r,n)$ at $k$ general points. 
Then the following are Mori dream spaces
\begin{itemize}
\item $\G(r,n)_1,r\geq 1,n\geq 2r+1$;
\item $\G(r,2r+1)_2,\G(r,2r+2)_2,r\geq 1,
\mbox{and } \G(1,n)_2,n\geq 5$;
\item $\G(1,5)_3$.
\end{itemize}
\end{Theorem}

In the first section we recall the definition of spherical variety, give some examples and properties. In the two next sections we determine when $\G(r,n)_k,$ the blow-up of $\G(r,n)$ at $k$ general points, is spherical.
We follow \cite[Section 4.5]{coxrings} for spherical varieties and \cite[Chapter IV]{Borel} for algebraic groups. 

\newpage
\section{Preliminary definitions}

We start recalling the definition of reductive group.

\begin{Definition}
An algebraic group $G$ is solvable when it is solvable as an abstract group.
A Borel subgroup $B$ of an algebraic group $G$ is a subgroup which is maximal among the connected
solvable algebraic subgroups of $G.$
The radical $R(G)$ of an algebraic group is the identity component of the intersection
of all Borel subgroups of $G.$
We say that $G$ is semi-simple if $R(G)$ is trivial.
We say that $G$ is reductive if the unipotent part of $R(G),$
i.e., the subgroup of unipotent elements of $R(G)$, is trivial.
\end{Definition}

Given an algebraic group $G$ there is a single conjugacy class of Borel subgroups.
For instance, in the group $GL_n$ ($n\times n$ invertible matrices),
the subgroup of invertible upper triangular matrices is a Borel subgroup.
The radical of $GL_n$ is the subgroup of scalar matrices, 
therefore $GL_n$ is reductive but not semi-simple.
On other hand $SL_n$ is semi-simple.

\begin{Definition}
A spherical variety is a normal variety $X$ together with an action $G\times X \to X$
of a connected reductive affine algebraic group $G,$ a Borel subgroup $B\subset G,$
and a base point $x_0\in X$ such that the orbit map $B\to X,\ g\mapsto g\cdot x_0$
is an open embedding. 
\end{Definition}

\begin{Example}
A toric variety is a spherical variety with $B=G$ equal to the torus.
Consider $\P^n$ with the natural action of $SL_{n+1}$ together with the Borel subgroup $B$ consisting of all
upper triangular matrices of $SL_{n+1}$ and the base point $[0:\dots:0:1].$
It is easy to see that this is a spherical variety.
\end{Example}

\begin{Example}\label{G(r,n) spherical}
Consider $X:=\G(r,n), \ G:=SL_{n+1}.$
Choose a complete Flag $\{0\}=V_{-1}\subset V_0\subset V_1\subset \cdots \subset V_n=\C^{n+1}$
of linear spaces in $\C^{n+1}$, with $V_r$ corresponding to a point $p\in \G(r,n).$ 
Let $B$ be the only Borel subgroup of $G$ that stabilizes this Flag, and choose a basis 
$e_0,\dots, e_n$ of $\C^{n+1}$ such that $B$ is the subgroup of upper triangular matrices in this basis.
Consider the divisor $D=(p_{n-r,n-r+1,\dots,n}=0)$ and the point $p_0=[\Sigma]\in \G(r,n)\backslash D$ where
$$\Sigma=[0\ Id]=\begin{bmatrix}
 0 & \cdots & 0 &1& & \\
 \vdots & \ddots & \vdots  &  &\ddots&\\
0 & \cdots & 0 & &&1 
\end{bmatrix}.$$
We claim that $B\cdot p_0=\G(r,n)\backslash D.$
Indeed, a point $q\in \G(r,n)\backslash D,$
associated to $\Sigma_q\subset \P^n,$ is of the form
$$\Sigma_q=\begin{bmatrix}
\sigma_{0,0}&\dots&\sigma_{0,n-r}&      &\\
\vdots &     &\vdots   &\ddots&\\
\sigma_{r,0}&\dots&\sigma_{r,n-r}&\dots &\sigma_{r,n}
\end{bmatrix}=\sigma,$$
with $\sigma_{0,n-r},\dots,\sigma_{r,n}$ non-zero.
Then taking the element $b\in B$ of the form 
$$b=
\begin{bmatrix}
	Id & \multirow{2}{*}{$\widetilde \sigma$} \\ 0 &
\end{bmatrix}, \mbox{ where } 
\widetilde \sigma=\begin{bmatrix}
\sigma_{0,0}&\dots&\sigma_{r,0}\\
\vdots&&\vdots\\
\sigma_{0,n-r}&\dots& \sigma_{r,n-r}\\
&\ddots&\vdots\\
&&\sigma_{r,n}
\end{bmatrix}$$
we have $b\cdot p_0=q.$ Thus $(X,G,B,p_0)$ is a spherical variety.
\end{Example}

Our interest in Spherical varieties comes from Theorem \ref{sphericalisMDS} below that follows from the work of Michel Brion on \cite{Brion93}.
A more explicit proof of this theorem can be found on \cite[Section 4]{Pe14}.

\begin{Theorem}\label{sphericalisMDS}
Every Spherical variety is a Mori dream space.
\end{Theorem}

Next, we see how the effective cone of a spherical variety can be described in terms of divisors which are invariant under the action of the Borel subgroup.

\begin{Definition}
Let $(X,G,B,p)$ be a spherical variety.
We distinguish two types of $B$-invariant prime divisors:
\begin{enumerate}
	\item A boundary divisor of $X$ is a $G$-invariant prime divisor on $X.$
	\item A color of $X$ is a $B$-invariant prime divisor that is not $G$-invariant.
\end{enumerate}
\end{Definition}

\begin{Example}
A toric variety is a spherical variety with $B=G$ equal to the torus. In this case there are no colors, and the boundary divisors are the usual toric invariant divisors.
\end{Example}

\begin{Example}
Consider $\P^n$ with the natural action of $SL_{n+1}$ together with the Borel subgroup $B$ consisting of all
upper triangular matrices of $SL_{n+1}$ and the base point $[0:\dots:0:1].$
This is a spherical variety without boundary divisors and precisely one color, namely $V(z_n)\subset \P^n.$
In Example \ref{G(r,n) spherical} there are no boundary divisors and the only color is $D.$
\end{Example}

For a toric variety the cone of effective divisors is generated by the classes of boundary divisors.
For a spherical variety we have to take into accont the colors as well.

\begin{Proposition}
\cite[Proposition 4.5.4.4]{coxrings}\label{effcone spherical}
Let $(X,G,B,p_0)$ be a spherical variety.
\begin{enumerate}
	\item There are only finitely many boundary divisors $E_1,\dots,E_r$ and only
	finitely many colors $D_1,\dots,D_s$ on $X$ and we have
	$$X\backslash \ B\cdot p_0=E_1\cup\dots\cup E_r \cup D_1\cup\dots\cup D_s.$$
	\item The classes of the $E_k$ and $D_i$ generate $\Eff(X)\subset \Pic (X)$ as a cone.
\end{enumerate}
\end{Proposition}

\section{Spherical blow-ups of Grassmannians at points}

Now we show that any Grassmannian blow-up at one point is a spherial variety and compute its effective cone of divisors. 
Before we do so, we recall a well known way to produce divisors on the Grassmannian.

\begin{Lemma}\label{divinG(r,n)}
Let $H\in \Pic(\G(r,n))$ be the class of a hyperplane section.  If $\Gamma \subset \P^n$ is a $(n-r-1)-$dimensional linear subspace, then
$$D:=\{[\Sigma]\in \G(r,n): \Sigma \cap \Gamma \neq \varnothing \}$$
is a divisor with class $H.$ Moreover, if $\Pi\subset \P^n$ is the $r-$dimensional subspace corresponding to $p\in \G(r,n)$
then
$$\mult_p(D)=\dim(\Pi \cap \Gamma)+1$$
where $\dim (\varnothing )=-1.$
\end{Lemma}
\begin{proof}
We may assume that 
$$\Gamma=\left\langle e_{r+1},\dots,e_{n} \right\rangle.$$
Denote by $p_I$ the Pl\"ucker coordinates on $\P^N.$ We claim that $D=(p_{01\dots r}=0).$
Indeed, a point $q\in \G(r,n)$ corresponding to a linear space $\Sigma_q$ has first Pl\"ucker coordinate non-zero, $p_{01\dots r}(\Sigma_q)\neq 0$, if and only if it can be written on the form
$$\Sigma_q=\begin{bmatrix}
1& & & a_{0 r+1} & a_{0 r+2} & \cdots & a_{0 n} \\
 & \ddots && \vdots & \vdots & \ddots & \vdots \\
 & & 1 &a_{r r+1} & a_{r r+2} & \cdots & a_{r n}
\end{bmatrix}$$
if and only if has no intersection with $\Gamma.$
This means that
$$p_{01\dots r}(\Sigma_q)\neq 0 \Leftrightarrow
\Sigma_q \notin D.$$
For the claim regarding the multiplicity of $p$ in 
$D$ we refer to Lemma \ref{keylemma}$(4).$
\end{proof}

We would like to mention that John Kopper independently found the effective cone of $\G(r,n)_1$ in \cite{Ko16}. Moreover, Kopper managed to find the cones of effective cycles of all dimensions for $\G(r,n)_1$ and $\G(r,n)_2.$

\begin{Lemma}\label{G(r,n)1 spherical}
$\G(r,n)_1$ is a spherical variety
and $$\Eff(\G(r,n)_1)=\cone(E,H-(r+1)E)$$
where $E$ is the exceptional divisor and $H$ is the pullback of a general hyperplane section.
\end{Lemma}
\begin{proof}
In the same notation of Example 
\ref{G(r,n) spherical} consider
$$G_1:=\{g\in G;\ g\cdot p = p\}.$$
$G_1$ is the set of matrices with the $(n-r)\times (r+1)$ left down block equal to zero.
This algebraic group $G_1$ is not reductive because its unipotent radical is the normal subgroup $U_1$ of matrices with the two diagonal blocks equal to the identity.
The quotient $G_1^{red}=G_1/U_1$ can be identified with the set of matrices in 
$SL_{n+1}$ with non-zero entries only on the two diagonal blocks, and the two non diagonal blocks zero.
We have an isomorphism
$$G_1^{red}\cong\{M=(M',M'')\in GL_{r+1}\times GL_{n-r};\ \det(M')\det(M'')=1\}.$$
Therefore, $G_1^{red}$ is reductive and acts on $\G(r,n)_1$.
Taking on $G_1^{red}$ the subgroup $B_1$ of matrices with upper triangular blocks
we get a Borel subgroup. These groups are illustrated below
$$
G_1\!=\!\left\{
\begin{pmatrix}
	A & B \\ 0 & C
\end{pmatrix}\right\},
U_1\!=\!\left\{
\begin{pmatrix}
	1 & B \\ 0 & 1
\end{pmatrix}\right\},
G_1^{red}\!=\!\left\{
\begin{pmatrix}
	A & 0 \\ 0 & C
\end{pmatrix}\right\},
B_1\!=\!\left\{
\begin{pmatrix}
	D & 0 \\ 0 & E
\end{pmatrix}\right\}
$$
where $A,D\!\in\! GL_{r+1},B\!\in\! M_{(r+1)\times(n-r)},C,E\!\in\! GL_{n-r},$
$D,E$ are upper triangular, and $\det(A)\cdot \det(C)=\det(D)\cdot \det(E)=1$.

Let $\alpha:\G(r,n)_1\to \G(r,n)$ be the blow-up map and $E$ the exceptional divisor.
Consider the point 
$$p_0\!=\!\begin{bmatrix}
	1 & 0 & \dots & 0 & 0 &\dots & 0 &     0 &\dots& 0   & 1 \\
	0 & 1 & \dots & 0 & 0 &\dots & 0 &     0 &\dots& 1   & 0 \\
	0 & 0 & \ddots& 0 & 0 &\vdots& 0 &\vdots&\rddots&0   & 0 \\
	0 &\dots& 0   & 1 & 0 &\dots & 0 & 1     &0    &\dots& 0 
\end{bmatrix}_{(n+1)\times(r+1)}\hspace{-40pt}=[Id\ 0\ \tilde{Id}]\in \G(r,n)$$
and $x_1\in \G(r,n)_1$ such that $\alpha(x_1)=p_0.$
By \cite[Corollary II.7.15]{Har}
the action of $G_1$ on $\G(r,n)$
induces an action of $G_1$ on $\G(r,n)_1.$
We claim that $B_1 \cdot x_1$ is dense on $X_1.$ 

In order to prove this we consider the following linear supspaces of $\P^n:$
$$\begin{cases}
\Gamma_0=\left\langle e_{r+1},\dots, e_n\right\rangle \\
\Gamma_1=\left\langle e_0,e_{r+1},\dots, e_{n-1}\right\rangle \\
\vdots\\
\Gamma_{r+1}=\left\langle e_0,\dots,e_r,e_{r+1},\dots, e_{n-r-1}\right\rangle 
\end{cases}
$$
and 
$$\begin{cases}
\Gamma_0'=\left\langle e_0,e_{r+1},\dots, e_n\right\rangle \\
\Gamma_1'=\left\langle e_0,e_1,e_{r+1},\dots, e_{n-1}\right\rangle \\
\vdots\\
\Gamma_{r}'=\left\langle e_0,\dots,e_r,e_{r+1},\dots, e_{n-r}\right\rangle 
\end{cases}
$$
For every $j$ we have
$\Gamma_j\cong \P^{n-r-1}$ and
$\Gamma_j'=\left\langle\Gamma_j,e_j\right\rangle\cong \P^{n-r},$
thus using Lemma \ref{divinG(r,n)} we can define divisors
$$D_j:=\{[\Sigma]\in \G(r,n): \Sigma \cap \Gamma_j \neq \varnothing \}$$
with class $H-jE.$ Note that the $D_j$ are colors.
Note that 
$$p_0\notin D_0\cup D_1 \cup \dots \cup \ D_{r+1}.$$
Therefore it is enough to prove that
\begin{equation}\label{orbitp_0inG(r,n)1}
B_1\cdot  p_0=\G(r,n)\backslash\left\{
D_0\cup D_1 \cup \dots \cup \ D_{r+1}\right\}
\end{equation}
because from this it will follows that 
\begin{equation*}\label{orbitp_0inG(r,n)1bis}
B_1\cdot p_0'=\G(r,n)_1\backslash\left\{
D_0\cup D_1 \cup \dots \cup \ D_{r+1}\cup E\right\}
\end{equation*}
and that $B_1\cdot p_0'$ is an open dense subset on $\G(r,n)_1,$
where $\alpha(p_0')=p_0.$
Now let 
$$q\in \G(r,n)\backslash
\left\{
D_0\cup D_1 \cup \dots \cup \ D_{r+1}\right\}$$
and consider
$$w_j\in \Gamma_j'\cap \Sigma_q\subset \P^n,
j=0,\dots, r,$$
where $\Sigma_q$ corresponds to $q.$
Then we may write
\begin{equation*}\label{G(r,n)1 Borbit}
\begin{bmatrix}
w_0\\ w_1\\ \vdots \\w_r
\end{bmatrix}=[d\  e]=
\begin{bmatrix}
d_{0,0} &             &           &   & e_{0,r+1} &   &\dots & &e_{0,n-1} & e_{0,n}\\
d_{1,0} &  d_{1,1}     &           &   & e_{1,r+1} &  &\dots & &e_{1,n-1} & \\
\vdots   &  &    \ddots       &    &\vdots  &   & & \rddots &&\\
d_{r,0} & \ldots & \dots& d_{r,r} & e_{r,r+1} &  \dots &e_{r,n-r} & &   &  
\end{bmatrix}
\end{equation*}
for some complex numbers $d_{ij},e_{ij}.$
Note that 
$$d_{jj}=0\Rightarrow w_j\in \Gamma_j\Rightarrow q\in D_j$$
and therefore $d_{jj}\neq 0$ for $j=0,\dots,r,$ 
and thus 
$\Sigma_q=\left\langle w_0,w_1,\dots,w_r \right\rangle.$
Similarly the $e_{ij}'s$ on the diagonal are also non-zero because
$$e_{j, n-j}=0\Rightarrow w_j\in \Gamma_{j+1}\Rightarrow q\in D_{j+1}.$$

Now to conclude that equation (\ref{orbitp_0inG(r,n)1}) is true it is enough to note that $q=b\cdot p_0$ where $b\in B_1$ is given by 
$$b\!=\!
\left[
\begin{tabular}{ccc}
$d^{T}$ & 0 & 0\\
0 & $Id$ & \multirow{2}{*}{$\widetilde e$}\\
0 & 0 & 
\end{tabular}
\right]
\mbox{ and }
\widetilde e=
\begin{bmatrix}
e_{r,r+1} & \dots & e_{0,r+1}\\
\vdots & & \\
e_{r,n-r}&&\vdots\\
&\ddots &\\
&& e_{0,n}
\end{bmatrix}.
$$
\end{proof}

\begin{Corollary}\label{G(r,n)Mori dream space}
$\G(r,n)_1$ is a Mori dream space.
\end{Corollary}

\begin{Lemma}\label{G(r,2r+1)2 spherical}
$\G(r,2r+1)_2$ is a spherical variety for $r\geq 1$ 
and
$$\Eff(\G(r,2r+1)_2)=
\cone(E_1,E_2,H-(r+1)E_1,H-(r+1)E_2)$$
where $E_1,E_2$ are the exceptional divisors and $H$ is the pullback of a general hyperplane section. 
\end{Lemma}
\begin{proof}
The proof is the same as Lemma 
\ref{G(r,n)1 spherical} but now the class of $D_j$ is $H-jE_1-(r+1-j)E_2.$
\end{proof}

The proofs of the remaining results in this section are similar to the proof of Lemma \ref{G(r,n)1 spherical}. Therefore, we omit some details.

\begin{Lemma}\label{G(1,n)2 spherical}
$\G(1,n)_2$ is a spherical variety for $n\geq 5$ 
and $$\Eff(\G(1,n)_2)=\cone(E_1,E_2,H-2(E_1+E_2))$$
where $E_1,E_2$ are the exceptional divisors and $H$ is the pullback of a general hyperplane section. 
\end{Lemma}
\begin{proof}
We may suppose that 
$$p_1=[\left\langle e_0,e_1 \right\rangle],
p_2=[\left\langle e_2,e_3 \right\rangle]$$
and the groups involved are
$$
G_2\!=\!\left\{
\begin{pmatrix}
	A & 0 & B\\
	0 & C & D\\
  0 & 0 & E\\	
\end{pmatrix}\right\},
G_2^{red}\!=\!\left\{
\begin{pmatrix}
	A & 0 & 0\\
	0 & C & 0\\
  0 & 0 & E\\	
\end{pmatrix}\right\},$$
$$
B_2\!=\!\left\{
\begin{pmatrix}
	b_{00}&b_{01}&  0 	&  0   &  0   &  0   &  0   \\
	  0		&b_{11}&  0		&  0   &  0   &  0   &  0   \\
	  0 	&  0 	 &b_{22}&b_{23}&  0   &  0   &  0   \\
	  0		&  0	 &  0		&b_{33}&  0   &  0   &  0   \\
    0 	&  0 	 &	0 	&  0 	 &b_{44}&\dots &b_{4n}\\
    0 	&  0 	 &	0		&  0	 &  0		&\ddots&\vdots\\
		0 	&  0 	 &	0		&  0	 &  0		&  0   &b_{nn}\\
	 
\end{pmatrix}\right\}.
$$
Consider the following linear supspaces of $\P^n:$
$$\begin{cases}
\Gamma_{02}=\left\langle e_2,e_3,\dots, e_n\right\rangle \\
\Gamma_{11}=\left\langle e_0,e_2,e_4,e_5,\dots,
e_n\right\rangle \\
\Gamma_{20}=\left\langle e_0,e_1,e_4,e_5,\dots,
e_n\right\rangle \\
\Gamma_{21}=\left\langle e_0,e_1,e_2,e_4,e_5,\dots,
e_{n-1}\right\rangle \\
\Gamma_{12}=\left\langle e_0,e_2,e_3,e_4,e_5,\dots,
e_{n-1}\right\rangle \\
\Gamma_{22}=\left\langle e_0,e_1,\dots, e_{n-2}\right\rangle 
\end{cases}
$$
and 
$$\begin{cases}
\Gamma_0'=\left\langle e_0,e_2,e_3,\dots,e_n\right\rangle \\
\Gamma_1'=\left\langle e_0,e_1,e_2,e_4,e_5,\dots,e_n\right\rangle 
\end{cases}.
$$
Thus
$\Gamma_{ij}\cong \P^{n-2}$ and
$\Gamma_j'\cong \P^{n-1},$
and using Lemma \ref{divinG(r,n)} we can define divisors
$$D_{ij}:=\{[\Sigma]\in \G(1,n): \Sigma \cap \Gamma_{ij} \neq \varnothing \}$$
with class $H-iE_1-jE_2$ which are colors.
Consider
$$p_0=\begin{bmatrix}
1&0&0&1&0&\dots&0&0&1\\
0&1&1&0&0&\dots&0&1&1\\
\end{bmatrix}
$$
and note that $p_0\in \G(1,n)\backslash \bigcup_{i,j} D_{ij}.$
It is enough to show that 
\begin{equation*}
B_2\cdot  p_0=\G(1,n)\backslash \bigcup_{i,j} D_{ij}.
\end{equation*}
Let $q\in \G(1,n)\backslash \bigcup_{i,j} D_{ij}$
and choose $w_j\in \Sigma_q \cap\Gamma_j'$
where $\Sigma_q\subset \P^n$ is the line
corresponding to $q\in \G(1,n).$ Then we have
$$\begin{bmatrix}
w_0 \\ w_1
\end{bmatrix}=\begin{bmatrix}
x_0&0&x_2&x_3&x_4&\dots&x_{n-2}&x_{n-1}&x_n\\
y_0&y_1&y_2&0&y_4&\dots&y_{n-2}&y_{n-1}&y_n\\
\end{bmatrix}.
$$
Observe that 
\begin{equation}\label{eqdivesf}
\begin{cases}
x_0\neq 0 \mbox{ because } q\notin D_{02}\\ 
y_1\neq 0\neq x_3 \mbox{ because } q\notin D_{11}\\ 
y_2\neq 0 \mbox{ because } q\notin D_{20}\\ 
x_n\neq 0 \mbox{ because } q\notin D_{12}\\ 
y_n\neq 0 \mbox{ because } q\notin D_{21}\\ 
x_{n-1}y_n-x_ny_{n-1}\neq 0 
\mbox{ because } q\notin D_{22}
\end{cases}.
\end{equation}
This yields that
$\Sigma_q=\left\langle w_0,w_1 \right\rangle$
and that scaling $w_0$ or $w_1$ we may suppose that $x_n=y_n\neq 0.$
Then using the last condition on (\ref{eqdivesf}) we have that $y_{n-1}\neq x_{n-1}$ and thus
$q=b\cdot p_0$ where

$$b\!=\!
\left[
\begin{tabular}{cccc}
$b_1$&0 & 0 & 0\\
0&$b_2$ & 0 & 0\\
0&0 & $Id$ & \multirow{2}{*}{$b_3$}\\
0&0 & 0 & 
\end{tabular}
\right]\in B_2
\mbox{ and }$$
$$
b_1=
\begin{bmatrix}
x_0&y_0\\ 0 &y_1	 
\end{bmatrix},
b_2=
\begin{bmatrix}
y_2&x_2\\ 0 &x_3
\end{bmatrix},
b_3=
\begin{bmatrix}
y_4-x_4&x_4\\
y_5-x_5&x_5\\
\vdots&\vdots\\
y_{n-1}-x_{n-1}&x_{n-1}\\
0&x_n\\
\end{bmatrix}.$$
\end{proof}

\begin{Lemma}\label{G(1,5)3 spherical}
$\G(1,5)_3$ is a spherical variety and
$$\Eff(\G(1,5)_3)=
\cone(E_1,E_2,E_3,H-2(E_1+E_2),H-2(E_1+E_3)
,H-2(E_2+E_3))$$
where $E_1,E_2,E_3$ are the exceptional divisors and $H$ is the pullback of a general hyperplane section. 
\end{Lemma}
\begin{proof}
The proof is the same as Lemma 
\ref{G(1,n)2 spherical} but now the class of $D_{i,j}$ is $H-iE_1-jE_2-(4-i-j)E_3.$
\end{proof}

\begin{Lemma}\label{G(r,2r+2)2 spherical}
$\G(r,2r+2)_2$ is spherical for every $r\geq 1$
and 
$$\Eff(\G(r,2r+2)_2)=
\cone(E_1,E_2,H-(r+1)E_1-E_2,H-E_1-(r+1)E_2)$$
where $E_1,E_2$ are the exceptional divisors and $H$ is the pullback of a general hyperplane section. 
\end{Lemma}
\begin{proof}
We may suppose that 
$$p_1=[\left\langle e_0,e_1,\dots,e_r \right\rangle],
p_2=[\left\langle e_{r+1},e_{r+2},\dots,e_{2r+1} \right\rangle]\mbox{ and }$$
$$
B_2\!=\!\left\{
\begin{pmatrix}
	A_1 & 0  & 0  \\
	0  & A_2 & 0  \\
  0  & 0   & A_3\\	
\end{pmatrix}\right\},
\mbox{ where } A_3\in \C^* ,A_i=\begin{bmatrix}
a_{0,0}^i & \ldots & a_{0,r}^i	\\
          & \ddots & \vdots   		\\
          &        & a_{r,r}^i
\end{bmatrix}
\mbox{ for } i=1,2.$$
Consider the following linear supspaces of 
$\P^{2r+2}:$
$$\begin{cases}
\Gamma_{0, r+1}=\left\langle e_{r+1},\dots, e_{2r+2}\right\rangle \\
\Gamma_{1,r}=\left\langle e_0,e_{r+1},\dots, e_{2r},e_{2r+2}\right\rangle \\
\vdots\\
\Gamma_{r+1, 0}=\left\langle e_0,\dots,e_r,e_{2r+2}\right\rangle 
\end{cases}
$$
$$\begin{cases}
\Gamma_{1, r+1}=\left\langle e_0,e_{r+1},\dots,e_{2r+1}\right\rangle \\
\Gamma_{2,r}=\left\langle e_0,e_1,e_{r+1},\dots,e_{2r}\right\rangle \\
\vdots\\
\Gamma_{r+1, 1}=\left\langle e_0,\dots,e_{r+1}\right\rangle 
\end{cases}
$$
and 
$$\begin{cases}
\Gamma_0'=\left\langle e_0,e_{r+1},\dots, e_{2r+2}\right\rangle \\
\Gamma_1'=\left\langle e_0,e_1,e_{r+1},\dots, 
e_{2r},e_{2r+2}\right\rangle \\
\vdots\\
\Gamma_{r}'=\left\langle e_0,\dots,e_r,e_{r+1},e_{2r+2}\right\rangle 
\end{cases}
$$
For every $i,j$ we have
$\Gamma_{i,j}\cong \P^{r+1}$ and
$\Gamma_j'\cong \P^{r+2},$
thus using Lemma \ref{divinG(r,n)} we can define divisors
$$D_{i,j}:=\{[\Sigma]\in \G(r,2r+2): \Sigma \cap \Gamma_{i,j} \neq \varnothing \}$$
with class $H-iE_1-jE_2$ which are colors.
Consider the point 
$$p_0\!=\!\begin{pmatrix}
1  &  &  &&      && 1   & 1 \\
  & \ddots&  &&&\rddots&   & \vdots \\
 &  & 1 && 1     &    && 1 
\end{pmatrix}=[Id\ \ \tilde{Id} \ 1]\in \G(r,2r+2)$$
and note that 
$$p_0\notin \bigcup_{i,j}D_{i,j}.$$
Observing that for each $1\leq k\leq r$ we have
$$H-(r+2-k)E_1-kE_2=
\dfrac{(D_{r+1,1})(r+1-k)+(D_{1,r+1})(k-1)}{r}$$
to conclude the proof it is enough to prove that
\begin{equation}\label{orbitp_0inG(r,2r+2)2}
B_2\cdot  p_0=\G(r,n)\backslash\bigcup_{i,j}D_{i,j}.
\end{equation}
Now let 
$q\in \G(r,n)\backslash\bigcup_{i,j}D_{i,j}$
and consider
$$w_j\in \Gamma_j'\cap \Sigma_q\subset \P^{2r+2},
j=0,\dots, r,$$
where $\Sigma_q$ corresponds to $q.$
Then we may write
\begin{equation*}\label{G(r,2r+2)2 Borbit}
\begin{bmatrix}
w_0\\ w_1\\ \vdots \\w_r
\end{bmatrix}=[d\  e\ f]=
\begin{bmatrix}
d_{0,0} &             &           &   & e_{0,r+1} & &e_{0,2r} & e_{0,2r+1}&f_0\\
d_{1,0} &  d_{1,1}     &           &   & e_{1,r+1}    & &e_{1,2r} & &f_1\\
\vdots   &  &    \ddots       &    &\vdots      & \rddots &&&\vdots\\
d_{r,0} & \ldots & \dots& d_{r,r} & e_{r,r+1}   & &   &  &f_{2r+2}
\end{bmatrix}
\end{equation*}
for some complex numbers $d_{ij},e_{ij},f_j.$
Note that 
$$d_{jj}=0\Rightarrow w_j\in \Gamma_{j,r+1-j}\Rightarrow q\in D_{j,r+1-j}$$
and therefore $d_{jj}\neq 0$ for $j=0,\dots,r,$ 
and thus 
$\Sigma_q=\left\langle w_0,w_1,\dots,w_r \right\rangle.$
Similarly the $e_{ij}'s$ on the diagonal are also non-zero because
$$e_{j, 2r+1-j}=0\Rightarrow w_j\in \Gamma_{j+1,r-j}\Rightarrow q\in D_{j+1,r-j},$$
and the $f_j$ are also non-zero because
$$f_{j}=0\Rightarrow w_j\in \Gamma_{j+1,r+1-j}\Rightarrow q\in D_{j+1,r+1-j}.$$
Therefore we may suppose that $f_0=\dots=f_r=f\neq 0$
and then $q=b\cdot p_0$ where $b\in B_2$ is as follows
$$b\!=\!
\left[
\begin{tabular}{ccc}
$d^{T}$ & 0 & 0\\
0 & $\widetilde e$ & 0\\
0 & 0 & $f$
\end{tabular}
\right]
\mbox{ and }
\widetilde e=
\begin{bmatrix}
e_{r,r+1} & \dots & e_{0,r+1}\\
&\ddots &\\
&& e_{0,2r+1}
\end{bmatrix}.
$$
\end{proof}

\section{Non-spherical blow-ups of Grassmannians}

In this section we study how many points of $\G(r,n)$ we can blow-up and still have a spherical variety.
For the convenience of the reader we recall the following well-known consequence of
Zariski's main theorem.

\begin{Lemma}\label{push}
Let $f:X\rightarrow Y$ be a  proper birational morphism of noetherian integral schemes, and assume that $Y$ is normal. Then $f_{*}\mathcal{O}_X \cong\mathcal{O}_Y$ and $f$ is proper with connected fibers.
\end{Lemma}
\begin{proof}
Let us consider the Stein factorization 
  \[
  \begin{tikzpicture}[xscale=1.5,yscale=-1.2]
    \node (A0_0) at (0, 0) {$X$};
    \node (A0_2) at (2, 0) {$Z$};
    \node (A1_1) at (1, 1) {$Y$};
    \path (A0_0) edge [->]node [auto] {$\scriptstyle{g}$} (A0_2);
    \path (A0_2) edge [->]node [auto] {$\scriptstyle{h}$} (A1_1);
    \path (A0_0) edge [->,swap]node [auto] {$\scriptstyle{f}$} (A1_1);
  \end{tikzpicture}
  \]
of $f$, where $\mathcal{O}_Z\cong g_{*}\mathcal{O}_X$, $g$ is proper with connected fibers, and $h$ is finite. Since $X$ is reduced and $Z = \Spec_Y(f_{*}\mathcal{O}_X)$ we have that $Z$ is reduced as well. Furthermore $X$ irreducible implies that $Z = g(X)$ is irreducible. Therefore, $Z$ is integral. Now, since $f$ is birational $h$ is birational as well. We get that $h:Z\rightarrow Y$ is a a birational finite morphisms between Noetherian
integral schemes. Now, since $Y$ is normal by Zariski's main theorem we conclude that $h$ is an isomorphism. 
\end{proof}

Thanks to a result due to M. Brion \cite{Br11} in the algebraic setting, and to A. Blanchard \cite{Bl} in the analytic setting we get the following result on the connected component of the identity of the the automorphism group of a blow-up.

\begin{Theorem}\label{brion}\cite[Proposition 2.1]{Br11}
Let $G$ be a connected group scheme, $X$ a scheme with an action of $G$, and $f:X\rightarrow Y$ a proper morphism such that $f_{*}\mathcal{O}_X \cong\mathcal{O}_Y$. Then there is a unique action of $G$ on $Y$ such that $f$ is equivariant.
\end{Theorem}

\begin{Proposition}\label{automorphism prop}
Let $X$ be a noetherian integral normal scheme, $Z\subset X$ a closed subscheme of codimension greater or equal than two, and $X_{Z}:= Bl_{Z}X$ the blow-up of $X$ along $Z$. Then the connected component of the identity of the automorphism group of $X_Z$ is isomorphic to the connected component of the identity of the subgroup $\Aut(X,Z)\subseteq \Aut(X)$ of automorphisms of $X$ stabilizing $Z$, that is
$$\Aut(X_Z)^{o}\cong \Aut(X,Z)^{o}$$ 
\end{Proposition}
\begin{proof}
Let $\pi:X_{Z}\to X$ be the blow-up of $X$ along $Z$. By Lemma \ref{push} we have $\pi_{*}\mathcal{O}_{X_{Z}}\cong \mathcal{O}_X$. Therefore, we may apply Theorem \ref{brion} with $f = \pi$ and $G = \Aut(X_Z)^{o}$.\\ 
By Theorem \ref{brion} any automorphism $\phi\in \Aut(X_Z)^{o}$ induces an automorphism $\overline{\phi}\in \Aut(X)$ such the diagram 
\[
  \begin{tikzpicture}[xscale=2.0,yscale=-1.2]
    \node (A0_0) at (0, 0) {$X_Z$};
    \node (A0_1) at (1, 0) {$X_Z$};
    \node (A1_0) at (0, 1) {$X$};
    \node (A1_1) at (1, 1) {$X$};
    \path (A0_0) edge [->]node [auto] {$\scriptstyle{\phi}$} (A0_1);
    \path (A1_0) edge [->]node [auto] {$\scriptstyle{\overline{\phi}}$} (A1_1);
    \path (A0_1) edge [->]node [auto] {$\scriptstyle{\pi}$} (A1_1);
    \path (A0_0) edge [->]node [auto,swap] {$\scriptstyle{\pi}$} (A1_0);
  \end{tikzpicture}
\]
commutes. Let $x\in Z$ be a point such that $\overline{\phi}(x)\notin Z$, and let $F_x, F_{\overline{\phi}(x)}$ be the fibers of $\pi$ over $x$ and $\phi(x)$ respectively. Then $\phi_{|F_x}:F_x\rightarrow F_{\overline{\phi}(x)}$ induces an isomorphism between $F_x$ and $F_{\overline{\phi}(x)}$. On the other hand $F_x$ has positive dimension while $F_{\overline{\phi}(x)}$ is a point. A contradiction. Therefore $\overline{\phi}\in \Aut(X,Z)$.  Furthermore, since $\phi\in \Aut(X)^{o},$ the automorphism $\overline{\phi}$ must lie in $\Aut(X,Z)^{o}$. This yields a morphism of groups 
$$
\begin{array}{cccc}
\chi: &\Aut(X_Z)^{o}& \longrightarrow & \Aut(X,Z)^{o}\\
      & \phi & \longmapsto & \overline{\phi}
\end{array}
$$
If $\overline{\phi} = Id_{X}$ then $\phi$ coincides with the identity on a dense open subset of $X_Z$, hence $\phi = Id_{X_Z}$. Therefore, the morphism $\chi$ is injective. 
Finally, by \cite[Corollary 7.15]{Har} any automorphism of $X$ stabilizing $Z$ lifts to an automorphism of $X_Z$, that is $\chi$ is surjective as well. 
\end{proof}

Given $k$ general points in $\G(r,n)$ 
corresponding to linear subspaces $P_1,\dots,P_k\in \P^n,$ denote by $G_k$ the group 
$$\{g\in GL_{n+1}: gP_1=P_1,\dots,gP_k=P_k\},$$ 
by $U_k$ its unipotent radical, by $G_k^{red}$ the quotient $G_k/U_k,$
and by $B_k$ a Borel subgroup of $G_k^{red}.$
Using Lemma \ref{Bkisenough} below our aim reduces to understand whether the action of $B_k$ on $\G(r,n)$ has a dense orbit or not.

\begin{Lemma}\label{Bkisenough}
$\G(r,n)_k$ is spherical if and only if $B_k$ has a dense orbit. Moreover, $\G(r,n)_k$ spherical implies $\G(r,n)_{k-1}$ spherical.
\end{Lemma}
\begin{proof}
If $B_k$ has a dense orbit then $\G(r,n)_k$ is spherical by definition.
Assume now that $\G(r,n)_k$ is spherical.
Then there is a reductive group 
$$G\subset \Aut^\circ(\G(r,n)_k)$$
with a Borel subgroup $B$ having a dense orbit.
By Proposition \ref{automorphism prop}
$$\Aut^\circ(\G(r,n)_k)\cong
\Aut^\circ(\G(r,n),p_1,\dots,p_k),$$
where $p_1,\dots,p_k\in \G(r,n)$ are general points.
Now, Chow proved in \cite{Ch49} that 
$$\Aut^\circ(\G(r,n))\cong \P(GL_{n+1}).$$
Therefore
$$\Aut^\circ(\G(r,n)_k)\cong
(\P(GL_{n+1}),p_1,\dots,p_k).$$
Since $G$ is a reductive (affine) algebraic group, we may assume that 
$$G\subset (GL_{n+1},p_1,\dots,p_k)=G_k.$$
Thus, after a conjugation if necessary, we have $B\subset B_k.$ Since by hypothesis $B$ has a dense orbit in 
$\G(r,n)_k$ then $B_k$ has also a dense orbit in $\G(r,n).$
The last statement follows from Proposition \ref{automorphism prop}. 
\end{proof}

When $k\leq \left\lfloor \frac{n+1}{r+1} \right\rfloor$
we may choose the points as
$$p_1=[\left\langle e_0,\dots,e_r \right\rangle],\dots,p_k=[\left\langle e_{k(r+1)},\dots,e_{(k+1)(r+1)-1}\right\rangle]$$
and the corresponding Borel subgroup $B_k$ with upper triangular blocks.

\begin{Lemma}\label{dimcount}
Define $f(r,n):=\dim(B_2)-\dim(\G(r,n)).$ Then
$$f(r,n)=\dfrac{(n-(2r+2))(n-(4r+1))}{2}.$$
In particular, if $2r+2<n<4r+1$ then $\G(r,n)_2$ is not spherical.
\end{Lemma}
\begin{proof}
This is just a dimension count. 
\begin{align*}
&\dim(B_2)-\dim(\G(r,n))=\\
&=\left[(r+1)(r+2)+\dfrac{(n-2r-1)(n-2r)}{2}-1\right]-(r+1)(n-r)\\
&=(r+1)(-n+2r+2)+\dfrac{(n-2r-2)(n-2r)}{2}+\dfrac{n-2r}{2}-1\\
&=\dfrac{(n-2r-2)(-2(r+1)+n-2r+1}{2}\\
&=\dfrac{(n-2r-2)(n-4r-1)}{2}.
\end{align*}
The last claim follows from Lemma \ref{Bkisenough}.
\end{proof}

If we take $r=1$ in Lemma \ref{dimcount} we get 
$$f(1,n)=\dfrac{(n-4)(n-5)}{2}\geq 0 \quad \forall \ n\geq 3$$
and we have seen in Lemma \ref{G(1,n)2 spherical} that in this case $\G(1,n)_2$ is spherical.

In order to determine, for some fixed value of 
$r,$ if $\G(r,n)_2$ is spherical we begin with the smallest possible $n = 2r+1$ and keep increasing $n$.
When $n = 2r+1$ the dimension of the Borel subgruop $B_2$ is greater than the dimension of $\G(r,2r+1)$, and therefore we may have a dense orbit. Indeed, by Lemma \ref{G(r,2r+1)2 spherical} this is the case.

When we consider the next $n,$ i.e. $n=2r+2,$ then $\dim(B_2)=\dim(\G(r,n))$ and $\G(r,n)_2$
may be spherical. Again this is the case as we saw in Lemma \ref{G(r,2r+2)2 spherical}.
And then, there is the gap $2r+2<n<4r+1$ where $\G(r,n)_2$ can not be spherical.
Note that in the case $r=1$ this gap does not exist.
To conclude in the case $k = 2$ it is enough to apply Lemma \ref{G(r,n)2 spherical} below.

\begin{Lemma}\label{G(r,n)2 spherical}
$\G(r,n)_2$ is not spherical for $r\geq 2,n\geq 4r+1.$
\end{Lemma}
\begin{proof}
We to proceed as in the proof of Lemma \ref{G(r,2r+2)2 spherical} with the following variations. Here we have again $B_2$ with three blocks,
but the last one is bigger:
$$A_3=\begin{bmatrix}
a_{0,0}^3 & a_{0,1}^3 & \ldots & a_{0,l}^3	\\
          & a_{1,1}^3 & \ldots & a_{1,l}^3  \\
          &           & \ddots & \vdots   		\\
          &           &a_{l-1,l-1}^3& a_{l-1,l}^3\\
  0       &           &        & a_{l,l}^3
\end{bmatrix}, \mbox{where } l=n-(2r+2)>r.
$$
We consider linear spaces
$$\begin{cases}
\Gamma_0'=\left\langle e_0,e_{r+1},\dots, e_{n}\right\rangle \\
\Gamma_1'=\left\langle e_0,e_1,e_{r+1},\dots, 
e_{2r},e_{2r+2},\dots,e_n\right\rangle \\
\vdots\\
\Gamma_{r}'=\left\langle e_0,\dots,e_r,e_{r+1},e_{2r+2},\dots,e_n\right\rangle 
\end{cases}
$$
of codimension $r$ in $\P^n.$
Now given a general $q\in \G(r,n)$ corresponding to
$\Sigma_q$ our goal is to compute the dimension of the stabilizer of $q$ by the action of $B_2.$
There is only one point $w_j$ in each $\Gamma_j'\cap \Sigma_q$ and
$\Sigma_q=\left\langle w_0,\dots,w_r\right\rangle.$
Write 
$$w_i=(x^i_0:\dots: x^i_r:y^i_0:\dots: y^i_r:z^i_0:\dots: z^i_l)$$ for every $i.$ And then
$$\Sigma_q=
\begin{bmatrix}
x_0^0 &  0      & 0  & y_0^0 & \dots  &y_r^0		 & z^0_0 & \dots & z^0_l\\
\vdots&    \ddots     &   0 & \vdots&    \rddots    &    	0	&\vdots &       &     \vdots \\
x_0^r &  \ldots &x_r^r&y_0^r&  0    & 0       & z^r_0 & \dots & z^r_l    
\end{bmatrix}.
$$

Notice that $\Gamma_0',\dots,\Gamma_r'$ are stabilized by $B_2,$
therefore $b\in B_2$ stabilizes $\Sigma_q$ if and only if $b$ fixes $w_0,\dots,w_r.$
Therefore, setting $a_{l,l}^3=1,$ $b\in B_2$ stabilizes $\Sigma_q$ if only if:
$$
\begin{cases}
a_{0,0}^1x_0^0															&\!=\!\lambda_0 x_0^0\\
a_{0,0}^1x_0^1+a_{0,1}^1x_1^1								&\!=\!\lambda_1 x_0^1\\
a_{0,0}^1x_0^2+a_{0,1}^1x_1^2+a_{0,2}^1x_2^2&\!=\!\lambda_2 x_0^2\\
\vdots   																		&\\
a_{0,0}^1x_0^r+\cdots +a_{0,r}^1x_r^r   		&\!=\!\lambda_r x_0^r\\
\end{cases}\
\begin{cases}
0																						&\!=\!0							 \\
a_{1,1}^1x_1^1															&\!=\!\lambda_1 x_1^1\\
a_{1,1}^1x_1^2+a_{1,2}^1x_2^2								&\!=\!\lambda_2 x_1^2\\
\vdots   																		&\\
a_{1,1}^1x_1^r+\cdots +a_{1,r}^1x_r^r   		&\!=\!\lambda_r x_1^r\\
\end{cases}
$$
$$
\dots
\begin{cases}
0																						&\!=\!0							 \\
0																						&\!=\!0							 \\
\vdots   																		&\\
0																						&\!=\!0							 \\
a_{r,r}^1x_r^r  													  &\!=\!\lambda_r x_r^r
\end{cases}\
\begin{cases}
a_{0,0}^2y_0^0+\cdots +a_{0,r-1}^2y_{r-1}^0 +a_{0,r}^2y_r^0&\!=\!\lambda_0 y_0^0\\
a_{0,0}^2y_0^1+\cdots +a_{0,r-1}^2y_{r-1}^1	&\!=\!\lambda_1 y_0^1\\
\vdots   																		&\\
a_{0,0}^2y_0^{r-1}+a_{0,1}^2y_1^{r-1}					&\!=\!\lambda_{r-1} y_0^{r-1}\\
a_{0,0}^2y_0^r 															&\!=\!\lambda_r y_0^r
\end{cases}
$$
$$
\begin{cases}
a_{r-1,r-1}^2y_{r-1}^0+a_{r-1,r}^2y_r^0			&\!=\!\lambda_0 y_{r-1}^0\\
a_{r-1,r-1}^2y_{r-1}^1									  	&\!=\!\lambda_1 y_{r-1}^1\\
0																						&\!=\!0\\							
\vdots   																		&\\
0																						&\!=\!0							
\end{cases}
\dots
\begin{cases}
a_{r,r}^2y_r^0  													  &\!=\!\lambda_0 y_r^0\\
0																						&\!=\!0							 \\
\vdots   																		&\\
0																						&\!=\!0							 \\
0																						&\!=\!0							 
\end{cases}
$$
$$
\begin{cases}
a_{0,0}^3z^0_0+\dots +a_{0,l}^3z^0_l&=z_0^0\\
a_{0,0}^3z^1_0+\dots +a_{0,l}^3z^1_l&=z_0^1\\
\vdots &\\
a_{0,0}^3z^r_0+\dots +a_{0,l}^3z^r_l&=z_0^r
\end{cases}
\begin{cases}
a_{1,1}^3z^0_1+\dots +a_{1,l}^3z^0_l&=z_1^0\\
a_{1,1}^3z^1_1+\dots +a_{1,l}^3z^1_l&=z_1^1\\
\vdots &\\
a_{1,1}^3z^r_1+\dots +a_{1,l}^3z^r_l&=z_1^r
\end{cases}
$$
$$
\dots
\begin{cases}
a_{l-1,l-1}^3z^0_{l-1}+a_{l-1,l}^3z^0_l&=z_{l-1}^0\\
a_{l-1,l-1}^3z^1_{l-1}+a_{l-1,l}^3z^1_l&=z_{l-1}^1\\
\vdots &\\
a_{l-1,l-1}^3z^r_{l-1}+a_{l-1,l}^3z^r_l&=z_{l-1}^r
\end{cases}
\begin{cases} 
z^0&=\lambda_0 z^0 \\
z^1&=\lambda_1 z^1 \\
\vdots&\\
z^r&=\lambda_r z^r
\end{cases}
$$
for some $\lambda_0,\dots\lambda_r\in \C^*.$

Therefore we 
get from the last system of equations $\lambda_0=\dots=\lambda_r=1.$
Using this on the first systems we easily get $a_{ij}^1=a_{ij}^2=\delta_{ij}$
for every $i,j=0,\dots, r.$ 
We are left with the following $l$ systems on the variables $a_{i,j}^3,$
with $r+1$ equations each:

$$
\begin{cases}
a_{0,0}^3z^0_0+\dots +a_{0,l}^3z^0_l&=z_0^0\\
a_{0,0}^3z^1_0+\dots +a_{0,l}^3z^1_l&=z_0^1\\
\vdots &\\
a_{0,0}^3z^r_0+\dots +a_{0,l}^3z^r_l&=z_0^r
\end{cases}
\begin{cases}
a_{1,1}^3z^0_1+\dots +a_{1,l}^3z^0_l&=z_1^0\\
a_{1,1}^3z^1_1+\dots +a_{1,l}^3z^1_l&=z_1^1\\
\vdots &\\
a_{1,1}^3z^r_1+\dots +a_{1,l}^3z^r_l&=z_1^r
\end{cases}
$$
$$
\dots
\begin{cases}
a_{l-1,l-1}^3z^0_{l-1}+a_{l-1,l}^3z^0_l&=z_{l-1}^0\\
a_{l-1,l-1}^3z^1_{l-1}+a_{l-1,l}^3z^1_l&=z_{l-1}^1\\
\vdots &\\
a_{l-1,l-1}^3z^r_{l-1}+a_{l-1,l}^3z^r_l&=z_{l-1}^r
\end{cases}
$$
The last $r$ ones yield $a_{ij}^3=\delta_{ij}$ for $i,j\geq l-r.$
We get then $l-r$  independent systems, all of them with more variables than equations:
$$
\begin{cases}
a_{0,0}^3z^0_0+\dots +a_{0,l}^3z^0_l&\!=\!z_0^0\\
a_{0,0}^3z^1_0+\dots +a_{0,l}^3z^1_l&\!=\!z_0^1\\
\vdots &\\
a_{0,0}^3z^r_0+\dots +a_{0,l}^3z^r_l&\!=\!z_0^r
\end{cases}\dots
\begin{cases}
a_{l-r-1,l-r-1}^3z^0_{l-r-1}+\dots +a_{l-r-1,l}^3z^0_l&\!=\!z_{l-r-1}^0\\
a_{l-r-1,l-r-1}^3z^1_{l-r-1}+\dots +a_{l-r-1,l}^3z^1_l&\!=\!z_{l-r-1}^1\\
\vdots &\\
a_{l-r-1,l-r-1}^3z^r_{l-r-1}+\dots +a_{l-r-1,l}^3z^r_l&\!=\!z_{l-r-1}^r
\end{cases}.
$$
Since the $z_i^j$ are general each system has linearly independent equations.
The first system has $l+1$ variables and $r+1$ conditions,
the second has $l$ variables and $r+1$ conditions, and so on,
up to the last system having $r+2$ variables and $r+1$ conditions.
Now we can calculate the dimension of the stabilizer 
$$\mbox{\# of variables} - \mbox{\# of conditions}\!=\!
(l-r)\!+\!(l-r-1)\!+\!\dots+1\!=
\!\dfrac{(l-r)(l-r+1)}{2}.$$
Then the dimension of the orbit is
\begin{align*}
&\dim(B_2)-\dim(\mbox{ stabilizer })\\
&=(r+1)(r+2)+\dfrac{(l+1)(l+2)}{2}-1-\dfrac{(l-r)(l-r+1)}{2}\\
&=(r+1)(r+2)\!+\!\dfrac{(n-(2r\!+\!1))(n-2r)}{2}\!-\!\dfrac{(n-(3r\!+\!2))(n-(3r+1))}{2}-\!1\!\\
&=(r+1)(r+2)+\dfrac{n(2r+2)-5r^2-7r-2}{2}-1\\
&=(r+1)(n+r+2)-\dfrac{5r^2+7}{2}-2
=(r+1)(n-r)-\dfrac{r^2}{2}+\dfrac{r}{2}\\
&=\dim(\G(r,n))-(r-1)\dfrac{r}{2}.
\end{align*}

We conclude that the codimension of a general orbit of $B_2$ is $(r-1)\dfrac{r}{2}>0.$
Therefore, $\G(r,n)_2$ is not spherical for $r\geq 2,n\geq 4r+1.$
\end{proof}

\begin{Lemma}\label{G(1,n)3 spherical}
$\G(1,n)_3$ is not spherical for $n\geq 7.$
\end{Lemma}
\begin{proof}
$$
G_3\!=\!\left\{
\begin{pmatrix}
	A & 0 & B\\
	0 & C & D\\
  0 & 0 & E\\	
\end{pmatrix}\right\},
G_3^{red}\!=\!\left\{
\begin{pmatrix}
	A & 0 & 0\\
	0 & C & 0\\
  0 & 0 & E\\	
\end{pmatrix}\right\},$$

$$
B_3\!=\!
\left\{
\begin{pmatrix}
	b_{00}&b_{01}&  0 	&  0   &  0   &  0   &  0   &  0   &  0\\
	  0		&b_{11}&  0		&  0   &  0   &  0   &  0   &  0   &  0\\
	  0 	&  0 	 &b_{22}&b_{23}&  0   &  0   &  0   &  0   &  0\\
	  0		&  0	 &  0		&b_{33}&  0   &  0   &  0   &  0   &  0\\
		0 	&  0 	 &	0 	&  0 	 &b_{44}&b_{45}&  0   &  0   &  0\\
    0 	&  0 	 &	0		&  0	 &  0		&b_{55}&  0   &  0   &  0\\
	  0   &  0   &  0   &  0 	 &	0 	&  0 	 &b_{66}&\dots &b_{6n}\\
    0   &  0   &  0 	&  0 	 &	0		&  0	 &  0		&\ddots&\vdots\\
		0   &  0   &  0 	&  0 	 &	0		&  0	 &  0		&  0   &b_{nn}\\
	 
\end{pmatrix}\right\}
$$
Note that 
$\dim(B_3)=3+3+3+(n-5)(n-4)/2-1=8+(n-5)(n-4)/2.$
Proceeding as in the proof of Lemma \ref{G(r,n)2 spherical} we get a similar system
with $2(n-7)$ equations in $(n-5)(n-4)/2-3$ variables:
$$
\begin{cases}
b_{66}x_6+\dots+b_{6n}x_n     &= x_6\\
b_{66}x_6+\dots+b_{6,n-1}y_{n-1}     &= y_6
\end{cases}
\quad\dots$$
$$
\begin{cases}
b_{n-2,n-2}x_{n-2}+b_{n-2,n-1}x_{n-1}+b_{n-2,n}x_n&= x_{n-2}\\
b_{n-2,n-2}y_{n-2}+b_{n-2,n-1}y_{n-1}&= y_{n-2}
\end{cases}
$$
In the case $n=7$ there are no equations, or conditions.
In any case the dimension of the stabilizer is $(n-5)(n-4)/2-3-2(n-7),$ note this number is zero when $n=7.$
Then the dimension of the general orbit is
$$\left[8+\dfrac{(n-5)(n-4)}{2}\right]-\left[\dfrac{(n-5)(n-4)}{2}-3-2(n-7)\right]\!=\!2(n-1)-1$$
and thus the general orbit has codimension one on $\G(1,n).$
\end{proof}

\begin{Lemma}\label{G(r,2r+2)3 spherical}
$\G(r,2r+2)_3$ is not spherical for $r\geq 1.$
\end{Lemma}
\begin{proof}
This  is a dimension count, we will compare the dimensions of $\G(r,2r+2)$ and $B_3.$
We know that $B_3$ is the subgroup of $SL_{2r+2}$ which stabilizes the three flags:
\begin{align*}
&\mathcal{F}_1: \left\langle e_0 \right\rangle \subset 
\left\langle e_0,e_1 \right\rangle \subset \dots \subset
\left\langle e_0,\dots,e_r \right\rangle\\
&\mathcal{F}_2: \left\langle e_{r+1} \right\rangle \subset 
\left\langle e_{r+1},e_{r+2} \right\rangle \subset \dots \subset
\left\langle e_{r+1},\dots,e_{2r+1} \right\rangle\\
&\mathcal{F}_3: \left\langle e_{2r+2} \right\rangle \subset 
\left\langle e_{2r+2},v_1 \right\rangle \subset \dots \subset
\left\langle e_{2r+2},v_1,\dots,v_r \right\rangle
\end{align*}
where $v_2,\dots,v_r\in \C^{2r+2}$ are general and
$v_1:=e_0+\dots+e_{2r+2}.$
Consider the subgroup $B_3'$ of $SL_{2r+2}$ which stabilizes the three flags 
$\mathcal{F}_1,\mathcal{F}_2,\mathcal{F}_3'$ where
$$\mathcal{F}_3':\left\langle e_{2r+2} \right\rangle \subset 
\left\langle e_{2r+2},v_1 \right\rangle.$$
Then it is clear that $B_2\supset B_3'\supset B_3.$
Furthermore, $\dim(B_2)<\dim(B_3')$ because
$$B_3=\{g\in B_2; g\cdot \left\langle e_{2r+2}\right\rangle=\left\langle e_{2r+2}\right\rangle,
g\cdot \left\langle e_{2r+2},v_1 \right\rangle=
\left\langle e_{2r+2},v_1 \right\rangle\}.$$
Then 
$$\dim(B_3)\leq \dim(B_3')<\dim(B_2)=(r+1)(r+2)=\dim(\G(r,2r+1)).$$
Thus $B_3'$ can not have a dense orbit in $\G(r,n),$
and consequently neither  can $B_3$.
\end{proof}

Now using the previous reults we prove the main theorem.

\begin{proof}
(\textit{of Theorem \ref{G(r,n)kspherical}})

The case $r\!=\!0$ follows from the fact that 
$\Bl_{p_1,\dots,p_k}(\P^n)$ is toric if $k\leq n+1$,
and that $\Aut(\P^n,p_1,\dots,p_{n+2})=\{Id\}.$
Suppose now $r\!\geq \!1.$ 

The case $k=1$ in Lemma \ref{G(r,n)1 spherical}.

The case $k=2$ follows from Lemmas \ref{G(1,n)2 spherical}, \ref{G(r,2r+1)2 spherical},
\ref{G(r,2r+2)2 spherical}, \ref{dimcount}, and \ref{G(r,n)2 spherical}.

Now, for the case $k=3,$ having in mind Lemma \ref{Bkisenough} and the case $k=2$ we already know that $\G(r,n)_3$ is not spherical for $r\geq 2,n>2r+2.$
Lemmas \ref{G(1,5)3 spherical},
\ref{G(1,n)3 spherical}, and
\ref{G(r,2r+2)3 spherical}
yield us the cases $r\geq 2,n=2r+2$ and $r=1,n\neq 3,6.$
On the remaining cases, i.e. 
$\G(1,6)$ and $\G(r,2r+1),$
can be easily checked that $\dim(B_3)<\dim \G(r,n).$

Finally, for $k\geq 4$ using Lemma \ref{Bkisenough} we just need to verify that $\G(1,5)_4$ is not spherical.
Again, it is easy to check that $\dim(B_4)<\dim \G(r,n).$
\end{proof}

\chapter{Fano and  weak Fano varieties}
\label{cap7}
In this chapter we classify which blowups of Grassmannians at points in general position are weak Fano varieties. We also do the same for smooth quadrics.

\begin{Proposition}
Let $\G(r,n)_k$ be the blow-up of the Grassmannian $\G(r,n)$ at $k\geq 1$ general points. Then
\begin{enumerate}
\item[(a)] $\G(r,n)_k$ is Fano if and only if $(r,n)=(1,3)$ and $k\leq 2.$
\item[(b)] $\G(r,n)_k$ is weak Fano if and only if one of the following holds.
\item[$\bullet$] $(r,n)=(1,3)$ and $k\leq 2;$ or
\item[$\bullet$] $(r,n)=(1,4)$ and $k\leq 4.$
\end{enumerate}
\end{Proposition}

\begin{Proposition}
Let $Q^n_k$ be the blow-up of a smooth quadric $Q^n\subset \P^{n+1}$ at $k\geq 1$ general points. Then
\begin{enumerate}
 \item[(a)]$Q_k^n$ is Fano if and only if either $k\leq 2$ or $n=2$ and $k\leq 7.$
 \item[(b)]$Q_k^n$ is weak Fano if and only if one of the following holds.
\item[$\bullet$] $n=2$ and $k\leq 7;$ or
\item[$\bullet$] $n=3$ and $k\leq 6;$ or
\item[$\bullet$] $n\geq 4$ and $k\leq 2$.
\end{enumerate}
\end{Proposition}

Since weak Fano varieties are MDS (\cite[Corollary 1.3.2]{BCHM}), we obtain in this way new examples of MDS.

\begin{Theorem}
Let $\G(r,n)_k$ be the blow-up of the Grassmannian $\G(r,n)$ at $k$ general points. 
Then
$\G(1,3)_1,\G(1,3)_2,\G(1,4)_1,\G(1,4)_2,\G(1,4)_3,$ and $\G(1,4)_4$ are Mori dream spaces.
Let $Q^n_k$ be the blow-up of a smooth quadric $Q^n\subset \P^{n+1}$ at $k$ general points.
Then $Q^3_1,\dots,Q^3_6,Q^n_1,Q^n_2,n\geq 4$
and Mori dream spaces as well.
\end{Theorem}

In the first section we recall the basic definitions and some well known properties of Fano and weak Fano varieties. In the second section of this chapter we determine the Mori cones of some varieties, and in the third we describe which blow-ups at points in general position of quadrics and Grassmannians are Fano or weak Fano.

\section{Basic definitions and known results}

We start defining Fano and weak Fano varieties.

\begin{Definition}
Let $X$ be a smooth projective variety and $-K_X$ its anticanonical divisor. We say that $X$ is:
\begin{itemize}
\item[-] \textit{Fano} if $-K_X$ is ample;
\item[-] \textit{weak Fano} if $-K_X$ is nef and big;
\end{itemize}
\end{Definition}

In particular, Fano implies weak Fano. A smooth hypersurface $X\subset\mathbb{P}^{n}$ of degree $d$ is Fano if and only if $d\leq n.$

Koll\'ar, Miyaoka, and Mori in \cite[Theorem 0.2]{KMM92} proved that the $n$-dimensional smooth Fano varieties form a bounded family, that is, there are finitely many deformation classes of Fano varieties of a fixed dimension. In dimension one the only one is $\P^1.$ In dimension two they are called del Pezzo surfaces and are isomorphic either to $\P^1 \times \P^1$ or to the blow-up of the projective plane in at most $8$ points in general position.
In dimension three there are 105 classes. An overview of this classification is given by Iskovskikh and Prokhorov in \cite{IP99}.

In \cite[Corollary 1.3.2]{BCHM} Birkar, Cascini, Hacon, and  McKernan proved that weak Fano varieties are Mori dream spaces. In fact, there is a larger class of varieties called log Fano varieties, which includes weak Fano varieties and which are Mori dream spaces as well.

To check if a given variety $X$ is Fano or weak Fano  is quite simple, one just need to know its anticanonical divisor and the cone of curves.
But to check if $X$ is log Fano can be difficult, one needs to find an effective divisor $D$ such that $-(K_X+D)$ is ample and then check that the pair $(X,D)$ is indeed Kawamata log terminal.
In order to deal with this second issue one needs to find a log resolution of the pair $(X,D).$
This was done for the blow-up of $\P^n$ at $k$ points in general position by Araujo and  Massarenti in \cite{AM16}.

\section{Mori Cones}\label{moricones}

In this section we describe Mori cones of some varieties and in the end we also describe the cone of movable curves of the Grassmannians' blow-up at one point.

Now we fix some notations to be used in this section and in the next one. Given a variety $X$ we denote by $X_k$ its blow-up at $k$ general points,
by $E_1,\dots,E_k$ the exceptional divisors, and by $e_i$ the class of a general line contained in $E_i,$ for $i=1,\dots,k.$
Therefore, 
$$\Pic(X_k)=\pi^*(\Pic(X))\oplus E_1\Z\oplus \cdots
\oplus E_k\Z$$
and similarly $N_1(X_k)$ is generated by the strict transform of curves in $X$ and by $e_1,\dots,e_k.$
When $k=1$ we use $E$  and $e$ intead of $E_1$ and $e_1.$

We denote by $H$ the class of a general hyperplane section in $X,$ and also denote by $H$ the class in $\Pic(X_k)$ of the strict transform of a general hyperplane section.
Similarly, we denote by $h$ the class of a general line in $X$ and by the same letter the class in $N_1(X_k)$ of the strict transform of a general line.

For convenience of the reader we state the following result which was used in \cite{AM16} to describe whether the blow-up of the projective space at general points is a weak Fano or a log Fano variety. It also works as warm up for our results.

\begin{Lemma}\cite[Proposition 1.4]{AM16}\label{conecurvesprojectivespace}
Let $\P^n_k$ be the blow-up of projective space $\P^n$ in $k$ general points.
If $2\leq k\leq 2n$, then the Mori cone of $\P^n_k$ is generated by $e_i, i=1,\dots,k,$ and the
strict transform  $l_{ij}$ of the lines through two blown-up points.
\end{Lemma}

In the following we prove two results similar to Lemma \ref{conecurvesprojectivespace}.

We would like to remark that  Kopper independently proved Lemma \ref{conecurvesgeneral} in \cite{Ko16}, he also extended it to $\G(r,n)_k$ for $k\leq \codim(X)+2$ and determined the effective cone of $2$-dimensional cycles for $\G(r,n)_k,k\leq \codim(X)+1$.

\begin{Lemma}\label{conecurvesgeneral}
Let $X\subset \P^N$ be a projective non degenerate variety of dimension at least $2$,
covered by lines and such that $N_1(X)=\R [h]$.
If $k\leq \codim(X)+1$, then the Mori cone of $X_k$ is generated by $e_i,l_i, i=1,\dots,k,$  where $l_i=h-e_i$ is the class of the strict transform of a general line through the $i-$th blown-up point $p_i$.
\end{Lemma}
\begin{proof}
Let $c=dh-m_1e_1-\cdots-m_ke_k, d,m_1,\dots,m_k\in \Z,$
be the class of an irreducible curve $\tilde{C}\subset X_k$,
and $C\subset X$ the image of $\tilde C.$
If $C$ is a point, then $C$ is contracted by some of the blow-ups.
Hence $\tilde C\subset E_i$ for some $i$,
and therefore $c=ne_i$ for some $n>0$.
If $C$ is a curve, then $d\geq m_i=mult_{P_i}C\geq 0$ for everey $i$.
Assume that this is the case from now on.

If $k=1$, write $c=ml+(d-m)h$ with $d-m,m\geq 0$ and we are done.

Suppose now that $k\geq 2$. 
If some $m_i=0$ we reduce to the case $k-1$, assume then that $m_1>0$. 
It is sufficient to prove that $m_1+\cdots+m_k\leq d$. In fact, with this we may write
$$c=m_1l_1+\cdots + m_kl_k+(d-m_1-\cdots -m_k)h$$ which works because $h=l_i+e_i$ for any $i$.
To prove the statement $m_1+\cdots+m_k\leq d$ assume by contradiction that $m_1+\cdots+m_k>d$.

If $k$ satisfies $2\leq k\leq \codim(X)+1$, then set  $\Pi=\overline{p_1\dots p_k}=\P^{k-1}\subset \P^N.$
Therefore $C$ and $\Pi$ intersect with multiplicity at least 
$$m_1+\cdots+m_k>d=\deg (C)\cdot\deg (\Pi)$$
and then $C\subset \Pi$, but $k\leq $ codim$(X)+1$ is equivalent to dim $\Pi$+dim$(X)\leq N$,
and because the points $p_1,\dots, p_k$ are in general position, $\Pi\bigcap X$ has dimension zero,
that is, is finite. A contradiction with $C\subset\Pi\bigcap X$.
\end{proof}

Next, we prove a result for quadrics. In order to do this we introduce more notation. Denote by  $l_i$ the class $h-e_i$ of the strict transform of a general line through the $i-$th blown-up point $p_i,$ and
by $c_{ijl}$ the class $2h-e_i-e_j-e_l$ of the strict transform  of the conic through $p_i,p_j,p_l.$

\begin{Lemma}\label{conecurvesquadrics}
Let $Q^n_k$ be the blow-up of a smooth quadric $Q^n\subset \P^{n+1}$ at $k\geq 3$ general points.
Suppose that $k\leq (3n+2)/2$ if $n$ is even, and 
that $k\leq (3n+3)/2$ if $n$ is odd.
Then the Mori cone of $Q_k^n$ is generated by $e_i,l_i$ for $i=1,\dots,k,$ and by  
$c_{i,j,l}$ for $1\leq i<j<l\leq k.$
\end{Lemma}
\begin{proof}
Observe that given three general points $p_i,p_j,p_l$ the plane generated by them
intersects the quadric $Q^n$ in a conic whose strict transform under 
the blow-up has class $c_{ijl}=2h-e_i-e_j-e_l$. It is enough to prove the result
for $k= (3n+2)/2$ if $n$ is even, and for $k= (3n+3)/2$ if $n$ is odd.
We use the same notation and strategy of Lemma \ref{conecurvesgeneral}:
assume from now on $d\geq m_k\geq \cdots \geq m_1>0.$ 
It is enough to write 
$$c=dh-m_1e_1-\cdots-m_ke_k$$
as a non negative integer combination of $c_{ijl},l_i,e_i,$ and $h$.

%

If $n=2n'-1$ is odd, then $k=3n'$, and we divide the points into $n'$ groups of three.
In each group we pick a conic through the three points with the multiplicity required for the first,
then we pick a line through the second and another line through the third.

If $n=2n'$ is even, then $k=3n'+1$, and we divide the points into a group
of four points and $n'-1$ groups of three points.
For the latter, we do as we did for the odd case.
For the former, we pick a conic through three points, say $A,B,C,$ and then another conic through the another triple, say $B,C,D,$ and then one line through $C$ and  one line through $D.$

Now, more precisely we proceed as follows.

For $k=3$ write $$c=m_1c_{123}+(m_2-m_1)l_2+(m_3-m_1)l_3+(d-m_2-m_3)h.$$
If $m_2+m_3>d$ then the line through $p_2,p_3$ is contained in $c$.
Since $c$ is irreducible, then $c$ is this line, and hence $m_1=0$ giving a contradiction.
Thus $m_2+m_3\leq d,$ that is $d-m_2-m_3\geq 0,$ and we are done. 

For $k=4$ write $$c=m_1c_{123}+(m_2-m_1)c_{234}+(m_3-m_2)l_3+(m_4-m_2+m_1)l_4+(d-(m_1+m_3+m_4))h.$$ 
If $m_1+m_3+m_4>d$ then $c$ is contained in the plane $\Pi=\overline{p_1p_3p_4}=\P^2$ and $m_2=0$,
contradiction. Then $m_1+m_3+m_4\leq d$ and we are done. 

Now assume $k\geq 5$.

First the case $n=2n'-1$ odd.
Write
\begin{align*}
c&=\sum _{i=1}^{n'}\Big[
m_{3i-2} c_{3i-2,3i-1,3i} +
(m_{3i-1}-m_{3i-2}) l_{3i-1}+(m_{3i}-m_{3i-2}) l_{3i} \Big]+\\
&+\left(d-\sum_{i=1}^{n'}(m_{3i-1}+m_{3i})\right)h
\end{align*}

If we have $d-(\sum_{i=1}^{n'}(m_{3i-1}+m_{3i}))\geq 0$, then we are done.
Assume by contradiction that $\sum_{i=1}^{n'}(m_{3i-1}+m_{3i})>d$.
Then $c$ is contained in 
$$\overline{p_2p_3\cdots p_{3n'-1}p_{3n'}}=\P^{2n'-1}=\P^n=\Pi.$$
Since the points are in general position,  $p_1\notin \Pi$ and $m_1=0$, contradiction again.

Now the case $n=2n'$ even.
Write
\begin{align*}
c&=\sum _{i=1}^{n'-1}\Big[m_{3i-2}
 c_{3i-2,3i-1,3i} +
(m_{3i-1}-m_{3i-2}) l_{3i-1}+(m_{3i}-m_{3i-2}) l_{3i} \Big]\\
&+m_{3n'-2}c_{3n'-2,3n'-1,3n'}+(m_{3n'-1}-m_{3n'-2})c_{3n'-1,3n',3n'+1}+\\
&+(m_{3n'}-m_{3n'-1})l_{3n'}+(m_{3n'+1}-m_{3n'-1}+m_{3n'-2})l_{3n'+1}+\\
&+\left(d-\sum_{i=0}^{n'-1}(m_{3i-1}+m_{3i})-(m_{3n'-2}+m_{3n'}+m_{3n'+1})\right)h
\end{align*}

If we have $d-\sum_{i=1}^{n'-1}(m_{3i-1}+m_{3i})-(m_{3n'-2}+m_{3n'}+m_{3n'+1})\geq 0$,
then we are done.
Assume by contradiction that
$$\sum_{i=1}^{n'-1}(m_{3i-1}+m_{3i})+(m_{3n'-2}+m_{3n'}+m_{3n'+1}))>d.$$
Then $c$ is contained in 
$$\overline{p_2p_3\dots p_{3n'-4}p_{3n'-3}p_{3n'-2}p_{3n'}p_{3n'+1}} =\P^{2n'}=\P^n=\Pi.$$
Since the points are in general position,  $p_1\notin \Pi$ and $m_1=0.$ A Contradiction.
\end{proof}

Finally, we describe the cone of moving curves of the blow-up at one point $\G(r,n)_1$ of the Grassmannian.
We denote by $\mov(X)$
the cone of moving curves of $X,$ namely the convex cone generated by classes
of curves moving in a family of curves covering
a dense subset of $X.$
We denote by $NE(X)$ the Mori cone of $X.$

\begin{Remark} Recall the dualities of cones of curves and divisors:
$$NE(X) = \Nef(X)^\vee  \mbox{ and } \mov(X) = \Eff(X)^\vee.$$
The first one follows directly from the definitions and the second one was proved in 
\cite[Theorem 2.2]{BDPP}.
\end{Remark}
\begin{Proposition}We have
\begin{itemize}
\item $NE(\G(r,n)_1)=\cone(e,h-e);$
\item $\Nef(\G(r,n)_1)=\cone( H, H-E);$
\item $\mov(\G(r,n)_1)=\cone( h,(r+1)h-e);$
\item $\Eff(\G(r,n)_1)=\cone( E, H-(r+1)E).$
\end{itemize}
\end{Proposition}
\begin{proof}
The first item follows from Lemma \ref{conecurvesgeneral} and the second by duality.

Note that $h$ is a moving curve because $\G(r,n)$ is covered by lines. The class $(r+1)h-e$ is also moving. Indeed, if $p\in \G(r,n)$ is the point blown-up and $q\in \G(r,n)$ is another point let $V_p,V_q\subset\P^n$ be the $r-$dimensional projective spaces corresponding to $p,q,$ and
consider $X\subset \P^n$ a rational normal scroll of dimension $r+1$ containing $V_p,V_q.$
Then the $r-$dimensional spaces contained in $X$ correspond to points of a rational normal curve
of degree $r+1$ inside $\G(r,n)$ connecting $p$ and $q,$ the strict transform of this curve has class $(r+1)h-e.$
This shows that $\G(r,n)_1\setminus E$ is covered by curves of class $(r+1)h-e.$ 

Therefore $\mov(X)\supset\cone(h,(r+1)h-e).$

By duality we have $\Eff(X)\subset \cone(E, H-(r+1)E).$
But we know that $E$ is effective, and the class $H-(r+1)E$ is effective as well.
Indeed, it is easy to explicit a divisor $D$ on $\G(r,n)$ that is a hyperplane section with
multiplicity $r+1$ at $p$ using Lemma \ref{divinG(r,n)}.

We conclude that $\Eff(X)= \cone(E, H-(r+1)E),$ and by duality $\mov(X)=\cone(h,(r+1)h-e).$
\end{proof}

\section{Weak Fano blow-ups of quadrics and Grassmannians at points}

In this section we determine when the blow-up of a smooth quadric and the blow-up of a Grassmannian at general points are Fano or weak Fano. We keep the notations introduced in the beginning of Section \ref{moricones}.

As motivation we recall the following result for the projective space.

\begin{Proposition}\label{P_k^n weak Fano}
Let $\P^n_k$ be the blow-up of the projective space $\P^n$ at $k\geq 1$ general points. Then
\begin{enumerate}
 \item[(a)]$\P_k^n$ is Fano if and only if either $k=1$ or $n=2$ and $k\leq 8.$
 \item[(b)]$\P_k^n$ is weak Fano if and only if one of the following holds.
\item[$\bullet$] $n=2$ and $k\leq 8;$ or
\item[$\bullet$] $n=3$ and $k\leq 7;$ or
\item[$\bullet$] $n\geq 4$ and $k=1.$
\end{enumerate}
\end{Proposition}

Now, we prove our result for quadrics.

\begin{Proposition}\label{Q_k^n weak Fano}
Let $Q^n_k$ be the blow-up of a smooth quadric $Q^n\subset \P^{n+1}$ at $k\geq 1$ general points. Then
\begin{enumerate}
 \item[(a)]$Q_k^n$ is Fano if and only if either $k\leq 2$ or $n=2$ and $k\leq 7.$
 \item[(b)]$Q_k^n$ is weak Fano if and only if one of the following holds.
\item[$\bullet$] $n=2$ and $k\leq 7;$ or
\item[$\bullet$] $n=3$ and $k\leq 6;$ or
\item[$\bullet$] $n\geq 4$ and $k\leq 2$.
\end{enumerate}
\end{Proposition}
\begin{proof}
We use the same notation of Lemma \ref{conecurvesquadrics} for curves.
For $n=2$ the result follows from the identification $\P^2_2\cong Q^2_1$ and from Proposition \ref{P_k^n weak Fano}.
Assume $n\geq 3$.

Noting that $-K_{Q_k^n}=nH-(n-1)(E_1+\cdots+E_k)$, for every $k$ we have
$$-K_{Q_k^n}\cdot e_i= n-1>0 \mbox{ and }
-K_{Q_k^n}\cdot l_i=n-(n-1)=1>0.$$
Then $-K_{Q_k^n}$ is ample for $k\leq 2$
by Lemma \ref{conecurvesgeneral}. Therefore, $Q_1^n$ and $Q_2^n$ are Fano.

Now assume $k\geq 3$, and observe that
$$-K_{Q_k^n}\cdot c_{ijl}\!=\!(nH-(n-1)(E_1+\cdots+E_k))\cdot(2h-e_i-e_j-e_l)\!=\!2n-3(n-1)=3-n.$$
Hence $-K_{Q_k^n}$ is not nef and $Q_k^n$ 
not weak Fano for $n\geq 4,k\geq 3$.

It remains the case $n=3.$ Assuming $n=3,$ Lemma \ref{conecurvesquadrics} gives us the Mori cone of $Q_k^3$ for $k$ up to $6$,
and we have that $-K_{Q_k^3}\cdot c_{ijl}=0$. Therefore $-K_{Q_k^3}$ is nef but not ample for $3\leq k\leq 6$. Note that $H^3=\deg(Q^3)=2$ and
$$   (-K_{Q_k^3})^3=(3H-2(E_1+\dots+E_k))^3=3^3\cdot 2-2^3\cdot k=
54-8k>0 \Leftrightarrow k\leq 6.  $$
Therefore, $Q_k^3$ is weak Fano if and only if $k\leq 6$ by \cite[Section 2.2]{Lazarsfeldvol1}.
\end{proof}

Next, we prove a result for the Grassmannian.

\begin{Proposition}\label{G(r,n)_k weak Fano}
Let $\G(r,n)_k$ be the blow-up of the Grassmannian $\G(r,n)$ at $k\geq 1$ general points. Then
\begin{enumerate}
\item[(a)] $\G(r,n)_k$ is Fano if and only if $(r,n)=(1,3)$ and $k\leq 2.$
\item[(b)] $\G(r,n)_k$ is weak Fano if and only if one of the following holds.
\item[$\bullet$] $(r,n)=(1,3)$ and $k\leq 2;$ or
\item[$\bullet$] $(r,n)=(1,4)$ and $k\leq 4.$
\end{enumerate}
\end{Proposition}

\begin{proof}
First we analyze wheter the blow-up at one point has anticanonial divisor nef or not.
By Lemma \ref{conecurvesgeneral},
it is sufficient to intersect $-K_{\G(r,n)_1}$ with $l$ and $e$.
We have $-K_{\G(r,n)_1}=(n+1)H-((r+1)(n-r)-1)E$ and
$$\begin{cases}
-K_{\G(r,n)_1}\cdot e= (r+1)(n-r)-1 >0 \\
-K_{\G(r,n)_1}\cdot l= (n+1)- ((r+1)(n-r)-1)=-nr+r^2+r+2
\end{cases}$$

If $r=1$, then $-K_{\G(r,n)_1}\cdot l=-n+4$, and therefore $-K_{\G(r,n)_1}$ is nef (ample) if and only if $n\leq 4$ ($n=3$).
Thus $\G(r,n)_1$ is not weak Fano (Fano) for
$n\geq 5$ ($n\geq 4$).

If $r\geq 2$, we have 
$$-K_{\G(r,n)_1}\cdot l\leq -(2r+1)r+r^2+r+2=-r^2+2<0.$$
Thus all these $\G(r,n)_1$ do not have anticanonical divisor nef and are not weak Fano.
 
Now we only have to check when $\G(1,3)_k$ is weak Fano or Fano and when $\G(1,4)_k$ is weak Fano.

For $\G(1,3)_k$ use Lemma \ref{Q_k^n weak Fano} with $n=4$.

For $\G(1,4)_k$ recall that $\deg(\G(1,4))=5$ and
$\codim(\G(1,4))=9-6=3.$
Then $H^6=5$ and we have the Mori cone of 
$\G(1,4)_k$ for $k$ up to $4$ by
Lemma \ref{conecurvesgeneral}.
Note now that for every $k$ we have 
$$\begin{cases}
-K_{\G(1,4)_k}\cdot e_i=5>0 \\
-K_{\G(1,4)_k}\cdot l_i=0 \\
(-K_{\G(1,4)_k})^6=(5H-5(E_1+\cdots+E_k)))^6=5^6(5-k)
\end{cases}$$
Then, again by \cite[Section 2.2]{Lazarsfeldvol1},
we have that
$\G(1,4)_k$ is weak Fano if and only if $k\leq 4.$ 
\end{proof}

Since \cite[Corollary 1.3.2]{BCHM} implies that weak Fano varieties are Mori dream spaces we get  the following as a direct consequence of Propositions \ref{Q_k^n weak Fano} and\ref{G(r,n)_k weak Fano}.

\begin{Theorem}
Let $\G(r,n)_k$ be the blow-up of the Grassmannian $\G(r,n)$ at $k$ general points. 
Then
$\G(1,3)_1,\G(1,3)_2,\G(1,4)_1,\G(1,4)_2,\G(1,4)_3,$ and $\G(1,4)_4$ are Mori dream spaces.
Let $Q^n_k$ be the blow-up of a smooth quadric $Q^n\subset \P^{n+1}$ at $k$ general points.
Then $Q^3_1,\dots,Q^3_6,Q^n_1,Q^n_2,n\geq 4$
and Mori dream spaces as well.
\end{Theorem}

\chapter{Conjectures  and future work}
\label{cap8}
In this chapter we collect some results and conjectures concerning the birational geometry of $\G(r,n)_k.$ 
In Section \ref{MCD G(r,n)_1} we describe a conjectural Mori chamber Decomposition of $\G(r,n)_1$ and a possible strategy to prove it.
In Section \ref{MDS G(r,n)_k} we discuss possible approaches  to determine for which triples $(r,n,k)$ the blow-up $\G(r,n)_k$ is a Mori dream space.

\section{Mori chamber decomposition of \texorpdfstring{$\Eff(\G(r,n)_1)$}{TEXT}}\label{MCD G(r,n)_1}\label{sec91}

In Chapter \ref{cap5} we constructed the Mori chamber decomposition of $\G(1,n)_1,$ now we try to extend the strategy to $\G(r,n)_1.$ In order to do this we introduce some maps. Recall that in Section \ref{osculating spaces grass} we described the osculating spaces $T^i$ to Grassmannians, defined the subvarieties 
$$R_i=\G(r,n)\bigcap T^i, i=0,\dots,r,$$
and showed some of its properties.

\begin{Notation}
Denote by $\alpha_0:\G(r,n)_1\to G(r,n)$ the blow-up of a point, we also write $\G(r,n)_{R_0}=\G(r,n)_1.$
Recursively, for $i=1,\dots,r,$ denote by $\alpha_i:\G(r,n)_{R_i}\to \G(r,n)_{R_{i-1}}$ the blow-up of the strict transform $\widetilde{R_{i}}\subset \G(r,n)_{R_{i-1}}$ of $R_i\subset \G(r,n).$
\end{Notation}

\begin{Notation}
For any $0\leq i\leq r$ consider the restriction 
$$\Pi_{T^i_p}
=\tau_i:\G(r,n)\dasharrow W_i,$$ 
of the linear projection from the osculating space $T^i_p(\G(r,n))$ and set $W_{-1}:=\G(r,n).$
For any $0\leq i \leq r$ define $\pi_i:W_{i-1}\dasharrow W_i$ as the only rational map
that makes the diagram 
\begin{center}
\begin{tikzcd}
&W_{i-1}\arrow[rd,dashed,"\pi_i"]&\\
\G(r,n)\arrow[ru,dashed,"\tau_{i-1}"]\arrow[rr,dashed,"\tau_i"']&&W_i
\end{tikzcd}
\end{center}
commutative.
\end{Notation}

Now we have two sequences of maps
\begin{equation}\label{seqofblowups}
\G(r,n)_{R_r}
\stackrel{\alpha_{r}}{\longrightarrow}\G(r,n)_{R_{r-1}}
\stackrel{\alpha_{r-1}}{\longrightarrow}
\cdots 
\stackrel{\alpha_{2}}{\longrightarrow}
\G(r,n)_{R_1}
\stackrel{\alpha_{1}}{\longrightarrow} 
\G(r,n)_{R_0}
\stackrel{\alpha_{0}}{\longrightarrow}
\G(r,n),
\end{equation}
$$\G(r,n)
\stackrel{\pi_{0}}{\dasharrow}
W_0
\stackrel{\pi_{1}}{\dasharrow}
W_1
\stackrel{\pi_{2}}{\dasharrow}
\cdots 
\stackrel{\pi_{r-1}}{\dasharrow} 
W_{r-1}
\stackrel{\pi_{r}}{\dasharrow}
W_r.$$

\begin{Remark}\label{connection}
We remark here that the need to check that these maps $\pi_i$ are birational was the starting point to all material in Part \ref{part1}. In fact, we needed Proposition \ref{oscprojbirational} and tried to prove it using Proposition \ref{cc}.
However, there were not many results concerning defectivity of Grassmannians.

As soon as we proved 
Proposition \ref{oscprojbirational}, it became clear that we could proceed the other way around, that is, to use Proposition \ref{oscprojbirational} together with Proposition \ref{cc} to obtain new results on defectivity of Grassmannians.
Then we realized that this method was more general, and could be applied for the Segre-Veronese varieties for instance, and so we obtained the results in Chapter \ref{cap4}.
\end{Remark}

\begin{Conjecture}\label{keyconjecture}
For each $1\leq i\leq r$ we have that
\begin{enumerate}
\item The singularities of $R_i$ can be resolved by the sequence of blow-ups along smooth centers below:
	$$R_i^i\stackrel{\alpha_{i-1}}{\longrightarrow}\cdots 
	\stackrel{\alpha_{1}}{\longrightarrow} R_i^1\stackrel{\alpha_{0}}{\longrightarrow}R_i.$$
	\item The sequence of blow-ups $\alpha_{i-1}\circ\dots \alpha_1\circ \alpha_0$ resolves the rational map $\tau_i.$
\end{enumerate}
\end{Conjecture}

Now we are ready to describe our strategy.
Conjecture \ref{keyconjecture} implies that the sequence of blow-ups (\ref{seqofblowups}) resolves the map $\tau_r.$
Therefore, there is a map $\xi:\G(r,n)_{R_r}\to W_r$
making the following diagram commutative.

\begin{center}
\begin{tikzcd}
\G(r,n)_{R_r}\arrow[dddrrrr,"\xi"]&&&&\\
\vdots&&&&\\
\G(r,n)_{R_0}\arrow[d,"\alpha_0"]
\arrow[dr]&&&&\\
\G(r,n)\arrow[r,dashed,"\pi_0"]&
W_0\arrow[r,dashed,"\pi_1"]&
\cdots \arrow[r,dashed,"\pi_{r-1}"]&
W_{r-1}\arrow[r,dashed,"\pi_{r}"]&W_r
\end{tikzcd}
\end{center}

The next step consists in factorizing the map $\xi$ as described in the following picture
\begin{center}
\begin{tikzcd}
\G(r,n)_{R_0}\arrow[d,"\alpha_0"]
\arrow[dr]\arrow[r,dashed,"\eta_1"]&
\G(r,n)_{R_0}^{(1)}\arrow[d]
\arrow[dr]\arrow[r,dashed,"\eta_2"]&
\cdots\arrow[r,dashed,"\eta_r"]&
\G(r,n)_{R_0}^{(r)}\arrow[d]
\arrow[dr]&\\
\G(r,n)\arrow[r,dashed,"\pi_0"]&
W_0\arrow[r,dashed,"\pi_1"]&
\cdots \arrow[r,dashed,"\pi_{r-1}"]&
W_{r-1}\arrow[r,dashed,"\pi_{r}"]&W_r
\end{tikzcd}
\end{center}
where the $\eta_i$ are flips.
Then we would like to conclude that $$\Nef(\G(r,n)_{R_0}^{(i)})\!=\!\cone(H-iE,H-(i+1)E).$$
Filling the gaps of this argument we expect to be able, using Corollary \ref{nef cone border3},  to prove the following.

\begin{Conjecture}
Let $\G(r,n)_1$ be the blow-up of the Grassmannian at one point. Then 
$$\Mov(\G(r,n))=\begin{cases}
\cone(H,H-rE) &\mbox{ if } n=2r+1\\
\cone(H,H-(r+1)E) &\mbox{ if } n>2r+1
\end{cases}$$
Moreover, $E,H,H-E,\dots,H-(r+1)E$ are the walls of the Mori chamber decomposition of $\Eff(\G(r,n)_1).$
\end{Conjecture}

\section{For which \texorpdfstring{$(r,n,k)$}{TEXT} is \texorpdfstring{$\G(r,n)_k$}{TEXT} a Mori dream space?}\label{MDS G(r,n)_k}

In the overview we discussed the following result due to Castravet, Mukai and Tevelev.

\begin{Theorem}
Let $X^n_k$ be the blow-up of $\P^n$ at $k$ points in general position, with $n\geq 2$ and $k\geq 0$.
Then  $X^n_k$ is a Mori dream space if and only if one of the following holds:
\begin{itemize}
	\item[-] $n=2$ and $k\leq 8$,
	\item[-] $n=3$ and $k\leq 7$,
	\item[-] $n=4$ and $k\leq 8$,
	\item[-] $n>4$ and $k\leq n+3$.
\end{itemize}
\end{Theorem}

This was the starting point of our interest in determining how many points in general position we can blow-up in the Grassmannian and still have a Mori dream space.

Given $r,n$ it can be shown that there is a maximum value of $k$ such that $\G(r,n)_k$ is a Mori dream space. Therefore, there exists a unique integer function $f(r,n)$ such that 
$\G(r,n)_k$ is a Mori dream space if and only if $k\leq f(r,n).$

Such a function $f(r,n)$ should recover the above result for the blow-up of the projective space, that is, $f(0,2)=8,f(0,3)=7,f(0,4)=8,f(0,n)=n+3,n>4.$ We would like to find such function $f.$ What we have so far using Theorem \ref{G(r,n)kspherical} and Proposition \ref{G(r,n)_k weak Fano} is the following.

\begin{Theorem}\label{G(r,n)k MDS weak}
$\G(r,n)_k$ is a Mori dream space if either
$$
\begin{cases}
k\!\!&\!=1;\ \mbox{ or }\\
k\!\!&\!=2 \mbox{ and } r=1 \mbox{ or } n=2r+1,n=2r+2; \mbox{ or }\\
k\!\!&\!=3 \mbox{ and } (r,n)\in \{(1,4),(1,5)\};\mbox{ or }\\
k\!\!&\!=4 \mbox{ and } (r,n)=(1,4)
\end{cases}
$$
\end{Theorem}

This Theorem translates as
$$
\begin{cases}
f(r,n)\geq 1, \mbox{ for } r\geq 1;\\
f(1,n), f(r,2r+1), f(r,2r+1)\geq 2, \mbox{ for } 
r\geq 1; \\
 f(1,5)\geq 3;\mbox{ and } f(1,4)\geq 4.
\end{cases}
$$

Now, a natural step to improve the lower bound on $f,$
that is find new MDS $\G(r,n)_k,$ is to analyze whether $\G(r,n)_k$ is log Fano. However, it could very well happen that for a given pair $(r,n)$ we have that $\G(r,n)_1,\dots,\G(r,n)_{k_1}$ are log Fano, and
$\G(r,n)_{k_1+1},\dots,G(r,n)_{f(r,n)}$ are Mori dream spaces but not log Fano. Therefore it is also important to find the unique integer function $g(r,n)$ such that $\G(r,n)_k$ is log Fano if and only if $k\leq g(r,n).$  
We would like to recall that the blow-up of the projective space in general points is a Mori dream space if only if it is log Fano, that is, $f(0,n)=g(0,n),$ see \cite{Mu04,AM16}.

In order to produce upper bounds for $f$ one needs a completely different strategy.
One way consists in adapting the work of \cite{Mu04} on $\Bl_{p_1,\dots,p_k}\P^n$ to prove that the effective cone of $\G(r,n)_k$ is not finitely generated and therefore $\G(r,n)_k$ is not a Mori dream space. To do that one needs to produce certain birational transformations in $\G(r,n)$ which will play the same role the Cremona transformations for $\P^n.$ To the best of our knowledge birational transformations in the Grassmannians have not been much studied yet and this could be an interesting subject of study by itself.

We mention here that in \cite{Ko16}  Kopper showed that a certain conjecture would imply that $\Eff(\G(1,4)_6)$ is not finitely generated, and therefore $\G(1,4)_6$ is not a Mori dream space.  Together with our Theorem \ref{G(r,n)k MDS weak} this means that either $f(1,4)=4$ or $f(1,4)=5.$

More generally, given a Mori dream space $X$ one could ask for an integer function $f_X$ such that the blow-up $X_k$ of $X$ at $k$ general points is a Mori dream space if and only if $k\leq f_X.$ At the best of our knowledge this is known only for $X = \P^n.$


\addcontentsline{toc}{chapter}{Bibliography}
\begin{thebibliography}{999999999}


\bibitem[Ab10]{Ab10} \bibaut{H. Abo}, \textit{On non-defectivity of certain Segre-Veronese varieties}, Symbolic Comput, 45, 2010, 1254-1269.

\bibitem[AB09]{AB09} \bibaut{H. Abo, M. C. Brambilla}, \textit{Secant varieties of Segre-Veronese varieties $\P^m \times \P^n$ embedded by $\mathcal{O}(1,2)$}, Experimental Mathematics, 18, 2009, 3, 369-384.

\bibitem[AB12]{AB12} \bibaut{H. Abo, M. C. Brambilla}, \textit{New examples of defective secant varieties of Segre-Veronese varieties}, Collectanea Mathematica. 63, 2012, 3, 287-297.

\bibitem[AB13]{AB13} \bibaut{H. Abo, M. C. Brambilla}, \textit{On the dimensions of secant varieties of Segre-Veronese varieties}, Annali di Matematica Pura ed Applicata, 192, 2013, 61-92.

\bibitem[AOP09a]{AOP09a} \bibaut{H. Abo, G. Ottaviani, C. Peterson}, \textit{Induction for secant varieties of Segre varieties}, Trans. Amer. Math. Soc, 361, 767-792, 2009.

\bibitem[AOP09b]{AOP09b} \bibaut{H. Abo, G. Ottaviani, C. Peterson}, \textit{Non-defectivity of Grassmannians of planes}, J. Algebraic Geom, 21, 2012, 1-20. 

\bibitem[Ab08]{Ab08}\bibaut{S. Abrescia},
\textit{About defectivity of certain Segre-Veronese varieties},
Canad. J. Math. 60 (2008), no. 5, 961-974.

\bibitem[AH95]{AH95} \bibaut{J. Alexander, A. Hirschowitz}, \textit{Polynomial interpolation in several variables}, J. Algebraic Geom. 4, 1995, no. 2, 201-222.

\bibitem[AGMO16]{AGMO16} \bibaut{E. Angelini, F. Galuppi, M. Mella, G. Ottaviani}, \textit{On the number of Waring decompositions for a generic polynomial vector}, 2016, \arXiv{1601.01869v1}.

\bibitem[AM16] {AM16}\bibaut{C. Araujo and A. Massarenti},
\textit{Explicit log Fano structures on blow-ups of projective spaces}, Proc. London Math. Soc. (2016) 113 (4): 445-473.

\bibitem[AMR16]{AMR16} \bibaut{C. Araujo, A. Massarenti, R. Rischter}, \textit{On non-secant defectivity of Segre-Veronese varieties}, 2016, \arXiv{1611.01674}.

\bibitem[ADHL14]{coxrings}\bibaut{I. Arzhantsev, U. Derenthal,   J. Hausen and A. Laface},
\textit{Cox Rings},
Cambridge Studies in Advanced Mathematics,  2014, Cambridge University Press.
  
\bibitem[AGL16]{AGL16}\bibaut{M. Artebani, A. Garbagnati and A. Laface},
\textit{Cox rings of extremal rational elliptic surfaces},
Trans. Amer. Math. Soc. 368 (2016), 1735-1757.

\bibitem[AL11]{AL11}\bibaut{M. Artebani and A. Laface},
\textit{Cox rings of surfaces and the anticanonical Iitaka dimension},  Advances in Math., 226, no. 6, (2011), 5252-5267. 

\bibitem[AHL10]{AHL}\bibaut{ M. Artebani, A. Laface and J. Hausen}
\textit{On Cox rings of K3-surfaces},
Compositio Math. 146 (2010), 964-998. 

\bibitem[Ba05]{Ba05} \bibaut{E. Ballico}, \textit{On the secant varieties to the tangent developable of a Veronese variety}, Journal of Algebra, 288, 2, 2005, 279-286.

\bibitem[BBC12]{BBC12} \bibaut{E. Ballico, A. Bernardi, M. V. Catalisano}, \textit{Higher secant varieties of $\mathbb{P}^n\times\mathbb{P}^1$ embedded in bi-degree $(a,b)$}, Communications in algebra, v. 40, n. 10, 2012, 3822-3840.

\bibitem[BF04]{BF04}\bibaut{E. Ballico, C. Fontanari}, \textit{On the osculatory behaviour of higher dimensional projective varieties}, Collect. Math, 55, 2, 2004, 229-236.

\bibitem[BPT92]{BPT92}\bibaut{E. Ballico, R. Piene, H. Tai}, \textit{A characterization of balanced rational normal surface scrolls in terms of their osculating spaces II}, Math. Scand, 70, 1992, 204-206.

\bibitem[BDdG07]{BDdG07} \bibaut{K. Baur, J. Draisma, W.A. de Graaf}, \textit{Secant dimensions of minimal orbits: computations and conjectures}, Experiment. Math, 16, 239-250, 2007.

\bibitem[BCC11]{BCC11} \bibaut{A. Bernardi, E. Carlini, M. V. Catalisano}, \textit{Higher secant varieties of $\mathbb{P}^n\times\mathbb{P}^m$ embedded in bi-degree $(1,d)$}, J. of pure and applied algebra, 215, n. 12, 2011, 2853-2858.

\bibitem[BGI11]{BGI11} \bibaut{A. Bernardi, A. Gimigliano, M. Id\`{a}}, \textit{Computing symmetric rank for symmetric tensors}, J. of symbolic computation, 46, n. 1, 2011, 34-53.

\bibitem[BL00]{Lak00}\bibaut{S. Billey,  and  V. Lakshmibai}, \textit{Singular Loci of Schubert Varieties}, Progress in Mathematics, Birkh{\"a}user Boston, 2000.

\bibitem[BCHM10]{BCHM}\bibaut{C. Birkar, P. Cascini, C. Hacon, J. McKernan},
\textit{Existence of minimal models
for varieties of log general type}, Journal of the American Mathematical Society, vol. 23, no. 2, 2010, 405-468.

\bibitem[Bl56]{Bl} \bibaut{A. Blanchard}, \textit{Sur les vari\'et\'es analytiques complexes}, Ann. Sci. \'Ecole Norm. Sup, 73, 1956, 157-202.

\bibitem[Bo19]{Bom19} \bibaut{E. Bompiani}, \textit{Determinazione delle superficie integrali di un sistema di equazioni parziali lineari e omogenee}, Rend. Ist. Lomb, 52, 1919, 610-636.

\bibitem[BKR16]{BKR16}\bibaut{J. Boehm, S. Keicher, Y. Ren},
\textit{Computing GIT-fans with symmetry and the Mori chamber decomposition of $\overline{M_{0,6}}$}, 2016, \arXiv{1603.09241}.

\bibitem[Bo13]{Bo13} \bibaut{A. Boralevi}, \textit{A note on secants of Grassmannians}, Rendiconti dell'Istituto Matematico dell'Universit\`{a} di Trieste, 45, 2013.

\bibitem[Bo91]{Borel}\bibaut{A. Borel},
\textit{Linear Algebraic Groups},
Graduate Texts in Mathematics, 1991, Springer New York.

\bibitem[BDPP12]{BDPP} \bibaut{S. Boucksom, J. P. Demailly, M. P\v{a}un and T. Peternell}, \textit{Journal of Algebraic Geometry} vol 22 (2013), 201-248.
Published electronically: May 31, 2012.
	
\bibitem[Br93]{Brion93}\bibaut{M. Brion},
\textit{Vari\'et\'es sph\'eriques et th\'eorie de Mori},
Duke Math. J. 72 (1993), no. 2, 369--404.

\bibitem[Br11]{Br11} \bibaut{M. Brion}, \textit{On automorphism groups of fiber bundles}, Publ. Math. Urug. 12, 2011, 39-66.

\bibitem[Ca12]{Ca12} \bibaut{C. Casagrande}
\textit{Mori dream spaces and Fano varieties}, minicourse (3 talks) at the conference Géométrie Algébrique et Géométrie Complexe, C.I.R.M., Marseille, March 2012.

\bibitem[Ca37]{Ca37} \bibaut{G. Castelnuovo}, \textit{Memorie scelte}, Zanichelli, Bologna, 1937.

\bibitem[Ca09]{Ca09} \bibaut{A. M. Castravet},
\textit{The Cox ring of $\overline{ M_{0,6}}$}, Trans. Amer. Math. Soc. 361 (2009), no. 7, 3851-3878.


\bibitem[CT06]{CT06} \bibaut{A. M. Castravet, J. Tevelev}, \textit{Hilbert's 14th problem and Cox rings}, Compositio Math. 142 (2006), 1479-1498.
 
\bibitem[CT15]{CT}\bibaut{A. M. Castravet, J. Tevelev}, \textit{$\overline M_{0,n}$
is not a Mori dream space,} Duke Math. J. Volume 164, Number 8 (2015), 1641-1667. 

\bibitem[CGG03]{CGG03}\bibaut{M. V. Catalisano, A. V. Geramita, A. Gimigliano}, \textit{Higher secant varieties of Segre-Veronese varieties}, in: Projective Varieties with Unexpected Properties,
Walter de Gruyter, Berlin, 2005, pp. 81-107.

\bibitem[CGG05]{CGG05} \bibaut{M. V. Catalisano, A. V. Geramita, A. Gimigliano}, \textit{ Secant varieties of Grassmann varieties}, Proc. Amer. Math. Soc, 133, 633-642, 2005.

\bibitem[CGG11]{CGG11} \bibaut{M. V. Catalisano, A. V. Geramita, A. Gimigliano}, \textit{Secant Varieties of $\mathbb{P}^1\times\dots\times\mathbb{P}^1$ ($n$-times) are NOT Defective for $n\geq 5$}, J. of Algebraic Geometry, 20, 2011, 295-327.

\bibitem[CC01]{CC01} \bibaut{L. Chiantini, C. Ciliberto}, \textit{Weakly Defective Varieties},
Transactions Of The American Mathematical Society, Volume 354, Number 1, Pages 151-178, 2001.

\bibitem[CC03]{CC03}\bibaut{L. Chiantini, C. Ciliberto}, \textit{On the classification of defective threefolds}, 2003, \arXiv{math/0312518}.

\bibitem[Ch49]{Ch49}\bibaut{W. L. Chow},
\textit{On the geometry of algebraic homogeneous spaces}, Ann. of Math. 50 (1949), 32-67.
 
\bibitem[CM98]{CM98} \bibaut{C. Ciliberto, R. Miranda}, \textit{Degenerations of planar linear systems}, J. Reine Angew. Math, 501, 191-220, 1998.

\bibitem[CM00]{CM00} \bibaut{C. Ciliberto, R. Miranda}, \textit{Linear systems of plane curves with base points of equal multiplicity}, Trans. Amer. Math. Soc, 352, 4037-4050, 2000.

\bibitem[CM05]{CM05} \bibaut{C. Ciliberto, R. Miranda}, \textit{Matching conditions for degenerating
plane curves and applications}, Projective varieties with unexpected properties, 177-197. Walter de Gruyter GmbH \& Co. KG, Berlin, 2005.

\bibitem[CR06]{CR06} \bibaut{C. Ciliberto, F. Russo}, \textit{Varieties of minimal secant degree and linear systems of maximal dimension on surfaces}, Advances in Mathematics 200, 2006, 1-50.

\bibitem[CGLM08]{CGLM} \bibaut{P. Comon, G. Golub, L. Lim, B. Mourrain}, \textit{Symmetric tensors and symmetric tensor rank}, SIAM J. Matrix Anal. Appl, 2008, no. 3, 1254-1279. 

\bibitem[CM96]{CM} \bibaut{P. Comon, B. Mourrain}, \textit{Decomposition of quantics in sums of power of linear forms}, Signal Processing 53(2), 93-107, 1996. Special issue on High-Order Statistics.

\bibitem[Co95]{Cox}\bibaut{D. A. Cox},
\textit{The Homogeneous Coordinate Ring of a Toric Variety}, J. Algebraic Geom-
etry 4 (1995), 17–50.

\bibitem[De01]{De01}\bibaut{O. Debarre},
\textit{Higher-Dimensional Algebraic Geometry}, 2001,
Springer New York.  
  
\bibitem[DeP96]{DP96} \bibaut{P. De Poi}, \textit{On higher secant varieties of rational normal scrolls}, Le Matematiche, Vol. LI, 1996,  Fasc. I, pp. 321.

\bibitem[DLHHK15]{DLHHK15}\bibaut{A. Laface, U. Derenthal, J. Hausen, A. Heim and S. Keicher},
Cox rings of cubic surfaces and Fano threefolds,  Journal of Algebra 436 (2015) 228-276. 

\bibitem[DiRJL15]{DiRJL15} \bibaut{S. Di Rocco, K. Jabbusch, A. Lundman}, \textit{A note on higher order Gauss maps}, 2015, \arXiv{1410.4811v2}.

\bibitem[Do04]{Do04} \bibaut{I. V. Dolgachev}, \textit{Dual homogeneous forms and varieties of power sums}, Milan J. Math. 72, 2004, pp. 163-187.

\bibitem[DK93]{DK93} \bibaut{I. V. Dolgachev, V. Kanev}, \textit{Polar covariants of plane cubics and quartics}, Adv. Math. 98, 1993 ,pp. 216-301.

\bibitem[GM16]{GM16} \bibaut{F. Galuppi, M. Mella}, \textit{Identifiability of homogeneous polynomials and Cremona Transformations}, 2016, \arXiv{1606.06895v1}.

\bibitem[GV85]{GV85} \bibaut{I. Gessel, G. Viennot}, \textit{Binomial determinants, paths, and hook length formulae}, Advances in Mathematics, Vol. 58, Issue 3, 1985, 300-321.

\bibitem[GK16]{GK} \bibaut{J. L. Gonz\`alez and K. Karu}, \textit{Some non-finitely generated Cox rings}, Compos. Math. 152 (2016), no. 5, 984–996.

\bibitem[GD64]{GD64} \bibaut{A. Grothendieck, J. Dieudonn\'e}, \textit{\'El\'ements de g\'eom\'etrie alg\'ebrique: IV. \'Etude locale des sch\'emas et des morphismes de sch\'emas, Premi\`ere partie}, Publications Math\'ematiques de l'IH\'ES 20, 1964, 5-259.

\bibitem[Ha92]{Harris92}\bibaut{J. Harris},
\textit{Algebraic Geometry: A First Course},
Graduate Texts in Mathematics, 1992, Spring.

\bibitem[Ha77]{Har} \bibaut{R. Hartshorne },
\textit{Algebraic Geometry}, Encyclopaedia of mathematical sciences, 1977, Springer.


\bibitem[HKL16]{HKL}\bibaut{J. Hausen, S. Keicher, A. Laface}, \textit{On blowing up the weighted projective plane}, 2016, arXiv:1608.04542. 

\bibitem[HKW16]{HKW}\bibaut{J. Hausen, S. Keicher, R. Wolf}, \textit{ 
Computing automorphisms of Mori dream spaces},
Mathematics of Computation, November 2015.

\bibitem[HK00]{HK} \bibaut{Y. Hu, S. Keel}, \textit{Mori dream spaces and GIT}, Michigan Math. J. 48, 2000, 331-348.

\bibitem[Hu15]{Hu15}\bibaut{C. L. Huerta},
\textit{Birational geometry of complete quadrics},
International Mathematics Research Notices (2015), 12563-12589. 

\bibitem[Hw06]{Hw06}\bibaut{J. M. Hwang},
\textit{Rigidity of rational homogeneous spaces},
Proceedings of the International Congress of Mathematicians
Madrid, August 22–30, 2006.

\bibitem[IR08]{IR08} \bibaut{P. Ionescu, F. Russo}, \textit{Varieties with quadratic entry locus II}, Compositio Mathematica, 144, 2008, 949-962.

\bibitem[IP99]{IP99}\bibaut{V.A. Iskovskikh and Yu.G. Prokhorov}, \textit{Fano varieties},
A. N. Parshin; I. R. Shafarevich, Algebraic Geometry, V. Encyclopedia Math. Sci., 47, Springer-Verlag, pp. 1-247.
 
\bibitem[KMM92]{KMM92}\bibaut{J. Koll\'ar, Y. Miyaoka, and S. Mori},
\textit{Rational connectedness and boundedness of Fano manifolds}, J. Differential Geom.,
Volume 36, Number 3 (1992), 765-779.

\bibitem[Ko16]{Ko16}\bibaut{J. Kopper},
\textit{Effective cycles on blow-ups of Grassmannians}, 2016, \arXiv{1612.01906}.

\bibitem[LP13]{LP13} \bibaut{A. Laface, E. Postinghel}, \textit{Secant varieties of Segre-Veronese embeddings of $(\mathbb{P}^1)^r$}, Math. Ann, 2013, 356 , no. 4, 1455-1470.

\bibitem[La12]{La12} \bibaut{J. M. Landsberg}, \textit{Tensors: Geometry and Applications}, Graduate Studies in Mathematics, vol. 128, 2012, ISBN-10: 0-8218-6907-8.

\bibitem[LM04]{LM} \bibaut{J. M. Landsberg, L. Manivel}, \textit{On the Ideals of Secant Varieties of Segre Varieties}, Foundations of Computational Mathematics November, 2004, Vol. 4, Issue 4, pp. 397-422.

\bibitem[LO15]{LO} \bibaut{J. M. Landsberg, G. Ottaviani}, \textit{New lower bounds for the border rank of matrix multiplication}, Theory of Computing, 11, 2015, 285-298.

\bibitem[La04]{Lazarsfeldvol1}\bibaut{R. K. Lazarsfeld }, \textit{Positivity in Algebraic Geometry I.}, 
Ergebnisse der Mathematik und ihrer Grenzgebiete.  A Series of Modern Surveys in Mathematics, 2004, Springer.

\bibitem[Ma16]{Ma16} \bibaut{A. Massarenti}, \textit{Generalized varieties of sums of powers}, Bulletin of the Brazilian Mathematical Society, 2016, \href{http://link.springer.com/article/10.1007/s00574-016-0113-6?wt_mc=internal.event.1.SEM.ArticleAuthorOnlineFirst}{DOI: 10.1007/s00574-016-0113-6}.

\bibitem[MM13]{MM13} \bibaut{A. Massarenti, M. Mella}, \textit{Birational aspects of the geometry of Varieties of Sums of Powers}, Advances in Mathematics, 2013, no. 243, pp. 187-202.

\bibitem[MR13]{MR} \bibaut{A. Massarenti, E. Raviolo}, \textit{On the rank of $n\times n$ matrix multiplication}, Linear Algebra and its Applications, no.438, 2013, 4500-4509.

\bibitem[MR16]{MRi16} \bibaut{A. Massarenti, R. Rischter}, \textit{Non-secant defectivity via osculating projections}, 2016, \arXiv{1610.09332v1}.

\bibitem[Me06]{Me06} \bibaut{M. Mella}, \textit{Singularities of linear systems and the Waring problem}, Trans. Amer. Math. Soc, 358, 2006, 5523-5538.

\bibitem[Me09]{Me09} \bibaut{M. Mella}, \textit{Base loci of linear systems and the Waring problem}, Proc. Amer. Math. Soc, 137, 2009, 91-98.

\bibitem[MMRO13]{MMRO13} \bibaut{E. Mezzetti, R. M. Mir\'o-Roig, G. Ottaviani}, \textit{Laplace Equations and the Weak Lefschetz Property}, Canadian J. of Math, 65, 3, 2013, 634-654.
   
\bibitem[Mu04]{Mu04} \bibaut{S. Mukai}, \textit{Geometric realization of T-shaped root systems and counterexamples to Hilbert's fourteenth problem}, Algebraic transformation groups and algebraic varieties, Encyclopaedia of
Mathematical Sciences, vol. 132 (Springer, Berlin, 2004), 123-130.

\bibitem[Mu05]{Mu05}\bibaut{S. Mukai},
\textit{Finite generation of the Nagata invariant rings in A-D-E cases}, RIMS Preprint 1502, 2005.

\bibitem[Ne09]{Ne09} \bibaut{M. Nesci}, \textit{Collisions of Fat Points}, Ph.D. thesis, Universit\`a degli Studi Roma Tre, 2009.

\bibitem[Ok15]{Oka} \bibaut{S. Okawa}, \textit{On images of Mori dream spaces}, S. Math. Ann., 2016, vol 364, pp 1315–1342. 

\bibitem[Pe14]{Pe14}\bibaut{N. Perrin},
\textit{On the Geometry of Spherical Varieties},
Transformation Groups, 2014, vol. 19, n.1, pp 171-223.

\bibitem[PT90]{PT90} \bibaut{R. Piene, H. Tai}, \textit{A characterization of balanced rational normal scrolls in terms of their osculating spaces}, Enumerative geometry (Sitges, 1987), 215-224, Lecture Notes in Math, 1436, Springer, Berlin, 1990.

\bibitem[RS00]{RS00} \bibaut{K. Ranestad, F. O. Schreyer}, \textit{Varieties of sums of powers}, J. Reine Angew. Math. 525, 2000, pp. 147-181.

\bibitem[Re83]{Re83} \bibaut{M. Reid},
\textit{Decompostition of toric morphisms},
Arithmetic and Geometry, vol.II, in Progr. Math., vol 36, Birk\"auser Boston, 1983, pp. 395-418.

\bibitem[Ru03]{Ru03} \bibaut{F. Russo}, \textit{Tangents and Secants of Algebraic Varieties}, IMPA Monographs in Mathematics, Rio de Janeiro, Brasil, 2003.

\bibitem[Sc08]{Sc08} \bibaut{G. Scorza}, \textit{Determinazione delle variet\`{a} a tre dimensioni di $S_r$, $r \geq 7$, i cui $S_3$ tangenti si intersecano a due a due}, Rend. Circ. Mat. Palermo, 25, 1908, 193-204.

\bibitem[Se07]{Seg07} \bibaut{C. Segre}, \textit{Su una classe di superficie degli iperspazi legate colle equazioni lineari alle derivate parziali di $2^{\circ}$ ordine}, Atti R. Accad. Scienze Torino, 42, 1907, 559-591.

\bibitem[Se01]{Se01} \bibaut{F. Severi}, \textit{Intorno ai punti doppi impropri di una superficie generale dello spazio a quattro dimensioni e ai suoi punti tripli apparenti}, Rend. Circ. Mat. Palermo, 15, 1901, 33-51.

\bibitem[TZ11]{TZ11} \bibaut{I. Takagi, F. Zucconi}, \textit{Scorza quartics of trigonal spin curves and their varieties of power sums}, Math. Ann. 349, no. 3, 2011, pp. 623-645.

\bibitem[Te11]{Te11} \bibaut{A. Terracini}, \textit{Sulle $V_k$ per cui la varieta degli $S_h$ $(h+1)$-seganti ha dimensione minore dell'ordinario}, Rend. Circ. Mat. Palermo, 31, 1911, 392-396.

\bibitem[Te12]{Te12} \bibaut{A. Terracini}, \textit{Sulle $V_k$ che rappresentano piu di $\frac{k(k-1)}{2}$ equazioni di Laplace linearmente indipendenti}, Rend. Circolo Mat. Palermo, 33, 1912, 176-186.

\bibitem[Te21]{Te21} \bibaut{A. Terracini}, \textit{Su due problemi concernenti la determinazione di alcune classi di superficie, considerate
da G. Scorza e F. Palatini}, Atti Soc. Natur. e Matem. Modena 6, (1921-22) 3-16.

\bibitem[To29]{To29} \bibaut{E. Togliatti}, \textit{Alcuni esempi di superficie algebriche degli iperspazi che rappresentano una equazione di Laplace}, Comm. Math. Helv, 1, 1929, 225-272.

\bibitem[To46]{To46} \bibaut{E. Togliatti}, \textit{Alcune osservazioni sulle superficie razionali che rappresentano equazioni di Laplace}, Ann. Mat. Pura Appl, 4, 25, 1946, 325-339.

\bibitem[Za93]{Za} \bibaut{F. L. Zak}, \textit{Tangents and Secants of Algebraic Varieties}, Translations of Mathematical Monographs, 1993, no. 127.
\end{thebibliography}
\end{document}